\documentclass[11pt]{aims}
\usepackage{amsmath, amssymb, mathrsfs}
\usepackage{bbm}
\usepackage{paralist}
\usepackage{graphics} %% add this and next lines if pictures should be in esp format
\usepackage{epsfig} %For pictures: screened artwork should be set up with an 85 or 100 line screen
\usepackage{graphicx} 
\usepackage{epstopdf}%This is to transfer .eps figure to .pdf figure; please compile your paper using PDFLeTex or PDFTeXify.
\usepackage[colorlinks=true]{hyperref}
\hypersetup{urlcolor=blue, citecolor=blue}
\usepackage[toc,page]{appendix}
\usepackage{chngcntr}
\usepackage{hyperref}

\usepackage{titlesec}
\titleformat{\section}{\Large\bfseries}{\thesection.}{4pt}{}
\titleformat{\subsection}{\large\bfseries}{\thesection.\arabic{subsection}.}{4pt}{}
\titleformat{\subsubsection}{\bfseries}{\thesection.\arabic{subsection}.\arabic{subsubsection}.}{4pt}{}
\titleformat*{\paragraph}{\bfseries}
\titleformat*{\subparagraph}{\bfseries}
\usepackage[margin=1in]{geometry}

\def\currentyear{}
\def\currentmonth{}
\def\ppages{}
\def\DOI{}

\newtheorem{theorem}{Theorem}[section]

\newtheorem{lemma}[theorem]{Lemma}
\newtheorem{proposition}[theorem]{Proposition}
\theoremstyle{definition}
\newtheorem{definition}[theorem]{Definition}
\newtheorem{remark}[theorem]{Remark}

\newcommand{\R}{\mathbb{R}}
\newcommand{\N}{\mathbb{N}}

\numberwithin{equation}{section}

\title[Blowup solutions for Yang Mills heat flow]{
	Non-self similar blowup  solutions to the higher dimensional Yang Mills heat flows }

\author[A. Bensouilah, G. K. Duong, and T. E. Ghoul  ]{}
\keywords{Blowup solutions, Finite time blowup, Type II blowup, Geometric heat flows, Singularity.  \\
	\textbf{Acknowledgement:} The work by Giao Ky Duong was supported by the International Center for Research and Postgraduate Training in Mathematics, Institute of Mathematics, Vietnam Academy of Science and Technology, code ICRTM04\_2021.05.	The work by   Tej-Eddine Ghoul  is based upon work supported by Tamkeen under the NYU Abu Dhabi Research Institute grant CG002.}

\thanks{\today}
\begin{document}

	\maketitle
	
	% Enter the first author's name and address:
	\centerline{A.   Bensouilah$^{(1),(2)}$, G.  K. Duong$^{(3),(4)}$, and  T. E. Ghoul$^{(1)}$} 
	\medskip
	{\footnotesize
		\centerline{ $^{(1)}$ NYUAD Research Institute, New York University Abu Dhabi, PO Box 129188, Abu Dhabi, UAE}
		\centerline{ $^{(2)}$ School of Mathematics and Data Science, Emirates Aviation University, PO Box 53044, Dubai, UAE}
		\centerline{ $^{(3)}$ International Centre for Research and Postgraduate Training, and}
		\centerline{ Institute of
			Mathematics, Vietnam Academy of Science and Technology, Hanoi, Vietnam.}
		\centerline{ $^{(4)}$ Institute of Applied Mathematics, University of Economics Ho Chi Minh City, Vietnam.}
	}

	\bigskip
	\begin{center}\thanks{\today}\end{center}

	\begin{abstract} 
		In this paper, we consider the Yang-Mills heat flow on $\R^d \times SO(d)$ with $d \ge 11$. Under a certain symmetry preserved by the flow, the Yang-Mills equation can be reduced to the following nonlinear  equation:
		$$ \partial_t u =\partial_r^2 u  +\frac{d+1}{r} \partial_r u -3(d-2) u^2  - (d-2) r^2 u^3, \text{ and } (r,t) \in \R_+ \times \R_+.  $$
		We are interested in describing the singularity formation of this parabolic equation. More precisely,  we aim to  construct non self-similar blowup solutions  in higher dimensions $d \ge 11$,  and prove that the asymptotic of the solution is of the form 
		$$   u(r,t)  \sim  \frac{1}{\lambda_\ell(t)}  \mathcal{Q} \left( \frac{r}{\sqrt{\lambda_\ell (t)}} \right),  \text{ as } t \to T ,$$
		where $\mathcal{Q}$ is  the  steady state  corresponding to the  boundary conditions $\mathcal{Q}(0)=-1, \mathcal{Q}'(0)=0$ and the blowup speed $\lambda_\ell$ verifies
		$$\lambda_\ell (t) = \left( C(u_0) +o_{t\to T}(1) \right) (T-t)^{\frac{2\ell }{\alpha}} \text{ as }  t \to T,~~ \ell \in \mathbb{N}^*_+, ~~\alpha>1.$$
		In particular,  the case   $\ell = 1$   corresponds to  the stable type II blowup regime, whereas for the cases $ \ell  \ge 2$ corresponds  to a finite co-dimensional stable regime.
		
		Our approach here is not based on energy estimates but on a careful construction of time dependent eigenvectors and eigenvalues combined with maximum principle and semigroup pointwise estimates.
	\end{abstract}
	
	\section{Introduction}
	Recently, geometric heat flows received a lot of attention from both the mathematics and physics communities. Among these geometric flows, the Yang-Mills heat flow is of a great interest. Let us give a brief survey of the physics behind it (more details can be found in \cite{GMZ01} and \cite{GMZ02}). 	The Yang-Mills theory is in some sense a \textit{non-commutative} version of Maxwell's electromagnetism where in the latter, the gauge group is the \textit{abelian} group $U(1)$. In order to describe the weak nuclear force, governing the nuclear decay of some particles, Yang and Mills proposed to substitute for the Maxwell's gauge group $U(1)$ the \textit{non-abelian} gauge group $SU(2)$. Let us describe the mathematical setting of the theory.
	Consider a Riemannian manifold $M$ of dimension $d$, with a structure group $G$ (i.e., a semi-simple Lie group) and denote by $\pi$ the canonical projection. Let $\mathcal{G}$ be the Lie algebra of $G$ and $E$ a principal fibre bundle over $M$. Let $D_A$ be a covariant derivative from $E$ to $Ad(E) \otimes T^*M$. On each  coordinate chart $U_\alpha$, $D_A$ can be represented by  the $\mathcal{G}$-valued 1-form of $\kappa + A_{\alpha} $  where $\kappa$ is some fixed reference connection (e.g. usual exterior derivative), and $A_{(\alpha)}$ a  $\mathcal{G}$-valued 1-form 
	$$ A_{(\alpha)} = \sum_{\mu =1}^d A_{\alpha, \mu} dx^\mu.  $$
	Since the transition functions are smooth, we can set $A_{(\alpha)} = A$. Physically, the vector $A$ represents the electromagnetic potential.
	
	Let the curvature $F_A$ be the tensor $D_A D_A $. By using a local chart $U_\alpha$, one can represent $F_A$  by the $\mathcal{G}$-valued 2-form 
	$$F_A = \sum_{\mu, \nu} F_{\mu,\nu} dx^\mu \wedge  dx^\nu,$$
	where 
	$$ F_{\mu,\nu}   =  \partial_\mu A_\nu - \partial_\nu A_\mu + \left[ A_\mu, A_\nu \right]  .$$
	The second rank covariant tensor $F_{\mu,\nu}$ is the well-known electromagnetic tensor. The Yang-Mills connections are defined as the \textit{critical points} of  Yang-Mills functional $\mathcal{F}_A$ given by
	$$ \mathcal{F}_A:= \int_{M} \left| F_A \right|^2 d vol_M.$$
	The Euler Lagrange equations corresponding to these critical points are
	$$ \sum_{\nu =1}^d D^\nu F_{\mu,\nu} =0 , \forall \mu =1,...,d,  $$
	where $D_\nu = \partial_\nu + \left[ A_\nu, \cdot \right] $.
	
	The Yang-Mills \textit{heat flow} is defined as the gradient flow associated to the above problem where $A$ is the Yang-Mills connection. By using a local chart, the time-dependent connection  locally satisfies 
	%Let us consider the equivariant Yang-Mills connections on $\R^d \times SO(d)$ described as the solution to the following heat flow equation
	\begin{equation}\label{equa-YM-connections}
		\left\{ \begin{array}{rcl}
			& & \partial_t A_\mu (x,t) + \partial^\nu F_{\mu \nu}(x,t) + \left[ A^\nu,  F_{\mu,\nu}\right](x,t) = 0,~~t >0,     \\
			& &  A_\mu(x,0) = A_{\mu,0}(x). 
		\end{array}  
		\right.
	\end{equation}
	Note that equation \eqref{equa-YM-connections} is invariant under the following scaling
	\begin{equation}\label{scaling-A-connections}
		A_{\lambda} (x,t)= \lambda A\left( \lambda x, {\lambda }^2 t\right), \text{ for }  \lambda >0.
	\end{equation}
	
	However, the Yang-Mills \textit{functional} is invariant under scaling symmetry for $d=4$, this is why we refer to this dimension as the energy critical one. For $d\geq 5$, we say that the equation is supercritical.  
	Results on the long time existence and uniqueness  were obtained in \cite{RJRAM92} for $d=2,3$, \cite{KMNNMJ95,SCVPDE94} for $d=4$ for weak solutions (see also \cite{SJRAM96} and \cite{SST-ZAJM98}  for the existence of smooth solutions).  In particular,  in the case $d=4$,  the authors in  \cite{SST-ZAJM98}  conjectured  finite time singularities do not occur on a compact manifold which recently confirmed by \cite{WIM19}. For the energy \textit{ supercritical} problem,  i.e. $d \ge 5$, there is few  results  on the global existence and this due to the  the gauge invariance of the Yang-Mills heat flow.\\
	%Due to, Yang-Mills heat flow is more difficult than other geometric heat flows.  Otherwise,   blowup results have been more studied, for example \cite{DSCM82} and \cite{GMZ02} and many works  mentioned in the below.  
	
	\medskip
	Let us restrict ourselves to  a special  case where $M=\R^d$ and $ E = \R^d \otimes SO(d) $ is the trivial bundle. In this case,  the Yang-Mills connection $A_\mu (\mu \in \{1,..,d\})$ is globally given by its $\mathcal{SO}(d)$-valued coefficient functions $A_\mu( \mu =1,...,d)$. In particular,  the Lie algebra  $\mathcal{SO}(d)$  is simply the space of   \textit{skew-symmetric}  $d \times d$ matrices endowed   with the commutator bracket. Let us denote  the coefficient functions by $ A_\mu = A^{i,j}_\mu$ and make (as in \cite{DSCM82}) the following $\mathcal{SO}(d)$-equivariant ansatz 
	$$A_\mu^{i,j}(x,t) = u(|x|,t) \sigma^{i,j}_\mu(x), \text{ where } \sigma^{i,j}_\mu(x) = \delta_\mu^i x^j - \delta_\mu^j x^i,i,j \in\{1,...,d\}. $$
	We emphasize  here that the covariant derivative of $\sigma$ is zero, so that the ansatz amounts to consider the problem in the Lorentz gauge. The equation reduces to (see \cite{GMZ01})
	\begin{equation}\label{equa-Yang-Mills-}
		\partial_t u =\partial_r^2 u  +\frac{d+1}{r} \partial_r u -3(d-2) u^2  - (d-2) r^2 u^3, \text{ and } (r,t) \in \R_+ \times \R_+.  
	\end{equation}
	The solution to this  equation is invariant under the scaling
	\begin{equation}\label{scaling-law}
		u_{\lambda} (x,t)= \frac{1}{\lambda} u\left( \frac{x}{\sqrt{\lambda}}, \frac{t}{\lambda } \right) 
	\end{equation}
	for $\lambda>0$.  Let us remark that \eqref{equa-Yang-Mills-} is locally well posed in the weighted $L^2$ space $$L^\infty_{1+r^\alpha}(\R_+) =\{  f  \text{ measurable  such that  }   \|(1+ r^\alpha)f(r) \|_{L^\infty(\mathbb{R}_+)} < +\infty, \alpha \ge \frac{2}{3} \}.$$ The local well-posedness can be easily proved  (via a fixed point argument) by extending to a  $\R^{d+2}$-problem. Thus,  with  an arbitrary initial data in $L^\infty_{1+r^\alpha}$ the solution is either global  or  it  develops a singularity in finite time $T$ in this  space. 
	\iffalse
	Due to the fact that
	$$ \|f\|_{L^\infty} \lesssim    \|f\|_{L^\infty_{1+r^2}}. $$
	we are usually  interested in the singularity in  $L^\infty$. In particular, \eqref{equa-Yang-Mills-} is quite similar to  the classical heat equations, then we  follows   \cite{MMcpam04} to
	
	We have the following the blowup alternative:
	\begin{itemize}
		\item the solution to  \eqref{equa-Yang-Mills-} blows up in finite time $T$ and  satisfies 
		$$ \sup_{r \in \R_+}  (T-t) \left| u(r,t) \right| <+\infty, $$
		such a case is referred to as Type I blowup.
		\item Otherwise, the blowup will be called of Type  II.
	\end{itemize}
	\fi
	As a matter of fact, we are interested in the blowup phenomenon and a variety of papers were devoted to the study of singularity formation.  First, in \cite{GMZ02}, the author constructed  self-similar blowup  solutions with   $ 5 \le d  \le 9$.  Besides that,  the authors in  \cite{WCVPDE04}  also gave explicit  examples  (so-called Weinkove solutions) 
	$$ u_W(x,t) = \frac{1}{T-t} W\left(\frac{r}{\sqrt{T-t}} \right), $$
	with 
	$$ W(r) = - \frac{1}{a_1(d) r^2 + a_2(d)}.$$
	Here $a_1(d)= \frac{\sqrt{d-2}}{2\sqrt{2}}, a_2(d) = \frac{1}{2} \left( 6d-16-(d+2)\sqrt{2d-4} \right) $. Recently, the authors in \cite{DSJDG19} have constructed non trivial  solutions in the range $5 \le d \le 9 $ which approach $u_W$ in $L^\infty(\R^+)$ and these  solutions corresponding to similar blowup setting. The stability of Weinkove solutions was also  proved by   \cite{DSJDG19}  and \cite{GSCPDE20}.
	For higher dimension $d \ge 10$,  the authors in \cite{BWIMRN15} excluded the existence of self similar blowup solutions and then  non-self similar solutions are expected.

	\medskip
	We have been successful in constructing non-self similar blowup solutions  (so-called Type II blowup solutions). Our results are stated in  the following.

	\begin{theorem}[Existence of stable blowup solution]\label{theorem-existence-Type-II-blowup}
		Let $ d \ge 11$ be an integer.  Then, there exist initial data  $u_0 \in C^\infty_0 (\R_+, \R)$ such that the  corresponding  solution to \eqref{equa-Yang-Mills-}  blows up in finite time $T(u_0)$. Moreover,  the following decomposition holds true
		\begin{equation}\label{blowup-profile}
			u(r,t)  = \lambda^{-1}(t) Q\left( \frac{r}{\sqrt{ \lambda(t)}} \right) + \tilde u(r,t), t \in [0,T),
		\end{equation}
		where  $ Q$ is the  ground state of \eqref{equa-Yang-Mills-} satisfying  $Q(0)= -1$ and $ Q'(0) =0$;  and  the  error $ \tilde u (r,t)$ satisfies 
		\begin{equation}\label{estima-error-tilde-u}
			\lambda(t) \| \tilde u(.,t)\|_{L^\infty(\R^+)}   \to 0 \text{ as   } t \to T,
		\end{equation}
		and  the blowup speed $\lambda(t)$ exactly behaves  as follows
		\begin{equation}\label{blowup-speed-lambda-t}
			\lambda(t) = C(u_0)( 1 + o(1)) (T-t)^{ \frac{2}{\alpha }  } .     
		\end{equation}
		as  $t \to T $ and $\alpha $ defined in   \eqref{defi-alpha-intro}. In particular, the constructed  blowup behavior is stable.
	\end{theorem}

	\iffalse
	\cite{CGMNARXIV20-a}
	in sense of the following: There exists a neighborhood $\mathcal{V}(u_{0,\ell})   $ in $L^\infty_{1+r^2}$ such that  for all $v_0 \in \mathcal{V}$,  the solution to \eqref{equa-Yang-Mills-} blows up in finite time $T(v_0)$ such that with the identify  \eqref{blowup-profile}, then, \eqref{estima-error-tilde-u} and \eqref{blowup-speed-lambda-t} hold by replacing $T(u_0)$ by $T(v_0)$. 
	\fi
	
	\iffalse 
	\begin{remark}[Stable result]
		We don't give in the paper the proof of the stability. However, it is quite the same in \cite{CGMNARXIV20-b} and we kindly refer the reader to check this fact. More precisely, the stability is accordance with $L^\infty$ norm in the following sense: There exists a neighborhood $\mathcal{O}$ of $u_0$ ( constructed initial data in Theorem \ref{theorem-existence-Type-II-blowup}) in $L^\infty_{1+r}$ such that for all $\tilde u_0 \in \mathcal{O}$
	\end{remark}
	
	\fi

	By a suitable expansion   the  construction technique in  Theorem \ref{theorem-existence-Type-II-blowup}, we can   construct  \textit{unstable} blowup solutions with different blowup speeds. More precisely, the result reads.
	\begin{theorem}[Existence of unstable blowup solutions]\label{theorem-existence-Type-II-blowup-instable}
		Let us consider integer numbers  $\ell \ge 2$ and    $ d \ge 11$. Then, there exist  initial data  $u_{0,\ell} \in C^\infty_0 (\R_+, \R)$ such that the  corresponding solution  $u_\ell$ to \eqref{equa-Yang-Mills-} blows up in finite time $T(u_{0,\ell})$. Moreover,  the  following decomposition holds true
		\begin{equation}\label{blowup-profile}
			u_\ell(r,t)  = \lambda_{\ell}^{-1}(t) Q\left( \frac{r}{\sqrt{ \lambda_{\ell}(t)}} \right) + \tilde u_\ell(r,t),
		\end{equation}
		where  $ Q$ is the  ground state satisfying  $Q(0)= -1$ and $ Q'(0) =0$;  and  the  error $ \tilde u_\ell(r,t)$ satisfies
		\begin{equation}\label{estima-error-tilde-u}
			\lambda_{\ell}(t) \| \tilde u(.,t)\|_{L^\infty(\R^+)}   \to 0 \text{ as in  } t \to T,
		\end{equation}
		and  the blowup speed $\lambda_{\ell}(t)$ exactly behaves  as follows
		\begin{equation}\label{blowup-speed-lambda-t}
			\lambda_{\ell}(t) = C(u_0)( 1 + o(1)) (T-t)^{ \frac{2\ell}{\alpha}  }    \text{ as } t \to T. 
		\end{equation}
	\end{theorem}

	\begin{remark}[Related blowup results for PDE's problem]
		Note that the Yang-Mills heat flow \eqref{equa-Yang-Mills-} has a lot of similarities with the harmonic map heat flow (under corotational symmetry):
		\begin{equation}\label{equ-Harmonicmap}
			\partial_t u=\partial_r^2 u  +\frac{d+1}{r} \partial_r u -(d-1)\frac{\sin(2u)}{2r^2}, \text{ and } (r,t) \in \R_+ \times \R_+.  
		\end{equation}
		The harmonic map heat flow forms also singularity in finite time, and the self-similar nature of the singularity appears only when $3\leq d\leq 6$, and for $d\geq 7$ self-similar blowup solutions don't exist \cite{BWIMRN15}.
		For $3\leq d\leq 6$, the existence of the self-similar solutions is known \cite{MR1694169} and the stability has been proved only in the case $d=3$ as in  \cite{MR3719557}. When $d=7$ the blowup is not self-similar  and the speed $\lambda$ has a log correction \cite{Ghould7}, it turns out that the non-self-similar regime is stable when $d=7$. If $d\geq 8$, in \cite{GINAPDE19} the authors proved similar results. The results in \cite{GINAPDE19}, also in \cite{BSIMRN19}, have been proved  with a different method.
		In \cite{GINAPDE19}, the result is based on  an energy based method, whereas in \cite{BSIMRN19} is based on the maximum principle which does not allow to abtain the stability. In  the present paper,  we present a new method that has been introduced previously in \cite{CGMNARXIV20-b,MR4073868} but combined with ideas from \cite{BSIMRN19}. We will explain those novelties in the remark below.
	\end{remark}
	
	\begin{remark}[Novelty of the paper]
		We point out that the approach pursued here is more intuitive than the one in \cite{GINAPDE19} for the heat flow map as it is based on a spectral approach rather than an energy method. Note that here, the selection by the flow of the blow up speed is linked to the eigenvalue $\lambda_\ell$ of the time dependent linearized  operator $\mathscr{L}_b$, after perturbing initially $Q$ in the direction of the eigenvectors $\phi_\ell$. Such an idea was not clear in \cite{GINAPDE19}. The length of the paper is due to the heavy and technical construction of the eigenvectors and eigenvalues of $\mathscr{L}_b$. In comparison with \cite{GINAPDE19}, the use of maximum principle reduce considerably the difficulty of the control of the infinite dimensional part $\varepsilon_-$. We believe that this method can be adapted to a large class of  parabolic problems. 
	\end{remark} 
	\begin{remark}[Structure of the paper]
		To be more convenient for the readers, we aim to give the structure of the paper here:   We introduce and explain the importance of the different set of variables: self-similar and blowup variables in the second section. In the third and fourth sections we explain the strategy of the proof, and the time dependent spectral analysis strategy. The fifth section aims to provide a proof of the main theorem without technical details where we show that the infinite dimensional problem can be reduced to a finite dimensional one.  In other words, we show that the solution can be split into two parts a finite dimensional part and an infinite dimensional one.
		In the sixth section we study the dynamic of the finite dimensional part under the assumption that the infinite dimensional part of the solution is decaying in a suitable  weighted $L^2$ norm. The seventh section shows that the assumption made in the section 6 on the infinite dimensional part holds after assuming an $L^\infty$ bound.
		In the 8th section we prove this $L^\infty$ bound assumed in the previous section by using maximum principle and pointwise estimates  which is based  on the semigroup associating to the linearised operator.
		The 9th section is devoted to prove the existence of the ground state $Q$ which solves an non-autonomous second order ODE. To do so,  we prove the existence of an heteroclinic trajectory by finding an appropriate trapping set. In the 10th section we sketch the proof of the existence of the unstable blowup solutions and the last section is devoted to the diagonalisation of the time dependent linearised operator $\mathscr{L}_b$.
	\end{remark}
	
	\section{Mathematical setting}\label{variable-setting}
	%pp1
	Let $u$ be   a  solution to the following equation on $[0,T)$ for some $T>0$
	\begin{equation}\label{equa-u-section-2}
		\partial_t u =\partial_r^2 u  +\frac{d+1}{r} \partial_r u -3(d-2) u^2  - (d-2) r^2 u^3, \text{ and } (r,t) \in \R_+ \times \R_+.  
	\end{equation}
	Let $\lambda$ be an unknown blow-up speed satisfying $\lambda(t)\to 0$ as $t\to T$ and write
	\begin{equation}
		u(t,r)=\frac{1}{\lambda(t)}v(\xi,s)
	\end{equation}
	where the blow-up variables $s$ and $\xi$ are such that 
	\[\frac{ds}{dt}=\frac{1}{\lambda}, \quad \xi=\frac{r}{\sqrt{\lambda}}.\]
	Simple computation yields
	\begin{equation}\label{equa-w}
		\partial_s v  = \partial_{\xi}^2 v + \frac{d+1}{\xi} \partial_\xi v +\frac{1}{2}\frac{\lambda_s}{\lambda} \Lambda_\xi v  - 3(d-2) v^2 - (d-2) \xi^2 v^3.
	\end{equation}
	We anticipate that $\frac{\lambda_s}{\lambda}\to 0$ as $s\to \infty$, since the blow-up mechanism is non-self similar, thus, $v$ is expected to converge to the ground state $Q$, which is a solution to 
	\begin{equation}
		Q''(\xi) + \frac{d+1}{\xi} Q'_\xi- 3(d-2) Q^2 - (d-2) \xi^2 Q^3 =0 
	\end{equation}
	with the boundary conditions $Q(0)=-1$ and $Q'(0)=0$.\\
	In order to establish the convergence of $v$ to the stationary solution $Q$, we linearize around the latter and study the operator
	\begin{equation}
		H_\xi+\frac{1}{2}\frac{\lambda_s}{\lambda} \Lambda_\xi,
	\end{equation}
	where 
	\begin{equation}
		\Lambda_\xi=2+\xi \partial_\xi,
	\end{equation}
	and
	\begin{eqnarray}
		H_\xi = \partial_\xi^2 + \frac{d+1}{\xi} \partial_\xi -3(d-2)(2Q(\xi)+\xi^2 Q^2(\xi)).
	\end{eqnarray}
	More precisely, we would like to determine the eigenvectors and eigenvalues of the linearized operator which depend on time. To do so, one has to switch to the so-called self-similar variables, i.e., we write the solution $u$ as
	\begin{equation}\label{defi-w-self-similar}
		u(r,t) = \frac{1}{T-t}   w \left(\frac{r}{\sqrt{T-t}},\tau \right), \quad  \tau=-\log(T-t).
	\end{equation}
	One then finds that $w$ satisfies
	\begin{equation}\label{equa-w}
		\partial_\tau w  = \partial_y^2 w + \frac{d+1}{y} \partial_y w -   \frac{1}{2} \Lambda_y w  - 3(d-2) w^2 - (d-2) y^2 w^3.
	\end{equation}
	Now, introduce a function $b$ of time such that
	\begin{equation}\label{defi-b-lambda-T-t}
		b=\frac{\lambda}{T-t}.
	\end{equation}
	If the blow-up is self-similar, $b$ would be a (non-zero) constant. In our case, the blow-up is foreseen to be non-self-similar and $b$ has then to tend to zero as $t \to T$.\\
	
	Stepping on the fact that our problem is invariant under time translation, we allow the blow-up time to vary. That is, we replace $T-t$ by some function $\mu$ and we prove that it behaves like $T-t$ for $t \to T$. Hence we relax $b=\frac{\lambda}{\mu}$ instead of $\frac{\lambda}{T-t}$. The parameter $b$ is measuring the non-self similarity of the solution.\\
	
	- \textit{Notation:}
	Based on the above, we write
	\begin{equation}\label{similarity-variable}
		u(r,t) = \frac{1}{\mu(t)}   w (y,\tau), \quad  y = \frac{r}{\sqrt{\mu(t)}} \quad \text{ and }  \frac{d \tau}{dt} = \frac{1}{\mu(t)}.
	\end{equation}
	The function $w$ now satisfies
	\begin{equation}\label{equa-w}
		\partial_\tau w  = \partial_y^2 w + \frac{d+1}{y} \partial_y w -   \beta(\tau) \Lambda_y w  - 3(d-2) w^2 - (d-2) y^2 w^3,
	\end{equation}
	where  
	\begin{equation}\label{defi-beta(tau)}
		\beta(\tau) = - \frac{1}{2} \frac{\mu_\tau}{\mu(\tau)},
	\end{equation}
	and 
	\begin{eqnarray}\label{defi-Lambda-f}
		\Lambda_y f = y \partial_y f + 2 f.
	\end{eqnarray}
	Note that in the self-similar scale $\mu$, one needs to linearise around $Q_b$ instead of $Q$, where 
	\begin{equation}\label{Q-b}
		Q_{b(\tau)} (y)= \frac{1}{b(\tau)}  Q \left(   \frac{y}{\sqrt{b(\tau)}}  \right).
	\end{equation} 
	In addition, $w$ is global but blows up in infinite time. Indeed, introduce the error
	\begin{equation}\label{defi-varep-Q-w}
		\varepsilon(y,\tau)   =  w(y,\tau) - Q_{b(\tau)}(y). 
	\end{equation}
	By a simple calculations, it  leads to 
	\begin{equation}\label{equa-varepsilon-appen}
		\partial_\tau  \varepsilon = \mathscr{L}_b ( \varepsilon)   + B(\varepsilon) + \Phi(y),
	\end{equation}
	where 
	\begin{equation}\label{defi-mathscr-L-b-radial}
		\mathscr{L}_b  =  \partial_y^2 + \frac{d+1}{y} \partial_y -\beta(\tau) \Lambda_y     - 3(d-2) \left(2 Q_{b} + Q_{b}^2 |y|^2 \right),
	\end{equation}
	and
	\begin{eqnarray}
		%\Phi(y,\tau) &=& \partial_y^2 Q_b  + \frac{d+1}{y} \partial_y Q_b  - \beta(\tau) \Lambda_y Q_b - 3(d-2) Q_b^2 - (d-2) y^2 Q_b^3 - \partial_\tau Q_b, \nonumber\\
		B(\varepsilon)  &=&   - 3(d-2) (1+ |y|^2 Q_b) \varepsilon^2 - (d-2) |y|^2  \varepsilon^3,
		\label{defi-B-quadratic-appendix}
	\end{eqnarray}
	and 
	\begin{eqnarray}\label{Phi-simple}
		\Phi (\cdot,\tau)
		& = & \frac{1}{2} \Lambda_y Q_{b(\tau)}  \left[\frac{b'(\tau)}{b(\tau)} - 2\beta(\tau) \right].
	\end{eqnarray}
	From the expression of  the operator    $H_\xi$, we have the relation
	\iffalse
	\begin{eqnarray}
		H_\xi = \partial_\xi^2 + \frac{d+1}{\xi} \partial_\xi -3(d-2)(2Q(\xi)+\xi^2 Q^2(\xi)).\label{relation-L-H}.
	\end{eqnarray}
	\fi
	\begin{eqnarray}
		\mathscr{L}_b w(y,\tau)= \frac{1}{b} \left( H_\xi    -  b \beta \Lambda_\xi\right)  v(\xi,\tau).
	\end{eqnarray}
	From Lemma  \ref{lemma-ground-state},  we infer that
	$$  3(d-2) \left[ 2Q_b(y) + y^2 Q_b^2   \right]  \to -\frac{3(d-2)}{y^2} \text{ as } b \to 0 \text{ with } y \ne 0  . $$
	\iffalse
	Following  \cite{CGMNARXIV20-b}, the parameter shall be controlled $\beta(\cdot) \to \beta_\infty$ where 
	$$ \beta_\infty = \frac{1}{2} + o_{\tau_0}(1).  $$
	\fi 
	We next introduce  the limit operator 
	\begin{equation}\label{defi-operator-L-infty}
		\mathscr{L}_\infty^{\beta} =  \partial_y^2 + \frac{d+1}{y}\partial_y - \beta\Lambda_y + \frac{3(d-2)}{y^2},
	\end{equation}
	and we set $\mathscr{L}^\frac{1}{2}_\infty =\mathscr{L}_\infty   $
	\begin{equation}\label{limit-1-2}
		\mathscr{L}_{\infty} =  \partial_y^2 + \frac{d+1}{y}\partial_y - \frac{1}{2}\Lambda_y   +  \frac{3(d-2)}{y^2} .
	\end{equation}
	Let $\rho=y^{d+1}e^{ -\beta \frac{y^2}{2}}$. Then a simple computation yields
	\begin{equation}
		\mathscr{L}_\infty^{\beta}\phi=\frac{1}{\rho} \frac{d}{dy}(\rho \phi')+\frac{3(d-2)}{y^2}\phi-2\beta \phi.
	\end{equation}
	\noindent
	In the present paper, we use  weighted Sobolev   spaces  $L^2_{\rho_\beta}$ and $H^1_{\rho_\beta}$  where  the weight $\rho_\beta$ is defined by
	\begin{equation}\label{defi-rho-y}
		\rho_\beta(y) = \frac{(2\beta)^\frac{d+2}{2}}{(4\pi)^{\frac{d+2}{2}}} y^{d+1} e^{ -(2\beta)\frac{y^2}{4}}.
	\end{equation}
	We also denote $ \rho_\frac{1}{2} = \rho  $.
	
	\noindent
	\medskip
	The space $L^2_{\rho_\beta} $\label{space-L-2-rho} is equipped with the norm 
	$$ \|f\|_{L^2_{\rho_\beta}(\R^+)}^2  = \int_{\mathbb{R_+}}  f^2(y) \rho_\beta(y) dy,  $$
	and $H^1_{\rho_\beta}(\R^+) $ has the norm
	$$\|f\|_{H^1_{\rho_\beta}(\R^+)}^2 =\|f\|_{L^2_{\rho_\beta}(\R^+)}^2 + \|\partial_y f\|_{L^2_{\rho_\beta}(\R^+)}^2.$$
	We also  define some  special  constants in our paper and  we assume the dimension  $d \ge 11$. Let
	\begin{eqnarray}
		\gamma & = & \frac{1}{2 } ( d - \sqrt{d^2 -12d +24}),\label{defi-gamma-intro}\\
		\alpha  &=&   \gamma -2, \label{defi-alpha-intro}
	\end{eqnarray}
	and 
	\begin{eqnarray}
		a_{i,j} &=& \frac{(-1)^{i-j}}{(i-j)!}  4^{i-j} \frac{i!}{j!} \frac{(\frac{d}{2} - \gamma)_i!}{ (\frac{d}{2} -\gamma)_j!} = c_{i,j} C_j\nonumber, \text{ for all } 0 \leq j \leq i,\label{defi-a-i-j-intro}  
	\end{eqnarray}
	with $c_{i,j}$ and $ C_j$ defined as follows
	\begin{eqnarray}
		c_{i,j} & = & \frac{(-1)^{i-j} 4^i  i!  \left( \frac{d}{2} -\gamma\right)_i! }{(i-j)!}\label{definition-c_i-j},\\
		C_j  & = & \frac{1}{ 4^{j} j! \left(\frac{d}{2}-\gamma \right)_j!}\label{definition-C-j-new},
	\end{eqnarray}
	where
	$$\left( \frac{d}{2} -\gamma\right)_i! =\left( \frac{d}{2} -\gamma +1 \right)\left( \frac{d}{2} - \gamma +2 \right)...\left( \frac{d}{2} -\gamma +i\right) \text{ and }  \left( \frac{d}{2} -\gamma\right)_0!  =1. $$
	We also use the notation
	\begin{equation*}
		\langle y \rangle=\sqrt{1+|y|^2}.
	\end{equation*}

	\section{Strategy of the proof}

	We aim  to summarize  in this paragraph   our  strategy for the proof of our results. As mentioned above, our goal is to prove that $v \to Q$ as $s\to \infty$ which is equivalent to  the control
	\begin{equation}\label{pourpose-non-varep-les-Q-b}
		\left\|  \varepsilon(\cdot, \tau) \right\|_{L^\infty} \ll  \|  Q_{b(\tau)}\|_{L^\infty}= b^{-1}(\tau) \text{ as } \tau \to \infty, 
	\end{equation}
	where $Q_{b}$ defined as in \eqref{Q-b}, $b$ determined as in \eqref{defi-b-lambda-T-t}, and $w = Q_b + \varepsilon$ with $w$ defined in \eqref{defi-w-self-similar}. Our problem  mainly focuses on the perturbative problem  \eqref{equa-varepsilon-appen}. In addition, the perturbative  spectral properties of  the linear operator  $\mathscr{L}_b$ is studied in  Proposition \ref{propo-mathscr-L-b} which allows us to expand the error $\varepsilon$  along its eigenmodes $\phi_{i,b,\beta}, i \in \{0,1,...,\ell\}$. More precisely, we   arrive at the following decomposition 
	\begin{equation}\label{decompo-stagegy-of-proof}
		\varepsilon (\tau)    = \sum_{j=1}^{\ell-1} \varepsilon_j(\tau) \phi_{j,b(\tau), \beta(\tau)} +  \varepsilon_\ell(\tau) \left[ \frac{\phi_{\ell,b(\tau), \beta(\tau)}}{c_{\ell,0}} -\phi_{0,b(\tau), \beta(\tau)} \right]  + \varepsilon_-(\tau),
	\end{equation}
	where $c_{\ell,0}$ defined in \eqref{definition-c_i-j}, $\varepsilon_-$ is the orthogonal part of $\varepsilon$ to $\phi_{i,b,\beta}$ for all $i \le \ell$ i.e.
	\begin{equation}\label{defi-varepsilon--b-beta}
		\left\langle \varepsilon_-, \phi_{i, b, \beta} \right\rangle_{L^2_{\rho_\beta}} =0, \forall j =0,...,\ell.
	\end{equation}
	Note that 	the decomposition in \eqref{defi-varepsilon--b-beta} is  crucial to our approach, as first introduced  in \cite{CMRJAMS20} (see also \cite{CGMNARXIV20-b}). On the one hand, this decomposition provides a good approximation to our solution, including   the main perturbative term i.e.
	$$  \varepsilon_\ell(\tau) \left[ \frac{\phi_{\ell,b(\tau), \beta(\tau)}}{c_{\ell,0}} -\phi_{0,b(\tau), \beta(\tau)} \right]  $$
	which offers a  better approximation compared to the   profile  when the solution is far from the singular domain. On the other hand, it plays an important role in driving the law of the blowup speed, $b(\tau)$. In order to ensure the decomposition  \eqref{decompo-stagegy-of-proof} be unique, we couple the problem \eqref{equa-varepsilon-appen} with 
	\begin{equation}\label{orthogonal-condition}
		c_{\ell,0} \|\phi_{\ell,\phi,\beta} \|^{-2}_{L^2_{\rho_\beta}} \langle \varepsilon,\phi_{\ell, b, \beta} \rangle_{L^2_{\rho_\beta}} =- \|\phi_{0,\phi,\beta} \|^{-2}_{L^2_{\rho_\beta}} \langle \varepsilon,\phi_{0, b, \beta} \rangle_{L^2_{\rho_\beta}} i.e. \varepsilon_\ell=-c_{\ell,0}\varepsilon_0,
	\end{equation}
	and the following compatibility condition (for only  the case $\ell =1$)
	\begin{equation}\label{the-compatibility-condition}
		\varepsilon_\ell (\tau) = -\frac{2}{\alpha} m_0   b^{\frac{\alpha}{2}}(\tau).
	\end{equation}
	Finally, the main  issue  is to  control  $(\varepsilon, b, \beta)$ by a suitable asymptotic behaviors. Specifically, we employ  the concept  of shrinking set,  $V_\ell[A, \eta, \tilde \eta]$ as defined in Definition  \ref{shrinking-set}  to handle the problem. It's worth noting that the set  bears resemblance to recent studies on   Type I blowup constructions, such as those found in
	\cite{BKnon94}, \cite{MZdm97}, \cite{NZCPDE2015}, \cite{DNZtunisian-2017}, \cite{DJFA2019}, \cite{DJDE2019}, \cite{DNZMAMS20}. More precisely, we control $\varepsilon_j, j = 0,..,\ell$, $\varepsilon_-$, the  blowup speed  $b$, the parameter $\beta(\cdot)$. Due to the nonlinearity $y^2 \varepsilon^3 $ in equation \eqref{equa-varepsilon-appen}, we need to control $\|\varepsilon_-\|_{L^2_{\beta}}$ to derive \textit{a priori estimates} on $\varepsilon_j$ and $\varepsilon_-$. Besides that, it is not enough to  imply \eqref{pourpose-non-varep-les-Q-b}  from $\varepsilon_j$ and $\varepsilon_-$, since the eigenmodes $\phi_{i, b, \beta},i \in \N$ are not bounded as $y \to +\infty$.  To address this challenge, we also regulate the outer   part $\varepsilon_e $  introduce in \eqref{defi-varepsilon-e}.  Furthermore, we  propose a simpler  way   for  constructing Type II blowup solutions for parabolic problems,  as an alternative to a direct brute force energy method. 
	
	\medskip
	Additionally, we also  point out main ideas of the proofs of Theorems \ref{theorem-existence-Type-II-blowup}  and \ref{theorem-existence-Type-II-blowup-instable}.
	
	\medskip
	\noindent
	\textit{- For $\ell =1$.} This case involves Theorem \ref{theorem-existence-Type-II-blowup}.  It is sufficient to control $(\varepsilon, b, \beta)(\tau) \in V_{1}[A, \eta, \tilde \eta](\tau),$ for all $\tau \ge \tau_0$ for some $\tau_0$ a sufficient large value.  The maim idea is to  construct a suitable initial choice $(\varepsilon, b, \beta)(\tau_0)$ (see more in subsection \ref{subsection-preparing-initial-data}), then we reply in \textit{a priori estimates}  provided in Lemmas \ref{lemma-ODE-finite-mode}, \ref{lemma-L-2-rho-var--}, \ref{lemma-priori-estima-varep--} and \ref{lemma-priori-estimate-outer-part}  to improve the bounds in the $V_1[A,\eta, \tilde \eta]$. Thus, by continuity of the solution in time, we easily conclude that the  maximum time  trapped in the shrinking set is $+\infty$. Finally, using the renormalisation in time given in \eqref{renormalisation-time}, we  conclude the proof of Theorem \ref{theorem-existence-Type-II-blowup}.  We also mention some interesting points in our proof. First, the control of $\varepsilon_-$ which we sufficiently do on  the interval $[0,b^{- \tilde \eta}(\tau)]$. On the one hand, on $[0,b^\frac{\eta}{4}]$ we use  the maximum principle, initialled in \cite{BSIMRN19}, to control it   in avoiding  a heavy control from energy method. On the other hand, on $[b^{\frac{\eta}{4}}, b^{-{\tilde \eta}} ]$ which is far the origin enough. Then, the result follows  pointwise estimates based on the Poisson semigroup, see more in section \ref{secion-poisson-kernel}.  Second, the control of the outer part $\varepsilon_e$, follows pointwise estimates based on the semigroup $\mathcal{K}_\beta(\tau,\tau')$.
	
	\medskip
	\noindent
	\textit{- For $\ell \ge 2$.} This case is related to Theorem \ref{theorem-existence-Type-II-blowup-instable}. Similar to the first one. We control the solution to be  trapped  in the shrinking set $V_{\ell}[A, \eta, \tilde \eta](\tau)$ by \textit{a priori estimates}. However,  this case includes unstable modes that $\varepsilon_{j}$ for all $j \in \{0,...,\ell\}$.   Thus, we reduce  our problem to a finite dimensional one  which is solvent by a classical topological argument.

\section{Spectral analysis}
The aim of this section is to study the linear operator $\mathscr{L}_b^\beta $. In order to do so, we begin with the limit operator $\mathscr{L}_\infty^\beta$.  
%First, let us  start to  the limit operator   by the following Lemma: 
\begin{proposition}[Diagonalisation of $\mathscr{L}_\infty^\beta$, \cite{HVunpublished-92},  \cite{BSIMRN19}, \cite{CMRJAMS20}]\label{proposition-spectral-L-infty} Let $d \geq 11$, $\beta \in \left( \frac{1}{4}, \frac{3}{4}\right)$ and $\mathscr{L}_\infty^\beta$ defined as in \eqref{defi-operator-L-infty}.  Then, $\mathscr{L}_\infty^\beta$  admits a  unique Friedrichs extension, still denoted by  $\mathscr{L}_\infty^\beta$ with domain $\mathcal{D}(\mathscr{L}_\infty^\beta) \subset H^1_{\rho_\beta} $ and $ H^2_{\rho_\beta} \subset \mathcal{D}(\mathscr{L}_\infty^\beta) $, 
is self-adjoint  with compact resolvent. Moreover, the following hold:
\begin{itemize}
	\item[$(i)$] Spectrum property: $\mathscr{L}_\infty^\beta$ consists of  countable many eigenvalues.  More precisely, the eigenvalues and eigenfunctions are given by
	\begin{equation}
		\lambda_{i, \infty, \beta} = 2\beta\left(\frac{\alpha}{2} - i \right), i \in \N,
	\end{equation}
	\begin{eqnarray}
		\phi_{i,\infty, \beta}(y) &=& \mathcal{N}_i\left(\sqrt{2\beta} y\right)^{-\gamma}L_i^{\left(\frac{d}{2}-\gamma\right)}\left(\frac{\beta y^2}{2}\right)=\sum_{j=0}^{i} a_{i,j} (2\beta)^j y^{2j-\gamma} \nonumber\\
		&=& \left\{ \begin{array}{rcl} & &
			a_{i, 0} y^{-\gamma} \left( 1 + O(y^2) \right) \text{ as } y \to 0,\\[0.2cm]
			& & a_{i,i} (2\beta)^{i} y^{2i -\gamma} \left( 1 + O(y^{-2}) \right)   \text{ as } y \to +\infty,
		\end{array}
		\right.\label{phi-i-infty}
	\end{eqnarray}
	where $L_i^{\left(\nu\right)}(z)$ denotes the generalized Laguerre polynomial,  $\mathcal{N}_i$ is a normalization constant and  $\gamma, \alpha, a_{i,j}$ are defined in \eqref{defi-gamma-intro},  \eqref{defi-alpha-intro} and \eqref{defi-a-i-j-intro}, respectively. 
	\item[$(ii)$] Spectral gap estimate: for all $u \in H^1_{\rho_\beta} $ satisfying  $ \langle   \phi_{i,\beta,\infty}, u \rangle_{L^2_{\rho_\beta}} =0,  \forall i\in \{1,...,\ell \}$,  then 
	$$ \langle  \mathscr{L}_\infty^{\beta} u, u \rangle_{L^2_{\rho_\beta}} \leq \lambda_{\ell+1, \infty, \beta} \| u\|^2_{L^2_{\rho_\beta}}.$$
\end{itemize}
\end{proposition}
\iffalse
\begin{proof}
The  proof is quite the same   Proposition 2.1 in \cite{CMRJAMS20} ( see also  \cite{CGMNARXIV20-a} and   \cite{HVunpublished-92} an unpublished one). Since  $\mathscr{L}_\infty^\beta$  has the same structure  to the following
$$ -\partial_y^2 - \frac{m-2}{y} \partial_y + \frac{1}{2 }\left( \frac{2}{p-1} +   y \cdot \partial_y  \right)  - \frac{p c_\infty^{p-1}}{y^2},$$
for some  constants $m, p, c_\infty$ and the operator completely  handled in those works.    For that reason, we kindly refer the readers to    the references for more details for item $(i)$. In addition to that, the proof of the spectral gap directly follows  the fact that $\mathscr{L}^\beta_\infty$ is self-adjoint with compact resolvent, eigenvalues $\lambda_{i,\infty, \beta}, i \in \{0,...,\ell \}$ are single and the others one is strictly smaller $ \lambda_{\ell +1}$.  However,  we need to    put a small  computation  involving the constants  in  the  Lemma, at Section \ref{Appendix-proof-spectrum-L-infty},  due to the changing of constants in $\mathscr{L}^\beta_\infty$.
\end{proof}
\fi
As has been noted above, $\mathscr{L}_\infty^\beta$ is formally  the limit ($ b\to 0$) of $\mathscr{L}_b$ defined in \eqref{defi-mathscr-L-b-radial}, and a priori it is a good approximation of the latter for large values of $\tau$. However, such an approximation is good only for $y$ large enough since $\phi_{i,\infty}$ is singular when $y$ approaches $0$. Hence, to understand well the operator $\mathscr{L}_b$ around zero (i.e., $y$ small), one has to use the blow-up variables $(\xi, s)$ and our operator then reads
\begin{eqnarray}
\mathscr{L}_b = \frac{1}{b} \left(H_\xi-b \beta \Lambda_\xi\right).
\end{eqnarray}
The strategy is is to construct the eigenvalues and eigenvectors of $\mathscr{L}_b$ in two different regions, namely, for $y>y_0$ (outer region) using the self-similar scale and and for $\xi\leq \xi_0$ (inner region) using the variable $\xi$. Once such a construction is achieved, we glue at $y_0$ and in a $C^1$-manner the obtained eigenvalues and eigenvectors. The result is summarized in the following proposition.  
\begin{proposition}[Diagonalisation of $\mathscr{L}_b$]\label{propo-mathscr-L-b} Let 
$d \geq 11, b >0, \beta \in \left(\frac{1}{4}, \frac{3}{4} \right) ,$  $\ell \in \N^*$ and $\mathscr{L}_b$ be defined as in \eqref{defi-mathscr-L-b-radial}. Then,  $\mathscr{L}_b$ admits a unique Friedrichs extension, still   denoted by $\mathscr{L}_b$  with  domain  $\mathcal{D}(\mathscr{L}_b)  \subset  H^1_{\rho_\beta} $ and  $H^1_{\rho_\beta} \subset \mathcal{D}(\mathscr{L}_b)$, is self-adjoint with compact resolvent. Moreover, for all $\ell \in \N^*$ there exists  $b^*(\ell) \ll 1$ such that for all $b \in (0, b^*)$ and $ j \le  \ell $, 
the following hold:
\begin{itemize}
	\item[$(I)$] Spectrum:  the eigenvalues and eigenfunctions are given by
	\begin{eqnarray}
		\lambda_{i,\beta,b} &=& 2\beta \left( \frac{\alpha}{2} -i \right) + \tilde{\lambda}_{i,\beta,b}, \forall i \in \N,\label{lambda_i-nu}\\
		\phi_{i,\beta,b}(y) &=& \sum_{j=0}^i c_{i,j} (2\beta)^{j} (\sqrt{b})^{2j -\gamma} T_j\left(\frac{y}{\sqrt{b}} \right)+\tilde \phi_{i,\beta, b},\label{phi-i-nu} 
	\end{eqnarray}
	where 
	$$  \|\tilde \phi_{i,\beta,b}\|_{H^1_{\rho_\beta}}  \lesssim  b^{1 -\frac{\epsilon}{2}}, \quad \rho_\beta  \text{ defined in } \eqref{defi-rho-y}.    $$
	In particular, we have
	\begin{equation}\label{estimat-tilde-lambda-b-beta-pro}
		|\tilde \lambda_{i,\beta,b}| \lesssim  b^{1 - \frac{\epsilon}{2}} \text{ and } \left|\partial_b \tilde \lambda_{i,\beta,b} \right| \lesssim  b^{ -\frac{\epsilon}{2}}, \text{ and } \left| \partial_\beta \tilde \lambda_{i,\beta,b} \right| \lesssim 1.
	\end{equation}
	\item[$(II)$]  Difference estimate:
	\begin{equation}\label{norm-phi-b-phi-nfty-i}
		\left\|  \phi_{i,\beta,b} - \phi_{i, \beta, \infty} \right\|_{H^1_\rho} \lesssim b^{1 -\frac{\epsilon}{2}},
	\end{equation}
	where $\phi_{i,b, \infty}$  defined as in \eqref{phi-i-infty}
	\item[$(III)$] Pointwise estimate: for $k \in \{0,1\} $ we have
	\begin{equation}\label{estimate-partial-y-k-phi-y-point-wise}
		\left|\partial_y^k {\phi}_{i,\beta,b}(y)\right|  \lesssim   \frac{\langle y \rangle^{2i+2}}{(\sqrt{b}+y)^{\gamma+k}}, 
	\end{equation}
	\begin{equation}\label{estimate-partial-y-k-b-partial-b-phi-y-point-wise}
		\left|\partial_y^k b \partial_b {\phi}_{i,\beta,b}(y)\right| \lesssim  \frac{\langle y \rangle^{2i+2}}{(\sqrt{b}+y)^{\gamma+k}},
	\end{equation}
	and 
	\begin{equation}\label{estima-partial-y-k-partial-b-b-tilde-phi-b}
		\left|\partial_y^k \tilde{\phi}_{i,\beta,b}(y)\right|+\left|\partial_y^k b \partial_b \tilde{\phi}_{i,\beta,b}(y)\right| + \left|\partial_y^k  \partial_\beta \tilde{\phi}_{i,\beta,b}(y)\right| \lesssim \frac{b^{1-\frac{\epsilon}{2}} \langle y \rangle^{2i+2}}{(\sqrt{b}+y)^{\gamma + k}}.
	\end{equation}
	\item[$(iv)$] Spectral gap estimate: assume that $ u \in H^1_{\rho_\beta} (\R^+)$ satisfies 
	$$  \langle u, \phi_{i,\beta,b}\rangle_{L^2_{\rho_\beta}} =0, \forall i \in \{0,1,...,\ell \}, $$
	then, there exists $c(\ell) > 0$ such that 
	\begin{eqnarray}
		\langle  \mathscr{L}_{b} u, u \rangle_{L^2_{\rho_\beta}}  \leq  - \left(\lambda_{\ell   , b, \beta } + c(\ell)\right)  \| u\|^2_{L^2_{\rho_\beta}} . \label{spectral-gap}
	\end{eqnarray}

\end{itemize}

\end{proposition}

\begin{proof}
The spectral analysis is quite the same as in  \cite{CGMNARXIV20-a} and \cite{CMRJAMS20}. We kindly refer the reader to check the details. In addition, we  also give  the matching ODE approach and the pointwise estimates    in Section \ref{proof-of-diagoligioncal-L-b}. 
\end{proof}

\section{Proof in the stable case without technical details} 

In this section, we aim to give the proof of Theorem  \ref{theorem-existence-Type-II-blowup} without technical details.
%More precisely, the conclusion follows  the  construction of  $(\varepsilon,  b, \beta)$   that $\varepsilon$  solves \eqref{equa-varepsilon-appen} and     satisfying \eqref{orthogonal-condition} and  \eqref{the-compatibility-condition}. In particular, we have the following decomposition
%$$ w(y,\tau) \sim Q_{b(\tau)}, \text{ as } \tau  \to +\infty.   $$

\iffalse
from \eqref{pourpose-non-varep-les-Q-b} and the fact that
$$ \| Q_{b(\tau)}(.)\|_{L^\infty} = \frac{1}{b(\tau)},$$
it is sufficient to construct $\varepsilon$ such that 
\begin{equation}\label{varepsilon-o-b-tau-1}
\| \varepsilon(\cdot, \tau) \|_{L^\infty} = o(b^{-1}), \text{ as } \tau \to +\infty.   
\end{equation}
The idea is to construct a universal shrinking set.

it needs to verify the following
$$ \|    b(\tau) \varepsilon(y,\tau)     \|_{L^\infty([0,+\infty))} \to 0, \text{ and } b(\tau) \to 0,  \text{ as } \tau \to +\infty.  $$
In order to get  more convenient, we summarize in the below the  organization of the section:
\begin{itemize}
\item  In  A shrinking set: The set will be created to control the solution such  to which the solution belongs for all $\tau \ge \tau_0$, for some $\tau_0 $ large enough.
\item[$(ii)$] Initial data: We will then construct a class of initial data depending on $l$ parameters to reduce the problem to a finite dimensional one.
\item[$(iii)$] We  finally solve the finite problem by using Index theory 
\end{itemize}
\fi 
\subsection{Shrinking set}  
We define below  the shrinking set that controls the asymptotic behavior of $(\varepsilon,b,\beta)$ leading to the global existence of the solution, and deriving Theorem \ref{theorem-existence-Type-II-blowup}. Let  $b$ and $\beta$ be positive functions satisfying the hypothesises in Proposition  \ref{propo-mathscr-L-b}, then we  decompose $\varepsilon$ as in \eqref{decompo-stagegy-of-proof} by taking $\ell=1$   
\begin{eqnarray}
\varepsilon(\tau) = \varepsilon_1 (\tau) \left( \frac{\phi_{1,b,\beta}}{c_{1,0}} -\phi_{0,b,\beta} \right) + \varepsilon_-(\tau) = \varepsilon_+(\tau) + \varepsilon_-( \tau). \label{decompose-varepsilon-ell=1}
\end{eqnarray}

\begin{definition}[Shrinking set]\label{shrinking-set}  Let  $A,\eta, \tilde \eta$  and $ \tau_0$ be positive constants. For each $\bar \tau > \tau_0$  we introduce  $V_1[A, \eta, \tilde \eta](\bar \tau)$ 
as the set of all triple time-dependent functions $(\varepsilon, b, \beta)$ on $[\tau_0, \bar \tau]$ such that    $(\varepsilon, b,\beta)(\tau) \in L^\infty(\mathbb{R}_+) \times \mathbb{R}^2 $ for all $\tau' \in [\tau_0, \bar \tau]$  and the following  estimates are satisfied:
\begin{itemize}
	\item[$(i)$]  The  dominating mode $\varepsilon_1 $ satisfies 
	\begin{eqnarray}
		\varepsilon_1(\tau) = - \frac{2}{\alpha} m_0 b^\frac{\alpha}{2}(\tau), \forall \tau \in [\tau_0, \bar \tau]
		,\label{estimate-varepsilon-1-V-A-s}
	\end{eqnarray}
	and functions $b$ and $\beta$  satisfy
	\begin{equation}\label{shrin-king-set-btau}
		\frac{1}{2} \le   b(\tau) \exp\left(\left( \frac{2}{\alpha} -1 \right) \left(\int_{\tau_0}^{\tau} 2\beta(\tilde \tau) d\tilde \tau +\tau_0\right) \right) \le 2, \forall \tau \in [\tau_0,\bar \tau],
	\end{equation}  
	and
	\begin{equation}\label{shrinking-set-beta-tau}
		\left| \beta(\tau)  -  \frac{1}{2} \right|  \le A I^{\eta}(\tau_0), \forall \bar \tau \in [\tau_0,\bar \tau],
	\end{equation}
	where $m_0$ is given by  \eqref{decompose-Phi}, and $I(\tau)$ is defined by 
	\begin{equation}\label{defi-I-tau}
		I(\tau) = e^{\left( 1 - \frac{2}{\alpha} \right)\tau}.
	\end{equation}
	\item[$(iii)$]  The part $\varepsilon_- $ of $\varepsilon $ defined as in \eqref{decompose-varepsilon-ell=1} satisfies 
	\begin{equation}
		\|\varepsilon_-( ., \tau)\|_{L^2_{\rho_{\beta(\tau)}}}  \le A^2 b^{\frac{\alpha}{2} +\eta}(\tau), \forall \tau \in [\tau_0, \bar \tau],
	\end{equation}
	and 
	\begin{equation}
		\left\| y^\gamma \frac{\varepsilon_-(.,\tau)}{ \langle y \rangle^{4}} \right\|_{L^\infty[0, b^{-\tilde \eta}(\tau)] }   \le A^3  b^{\frac{\alpha}{2} + \tilde \eta}(\tau), \forall \tau \in [\tau_0, \bar \tau].
	\end{equation}
	\item[$(iv)$] The part  $ \varepsilon_e$ 
	satisfy 
	\begin{eqnarray}
		\| |y| \varepsilon_e(.,\tau)\|_{L^\infty} & \le &  A^4  b^{ \frac{\alpha}{2} +(\gamma-4)\tilde \eta }(\tau), \forall \tau \in [\tau_0, \bar \tau],
	\end{eqnarray} 
	where 
	\begin{equation}\label{defi-varepsilon-e}
		\varepsilon_e(y,\tau)  = (1 - \chi_0(2yb^{\tilde \eta}(\tau))) \varepsilon(y,\tau), \text{ and } \text{supp}(\varepsilon_e) \subset \left\{ |y| \ge \frac{1}{2} b^{-\tilde \eta} \right\},
	\end{equation}
	and $\chi_0$ defined by 
	\begin{equation}\label{defi-chi-0}
		\chi_0  \in C^\infty, \chi_0 (x) = 1,  \forall x \in [0,1],  \text{ and } \chi_0(x) = 0, \forall x \ge 2.
	\end{equation}
\end{itemize} 
\end{definition}
\iffalse
\begin{remark}  In above definition,  $A$ and $\tau_0$ will be taken  large and  $ 0 < \tilde \eta \ll \eta \ll  1$. In addition, we aim to define constant $C$ for the  universal constant in the paper which doesn't  depend  $A, \tau_0, \tilde \eta, \tau$ and $ \bar \tau$.
\end{remark}

In the following,  we have some  pointwise estimates once the function  trapped in $V_1[A, \eta, \tilde \eta](\bar \tau)$ for some $\bar \tau$:
\fi
Consequently,  once $(\varepsilon, b, \beta)$ belongs to  $V_1[A, \eta, \tilde \eta]$, one can easily deduce the following pointwise estimates.
\begin{lemma}[Pointwise estimates]\label{lemma-rough-estimate-bounds in shrinking-set}
For all $A \ge 1$ and  $ 0< \tilde \eta < \eta \ll 1$, then there exists $\tau_1(A, \tilde \eta, \eta) \ge 1$ such that for all $ \tau_0 \ge \tau_1$ the following holds: Assume   $(\varepsilon, b, \beta)(\tau) \in V_1[A, \eta, \tilde \eta]( \tau)$ for all $\tau \in [\tau_0, \bar \tau]$ with   $\bar \tau > \tau_0$ arbitrarily given,  then we  have 
\begin{eqnarray}
	I^{1 +\frac{\tilde \eta}{10}}(\tau) \le b(\tau) \le I^{1 -\frac{\tilde \eta}{10}}(\tau), \forall \tau \in [\tau_0,\bar \tau].\label{esti-b-equivalent-I-1-tilde-eta}
\end{eqnarray}
Accordingly   \eqref{decompose-varepsilon-ell=1},     $ \varepsilon_+ $ and $ \varepsilon_-$ satisfy  the following pointwise estimates
\begin{equation}
	|\varepsilon_+(y,\tau)| \le  \frac{Cb^\frac{\alpha}{2}}{y^{\gamma}} \left\{ y^2 +  b^{ 10\eta}(\tau) \langle y\rangle^{4}  \right\}  , \forall y >0, \tau \in [\tau_0, \bar \tau],
\end{equation}
and 
\begin{equation}
	\left| \varepsilon_-(y,\tau) \right| \le  CA^4 b^{\frac{\alpha}{2} +\tilde \eta}(\tau ) \frac{\langle y \rangle^{4}}{y^\gamma}, \forall y >0, \tau \in [\tau_0, \bar \tau].
\end{equation}
\end{lemma}
\begin{proof}
The results immediately follow from the bounds given in  Definition  \ref{shrinking-set}  of the shrinking set  $V_1[A, \eta, \tilde \eta ](\tau)$.
%and  pointwise estimates  in Proposition \ref{propo-mathscr-L-b}. We kindly refer the reader to check the details. 
\end{proof}

\subsection{ Preparing  initial data }\label{subsection-preparing-initial-data}

In this part, we  aim to construct a suitable falimy of  initial data  $(\varepsilon, b,\beta)(\tau_0)$   such that the solution to the problem (\ref{equa-varepsilon-appen}-\ref{orthogonal-condition}-\ref{the-compatibility-condition})  globally exists and satisfies
$$ (\varepsilon, b, \beta) \in V_1[A, \eta, \tilde \eta](\bar \tau), \forall \bar \tau > \tau_0 .$$
Let us define $\beta_0 = \beta(\tau_0) = \frac{1}{2}$,   $ b_0=b(\tau_0) = I^\frac{\alpha}{2}(\tau_0)$  where $I(\tau)$ introduced as in \eqref{defi-I-tau},   $ \delta \ll 1$ satisfying $ 0< \tilde \eta \ll \eta \ll \delta  \ll 1$. In addition, we recall  $\chi_0$ defined as in \eqref{defi-chi-0} and   we introduce  then
\begin{equation}\label{defi-initial-vaepsilon-l=1}
\psi(\tau_0)    =    \chi_0\left( y  b_0^{\delta} \right) \left( 1 -  \chi_0\left( \frac{y}{ b_0^\delta}   \right)  \right)
\left(-\frac{2}{\alpha} m_0 b_0^{\frac{\alpha}{2}} \right) \left\{  
[1+ \hat \psi(\tau_0)] \frac{\phi_{1,b_0,\beta_0}}{c_{1,0}}    -  \left[1+\tilde \psi(\tau_0)  \right]  \phi_{0,b_0,\beta_0}   \right\},  
\end{equation}
where the  corrections  $ \tilde \psi(\tau_0)$ and $ \hat \psi(\tau_0)  $  are  uniquely determined such that   (\ref{orthogonal-condition}-\ref{the-compatibility-condition})  are satisfied   at  $\tau =\tau_0$.   More precisely, via a direct computation, they satisfy     
$$  \left| \tilde \psi(\tau_0) \right| + \left| \hat{\psi}(\tau_0) \right|    \lesssim b^\delta(\tau_0). $$
Thus, our initial data is of the form
\begin{equation}\label{initial-data-varep-tau-0}
(\varepsilon,b,\beta)(\tau_0) = (\psi(\tau_0), b_0, \beta_0).
\end{equation}

\medskip
\noindent
In addition, the  initial data  for problem   \eqref{equa-w}   will be  of the form 
\begin{equation}\label{initial-data}
w(y,\tau_0) = Q_{b(\tau_0)}(y) + \varepsilon(\tau_0).
\end{equation}

\bigskip
In the sequel, we prove by using modulation that we can propagate \eqref{orthogonal-condition} and \eqref{the-compatibility-condition}.
\begin{lemma}[Modulation technique]\label{lemma-implicit-function} There exists $\delta_2 \ll 1$ such that  for all $\delta \le \delta_2$ there exists    $A_2 \ge 1$  such that for all $A  \ge A_2$ there exists  $\eta_2(A, \delta) $  such that for all $\eta \le \eta_2$ there exists $\tilde \eta_2(A, \delta, \eta)$ such that for all $\tilde \eta \le \tilde \eta_2$ there exists $\tau_2(A, \delta, \eta, \tilde \eta) \ge 1$ such that the following property holds:   Assume that  initial datum  is of the form in  \eqref{defi-initial-vaepsilon-l=1}, then   there exists $\tau_{loc}^* > \tau_0$ and  smooth functions $(b,\beta)  \in \left(C\left[\tau_0,\tau_{loc}^*\right],\mathbb{R}^2\right) \cap C^{1}\left((\tau_0,\tau_{loc}^*],\mathbb{R}^2\right)) $ such that  the solution $w$ (corresponding to initial data in \eqref{initial-data}) to equation \eqref{equa-w}, locally exists on $[\tau_0,\tau_{loc}^*]$   and uniquely admits the following decomposition
\begin{equation}\label{local-decomposition-w-Q-varpesilon}
	w(\tau) = Q_{b(\tau)} + \varepsilon(\tau), 
\end{equation}
where $(\varepsilon, b,\beta)$  satisfying 
(\ref{equa-varepsilon-appen}-\ref{orthogonal-condition}-\ref{the-compatibility-condition}) and $Q_{b(\tau)}$ defined as in \eqref{Q-b}.  In addition, it holds that  $(\varepsilon, b, \beta) \in V_1[A,\eta, \tilde \eta](\bar \tau), $  for all $ \bar \tau \in [\tau_0,\tau_{loc}^*]$.   In particular, the existence of 
$(\varepsilon, b, \beta)$ can be propagated to the interval $[\tau_0, \tau_{loc}^* + \tilde\sigma]$ for some $\tilde \sigma$ small thank to the bounds in $V_1$.
\end{lemma}
\begin{proof} 
Let us consider  initial data $w(\tau_0)$  defined as in \eqref{initial-data}. Thanks to the local well-posedness in $L^\infty_{1+r^\alpha}, \alpha \ge \frac{2}{3}$ of the problem \eqref{equa-Yang-Mills-}, there exists      $\tilde \tau  > \tau_0$ such that the solution $w$ to  equation \eqref{equa-w}  uniquely exists on $[\tau_0,\tilde \tau]$. We mention that the existence of modulations  $b$ and $ \beta$ and decomposition \eqref{local-decomposition-w-Q-varpesilon}  is a direct consequence of the implicit function theorem. Let us introduce the following  maps 

\begin{equation}\label{function-F}
	\left\{ \begin{array}{rcl}
		F_1(\tau, b,\beta) &:= &    \langle w(\tau) - Q_{b}, \|\phi_{1,b,\beta}\|^{-2}_{L^2_{\rho_\beta}} c_{1,0}\phi_{1, b, \beta} +\|\phi_{0,b,\beta}\|^{-2}_{\rho_\beta} \phi_{0, b, \beta} \rangle_{L^2_{\rho_\beta}},   \\
		F_2(\tau, b, \beta) &:=&  \langle w(\tau) - Q_b, \phi_{1, b, \beta } \rangle_{L^2_{\rho_\beta}} + \frac{2}{\alpha} m_0 \|\phi_{1, b,\beta} \|^{2}_{L^2_{\rho_\beta}} b^\frac{\alpha}{2}(\tau).
	\end{array}
	\right.
\end{equation}
Since $\varepsilon(\tau_0) = \psi(\tau_0)$ defined as in \eqref{defi-initial-vaepsilon-l=1}, it immediately follows   
$$ F(\tau_0, b_0, \beta_0) = (F_1,F_2) (\tau_0, b_0, \beta_0) =0.$$
Now we admit the following expansion (the proof will be given below)
\begin{eqnarray}
	\text{Det}(\mathbf{J})(\tau_0, b_0,\beta_0)
	&=& m_0 \|\phi_{1, b_0, \beta_0}\|_{L^2_{\rho_{\beta_0}}}^2 b_0^{\frac{\alpha}{2}-1} \varepsilon_1 (\tau_0) \left( \frac{\|\phi_{1, b_0, \beta_0}\|_{L^2_{\rho_{\beta_0}}}^2\|\phi_{0, b_0, \beta_0}\|_{L^2_{\rho_{\beta_0}}}^{-2}}{4\beta_0} \right) \label{computation-Det-J-b-0} \\
	&+& o(b_0^{\alpha -1}),\nonumber 
\end{eqnarray}
which implies  
$$ \text{Det}(\mathbf{J})(\tau_0, b_0,\beta_0) \ne 0,   $$
provided that $b_0$ small enough i.e. $\tau_0 \ge \tau_{2,1}(\delta) $.  Finally, we  apply the implicit function theorem to conclude  the unique existence of the functions $(b, \beta) \in C([\tau_0,\tilde \tau_1],\mathbb{R}^2) \cap C^1((\tau_0,\tilde \tau_1],\mathbb{R}^2) $  for some  $\tilde \tau_1 >\tau_0$ such that
$$ F (\tau, b(\tau), \beta(\tau)) =0, \forall \tau \in [\tau_0, \tilde \tau_1].$$
Now, we define  $\tau_{loc}^* = \min (\tilde \tau_1, \tilde \tau)$ and   we define $\varepsilon (\tau) = w (\tau)- Q_{b(\tau)}, \tau \in [\tau_0, \tau_1]$. Thus,    $(\varepsilon, b,\beta)$  reads
(\ref{equa-varepsilon-appen}-\ref{orthogonal-condition}-\ref{the-compatibility-condition})  for all $\tau \in [\tau_0, \tau^*_{loc}]$.  

\medskip
Besides that, from definition \eqref{defi-initial-vaepsilon-l=1} and the continuity of the solution, there exists $ A_2(\delta)$  such that for all  $A  \ge A_2$ there exists  $\eta_2(A, \delta) $  such that for all $\eta \le \eta_2$ there exists $\tilde \eta_2(A, \delta, \eta)$ such that for all $\tilde \eta \le \tilde \eta_2$ there exists $\tau_2(A, \delta, \eta, \tilde \eta) \ge 1$ such that  for all $\tau_0 \ge \tau_2$ and $\tau_{loc}^* >\tau_0$ such that  $ (\varepsilon, b, \beta) \in V_1[A, \eta, \tilde \eta](\bar \tau), \forall \bar \tau \in (\tau_0, \tau_{loc}^*]$.  To finish the proof we now complete the proof of \eqref{computation-Det-J-b-0} provided that 
$\delta \le \delta_2, \eta \le \eta_2(\delta), \tilde \eta \le \tilde \eta_2(\delta, \eta)$ and $\tau_0 \ge \tau_2(\delta, \eta, \tilde \eta)$. 

Let us recall  Jacobian matrix $\mathbf{J}$ defined by 
$$\mathbf{J} (\tau, b,\beta)= \left( \begin{matrix}
	\frac{\partial F_1}{\partial b} & \frac{\partial F_1}{\partial \beta} \\
	\frac{\partial F_2}{\partial b} & \frac{\partial F_2}{\partial \beta}
\end{matrix} \right)(\tau, b,\beta).
$$

We now explicitly write the partial  derivatives:
\begin{eqnarray}
	& &\frac{\partial F_1}{ \partial b} = \int \frac{1}{2b} \Lambda_y Q_b \left( \|\phi_{1,b,\beta}\|^{-2}_{L^2_{\rho_\beta}} c_{\ell,0}\phi_{1, b, \beta} +\|\phi_{0,b,\beta}\|^{-2}_{\rho_\beta} \phi_{0, b, \beta} \right) \rho_\beta dy\label{defi-partial-F-1-b}\\
	& + & \int (w(\tau) -Q_b) \partial_b \left( \|\phi_{1,b,\beta}\|^{-2}_{L^2_{\rho_\beta}} c_{\ell,0}\phi_{1, b, \beta} +\|\phi_{0,b,\beta}\|^{-2}_{\rho_\beta} \phi_{0, b, \beta} \right)  \rho_\beta dy \nonumber,
\end{eqnarray}
and 
\begin{eqnarray}
	& &\frac{d F_1}{ \partial \beta} = \partial_\beta (\|\phi_{1, b, \beta}\|^{-2}_{L^2_{\rho_\beta}} )c_{\ell,0} \int  (w - Q_b) \phi_{1, b, \beta} \rho_\beta dy + \|\phi_{1, b, \beta}\|^{-2}_{L^2_{\rho_\beta}} c_{1,0} \int  (w - Q_b) \partial_\beta \phi_{1, b, \beta} \rho_\beta dy \nonumber \\
	& + & \|\phi_{1, b, \beta}\|^{-2}_{L^2_{\rho_\beta}} c_{1,0} \int  (w - Q_b) \phi_{1, b, \beta} \partial_\beta \rho_\beta dy +  \partial_\beta (\|\phi_{0, b, \beta}\|^{-2}_{L^2_{\rho_\beta}} ) \int  (w - Q_b) \phi_{0, b, \beta} \rho_\beta dy \label{partial-F-1-beta} \\
	&+& \|\phi_{0, b, \beta}\|^{-2}_{L^2_{\rho_\beta}} \int  (w - Q_b) \partial_\beta \phi_{0, b, \beta} \rho_\beta dy
	+  \|\phi_{0, b, \beta}\|^{-2}_{L^2_{\rho_\beta}} \int  (w - Q_b) \phi_{0, b, \beta} \partial_\beta \rho_\beta dy\nonumber,
\end{eqnarray}
and 
\begin{eqnarray}
	\frac{\partial F_2}{ \partial b} &=& \int \frac{1}{2b} \Lambda Q_b \phi_{1,b, \beta} \rho_\beta dy + \int ( w(\tau) -Q_\beta ) \partial_b \phi_{1, b, \beta} \rho_\beta dy + \frac{2}{\alpha} m_0 \partial_b \| \phi_{1, b, \beta} \|^{2}_{L^2_{\rho_\beta}} b^{\frac{\alpha}{2}} \label{partial-F-2-b}\\
	& + &   m_0    \| \phi_{1, b, \beta} \|^{2}_{L^2_{\rho_\beta}} b^{\frac{\alpha}{2} -1}, \nonumber
\end{eqnarray}
and
\begin{eqnarray}
	\frac{\partial F_2}{ \partial \beta}  =  \int ( w(\tau) -Q_b )   \partial_\beta \phi_{1, b, \beta} \rho_\beta dy + \int ( w(\tau) -Q_b )   \phi_{1, b, \beta} \partial_\beta  \rho_\beta dy  + \frac{2}{\alpha} m_0  \partial_\beta \| \phi_{1, b, \beta} \|^{2}_{L^2_{\rho_\beta}} b^\frac{\alpha}{2} .\label{partial-F-2-beta}
\end{eqnarray}
We now  claim the following (which will be proved later)
\begin{eqnarray}
	\frac{\partial F_1}{\partial b}(\tau_0,b(\tau_0),\beta(\tau_0)) &=& m_0b^{\frac{\alpha}{2} -1} + o(b^{\frac{\alpha}{2} -1}), \label{estimate-partial-F-1-b}\\
	\frac{\partial F_1}{ \partial  \beta}(\tau_0, b(\tau_0), \beta(\tau_0)) &=&  \begin{array}{rcl}
		- \frac{1}{\beta}  \varepsilon_1(\tau_0) - \| \phi_{1, b, \beta}\|^{2} \| \phi_{0,b, \beta}\|^{-2} \frac{\varepsilon_1(\tau_0)}{4\beta} + o(b^\frac{\alpha}{2}) 
	\end{array}
	,\label{estimate-partial-F-1-beta}\\
	\frac{\partial F_2}{\partial b} (\tau_0, b(\tau_0), \beta(\tau_0)) &=& m_0 \| \phi_{1, b, \beta}\|^2_{L^2_{\rho_\beta}} b^{\frac{\alpha}{2} -1} + o(b^{\frac{\alpha}{2}-1}),\label{estimate-partial-F-2-b}\\
	\frac{\partial F_2}{\partial  \beta}(\tau_0, b(\tau_0), \beta(\tau_0))   &=& -\frac{1}{\beta} \|\phi_{1,b,\beta}\|^2_{L^2_{\rho_\beta}} \varepsilon_1 + o(b^{\frac{\alpha}{2}}). \label{estimate-partial-F-2-beta}
\end{eqnarray}
Indeed, using  estimates (\eqref{estimate-partial-F-1-b}-\eqref{estimate-partial-F-2-beta}) with $(b(\tau_0),\beta(\tau_0))= (b_0,\beta_0)$,   we  derive 
\begin{eqnarray*}
	\text{Det}(\mathbf{J})(\tau_0, b_0,\beta_0)
	= m_0 \|\phi_{1, b_0, \beta_0}\|_{L^2_{\rho_{\beta_0}}}^2 b^{\frac{\alpha}{2}-1} \varepsilon_1 (\tau_0) \left( \frac{\|\phi_{1, b_0, \beta_0}\|_{L^2_{\rho_{\beta_0}}}^2\|\phi_{0, b_0, \beta_0}\|_{L^2_{\rho_{\beta_0}}}^{-2}}{4\beta_0} \right) +o(b^{\alpha-1}).
\end{eqnarray*}
Thus, we get 
$$ \text{Det}(\mathbf{J})(\tau_0, b(\tau_0),\beta(\tau_0)) \ne 0,   $$
provided that $b_0$ is small enough.  Finally, we  apply the implicit function theorem to get  existence of $(b, \beta) \in C([\tau_0,\tau_1],\mathbb{R}^2) \cap C^1((\tau_0,\tau_1],\mathbb{R}^2) $  for some  $\tau_1>\tau_0$  and  $\varepsilon = w - Q_{b}$ satisfying the decomposition \eqref{defi-varepsilon-j} and   the compatibility \eqref{the-compatibility-condition}  for $\tau \in [\tau_0,\tau_1]$. In addition to that, since   $\varepsilon = w - Q_{b}$, $\varepsilon$ evidently solves \eqref{equa-varepsilon-appen}

%It is expected that for $\ell\geq 2$ $\text{Det}(\mathbf{J}) \neq0$, but this requires to carry out the expansion to the next order of the eigenvalues and eigenvectors.\\

%In particular, follow the implicit function theorem, there exist $\tau_{max} > \tau_0$, and $(b,\beta)(\tau) \in C^1((\tau_0, \tau_{max}])$ such that \eqref{orthogonal-condition-on-b-tau} and \eqref{the-compatibility-condition}  are valid with 
%$$ \varepsilon(\tau) = w(\tau) - Q_{b(\tau)}.$$
%In addition to that, by the continuity of the solution $w$ and $ Q_{b(\tau)}$, it follows that $(\varepsilon,b,\beta)(\tau) \in V_1[A, \eta, \tilde \eta](\tau),$ for all $ \tau \in [\tau_0, \tau_{max}]$. 

\medskip
\noindent
\iffalse
For $\ell \ge 2$, as  we  pointed out before, the compatibility is removed and $\beta \equiv \frac{1}{2}$, then, we remain only with the functional $F_1(\tau, b)$  in \eqref{function-F}. Following \eqref{estimate-partial-F-1-b}, we  immediately  derive
$$ \frac{d F_1}{d b} (\tau_0, b(\tau_0))   \ne 0,    $$
provided that $b(\tau_0)$ is small enough. Hence, the result also follows from the implicit function theorem  again.  In particular, if  $(\varepsilon, b, \beta )(\tau^*) \in V[A, \eta, \tilde \eta ] (\tau^*)   $ for some $ \tau^* > \tau_0$, thanks to the smallness of $b(\tau^*)$ in the shrinking set, we apply  the asymptotic   (\eqref{estimate-partial-F-1-b}-\eqref{estimate-partial-F-2-beta}), then, the implicit  function theorem is valid. Thus, the existence and uniqueness remains holds on $[\tau^*, \tau^* + \nu]$, for some $\nu $ small.  Finally, the result of the Lemma follows provided (\eqref{estimate-partial-F-1-b}-\eqref{estimate-partial-F-2-beta}). 
\fi
Let us now give the details of the computation.\\

- For \eqref{estimate-partial-F-1-b}:
%From \eqref{equality-Lambda-Q-phi_0}, and the fact that 

we have 
\begin{eqnarray*}
	\int \frac{1}{2b} \Lambda_y Q_b \left( \|\phi_{1,b,\beta}\|^{-2}_{L^2_{\rho_\beta}} c_{1,0}\phi_{1, b, \beta} +\|\phi_{0,b,\beta}\|^{-2}_{\rho_\beta} \phi_{0, b, \beta} \right) \rho_\beta dy = m_0 b^{\frac{\alpha}{2}-1} + o(b^{\frac{\alpha}{2}-1}).
\end{eqnarray*}
Next,  we  estimate 
$$ \partial_b \| \phi_{1, b,
	\beta}\|_{L^2_{\rho_\beta}}^{-2}
= - \| \phi_{1, b, \beta}\|_{L^2_{\rho_\beta}}^{-4}  \left(2 \int \phi_{\ell, b, \beta} \partial_b \phi_{\ell, b, \beta} \rho_\beta dy \right).$$
From \eqref{estimate-bpartial-b-phi-j-b}, we get
\begin{eqnarray}
	b \partial_b \| \phi_{1, b, \beta}\|_{L^2_{\rho_\beta}}^{-2} = O(b^{1-\frac{\epsilon}{2}}). \label{b-partial-b-phi_l-beta}
\end{eqnarray}
Since  we choose an initial data  $\varepsilon(
\tau_0) = \psi(\tau_0)$ defined as in  
\eqref{defi-initial-vaepsilon-l=1} and satisfying $w(\tau_0) - Q_{b_0} =\varepsilon(\tau_0) $, we  obtain 
that  
\begin{eqnarray*}
	\left| \varepsilon(\tau_0) \right| = \left| w(\tau_0) - Q_b \right| \le C b^{\frac{\alpha}{2}} \frac{(1+y^{4})}{y^\gamma}.
\end{eqnarray*}
Hence, from \eqref{estimate-bpartial-b-phi-j-b}, we infer that
\begin{eqnarray*}
	\left| \int (w - Q_b) \partial_b \phi_{\ell, b, \beta} \rho_\beta dy \right| = o(b^{\frac{\alpha}{2} -1}). 
\end{eqnarray*}
It follows that, 
\begin{eqnarray*}
	\int (w(\tau_0) - Q_b) \partial_b \left(\|\phi_{1,b,\beta}\|^{-2}_{L^2_{\rho_\beta}} c_{1,0}\phi_{1, b, \beta}  \right)  \rho_\beta dy = o(b^{1 -\frac{\epsilon}{2}}).     
\end{eqnarray*}
Similarly, 
\begin{eqnarray*}
	\int (w(\tau_0) - Q_b) \partial_b \left(\|\phi_{0,b,\beta}\|^{-2}_{L^2_{\rho_\beta}}\phi_{0, b, \beta}  \right)  \rho_\beta dy = o(b^{1 -\frac{\epsilon}{2}}).     
\end{eqnarray*}
Finally, by adding all integrals in \eqref{defi-partial-F-1-b}, \eqref{estimate-partial-F-1-b} follows.

\medskip
- For \eqref{estimate-partial-F-1-beta}: from  \eqref{partial-F-1-beta}, we will establish the following estimates
\begin{eqnarray}
	\partial_\beta \|\phi_{1, b, \beta}\|_{L^2_{\rho_\beta}}^{-2} = \left[  \frac{\left(\frac{d}{2} - \gamma +1 \right)}{\beta}  - \frac{d+2}{2\beta} \right]\|\phi_{1, b, \beta}\|_{L^2_{\rho_\beta}}^{-2} + O(b^{1-\frac{\epsilon}{2}})\label{partial-beta-Phi_l-b-2},\\
	\partial_\beta \|\phi_{0, b, \beta}\|_{L^2_{\rho_\beta}}^{-2} = \left[  \frac{\left(\frac{d}{2} - \gamma +1 \right)}{\beta}  - \frac{d+2}{2\beta} \right]\|\phi_{0, b, \beta}\|_{L^2_{\rho_\beta}}^{-2} + O(b^{1-\frac{\epsilon}{2}}).\label{partial-beta-Phi_0-b-2}
\end{eqnarray}
We remark that these estimates are similar, so we only give the proof of \eqref{partial-beta-Phi_l-b-2}. Indeed,  we write 
\begin{eqnarray*}
	\partial_{\beta} \|\phi_{1, b, \beta}\|_{L^2_{\rho_\beta}}^{-2} = -  \|\phi_{1, b, \beta}\|_{L^2_{\rho_\beta}}^{-4} \left(2 \int  \phi_{1, b, \beta} \partial_\beta \phi_{1, b, \beta} \rho_\beta  dy +\int \phi_{1, b, \beta}  \phi_{1, b, \beta} \partial_\beta \rho_\beta dy \right).  
\end{eqnarray*}
%Now, we calculate  the first integral
%$$  \int_{0}^{\infty} 2  \phi_{\ell, b,\beta}  \partial_\beta \phi_{\ell, b,\beta} \rho_\beta dy . $$
From the construction of $\phi_{\ell, b, \beta}$ in Proposition \ref{propo-mathscr-L-b}, we have
\begin{eqnarray}
	\partial_\beta \phi_{1, b,\beta} &=&  \frac{1}{\beta} c_{1,1} (2\beta) (\sqrt{b})^{2 -\gamma} T_1 \left(\xi \right) + \partial_\beta \tilde \phi_{\ell, b,\beta} \nonumber \\
	& = & \frac{1}{\beta}  \phi_{1, b,\beta} + \sum_{j=0}^{\ell -1}\tilde c_j(\beta) \phi_{j,b,\beta} + \tilde \Phi_{1, b, \beta}, \text{ with } \|\tilde \Phi_1\|_{L^2_{\rho_\beta}} \le C b^{1 -\frac{\epsilon}{2}}.\label{decom-partial-beta-phi-l}  
\end{eqnarray}
Then,
\begin{eqnarray}
	\int_{0}^{\infty} 2  \phi_{1, b,\beta}  \partial_\beta \phi_{1, b,\beta} \rho_\beta dy = 2 \left( \frac{1 }{\beta}  \|\phi_{1, b,\beta}\|_{\rho_\beta}^2 + O(b^{1-\frac{\epsilon}{2}}) \right). \label{inte-2phi-l-parti-beta-phi-l}
\end{eqnarray}
For the second integral, we use the identity
\begin{eqnarray}
	\partial_\beta \rho_\beta = \frac{d+2}{2\beta} \rho_\beta - \frac{y^2}{2} \rho_\beta,\label{equality-rho-beta-partial}
\end{eqnarray}
to derive 
$$\int_{0}^{\infty}   \phi_{1, b, \beta}^2 \partial_\beta \rho_\beta dy   = \frac{d+2}{2\beta} \|\phi_{1, b,\beta}\|_{L^2_{\rho_\beta}}^2  - \int_0^\infty \phi_{1, b, \beta} \phi_{1, b, \beta} \frac{y^2}{2} \rho_\beta dy.   $$
Besides that, by Proposition \ref{propo-mathscr-L-b} we have 
$$ \| \phi_{1, b, \beta} - \phi_{1, \infty, \beta} \|_{L^2_{\rho_\beta}} \le C b^{1 -\frac{\epsilon}{2}},  $$
which yields  
\begin{eqnarray}
	\int_0^\infty \phi_{1, b, \beta} \phi_{1, b, \beta} \frac{y^2}{2} \rho_\beta dy = \int_0^\infty \phi_{1, \infty, \beta} \phi_{1, \infty, \beta} \frac{y^2}{2} \rho_\beta dy + O (b^{1 -\frac{\epsilon}{2}}).\label{inte-phi-2-y-2=infinit}
\end{eqnarray}
In addition,  we use $\phi_{\ell, \infty, \beta}$ as in \eqref{phi-i-infty} to get
\begin{eqnarray*}
	\frac{y^2}{2} \phi_{1, \infty, \beta} &=& \frac{y^2}{2} \left\{  a_{1, 1} (2\beta) y^{2 -\gamma}  + a_{1,0} (2\beta)^{\ell -1} y^{-\gamma} \right\}    \\
	%\sum_{j=0}^{\ell-2} a_{\ell, j} (2\beta)^{j} y^{2j-\gamma} \frac{y^2}{2}
	& = & \frac{1}{4 \beta} (2\beta)^{2} y^{4 -\gamma}  - \frac{1}{\beta}  (\frac{d}{2} -\gamma +1) (2\beta) y^{2 -\gamma} \\
	% \sum_{j=0}^{\ell-2} a_{\ell, j} (2\beta)^{j} y^{2j-\gamma} \frac{y^2}{2}
	& = & \frac{1}{4\beta} \phi_{2, \infty,\beta} - \frac{1}{4\beta} a_{2,1} \phi_{1, \infty, \beta} - \frac{1}{\beta} \left(\frac{d}{2} -\gamma + 1 \right) \phi_{1, \infty,\beta}  \\
	%\sum_{j=0}^{\ell-1} \tilde c_{\ell, j} \phi_{j,\infty,\beta}
	& = & \frac{1}{4\beta} \phi_{2,\infty,\beta}  + \left( \frac{2}{\beta} \left(\frac{d}{2} -\gamma+2 \right) - \frac{1}{\beta}\left(\frac{d}{2} -\gamma+1  \right) \right) \phi_{1, \infty,\beta}.
	%\sum_{j=0}^{\ell-1} \tilde c_{\ell, j} \phi_{j,\infty,\beta}
\end{eqnarray*}
Then
\begin{eqnarray}
	& &\int_0^\infty \phi_{1, \infty, \beta} \phi_{1, \infty, \beta} \frac{y^2}{2} \rho_\beta dy  = \left( \frac{2}{\beta} \left(\frac{d}{2} -\gamma+2 \right) - \frac{1}{\beta}\left(\frac{d}{2} -\gamma+1  \right) \right)  \|\phi_{1,\infty,\beta}\|_{L^2_{\rho_\beta}}^{2} + O(b^{1-\frac{\epsilon}{2}}) \nonumber \\
	&=&  \left( \frac{2}{\beta} \left(\frac{d}{2} -\gamma+2 \right) - \frac{1}{\beta}\left(\frac{d}{2} -\gamma+1  \right) \right)  \|\phi_{1,b,\beta}\|_{L^2_{\rho_\beta}}^{2} +O(b^{1-\frac{\epsilon}{2}}).\label{phi-l-2-infty-y-2}
\end{eqnarray}
This concludes the proof of \eqref{partial-beta-Phi_l-b-2}.

\medskip
Next, we will prove the following
\begin{eqnarray}
	\int_0^\infty \phi_{0,b,\beta} \partial_\beta \phi_{1, b, \beta} \rho_\beta dy = \frac{1}{4\beta} \| \phi_{1, b, \beta}\|^{2}_{L^2_{\rho_\beta}} + O(b^{1-\frac{\epsilon}{2}}).
	\label{int-phi-0-partial0beta-phi-l}
\end{eqnarray}
Indeed, using the orthogonality between $\phi_{0, b, \beta}$ and $ \phi_{1, b, \beta}$ we get
\begin{eqnarray*}
	0 &=& \partial_\beta \int  \phi_{0, b, \beta} \phi_{1, b, \beta} \rho_\beta dy = \int \partial_\beta \phi_{0, b, \beta} \phi_{1, b, \beta} \rho_\beta dy  +  \int \phi_{0, b, \beta}  \partial_\beta \phi_{1, b, \beta} \rho_\beta dy \\
	& + &   \int \phi_{0, b, \beta}   \phi_{1, b, \beta} \partial_\beta \rho_\beta dy \\
	& = & \int \phi_{0, b, \beta}  \partial_\beta \phi_{1, b, \beta} \rho_\beta dy + \int \phi_{0, b, \beta}   \phi_{1, b, \beta} \partial_\beta \rho_\beta dy +O(b^{1-\frac{\epsilon}{2}}).
\end{eqnarray*}
From \eqref{equality-rho-beta-partial}, we obtain
\begin{eqnarray}
	\int \phi_{0, b, \beta}  \partial_\beta \phi_{1, b, \beta} \rho_\beta dy = \int \frac{y^2}{2} \phi_{0, b, \beta}   \phi_{1, b, \beta} \rho_\beta dy + O(b^{1-\frac{\epsilon}{2}}).\label{int-phi-0partial-betaphi-l}
\end{eqnarray}
In addition to that, we have the following identity
\begin{eqnarray}
	\frac{y^2}{2} \phi_{0,\infty, \beta} = \frac{1}{4\beta} \phi_{1,\infty,\beta} + \frac{\frac{d}{2} -\gamma+1}{\beta} \phi_{0, \infty,\beta } ,
\end{eqnarray}
and \eqref{int-phi-0-partial0beta-phi-l} follows. 
\iffalse
Now, we come back to  $ \frac{d F_1}{d\beta}$: We write  
\begin{eqnarray*}
	F_1(\tau, b, \beta) = \|\phi_{\ell, b, \beta}\|^{-2}_{L^2_{\rho_\beta}} c_{\ell,0}\int  (w - Q_b) \phi_{\ell, b, \beta} \rho_\beta dy    + \|\phi_{0, b, \beta}\|^{-2}_{L^2_{\rho_\beta}} \int  (w - Q_b) \phi_{0, b, \beta} \rho_\beta dy,
\end{eqnarray*}
which yields\\
\fi

- For \eqref{estimate-partial-F-2-b}:  from \eqref{partial-F-2-b}, we have 
\begin{eqnarray}
	\int \frac{\Lambda Q_b}{b} \phi_{1, b, \beta} \rho_\beta dy = o(b^{\frac{\alpha}{2}-1} ),
\end{eqnarray}
from \eqref{phi-i-nu} and the orthogonality between  $\phi_{0,b, \beta}$ and $ \phi_{1, b, \beta}$. Moreover,  \eqref{estimate-bpartial-b-phi-j-b} ensures that 
\begin{eqnarray*}
	\int (w -Q_b) \partial_b \phi_{1, b, \beta} \rho_\beta = o(b^{\frac{\alpha}{2}-1}),
\end{eqnarray*}
and \eqref{b-partial-b-phi_l-beta} implies $$ \frac{2}{\alpha} m_0 \| \phi_{1, b, \beta} \|^{2}_{L^2_{\rho_\beta}} b^{\frac{\alpha}{2}} = o(b^{\frac{\alpha}{2}-1}) .$$
Finally, we get
\begin{eqnarray}
	\frac{\partial F_2}{ \partial b} (\tau_0, b(\tau_0), \beta(\tau_0)) = m_0 \| \phi_{1, b, \beta}\|^2_{L^2_{\rho_\beta}} b^{\frac{\alpha}{2} -1} + o(b^{\frac{\alpha}{2}-1}),
\end{eqnarray}
which concludes \eqref{estimate-partial-F-2-b}.

\medskip
- For \eqref{estimate-partial-F-2-beta}
we use
%\left| \varepsilon(\tau_0) \right| =
\begin{eqnarray*}
	\left| w(\tau_0) - Q_b \right| \le C b^{\frac{\alpha}{2}} \frac{(1+y^{4})}{y^\gamma}, 
\end{eqnarray*}
and
\begin{eqnarray*}
	\partial_\beta \phi_{1, b,\beta} & = & \frac{1}{\beta}  \phi_{1, b,\beta} + \tilde \Phi_{1, b, \beta}, \text{ with } \|\tilde \Phi_1\|_{L^2_{\rho_\beta}} \le C b^{1 -\frac{\epsilon}{2}}.
\end{eqnarray*}
%\sum_{j=0}^{\ell -1}\tilde c_j(\beta) \phi_{j,b,\beta} +
Using 
$$  w(\tau_0)  - Q_b = \varepsilon(\tau_0) = \sum_{j=0}^1 \varepsilon_j (\tau_0) \phi_{j, b, \beta} + \varepsilon_-(\tau_0), $$
we infer
\begin{eqnarray*}
	\int ( w(\tau_0) -Q_b )   \partial_\beta \phi_{1, b, \beta} \rho_\beta dy=\varepsilon_\ell(\tau_0) \|\phi_{\ell, b, \beta}\|^2+o(b^{\frac{\alpha}{2}}).
\end{eqnarray*}
\iffalse
then
\begin{equation}
	\left|\int ( w(\tau_0) -Q_b )   \partial_\beta \phi_{1, b, \beta} \rho_\beta dy \right|\leq C b^{\frac{\alpha}{2}}
\end{equation}
\fi
It follows that
\begin{equation*}
	\int ( w(\tau_0) -Q_b )   \phi_{1, b, \beta} \partial_\beta  \rho_\beta dy=\frac{d+2}{2\beta}\int ( w(\tau_0) -Q_b )   \phi_{1, b, \beta} \rho_\beta dy -\frac{1}{2}\int ( w(\tau_0) -Q_b )   \phi_{1, b, \beta} y^2 \rho_\beta dy. 
\end{equation*}
Since
$\varepsilon_j = O(b^{\frac{\alpha}{2} +\eta}(\tau_0)), |\varepsilon_-|_{L^2_\rho} \le Cb^{\frac{\alpha}{2} +\tilde \eta}$, we get for the first integral is 
\begin{equation}
	\int ( w(\tau_0) -Q_b )   \phi_{1, b, \beta} \rho_\beta dy=\varepsilon_1(\tau_0) \|\phi_{1, b, \beta}\|^2.
\end{equation}
For the second integral, we use the expansion of $w(\tau_0) -Q_b$ to obtain
\begin{equation}
	\int ( w(\tau_0) -Q_b )   \phi_{1, b, \beta} \frac{y^2}{2} \rho_\beta dy=\varepsilon_0(\tau_0)\int \phi_{0, b, \beta}\phi_{1, b, \beta} \frac{y^2}{2} \rho_\beta dy+\varepsilon_1(\tau_0)\int \phi_{1, b, \beta}^2 \frac{y^2}{2} \rho_\beta dy+o(b^{\frac{\alpha}{2}}).
\end{equation}
In addition, we have that
\begin{eqnarray*}
	\int_0^\infty \phi_{1, b, \beta} \phi_{1, b, \beta} \frac{y^2}{2} \rho_\beta dy = \int_0^\infty \phi_{1, \infty, \beta} \phi_{1, \infty, \beta} \frac{y^2}{2} \rho_\beta dy + O (b^{1 -\frac{\epsilon}{2}})
\end{eqnarray*}
and
\begin{eqnarray*}
	\int_0^\infty \phi_{1, \infty, \beta} \phi_{1, \infty, \beta} \frac{y^2}{2} \rho_\beta dy &=&  \left( \frac{2}{\beta} \left(\frac{d}{2} -\gamma+2 \right) - \frac{1}{\beta}\left(\frac{d}{2} -\gamma+1  \right) \right)  \|\phi_{1,b,\beta}\|_{L^2_{\rho_\beta}}^{2} +O(b^{1-\frac{\epsilon}{2}}).
\end{eqnarray*}
Hence
\begin{equation}
	\int ( w(\tau_0) -Q_b )   \phi_{1, b, \beta} \frac{y^2}{2} \rho_\beta dy= \left( \frac{2}{\beta} \left(\frac{d}{2} -\gamma+2 \right) - \frac{1}{\beta}\left(\frac{d}{2} -\gamma+1  \right) \right)  \|\phi_{1,b,\beta}\|_{L^2_{\rho_\beta}}^{2} +O(b^{1-\frac{\epsilon}{2}}) .  
\end{equation}
It remains to estimate the last term in 
$\frac{ dF_2}{ d  \beta}$, namely, $\frac{2}{\alpha} m_0  \partial_\beta \| \phi_{1, b, \beta} \|^{2}_{L^2_{\rho_\beta}} b^\frac{\alpha}{2}$. Indeed, we have
\begin{equation*}
	\partial_\beta \| \phi_{1, b, \beta} \|^{2}_{L^2_{\rho_\beta}}=2\int\partial_\beta \phi_{1, b, \beta} \phi_{1, b, \beta} \rho_\beta+\int \phi_{1, b, \beta}^2 \partial_\beta \rho_\beta.  
\end{equation*}
Arguing as above, we get
\begin{equation}
	\partial_\beta \| \phi_{1, b, \beta} \|^{2}_{L^2_{\rho_\beta}}=\left(2\frac{1}{\beta}+\frac{d+2}{2\beta}- \left( \frac{2}{\beta} \left(\frac{d}{2} -\gamma+2 \right) - \frac{1}{\beta}\left(\frac{d}{2} -\gamma+1  \right) \right) \right) \|\phi_{1, b,\beta}\|_{\rho_\beta}^2+O(b^{1-\frac{\epsilon}{2}}) 
\end{equation}
Putting the different contributions of $\frac{ dF_2}{ d  \beta}$ together, we arrive at
\iffalse
\begin{eqnarray*}
	\int ( w(\tau_0) -Q_b )   \partial_\beta \phi_{\ell, b, \beta} \rho_\beta dy=\varepsilon_\ell(\tau_0) \|\phi_{\ell, b, \beta}\|^2+o(b^{\frac{\alpha}{2}})
\end{eqnarray*}
and
\begin{eqnarray*}
	\int ( w(\tau_0) -Q_b )   \phi_{\ell, b, \beta} \partial_\beta  \rho_\beta dy=\left(\frac{d+2}{2\beta}-\left( \frac{\ell+1}{\beta} \left(\frac{d}{2} -\gamma+\ell+1 \right) - \frac{\ell}{\beta}\left(\frac{d}{2} -\gamma+\ell  \right) \right)\right)  \|\phi_{\ell,b,\beta}\|_{L^2_{\rho_\beta}}^{2}\varepsilon_\ell
\end{eqnarray*}
and
\begin{eqnarray*}
	\frac{2}{\alpha} m_0  \partial_\beta \| \phi_{\ell, b, \beta} \|^{2}_{L^2_{\rho_\beta}} b^\frac{\alpha}{2}=\frac{2\gamma}{\beta}\|\phi_{\ell,b,\beta}\|_{L^2_{\rho_\beta}}^{2}\varepsilon_\ell+o(b^{\frac{\alpha}{2}})
\end{eqnarray*}
this yields to
\fi
\begin{equation*}
	\frac{ dF_2}{ d  \beta}=-\frac{1}{\beta}  \|\phi_{1,b,\beta}\|_{L^2_{\rho_\beta}}^{2}\varepsilon_1 + o(b^\frac{\alpha}{2}).
\end{equation*}
as claimed.

\end{proof}

\subsection{The proof of Theorem \ref{theorem-existence-Type-II-blowup}}\label{sub-sec-proof-Theorem-1.1}
In this part, we focus on the proof of Theorem \ref{theorem-existence-Type-II-blowup} which immediately the following result:
%Firstly, we  need to establish  the existence on $ (\varepsilon, b, \beta) (\tau) \in V_1[A,\eta, \tilde \eta](\tau), \forall \tau \ge \tau_0.$
\begin{proposition}\label{proposition-trapped-solu-l=1} There exist $A, \eta, \tilde \tau $ such that we can find $\delta \gg \eta, \tilde \eta$ and  $\tau_0(A,\eta, \tilde \eta, \delta) \large $ small enough such that with initial data $\varepsilon(\tau_0)$  defined as in \eqref{defi-initial-vaepsilon-l=1} and $(b,\beta)(\tau_0) = \left(e^{\left(1-\frac{2}{\alpha} \right) \tau_0}, \frac{1}{2} \right)$, the solution $(\varepsilon, b, \beta)$ exists  for all $\tau \ge \tau_0$ and satisfies 
$$ (\varepsilon, b, \beta)(\tau) \in V_1[A,\eta, \tilde \eta](\tau), \, \forall \tau \ge \tau_0.$$
\end{proposition}
\begin{proof}
%We  observe  that initial data in  \eqref{defi-initial-vaepsilon-l=1} and Lemma \ref{lemma-implicit-function}, we can

Let us define $\tau^* $ by
\begin{equation}\label{defi-tau-*}
	\tau^* = \sup \{ \tau_1 \ge \tau_0 \text{ such that } (\varepsilon, b, \beta)(\tau) \in V_1[A, \eta, \tilde \eta](\tau) , \forall \tau \in [\tau_0,\tau_1]  \}.
\end{equation}
By contradiction we suppose that $\tau^* < +\infty$. Lemmas \ref{lemma-L-2-rho-var--}, \ref{lemma-priori-estima-varep--}  and  \ref{lemma-priori-estimate-outer-part}, the bounds in $V_1[A, \eta,\tilde \eta](\tau^*)$ involving $\|\varepsilon\|_{L^2_{\rho_{\beta(\tau)}}}$, $\varepsilon_-$ and $\varepsilon_e$ are improved by a factor $\frac{1}{2}$. In addition to that the improvement for $b$ and $\beta$ comes from Lemma \ref{lemma-ODE-finite-mode}. Indeed, we have
%mention to \eqref{estimate-beta-derive-tau} that implies
\begin{eqnarray}
	\left| \beta(\tau^*) - \beta(\tau_0)  \right| \lesssim A \int_{\tau_0}^{\tau^*}b^{4\eta}(\tau') d\tau' \le \frac{A}{2} b^\eta(\tau_0),  \label{estimate-beta-tau-*-} 
\end{eqnarray}
provided that $\tau_0$ is large enough.
%  since the fact $ (\varepsilon, b, \beta)(\tau^*) \in V_1[A, \eta, \tilde \eta](\tau^*)$ that ensures \eqref{shrin-king-set-btau} and  \eqref{shrinking-set-beta-tau}. 
For the bound on  $b(\tau)$, we introduce
\begin{equation}\label{defi-Psi}
	\Psi(\tau)  = b(\tau) \exp\left(\left(\frac{2}{\alpha}-1 \right) \left( \int_{\tau_0}^\tau 2\beta(\tau') d\tau' +\tau_0 \right)  \right) \text{ and } \Psi(\tau_0) =1,
\end{equation}
and from  \eqref{ODE-b-tau-proposition},   we get
\begin{eqnarray}
	\left| \Psi(\tau^*) -  1  \right| \lesssim \int_{\tau_0}^{\tau^*} \left| \Psi(\tau') b^{4\eta}(\tau') \right| d\tau' \le \frac{1}{10},\label{estimate-Psi-tau-*-1}
\end{eqnarray}
provided that $ \tau_0 $ large enough. Thus,  the bound of $ b$ in the shrinking set is improved by the factor $\frac{1}{2}$.  Besides that, by continuity of the solution,  there exists $\nu >0$ small such that $(\varepsilon, b, \beta)(\tau) \in V_1[A, \eta, \tilde \eta](\tau), \forall \tau \in [\tau^*, \tau^*+\nu]$ which contradicts to $\tau^*$'s definition. 
%Finally, we conclude $\tau^* =+\infty$  and Proposition \ref{proposition-trapped-solu-l=1} follows. 
\end{proof}
%Now, we are ready to prove Theorem \ref{theorem-existence-Type-II-blowup}:

%\begin{center}
%    \textbf{ Case $\ell=1$}
%\end{center}

Now, we aim to give a proof of  Theorem \ref{theorem-existence-Type-II-blowup}\label{proof-Theorem-1}: Let us consider suitable constants such that Proposition \ref{proposition-trapped-solu-l=1} holds that  
$(\varepsilon, b,\beta) \in V[A,\eta,\tilde \eta](\tau)$ for all $\tau > \tau_0$. Next, we derive the laws of $b $ and $\beta$ as follows:

(i): The law of $b(\tau)$: Let us introduce 
$$ \Psi(\tau)  = b(\tau) \exp{\left(\left( \frac{2}{\alpha} - 1  \right) \left[ \int_{\tau_0}^\tau 2\beta(\tau')  + \tau_0\right] \right)}, \tau \in [\tau_0, +\infty), \text{ with } \Psi(\tau_0) = 1.$$
From Lemma \ref{lemma-ODE-finite-mode},    we have
$$  \left|  \Psi'(\tau) \right| \lesssim |\Psi(\tau)| b^{4\eta}(\tau), \forall \tau \ge \tau_0, $$
since $(\varepsilon, b, \beta) \in V_1[A, \eta, \tilde \eta](\tau)$ for all $\tau > \tau_0$, we get
$$ \left| \Psi(\tau)  \right| \le C, \text{ and } b^{4 \eta }(\tau) \lesssim I^\eta(\tau) \text{ where } I(\tau) \text{ defined in } \eqref{defi-I-tau},$$
which yields
$$ \Psi(\tau) =   \Psi(\tau_0) + \int_{\tau_0}^\infty \Psi'(\zeta) d\zeta - \int_\tau^\infty \Psi'(\zeta) d\zeta  =\Psi_\infty + O(I^\eta(\tau)) \text{ as } \tau \to +\infty,   $$
with  $\Psi_\infty = \Psi(\tau_0) + \int_{\tau_0}^\infty \Psi'(\zeta) d\zeta = 1 +\int_{\tau_0}^\infty \Psi'(\zeta) d\zeta$.
Thus, we get 
\begin{equation}\label{equavalent-on-b-tau}
b(\tau) =    \Psi_\infty  \exp{ \left( \left(  1 -\frac{2}{\alpha}   \right)\left( \int_{\tau_0}^\tau 2\beta(\zeta) d\zeta + \tau_0 \right) \right)} \left[1 + O\left(I^\eta(\tau)\right) \right] \text{ as } \tau \to +\infty.  
\end{equation}

(ii) Renormalized flow $\beta(\tau)$ and derivation the law of  $\mu(t)$ defined in  \eqref{similarity-variable}: Using Lemma \ref{lemma-ODE-finite-mode} again, we have
$$ |\beta'(\tau)| \lesssim b^{4\eta}(\tau),  $$
then we deduce 
$$ \beta(\tau ) = \beta(\tau_0) + \int_{\tau_0}^\infty\beta'(\tau') d\tau' - \int_\tau^\infty \beta '(\tau') d\tau' = \beta_\infty + \int_\tau^\infty \beta'(\tau') d\tau' = \beta_\infty + O(I^{\eta}(\tau)), \text{ as } \tau \to +\infty,     $$
with  
$$ \beta_\infty =  \frac{1}{2} + \int_0^\infty \beta'(\zeta) d\zeta.$$
We introduce the  renormalized time $\tilde \tau$ by 
\begin{equation}\label{renormalisation-time}
\tilde \tau = 2\beta(\tau) \tau,
\end{equation}
which is an invertible function of $\tau$. Indeed, 
\begin{equation}\label{intertible-tau=-tilde-tau}
\tau  =  (2\beta_\infty)^{-1} \tilde \tau (1 + O(I^\eta(\tilde \tau))), \text{ as } \tilde \tau \to +\infty. 
\end{equation}
We shall remark that we will make an abuse of notation
$\mu(\tau)=\mu(\tilde\tau)=\mu(t)$.
%so the reader should bear that in mind to  avoid confusion. 
The relation
$$ \frac{d \mu}{d \tilde \tau} = \frac{d \mu}{ d  \tau } \frac{d \tau}{d \tilde  \tau},     $$
implies, from the fact  that $ \mu_\tau = -2\beta \mu$ and \eqref{renormalisation-time}
$$  \frac{d \mu}{ d \tilde \tau } = - \mu(\tilde \tau) \left[1 + O(I^{\eta}(\tilde \tau)) \right] \text{ as } \tilde \tau \to +\infty.   $$
Thus, we get 
$$ \mu(\tilde \tau) = e^{-\tilde \tau} (1+ O( I^{\eta}(\tilde \tau))), \text{ as } \tilde \tau  \to +\infty.$$
In addition,  we derive from \eqref{similarity-variable} that 
\begin{eqnarray*}
\frac{d \tilde \tau}{ d t } = \frac{d \tilde \tau}{ d \tau} \frac{d \tau }{d t} = 2 \beta_\infty e^{\tilde \tau(t)} (1 + O(I^{ \eta}(\tilde \tau(t)))),
\end{eqnarray*}
which implies 
$$ \tilde \tau (t) = - \ln(2\beta_\infty(T-t)) (1 + O((T-t)^{\tilde \eta})), \text{ as } t \to T, $$
for  some $T=T(\tau_0) >0$. From   \eqref{intertible-tau=-tilde-tau}, we have
\begin{eqnarray*}
\tau(t) = (2\beta_\infty)^{-1}  \tilde \tau (t) (1 + I^{\tilde \eta}(\tilde \tau(t))) =  -(2 \beta_\infty)^{-1} \ln(2\beta_\infty(T-t)) (1 + O((T-t)^{\tilde \eta})), \text{ as } t \to T.
\end{eqnarray*}
Substituting $\mu(t)$'s formula, we get  
$$ \mu(t) = \mu(\tilde \tau(t)) = 2 \beta_\infty (T-t) \left( 1  + O((T-t)^{\tilde \eta}) \right) \text{ as } t \to T.   $$
Recall that  
$$ \int_{\tau_0}^\tau 2 \beta(\tau')d\tau' = 2\beta_\infty \tau (1 + O(\tau^{-1})) \text{ as } \tau \to +\infty,      $$
from which we deduce, with the use of \eqref{equavalent-on-b-tau}, that 
\begin{eqnarray*}
b(t)  &= &  \Psi_\infty  \exp{ \left(\left(   1 -\frac{2}{\alpha}  \right)\left( \int_{\tau_0}^\tau 2\beta(\zeta)d\zeta + \tau_0\right) \right)} (1 + O(I^\eta)(\tau(t)))  \\
& = & \Psi_\infty  \left( 2\beta_\infty \right)^{\frac{2}{\alpha} -1} (T-t)^{\frac{2}{\alpha} -1} \left( 1+O(\left|\ln(T-t) \right|^{-1}) \right) \text{ as} t \to T.
\end{eqnarray*}
Introducing  $\lambda(t) =\mu(t) b(t)$ which satisfies  
\begin{equation}\label{asymptotic-lambda-t}
\lambda(t) = C(u_0) (T-t)^{\frac{2  }{\alpha }} (1 + O(\left|\ln(T-t) \right|^{-1}) ) \text{ as } t \to T.  
\end{equation}
Finally,  the conclusion of the Theorem \ref{theorem-existence-Type-II-blowup} immediately follows \eqref{similarity-variable}, \eqref{defi-varep-Q-w}, the fact $(\varepsilon, b, \beta) \in V_1[A, \eta, \tilde \eta](\tau) $ for all   $\tau > \tau_0$, and \eqref{asymptotic-lambda-t}.  $\square$

\section{Finite dimensional system}

In this part, we study the  dynamics of finite modes $\varepsilon_j(\tau)$ and the modulation parameters $b$ and $\beta$.
\begin{lemma}\label{lemma-ODE-finite-mode}     Consider $A \ge 1, \eta > 0, \tilde \eta >0$, there exists  $ \tau_2(A, \eta, \tilde \eta)$ such that for all $\tau_0 \ge \tau_2$,   the following holds: Assume that $(\varepsilon, b, \beta)(\tau)  \in V_1[A, \eta, \tilde \eta](\tau),  \forall \tau \in [\tau_0,\tau_1]$, for some $\tau_1 > \tau_0$, then,  we have 
\begin{itemize}
	\item[$(i)$] The dominating mode $\varepsilon_1$ satisfies 
	\begin{equation}\label{system-vare-0-1}
		\left\{ \begin{array}{rcl}
			& & \partial_\tau \varepsilon_1  -\Bigg[2\beta\left(  \frac{\alpha}{2} - 1  \right)\Bigg]   \varepsilon_1 =  O\left(  b^{\frac{\alpha}{2} + 4 \eta}(\tau) [ |\beta'| + \left| \frac{b_\tau}{b} \right| +1]\right),\\
			& & \partial_\tau \varepsilon_1 - \Bigg[2\beta\left( \frac{\alpha}{2}    \right)\Bigg]   \varepsilon_1  + m_0 \left( \frac{b_\tau}{b} -2\beta \right)b^\frac{\alpha}{2} =  O\left(  b^{\frac{\alpha}{2} + 4 \eta}(\tau) [ |\beta'| + \left| \frac{b_\tau}{b} \right| +1]\right),
		\end{array}
		\right.
	\end{equation}
	for all $\tau \in [\tau_0,\tau_1]$.
	\item[$(iii)$] For the $b$ and $\beta$, we obtain
	\begin{eqnarray}
		\left|  \beta'(\tau)\right| \le CA b^{4\eta}(\tau), \label{estimate-beta-derive-tau}
	\end{eqnarray}
	and 
	\begin{equation}\label{ODE-b-tau-proposition}
		\left|  \frac{b'(\tau)}{b(\tau)}  - 2\beta  \left(  1  - \frac{2 }{\alpha}\right)  \right|   \le  C A b^{4\eta}(\tau),
	\end{equation}
	for all $ \tau \in \left( \tau_0, \tau_1\right)   $.
\end{itemize}
\end{lemma}
\begin{proof}
Let us consider $(\varepsilon, b, \beta)(\tau) \in V_1[A, \eta, \tilde \eta](\tau), \forall \tau [\tau_0, \tau_1] $ and $\varepsilon (\tau)$  decomposed as in \eqref{decompose-varepsilon-ell=1}, and we also recall that
\begin{equation}\label{defi-varepsilon-j-full}
	\varepsilon_j = \|\phi_{j,b,\beta}\|^{-2}_{L^2_{\rho_{\beta}}} \langle \varepsilon, \phi_{j,b,\beta} \rangle_{L^2_{\rho_\beta}}.
\end{equation}
Then, we  obtain from \eqref{orthogonal-condition} that 
\begin{equation}\label{system-varepsilon-1}
	\left\{ \begin{array}{rcl}
		\varepsilon_1 (\tau) &=& c_{1,0} \|\phi_{\ell, b, \beta}\|^{-2}_{L^2_{\rho_\beta}} \langle \varepsilon, \phi_{1, b, \beta} \rangle_{L^2_{\rho_\beta}},\\
		\varepsilon_1(\tau) & = & - \|\phi_{0, b, \beta}\|^{-2}_{L^2_{\rho_\beta}} \langle \varepsilon, \phi_{0, b, \beta} \rangle_{L^2_{\rho_\beta}}.
	\end{array}
	\right.    
\end{equation}
%Note that we aim to get  simplification  that  $\beta =\beta(\tau)$ and $b = b(\tau)$.
By taking $\tau$-derivative of  the second equation of the above  system, we get 
\begin{eqnarray}
	- \partial_\tau \varepsilon_1  &=& \partial_\tau \| \phi_{0,b,\beta}\|^{-2}_{L^2_{\rho_\beta}} \langle \varepsilon, \phi_{0, b, \beta} \rangle_{L^2_{\rho_\beta}} + \| \phi_{0,b,\beta}\|^{-2}_{L^2_{\rho_\beta}} \langle  \partial_\tau \varepsilon, \phi_{0,b,\beta}\rangle_{L^2_{\rho_\beta}} + \| \phi_{0,b,\beta}\|^{-2}_{L^2_{\rho_\beta}} \langle \varepsilon, \partial_\tau \phi_{0, b, \beta} \rangle_{L^2_{\rho_\beta}} \nonumber\\
	& + & \| \phi_{0,b,\beta}\|^{-2}_{L^2_{\rho_\beta}} \left\langle \varepsilon, \frac{\partial_\tau \rho_\beta}{\rho_\beta}\right\rangle_{L^2_{\rho_\beta}},\label{equality-partial-tau-varphi_j}
\end{eqnarray}
where $\rho_\beta$ defined as in \eqref{defi-rho-y}. A direct computation gives
\begin{eqnarray*}
	&&\partial_\tau \| \phi_{0, b, \beta}\|^{-2}_{L^2_{\rho_\beta}} = \partial_\tau\frac{1}{\int_{\R_+} \phi_{0,b,\beta}^2 \rho_\beta dy  } = - \frac{ 2\int_{\R_+} \phi_{0, b, \beta} \partial_\tau \phi_{0, b, \beta} \rho_\beta dy + \int_{\R_+} \phi_{0, b, \beta}  \phi_{0, b, \beta} \partial_\tau \rho_\beta dy }{\left(\int_{\R_+} \phi_{0, b,\beta}^2 \rho_\beta dy\right)^{2}} \\
	& = & - \| \phi_{0,b,\beta}\|^{-4}_{L^2_{\rho_\beta}} \left(2\frac{b'}{b} \langle \phi_{0, b, \beta}, b \partial_b \phi_{0, b, \beta} \rangle_{L^2_{\rho_\beta}} + 2 \beta'\langle \phi_{0, b, \beta},  \partial_\beta \phi_{0, b, \beta} \rangle_{L^2_{\rho_\beta}}    \right.  \\
	&+& \left.  \beta'\left\langle \phi_{0, b, \beta},   \phi_{0, b, \beta} \left( \frac{d+2}{2\beta} - \frac{y^2}{2} \right) \right\rangle_{L^2_{\rho_\beta}}       \right),
\end{eqnarray*}
and 
\begin{eqnarray*}
	\langle \varepsilon, \partial_\tau \phi_{0,b,\beta}  \rangle_{L^2_{\rho_\beta}} &=& \frac{b'}{b} \langle \varepsilon, b \partial_b \phi_{0, b, \beta} \rangle_{L^2_{\rho_\beta}} + \beta'  \langle \varepsilon,  \partial_\beta \phi_{0, b, \beta} \rangle_{L^2_{\rho_\beta}},\\
	\left\langle \varepsilon, \phi_{0, b,\beta}\frac{\partial_\tau \rho_\beta}{\rho_\beta}\right\rangle_{L^2_{\rho_\beta}} & = & \beta'\left\langle \varepsilon,   \phi_{0, b, \beta} \left( \frac{d+2}{2\beta} - \frac{y^2}{2}\right) \right\rangle_{L^2_{\rho_\beta}}.
\end{eqnarray*}
We plug these equalities  into \eqref{equality-partial-tau-varphi_j} to derive 
\begin{eqnarray}
	-\partial_\tau \varepsilon_1 =  \|\phi_{0, b, \beta}\|^{-2}_{L^2_{\rho_\beta}}\langle \partial_\tau \varepsilon,   \phi_{0, b, \beta} \rangle_{L^2_{\rho_\beta}} + \tilde K_0,
\end{eqnarray}
where
\begin{eqnarray}
	\tilde K_0 & = &  
	- \| \phi_{0, b, \beta}\|_{L^2_{\rho_\beta}}^{-4} \langle \varepsilon, \phi_{0,b, \beta}\rangle_{L^2_{\rho_\beta}}   \left\{ 2\frac{b'}{b} \langle \phi_{0, b, \beta}, b \partial_b \phi_{0, b, \beta} \rangle_{L^2_{\rho_\beta}} +2 \beta'  \langle \phi_{0, b, \beta},  \partial_\beta \phi_{0, b, \beta} \rangle_{L^2_{\rho_\beta}}  \right.\label{defi-tilde-K-j} \\
	&+& \left.  \beta'\left\langle \phi_{0, b, \beta},   \phi_{0, b, \beta} \left( \frac{d+2}{2\beta} - \frac{y^2}{2}\right) \right\rangle_{L^2_{\rho_\beta}} \right\} +\| \phi_{0, b, \beta}\|_{L^2_{\rho_\beta}}^{-2}\left\{ \frac{b'}{b} \langle \varepsilon, b \partial_b \phi_{0, b, \beta} \rangle_{L^2_{\rho_\beta}}  \right. \nonumber\\
	& + &  \left.  \beta'  \langle \varepsilon,  \partial_\beta \phi_{0, b, \beta} \rangle_{L^2_{\rho_\beta}} +  \beta'\left\langle \varepsilon,   \phi_{0, b, \beta} \left( \frac{d+2}{2\beta} - \frac{y^2}{2}\right) \right\rangle_{L^2_{\rho_\beta}} \right\} \nonumber.
\end{eqnarray}
Similarly, we derive from the first one in \eqref{system-varepsilon-1} that 

$$   c_{1,0}^{-1} \partial_\tau \varepsilon_1  = \|\phi_{1, b, \beta}\|^{-2}_{L^2_{\rho_\beta}}\langle \partial_\tau \varepsilon,   \phi_{1, b, \beta} \rangle_{L^2_{\rho_\beta}} + \tilde K_1, $$
where 
\begin{eqnarray}
	\tilde K_1 & = &  
	- \| \phi_{1, b, \beta}\|_{L^2_{\rho_\beta}}^{-4} \langle \varepsilon, \phi_{1,b, \beta}\rangle_{L^2_{\rho_\beta}}   \left\{ 2\frac{b'}{b} \langle \phi_{1, b, \beta}, b \partial_b \phi_{1, b, \beta} \rangle_{L^2_{\rho_\beta}} +2 \beta'  \langle \phi_{1, b, \beta},  \partial_\beta \phi_{1, b, \beta} \rangle_{L^2_{\rho_\beta}}  \right.\label{defi-tilde-K-1} \\
	&+& \left.  \beta'\left\langle \phi_{1, b, \beta},   \phi_{1, b, \beta} \left( \frac{d+2}{2\beta} - \frac{y^2}{2}\right) \right\rangle_{L^2_{\rho_\beta}} \right\} +\| \phi_{1, b, \beta}\|_{L^2_{\rho_\beta}}^{-2}\left\{ \frac{b'}{b} \langle \varepsilon, b \partial_b \phi_{1, b, \beta} \rangle_{L^2_{\rho_\beta}}  \right. \nonumber\\
	& + &  \left.  \beta'  \langle \varepsilon,  \partial_\beta \phi_{1, b, \beta} \rangle_{L^2_{\rho_\beta}} +  \beta'\left\langle \varepsilon,   \phi_{1, b, \beta} \left( \frac{d+2}{2\beta} - \frac{y^2}{2}\right) \right\rangle_{L^2_{\rho_\beta}} \right\} \nonumber.
\end{eqnarray}
We have the following system 
\begin{equation}\label{system-partial-tau-varepsilon-ell}
	\left\{ \begin{array}{rcl}
		c_{1, 0}^{-1} \partial_\tau \varepsilon_1 (\tau) &=&  \|\phi_{1, b, \beta}\|^{-2}_{L^2_{\rho_\beta}}\langle \partial_\tau \varepsilon,   \phi_{1, b, \beta} \rangle_{L^2_{\rho_\beta}} + \tilde K_1, \\
		- \partial_\tau \varepsilon_1 (\tau) &=&  \|\phi_{0, b, \beta}\|^{-2}_{L^2_{\rho_\beta}}\langle \partial_\tau \varepsilon,   \phi_{0, b, \beta} \rangle_{L^2_{\rho_\beta}} + \tilde K_0.
	\end{array}
	\right.
\end{equation}
Since   $\varepsilon$ solves   \eqref{equa-varepsilon-appen}, we obtain  
\begin{eqnarray*}
	\langle \partial_\tau \varepsilon,   \phi_{0, b, \beta} \rangle_{L^2_{\rho_\beta}} = \langle \mathscr{L}_b\varepsilon ,\phi_{0, b, \beta} \rangle_{L^2_{\rho_\beta}} + \langle B(\varepsilon),\phi_{0, b, \beta} \rangle_{L^2_{\rho_\beta}} + \langle \Phi, \phi_{0, b, \beta}\rangle_{L^2_{\rho_\beta}}, 
\end{eqnarray*}
which implies 
\begin{equation}\label{ODE-varep-j-rough}
	\left\{ \begin{array}{rcl}
		c_{1,0}^{-1} \partial_\tau \varepsilon_1 &=& \|\phi_{1,b,\beta}\|^{-2}_{L^2_{\rho_\beta}} \left\{ \langle \mathscr{L}_b\varepsilon ,\phi_{1, b, \beta} \rangle_{L^2_{\rho_\beta}} + \langle B(\varepsilon),\phi_{1, b, \beta} \rangle_{L^2_{\rho_\beta}} + \langle \Phi, \phi_{1, b, \beta}\rangle_{L^2_{\rho_\beta}} \right\}+ \tilde K_1\\
		-\partial_\tau \varepsilon_1 & = &   \|\phi_{0,b,\beta}\|^{-2}_{L^2_{\rho_\beta}} \left\{ \langle \mathscr{L}_b\varepsilon ,\phi_{0,b, \beta} \rangle_{L^2_{\rho_\beta}} + \langle B(\varepsilon),\phi_{0,b, \beta} \rangle_{L^2_{\rho_\beta}} + \langle \Phi, \phi_{0,b, \beta}\rangle_{L^2_{\rho_\beta}} \right\}+ \tilde K_0
	\end{array}
	\right..
\end{equation}
We only estimate  all  terms  of the first equation in  \eqref{ODE-varep-j-rough}, the rest is left to the reader.   

- \textit{ For $\langle \mathscr{L}_b \varepsilon, \phi_{j,b, \beta} \rangle_{L^2_{\rho_\beta}}$}: Using the fact that  $\mathscr{L}_b$ is self-adjoint and the special decomposition \eqref{defi-varepsilon-j}, we have  
\begin{eqnarray}
	\langle \mathscr{L}_b \varepsilon, \phi_{1,b, \beta}  \rangle_{L^2_{\rho_\beta}} = \langle  \varepsilon, \mathscr{L}_b \phi_{1,b, \beta}  \rangle_{L^2_{\rho_\beta}}  = \lambda_{1, b,  \beta}   \|\phi_{1, b, \beta}\|^2_{L^2_{\rho_\beta}} 
	\frac{\varepsilon_1}{c_{1, 0}}  
	\label{mathscr-L-b-epsilon-0}
\end{eqnarray}
where 
\begin{eqnarray*}
	\lambda_{1, b, \beta} = 2\beta \left( \frac{\alpha}{2} - 1\right) + \tilde \lambda_{1,b,\beta},  \text{ with } \left|\tilde \lambda_{1, b, \beta}\right| \lesssim  b^{ 1 -\frac{\epsilon}{2} }.
\end{eqnarray*}

- \textit{For   $\langle B(\varepsilon), \phi_{j,b,\beta}  \rangle_{L^2_{\rho_\beta}}$}:  We recall  $B(\varepsilon)$   \eqref{defi-B-quadratic-appendix} in the below
\[B(\varepsilon)  =   - 3(d-2) (1+ |y|^2 Q_b) \varepsilon^2 - (d-2) |y|^2  \varepsilon^3.\]
From \eqref{decompose-varepsilon-ell=1} we have   
\begin{eqnarray*}
	\left| 3(d-2) (1 + y^2 Q_b) \varepsilon^2 \right|  & \lesssim &   \varepsilon_1^2 \left( \frac{\phi_{1, b, \beta}}{c_{1, 0}} - \phi_{0, b, \beta}  \right)^2 + \varepsilon_-^2, \\
	\left|  (d-2) |y|^2 \varepsilon^3  \right|   & \lesssim & y^2  \left(  \left|\varepsilon_\ell \right|^3 \left| \frac{\phi_{1, b, \beta}}{c_{1, 0}} - \phi_{0, b, \beta}  \right|^3 + \left|\varepsilon_-\right|^3 \right).
\end{eqnarray*}
Since   $(\varepsilon, b, \beta)(\tau) \in V_1[A,\eta, \tilde \eta](\tau),$ for all $\tau \in [\tau_0,\tau_1]$  which ensures the  pointwise 
estimates given in Proposition \eqref{propo-mathscr-L-b} to deduce that
\begin{eqnarray*}
	\left| \left\langle  B(\varepsilon),  \phi_{1, b, \beta}   \right\rangle_{L^2_{\rho_\beta}}   \right| \lesssim  \left| \varepsilon_1 \right|^2 + \int_{\R}  \left[ \varepsilon_-^2 + y^2 \left| \varepsilon_-\right|^3 \right] \left|\phi_{j,b,\beta} \right| \rho_\beta dy.
\end{eqnarray*}
Lemma  \ref{lemma-rough-estimate-bounds in shrinking-set} yields
\begin{eqnarray}
	\left| \varepsilon_-(y,\tau) \right|  \le CA^4 b^{\frac{\alpha}{2} +\tilde \eta}(\tau) \frac{\langle y\rangle^{2\ell +2}}{y^\gamma}, \forall y\in \R^*_+,\label{estimate-varepsilon--}
\end{eqnarray}
which yields
\begin{eqnarray*}
	\int_{\R_+}  \left[ \varepsilon_-^2 + y^2 \left| \varepsilon_-\right|^3 \right] \left|\phi_{j,b,\beta} \right| \rho_\beta dy \lesssim b^{\frac{\alpha}{2} +4\eta}(\tau).
\end{eqnarray*}
Hence, we  get 
\begin{equation}\label{project-B(epsilon)-phi-0}
	\langle B(\varepsilon), \phi_{1,b,\beta}    \rangle_{L^2_{\rho_\beta}}  \lesssim b^{ \frac{\alpha}{2} + 4\eta}(\tau).
\end{equation}

\medskip
\noindent
- \textit{For  $\langle \Phi(.,\tau) ,\phi_{ 1, b, \beta}\rangle_{L^2_{\rho_\beta}}$}: \iffalse From the special property of  $\Phi$, defined as in \eqref{Phi-simple}, we should consider the scalar products in two separate cases $j=0$ and $j \ge 1$.\fi
Following $\Phi$'s definition in \eqref{Phi-simple}, we have
\begin{eqnarray*}
	\Phi (y,\tau) = \left[ \frac{b'(\tau)}{b(\tau)} - 2\beta \right] \frac{1}{2 b} \Lambda_y Q\left(\frac{y}{\sqrt{b}}\right) = \left[\frac{b'(\tau)}{b(\tau)} - 2\beta \right] \frac{1}{2b} \Lambda_\xi Q(\xi), \text{ with } \xi = \frac{y}{\sqrt{b}}.
\end{eqnarray*}
Accordingly to $\phi_{0,b,\beta}$'s formula in  Proposition \ref{propo-mathscr-L-b}, \iffalse
\begin{eqnarray}
	\phi_{0,b}  = c_{0,0} b^{-\gamma}T_0(\xi)  + \tilde \phi_{0,b}(y),
\end{eqnarray}
which yields \fi
and the construction of $T_0$ in Lemma \ref{lemma-Generation-H}, we write $\Phi$ as follows
\begin{eqnarray}
	\Phi  =  \left[  \frac{b'}{b} -2\beta\right]  m_0 b^\frac{\alpha}{2}\phi_{0,b, \beta} + \tilde \Phi,\label{decompose-Phi}
\end{eqnarray}
for some  $m_0 \ne 0$,  and 
\begin{eqnarray*}
	\| \tilde \Phi \|_{L^2_{\rho_\beta}} \lesssim  \left| \frac{b_\tau}{b} -2\beta \right| b^{\frac{\alpha}{2} +1-\frac{\epsilon}{2}}.
\end{eqnarray*}
This immediately implies 
\begin{equation}\label{scalar-Phi-phi_j}
	\left\langle \Phi(.,\tau),\phi_{j, b, \beta}   \right\rangle_{L^2_{\rho_\beta}} = \left\{  \begin{array}{rcl}
		m_0 b^\frac{\alpha}{2} \left[ \frac{b'(\tau)}{b(\tau)} -2\beta \right]\|\phi_{0, b, \beta}\|^2_{L^2_{\rho_\beta}}  + O\left( \left| \frac{b_\tau}{b} -2\beta \right| b^{\frac{\alpha}{2} +1-\frac{\epsilon}{2}} \right)  & \text{ if  }& j =0   \\[0.2cm]
		O\left( \left| \frac{b_\tau}{b} -2\beta \right| b^{\frac{\alpha}{2} +1-\frac{\epsilon}{2}} \right)  & \text{ if } & j =1
	\end{array}
	\right..
\end{equation}

- \textit{For $ \tilde K_1$}:  Let us consider $ \delta \ll \min\left(\frac{1}{2}, 1-\frac{\epsilon}{2} \right)   $ with  $\epsilon$ defined as in Proposition \eqref{propo-mathscr-L-b}   and    $\delta \gg \eta \gg \tilde \eta$. From $\tilde K_1$'s definition given in \eqref{defi-tilde-K-1}, we will  prove the following bounds:
\begin{eqnarray}
	\langle \phi_{1, b, \beta}, b\partial_b  \phi_{1,b,\beta} \rangle_{L^2_\rho} & \lesssim & b^{\delta},  \label{esti-phi_ell-partial-b-phi-ell}\\
	\langle \varepsilon, b\partial_b  \phi_{1,b,\beta} \rangle_{L^2_\rho} & \lesssim & b^{\delta} \left( |\varepsilon_1| + \left\| y^\gamma \frac{\varepsilon_-(\cdot,\tau)}{1+y^{4}}\right\|_{L^\infty_{[0, b^{-\tilde \eta} ]}} + \|\varepsilon_-\|_{L^2_{\rho_\beta}} \right),\label{esti-varepsilon-bpartial-b-phi-l}\\
	\langle \phi_{1, b, \beta},  \partial_\beta \phi_{1, b, \beta} \rangle_{L^2_{\rho_\beta}} &=& \frac{1}{\beta} \|\phi_{1, b,\beta}\|^2_{L^2_{\rho_\beta}} + O(b^{\delta}),\label{esti-phi_ell-partial-beta-phi-ell}\\
	\left\langle \phi_{1, b, \beta},  \phi_{1, b, \beta} \left( \frac{d+2}{2\beta} -\frac{y^2}{2}\right)   \right\rangle_{L^2_{\rho_\beta}}&=& \left[ \frac{d+2}{2\beta} - \frac{2}{\beta}\left(\frac{d}{2} -\gamma +2 \right) + \frac{1}{\beta}\left(\frac{d}{2} -\gamma +1 \right) \right] \|\phi_{1,b,\beta}\|^2_{L^2_{\rho_\beta}} \nonumber \\ 
	&+& O(b^{\delta}),\label{esti-phi_ell--phi-ell-deriv-rho}\\ \nonumber
\end{eqnarray}
and with  the equality \eqref{defi-varepsilon-j}
\begin{eqnarray}
	\langle \varepsilon,  \partial_\beta \phi_{1, b, \beta} \rangle_{L^2_{\rho_\beta}} &=& \frac{1}{\beta} \frac{\varepsilon_1}{c_{1, 0}} \|\phi_{1, b, \beta}\|^2_{L^2_{\rho_\beta}} - \frac{\varepsilon_1}{4\beta} \|\phi_{1, b,\beta}\|^2_{L^2_{\rho_\beta}} + O\left(\left\| y^\gamma \frac{\varepsilon_-}{1+y^{4}}\right\|_{L^\infty} \right) \\
	& + &  O\left( b^{\delta} \left( |\varepsilon_\ell| + \left\| y^\gamma \frac{\varepsilon_-(\cdot,\tau)}{1+y^{4}}\right\|_{L^\infty_{[0, b^{-\tilde \eta} ]}} + \|\varepsilon_-\|_{L^2_{\rho_\beta}} \right)\right),    \label{esti-varepsilon-partial-beta-phi-l}   
\end{eqnarray}
and
\begin{equation}\label{esti-varepsilon-partial-beta-rho}
	\left\langle \varepsilon,   \phi_{1, b, \beta} \left( \frac{d+2}{2\beta} - \frac{y^2}{2}\right) \right\rangle_{L^2_{\rho_\beta}} = \left\{ \begin{array}{rcl}
		& & \hspace{-1cm} \varepsilon_1 \|\phi_{1, b, \beta}\|^2_{L^2_{\rho_\beta}} \left[ \frac{1}{c_{1,0}} \left( \frac{d+2}{2\beta} - \frac{2}{\beta}\left(\frac{d}{2} -\gamma +2 \right) + \frac{1}{\beta}\left(\frac{d}{2} -\gamma +1 \right) \right) +\frac{1}{4\beta} \right] \\
		& + &  O\left( b^{\delta} \left( |\varepsilon_1| + \left\| y^\gamma \frac{\varepsilon_-(\cdot,\tau)}{1+y^{4}}\right\|_{L^\infty_{[0, b^{-\tilde \eta} ]}} + \|\varepsilon_-\|_{L^2_{\rho_\beta}} \right)\right),
	\end{array} \right.
\end{equation}
and
\begin{equation}\label{equality-varep-phi-0-partial-beta-rho}
	\left\langle \varepsilon,   \phi_{0, b, \beta} \left( \frac{d+2}{2\beta} - \frac{y^2}{2}\right) \right\rangle_{L^2_{\rho_\beta}}=\frac{\gamma \varepsilon_1}{\beta}\|\phi_{0, b, \beta}\|^2_{L^2_{\rho_\beta}} +O(b^{\frac{\alpha}{2} +\delta}).
\end{equation}

+ For \eqref{esti-phi_ell-partial-b-phi-ell}: As we assumed $\delta \ll  \frac{1}{2}$, then,   for all $y \ge b^\delta$, we have 
$$ \xi = \frac{y}{\sqrt{b}} \to +\infty, \text{ as } b \to 0.  $$
Hence, we have 
\begin{eqnarray*}
	\left| \partial_b \sum_{j=0}^1 c_{1,j} (\sqrt{b})^{2j-\gamma} T_j(\xi) \right|&=&\left|\frac{1}{2b}\sum_{j=0}^1 c_{1,j} (2j-\gamma) (\sqrt{b})^{2j-\gamma}T_j(\xi)-\frac{1}{2b}\sum_{j=0}^1 c_{1,j} (\sqrt{b})^{2j-\gamma} \xi \partial_\xi T_j(\xi)\right|\\
	& \lesssim & \sum_{j=0}^1 \sqrt{b}^{2j -2-\gamma} \xi^{2j -\gamma - 2 } |\ln \xi |, \text{ as } \xi \to +\infty,
\end{eqnarray*}
and   \eqref{asymptotic-Theta} ensures that
\begin{equation*}
	\xi \partial_\xi T_j(\xi)=(2j-\gamma)T_j(\xi)+O(\xi^{-\gamma+2j-2}\ln\xi), \text{ as } \xi \to +\infty.
\end{equation*}
Then 
\begin{eqnarray*}
	\left| b \partial_b \sum_{j=0}^i c_{i,j} (\sqrt{b})^{2j-\gamma} T_j(\xi) \right| 
	\lesssim  \sum_{j=0}^i \sqrt{b}^{2j -\gamma} \xi^{2j -\gamma - 2 } |\ln \xi |, \text{ as } \xi \to +\infty,
\end{eqnarray*}
which implies,  from  \eqref{phi-i-nu} and the above inequalities, that
\begin{eqnarray}
	\left| b \partial_b \phi_{j,b,\beta}(y) \right| \lesssim b^{1 -\frac{\epsilon}{2}} \frac{\langle y \rangle^{2} |\ln y|}{y^{\gamma+2}},  \forall y \ge b^\delta \text{ and  } j \le 1.
	\label{estimate-bpartial-b-phi-j-b}
\end{eqnarray}
Hence, we derive  on the one hand
\begin{equation}\label{estima-intergral-phi-jb-par-b-}
	\left| \int_{y \ge b^\delta} \phi_{j, b, \beta} b\partial_b \phi_{i, b, \beta} \rho_{\beta} dy \right| \le  C b^{\delta}, \forall i,j \le 1. 
\end{equation}
On the other hand,  we  apply the pointwise estimates given in Proposition \ref{propo-mathscr-L-b} that yields
\begin{eqnarray}
	\left| \int_{y \le b^\delta} \phi_{j, b, \beta} b\partial_b \phi_{i, b, \beta} \rho_\beta dy \right|\lesssim  \int_{y \le b^\delta}  y^{d+1-2\gamma} e^{-\frac{2\beta y^2}{4}} dy \lesssim \int_{y \le b^\delta}  y^{d+1-2\gamma} dy 
	\lesssim  b^{\delta(d+2-2\gamma)}.  \label{estima-phi-j-bparti-b-y-less-bdelta}
\end{eqnarray}
Combining \eqref{estima-intergral-phi-jb-par-b-} and \eqref{estima-phi-j-bparti-b-y-less-bdelta}, we obtain 
\begin{equation}\label{scalar-product-phi-j-bpartial-bphi-i}
	\left| \langle \phi_{j,b,\beta}, b \partial_b \phi_{i,b, \beta} \rangle_{L^2_{\rho_\beta} }    \right|  \lesssim b^{1 -\frac{\epsilon}{2}} + b^{\delta(d+2-\gamma)} \lesssim b^{\delta},i,j\leq 1,
\end{equation}
thus, \eqref{esti-phi_ell-partial-b-phi-ell} follows.

\iffalse

In particular, we get a general estimate as follows
\begin{equation}\label{asti-phi-bpartial-b-phi-i}
	\left| \langle \phi_{j, b, \beta},  \partial_\tau \phi_{i, b, \beta} \rangle_{L^2_{\rho_\beta} }    \right| \le C \left| \frac{b'(\tau)}{b}  \right| b^{\delta},\, \text{for all}\, i,j\leq \ell.
\end{equation}
Finally, 
\fi 
\iffalse
By a similar computation, we also get
\begin{eqnarray*}
	\left| \partial_\tau  \| \phi_{j,b}\|_{L^2_\rho}^2  \right| \le C \left| \frac{b'(\tau)}{b}  \right| b^{\delta}.
\end{eqnarray*}
and \eqref{esti-phi_ell-partial-b-phi-ell} is obtained.
\fi

+ For  \eqref{esti-varepsilon-bpartial-b-phi-l}: 
Using \eqref{defi-varepsilon-j}, we  estimate  as follows
\begin{eqnarray*}
	\left|\langle  \varepsilon, b\partial_b  \phi_{1,b,\beta} \rangle \right|  \lesssim \sum_{j=0}^{1} \left|\varepsilon_j \right| \left|\langle  \phi_{j,b, \beta}, b \partial_b \phi_{1, b, \beta}  \rangle_{L^2_{\rho_\beta}} \right|  + \left|\langle \varepsilon_{-},   b \partial_b \phi_{1} \rangle_{L^2_{\rho_\beta}}  \right|.
\end{eqnarray*}
Using  \eqref{scalar-product-phi-j-bpartial-bphi-i}, we get 
\begin{eqnarray*}
	\sum_{j=0}^{1} \left|\varepsilon_j \right| \left|\langle  \phi_{j, b, \beta}, b \partial_b \phi_{1}  \rangle_{L^2_{\rho_\beta}} \right| \lesssim b^{\delta} \sum_{j=0}^1 \left| \varepsilon_j \right|.  
\end{eqnarray*}
Next, we estimate the projection on $\partial_b \phi_1$  of $\varepsilon_-$. Indeed, we split the integral: 
\begin{eqnarray*}
	\left|\langle \varepsilon_{-}(\tau),   b \partial_b \phi_{1, b, \beta} \rangle_{L^2_{\rho_\beta}} \right| 
	& \le & \int_{0}^{b^\delta} \left| \varepsilon_-(\tau) \right| \left| b \partial_b \phi_{1, b, \beta}  \right| \rho_\beta  dy + \int_{b^{\delta}}^{b^{-\tilde \eta}} \left| \varepsilon_-(\tau) \right| \left| b \partial_b \phi_{1, b, \beta}  \right| \rho_\beta  dy  \\
	&+& \int_{b^{-\tilde \eta}}^{\infty} \left| \varepsilon_-(\tau) \right| \left| b \partial_b \phi_{1, b, \beta}  \right| \rho_\beta  dy .   
\end{eqnarray*}
For the integral on $[0,b^{\delta}]$, 
we estimate as follows
\begin{eqnarray*}
	& & \left|\int_{0}^{b^\delta} \varepsilon_{-}(\tau)  b \partial_b \phi_{1, b, \beta} \rho_\beta dy\right|\leq \left\| y^\gamma \frac{\varepsilon_-(\cdot,\tau)}{ \langle y\rangle^{4}}\right\|_{L^\infty{[0, b^{-\tilde \eta} ]}}  \int_{0}^{b^\delta} y^{d+1-2\gamma}
	\\
	& \leq & \left\| y^\gamma \frac{\varepsilon_-(\cdot,\tau)}{\langle y\rangle^{4}}\right\|_{L^\infty{[0, b^{-\tilde \eta} ]}}  b^{\delta(d+2-2\gamma)} \le  \left\| y^\gamma \frac{\varepsilon_-(\cdot,\tau)}{\langle y\rangle^{4}}\right\|_{L^\infty{[0, b^{-\tilde \eta} ]}} b^\delta. 
\end{eqnarray*}
On the interval $ [b^\delta, b^{ -\tilde \eta }]$, we  estimate 
\begin{eqnarray*}
	& & \left|\int_{b^{\delta}}^{b^{-\tilde \eta}} \varepsilon_{-}(\tau)  b \partial_b \phi_{1,b,\beta} \rho_\beta dy\right| \leq  b^{1-\frac{\epsilon}{2}} \left\| y^\gamma \frac{\varepsilon_-(\cdot,\tau)}{1+y^{4}}\right\|_{L^\infty{[0, b^{-\tilde \eta} ]}} \int_{b^{\delta}}^{b^{-\tilde \eta}} \frac{(1+y^{4})(1+y^{2})|\ln y|}{y^{2\gamma+2}}y^{d+1}  e^{-\frac{2\beta y^2}{4}
	} dy   \\
	& \lesssim & b^{1-\frac{\epsilon}{2}} \left\| y^\gamma \frac{\varepsilon_-(\cdot,\tau)}{1+y^{4}}\right\|_{L^\infty{[0, b^{-\tilde \eta} ]}} \left\{  \int_{b^\delta}^1 y^{d-2\gamma-1}dy+ \int_1^{+\infty}\frac{(1+y^{4})(1+y^{2})|\ln y|}{y^{2\gamma+2}}y^{d+1}  e^{-\frac{2\beta y^2}{4}
	} dy  \right\}\\
	&\lesssim& b^{\delta} \left\| y^\gamma \frac{\varepsilon_-(\cdot,\tau)}{1+y^{4}}\right\|_{L^\infty{[0, b^{-\tilde \eta} ]}} .
\end{eqnarray*}
For the interval $[b^{-\tilde \eta},+\infty)$, we use  Cauchy-Schwarz inequality  and \eqref{estimate-bpartial-b-phi-j-b}  to arrive at
\begin{eqnarray*}
	\left| \int_{b^{-\tilde \eta}}^{+\infty} \left| \varepsilon_-\right| |b \partial_b \phi_{1, b, \beta}| \rho_\beta \right| \lesssim \|\varepsilon_-\|_{L^2_{\rho_\beta}} \left(\int_{b^{-\tilde \eta}}^{\infty}   |b \partial_b \phi_{1, b, \beta}|^2\rho_\beta \right)^{\frac{1}{2}}  \lesssim b^{1-\frac{\epsilon}{2}}  \|\varepsilon_-\|_{L^2_{\rho_\beta}} \lesssim b^{\delta}  \|\varepsilon_-\|_{L^2_{\rho_\beta}},
\end{eqnarray*}
which concludes   \eqref{esti-varepsilon-bpartial-b-phi-l} by adding 
all related terms.

\noindent
\medskip
+ For estimates \eqref{esti-phi_ell-partial-beta-phi-ell} and \eqref{esti-phi_ell--phi-ell-deriv-rho}: Indeed, from \eqref{decom-partial-beta-phi-l}, we immediately deduce \eqref{esti-phi_ell-partial-beta-phi-ell}. In  addition to that, by  combining \eqref{inte-phi-2-y-2=infinit} and \eqref{phi-l-2-infty-y-2} one gets \eqref{esti-phi_ell--phi-ell-deriv-rho}.

\noindent
\medskip
+ For \eqref{esti-varepsilon-partial-beta-phi-l}: According to \eqref{defi-varepsilon-j}, we decompose $\varepsilon$ as follows 
\begin{eqnarray*}
	\langle \varepsilon, \partial_\beta \phi_{1, b, \beta}  \rangle_{L^2_{\rho_\beta}}  & = &  \varepsilon_1 (\tau) \left\langle \frac{\phi_{1,b,\beta}}{c_{1,0}} - \phi_{0,b,\beta},\partial_\beta \phi_{1, b, \beta} \right\rangle_{L^2_{\rho_\beta}}+\left\langle\varepsilon_-(\tau),\partial_\beta \phi_{1, b, \beta}\right\rangle_{L^2_{\rho_\beta}}.
\end{eqnarray*}
Similarly to \eqref{esti-varepsilon-bpartial-b-phi-l},  one can deduce that
\begin{eqnarray*}
	\langle\varepsilon_-(y,\tau),\partial_\beta \phi_{1, b, \beta}\rangle_{L^2_{\rho_\beta}}  \lesssim  b^{\delta} \left(  \left\| y^\gamma \frac{\varepsilon_-(\cdot,\tau)}{1+y^{4}}\right\|_{L^\infty_{[0, b^{-\tilde \eta} ]}} + \|\varepsilon_-\|_{L^2_{\rho_\beta}} \right). 
\end{eqnarray*}
In addition to that, we derive from 
\eqref{inte-2phi-l-parti-beta-phi-l}
\begin{eqnarray*}
	\left\langle \frac{\phi_{1,b,\beta}}{c_{1,0}} ,\partial_\beta \phi_{1, b, \beta} \right \rangle_{L^2_{\rho_\beta}}  =    \frac{1}{c_{1,0}} \| \phi_{1, b, \beta}\|_{L^2_{\rho_\beta}}^2 + O(b^{1-\frac{\epsilon}{2}}),
\end{eqnarray*}
and from \eqref{int-phi-0-partial0beta-phi-l}, we have
$$ \langle \phi_{0,b,\beta}, \partial_\beta \phi_{1, b, \beta}   \rangle_{L^2_{\rho_\beta}}  =  \frac{1}{4\beta} \| \phi_{1, b, \beta}\|_{L^2_{\rho_\beta}}^2   +O(b^{1-\frac{\epsilon}{2}}). $$
Finally, we use   the above facts  to get \eqref{esti-varepsilon-partial-beta-phi-l}.

+ For \eqref{esti-varepsilon-partial-beta-rho}:   
We firstly write as follows
\begin{eqnarray*}
	& &\left\langle \varepsilon,   \phi_{1, b, \beta} \left( \frac{d+2}{2\beta} - \frac{y^2}{2}\right) \right\rangle_{L^2_{\rho_\beta}}
	=   \varepsilon_1 (\tau) \left\langle \frac{\phi_{1,b,\beta}}{c_{1,0}} - \phi_{0,b,\beta},\phi_{1, b, \beta} \left( \frac{d+2}{2\beta} - \frac{y^2}{2}\right) \right\rangle_{L^2_{\rho_\beta}}\\
	& & +\left\langle\varepsilon_-(y,\tau),\phi_{1, b, \beta} \left( \frac{d+2}{2\beta} - \frac{y^2}{2}\right)\right\rangle_{L^2_{\rho_\beta}}.
\end{eqnarray*}
On one hand, we have 
\begin{eqnarray*}
	& &\left\langle\varepsilon_-(y,\tau),\phi_{1, b, \beta} \left( \frac{d+2}{2\beta} - \frac{y^2}{2}\right)\right\rangle_{L^2_{\rho_\beta}}\\
	& \lesssim &  \sum_{j=1}^{1-1} |\varepsilon_j|  +  b^{\delta} \left(   \left\| y^\gamma \frac{\varepsilon_-(\cdot,\tau)}{1+y^{21+2}}\right\|_{L^\infty_{[0, b^{-\tilde \eta} ]}} + \|\varepsilon_-\|_{L^2_{\rho_\beta}} \right). 
\end{eqnarray*}
For the rest, we obtain
\begin{eqnarray*}
	\left\langle \frac{\phi_{1,b,\beta}}{c_{1,0}} - \phi_{0,b,\beta},\phi_{1, b, \beta} \left( \frac{d+2}{2\beta} - \frac{y^2}{2}\right) \right\rangle_{L^2_{\rho_\beta}} &=&\frac{1}{c_{1, 0}} \frac{d+2}{2\beta} \| \phi_{1, b, \beta}\|^2_{L^2_{\rho_\beta}} - \frac{1}{c_{1,0}} \left\langle \phi_{1,b,\beta}, \phi_{1, b, \beta}  \frac{y^2}{2} \right\rangle_{L^2_{\rho_\beta}}\\
	&  +  & \left\langle \phi_{0,b,\beta}, \phi_{1, b, \beta}  \frac{y^2}{2} \right\rangle_{L^2_{\rho_\beta}}. 
\end{eqnarray*}
Using \eqref{inte-phi-2-y-2=infinit}, \eqref{phi-l-2-infty-y-2},\eqref{int-phi-0-partial0beta-phi-l} and \eqref{int-phi-0partial-betaphi-l}, 
we have that
\begin{eqnarray*}
	\left\langle \phi_{0,b,\beta}, \phi_{1, b, \beta}  \frac{y^2}{2} \right\rangle_{L^2_{\rho_\beta}}=  \frac{1}{4\beta}\|\phi_{1, b, \beta}\|_{_{L^2_{\rho_\beta}}}^2+ O(b^{1-\frac{\epsilon}{2}}) 
\end{eqnarray*}
and  
\begin{eqnarray*}
	\left\langle \phi_{1,b,\beta}, \phi_{1, b, \beta}  \frac{y^2}{2} \right\rangle_{L^2_{\rho_\beta}} = \left( \frac{2}{\beta} \left(\frac{d}{2} -\gamma+2 \right) - \frac{1}{\beta}\left(\frac{d}{2} -\gamma+1  \right) \right)  \|\phi_{1, b,\beta}\|_{L^2_{\rho_\beta}}^{2} +O(b^{1-\frac{\epsilon}{2}}),
\end{eqnarray*}
which implies \eqref{esti-varepsilon-partial-beta-rho}.

\medskip
- For \eqref{equality-varep-phi-0-partial-beta-rho}: The proof is  similar to \eqref{esti-varepsilon-partial-beta-rho} which also follows from \eqref{inte-phi-2-y-2=infinit}, \eqref{phi-l-2-infty-y-2},\eqref{int-phi-0-partial0beta-phi-l} and \eqref{int-phi-0partial-betaphi-l}.

\medskip
Now, combining    \eqref{esti-phi_ell-partial-b-phi-ell} to \eqref{equality-varep-phi-0-partial-beta-rho}, we derive   
\begin{eqnarray*}
	\tilde K_{1} &=& -\frac{ \varepsilon_1 \beta'}{c_{1, 0} \beta} +  O\left(   \left| \frac{b'}{b}  \right|+1 \right) b^{\delta}  \left(  |\varepsilon_1| + \left\| y^\gamma \frac{\varepsilon_-(\cdot,\tau)}{1+y^{4}}\right\|_{L^\infty_{[0, b^{-\tilde \eta} ]}} + \|\varepsilon_-\|_{L^2_{\rho_\beta}} \right)\\
	&+& O \left( \beta'     b^\delta \left( \left|\varepsilon_1\right| +  \left\| y^\gamma \frac{\varepsilon_-(\cdot,\tau)}{1+y^{4}}\right\|_{L^\infty_{[0, b^{-\tilde \eta} ]}} + \|\varepsilon_-\|_{L^2_{\rho_\beta}} \right)\right) . \nonumber
\end{eqnarray*}
In a  convenient way, we denote 
\begin{eqnarray*}
	L & = &  \left(   \left| \frac{b'}{b}  \right| +1 \right) b^{\delta}  \left(  |\varepsilon_1| + \left\| y^\gamma \frac{\varepsilon_-(\cdot,\tau)}{1+y^{4}}\right\|_{L^\infty_{[0, b^{-\tilde \eta} ]}} + \|\varepsilon_-\|_{L^2_{\rho_\beta}} \right) \label{tilde-L-error}\\
	&+&  \left( \left|\beta' \right|  \left(   b^\delta \left| \varepsilon_1 \right| + \left\| y^\gamma \frac{\varepsilon_-(\cdot,\tau)}{1+y^{4}}\right\|_{L^\infty_{[0, b^{-\tilde \eta} ]}} + \|\varepsilon_-\|_{L^2_{\rho_\beta}} \right) \right).
\end{eqnarray*}
Then, we have 
\begin{eqnarray}
	\tilde K_1 =   -\frac{ \beta' \varepsilon_1 }{c_{1, 0} \beta} + O(L)\label{tilde -K-ell}.
\end{eqnarray}

\medskip
\noindent
-  Applying $ \tilde K_1 $'s process  to   $\tilde K_0$, we get
\begin{equation}\label{tilde-K-0}
	\tilde K_0 =  -\frac{\| \phi_{1, b, \beta }\|^{2} \| \phi_{0, b, \beta }\|^{-2}\beta' }{4\beta} \frac{\varepsilon_1}{c_{1,0}}  + O(L).
\end{equation}

\bigskip
Now, we are ready to start to the proof of the Lemma.

- \textit{ Proof for (i):} We use system \eqref{ODE-varep-j-rough} combined with all  of the previous estimates to derive  
\begin{equation}\label{system-varepsilon-1}
	\left\{  \begin{array}{rcl}
		\partial_\tau \varepsilon_1   & = & 2\beta \left( \frac{\alpha}{2} -1  \right) -\frac{\beta'}{\beta} \varepsilon_1   + O(L)  + O \left( \left|\frac{b_\tau}{b} -2\beta \right|b^{\frac{\alpha}{2} +\delta} \right),    \\
		\partial_\tau  \varepsilon_1   &  =  & 2\beta \frac{\alpha}{2} \varepsilon_1 - m_0 \left[ \frac{b'}{b} -2\beta \right] b^{\frac{\alpha}{2} }+ O(L)  + O \left( \left|\frac{b_\tau}{b} -2\beta \right|b^{\frac{\alpha}{2} +\delta} \right) .  
	\end{array}
	\right.
\end{equation}
In particular,  since $(\varepsilon, b, \beta)(\tau) \in V[A,\eta, \tilde \eta](\tau), \forall \tau \in [\tau_0, \tau_1]$,  the pointwise estimates  in Lemma \ref{lemma-rough-estimate-bounds in shrinking-set},    imply \eqref{system-vare-0-1}.

\medskip
\noindent
- \textit{Proof for (ii):} The results immediately follows item (i). 
%Indeed,  we recall the compatibility condition that 
%$$ \varepsilon_1(\tau) = - \frac{2}{\alpha} m_0  b^{\frac{\alpha}{2} }(\tau), \forall \tau \in [\tau_0,\tau_1].  $$
%By plugging the above condition into   \eqref{system-vare-0-1} we completely  conclude the proof of item [ii].
\end{proof}

\medskip
\section{Control of the infinite dimensional part}
In this part, we aim to give \textit{a priori estimates} involving $\varepsilon_-$ and $\varepsilon_e$
\subsection{Energy estimate}
In below, we will prove \textit{ a priori estimates } on  $\|\varepsilon_-\|^2_{\rho_\beta}$.
\begin{lemma}[A $L^2_\rho$-priori estimates on $\varepsilon_-$]\label{lemma-L-2-rho-var--}
For all $ A \ge 1, \eta, $ and $\tilde \eta$ satisfying $ 1  \ll \eta  \ll \tilde \eta$,  there exists $\tau_3 (A, \eta, \tilde \eta)$ and $\tau^*$ such that for all $ \tau_0  \ge \tau_4$ and  the solution  $ (\varepsilon, b, \beta)(\tau) \in V_1[A, \eta, \tilde \eta](\tau^*), \forall \tau \in [\tau_0, \tau^*] $ and
\begin{equation}\label{estima-epsilon--new}
	\| \varepsilon_-(\tau)\|_{L^2_{\rho_\beta}} \le C A b^{\frac{\alpha}{2} + \eta}(\tau), \forall \tau \in [\tau_0, \tau^*].
\end{equation}
\end{lemma} 
\begin{proof}
The result is mainly based on the \textit{spectral  gap} property.   First,  we claim that   \eqref{estima-epsilon--new} follows from
\begin{equation}\label{estima-derive-estima-epsilon---new}
	\frac{1}{2}\frac{d}{d \tau}\|\varepsilon_-\|_{L^2_{\rho_\beta}}^2    -   \left( \frac{\alpha}{2} - 2 \right)  \|\varepsilon_-\|_{L^2_{\rho}}^2   \le  CA    b^{ \alpha + 3 \eta}.
\end{equation}
Indeed, let us assume \eqref{estima-derive-estima-epsilon---new} holds, we infer that
\begin{eqnarray*}
	\frac{d}{d\tau} \left(   e^{2\left(2 - \frac{\alpha}{2}   \right) \tau}   \| \varepsilon_-(\tau) \|^2_{L^2_\rho}   \right)    \le C A e^{2\left(2 - \frac{\alpha}{2}   \right) \tau} b^{ \frac{\alpha}{2} +3 \eta}, \forall \tau \in [\tau_0, \tau^*].
\end{eqnarray*}
From the fact $(\varepsilon, b, \beta)(\tau) \in V_1[A, \eta, \tilde \eta](\tau),$ for all $\tau \in [\tau_0, \tau^*]$, we can apply Lemma \ref{lemma-rough-estimate-bounds in shrinking-set} to deduce
\begin{eqnarray*}
	\| \varepsilon_-\|^2_{L^2_\rho}  &\le & e^{ -2\left( 2  - \frac{\alpha}{2}   \right) (\tau -\tau_0) } \| \varepsilon_-(\tau_0)\|^2_{L^2_\rho} + CA e^{ -2\left( 2 - \frac{\alpha}{2}  \right) \tau  } \int_{\tau_0}^\tau e^{2\left(2 - \frac{\alpha}{2}   \right) \tau'} b^{ \frac{\alpha}{2} +3 \eta}(\tau')   d\tau'\\
	& \le & CA b^{\alpha +2 \eta}(\tau), \forall \tau \in [\tau_0, \tau^*].  
\end{eqnarray*}
Then, \eqref{estima-epsilon--new} follows. Now,  it remains to give the proof of  \eqref{estima-derive-estima-epsilon---new}. Indeed, we multiply equation \eqref{equa-varepsilon-appen} by $\varepsilon_-$  and integrate 
\begin{equation}
	\frac{1}{2}\frac{d}{d \tau}\|\varepsilon_-\|_{L^2_{\rho}}^2 = \left\langle \partial_\tau \varepsilon_-, \varepsilon_- \right\rangle_{L^2_{\rho_\beta}} + \beta' \left\langle  \varepsilon_- , \varepsilon_- \left( \frac{d+2}{2\beta} -\frac{y^2}{2} \right)  \right\rangle_{L^2_{\rho_\beta}}.    
\end{equation}
Next, we will prove that for all $ \tau \in [\tau_0, \tau^*]$
\begin{eqnarray}
	\left\langle \partial_\tau \varepsilon_-, \varepsilon_- \right\rangle_{L^2_{\rho_\beta}} &\le & \left( \frac{\alpha}{2} - 2\right)\| \varepsilon_-\|_{L^2_{\rho_\beta}}^2  + O(b^{\alpha +3 \eta})(\tau),\label{esti-partial-varp--varep--} \\
	\left|\beta' \left\langle  \varepsilon_- , \varepsilon_- \left( \frac{d+2}{2\beta} -\frac{y^2}{2} \right)  \right\rangle_{L^2_{\rho_\beta}} \right|  &\lesssim &  b^{\alpha + 3 \eta}(\tau).\label{esti-partialbeta--varp--varep--beta'}
\end{eqnarray}
Let us  start with \eqref{esti-partial-varp--varep--}. Indeed,  from \eqref{equa-varepsilon-appen}, and the decomposition
\begin{eqnarray}
	\varepsilon =  \varepsilon_\ell  \left(\frac{\phi_{\ell,b,\beta}}{c_{\ell,0}} - \phi_{0,b,\beta}\right) + \sum_{j=1}^{\ell-1}  \varepsilon_j \phi_{j,b,\beta} + \varepsilon_-:= \varepsilon_+  + \varepsilon_-,\label{defi-varepsilon-j} 
\end{eqnarray}
$\varepsilon_-$ solves 
\begin{eqnarray*}
	\partial_\tau \varepsilon_- = \mathscr{L}_b (\varepsilon_-) + B(\varepsilon_+ +\varepsilon_-)+ \Phi - \partial_\tau \varepsilon_+ + \mathscr{L}_b \varepsilon_+.
\end{eqnarray*}
Taking $L^2_{\rho_\beta}$ scalar product to the both sides of the above equation, we deduce 
\begin{eqnarray*}
	\langle \partial_\tau \varepsilon_, \varepsilon_- \rangle_{\rho_\beta} 
	= \langle \mathscr{L}_b \varepsilon_-,\varepsilon_-  \rangle_{\rho_\beta}   +   \langle B(\varepsilon_+ + \varepsilon_-), \varepsilon_- \rangle_{\rho_\beta} + \langle \Phi - \partial_\tau \varepsilon_+,\varepsilon_- \rangle_{\rho_\beta},
\end{eqnarray*}
since  $ \langle \mathscr{L}_b \varepsilon_+,\varepsilon_- \rangle_{L^2_\rho}=0.$

\medskip
+ Estimate to $ \langle \mathscr{L}_b \varepsilon_-, \varepsilon_- \rangle_{L^2_\rho} $: Using the orthogonality
$$ \langle \phi_{j,b, \beta}, \varepsilon_-(\tau) \rangle_{L^2_{\rho_\beta}} =0, \text{for}  j =0 \text{ and } j=1,$$
the spectral gap  in   Proposition \ref{propo-mathscr-L-b}  ensures 
\begin{eqnarray*}
	\left\langle \mathscr{L}_b \varepsilon_-,\varepsilon_- \right\rangle_{L^2_{\rho_\beta}} \le \lambda_{2 , b, \beta} \| \varepsilon_-\|^2_{L^2_{\rho_\beta}}.
\end{eqnarray*}
In addition, we have 
$$ \lambda_{2, b, \beta} = \frac{\alpha}{2} - 2  + O(b^{1-\frac{\epsilon }{2}}),$$
which yields
\begin{eqnarray*}
	\langle \mathscr{L}_b \varepsilon_-,\varepsilon_- \rangle_{L^2_{\rho_\beta}}  \le \left(\frac{\alpha}{2} - (\ell +1) \right)\|\varepsilon_-\|^2_{L^2_\rho} + Cb^{\alpha + 3 \eta}.
\end{eqnarray*} 

\medskip

+ Estimate for  $\langle \Phi - \partial_\tau \varepsilon_+,\varepsilon_- \rangle_{L^2_{\rho_\beta}} $: Recall that $$\varepsilon(\tau) = \varepsilon_1 (\tau) \left( \frac{\phi_{1,b,\beta}}{c_{1,0}} -\phi_{0,b,\beta} \right) + \varepsilon_-(\tau) = \varepsilon_+ + \varepsilon_-.$$
We  decompose
\begin{eqnarray}
	\langle \Phi - \partial_\tau \varepsilon_+,\varepsilon_- \rangle = \langle \Phi, \varepsilon_- \rangle - \langle \partial_\tau \varepsilon_+, \varepsilon_-  \rangle.
\end{eqnarray}
For  $  \langle \Phi, \varepsilon_- \rangle_{L^2_{\rho_\beta}}$, we use \eqref{decompose-Phi} and by  Cauchy-Schwarz inequality to deduce that
\begin{eqnarray*}
	\left|\langle \Phi, \varepsilon \rangle_{L^2_{\rho_\beta}} \right| = \left| \langle \tilde \Phi, \varepsilon_-   \rangle_{L^2_{\rho_\beta}}  \right| \lesssim  \|\tilde \Phi\|_{L^2_{\rho_\beta}} \|\varepsilon_- \|_{L^2_{\rho_\beta}} \le C A^3 b^{\alpha + \tilde \eta + 1- \frac{\epsilon}{2}  } \le b^{\alpha +6  \eta}(\tau), \forall \tau \in [\tau_0, \tau^*].
\end{eqnarray*}
For the second term, we have
\begin{eqnarray*}
	\partial_\tau \varepsilon_+ &=&  \varepsilon'_1 \left[ \frac{\phi_{1, b, \beta}}{c_{1, 0}} - \phi_{0,b, \beta } \right]  +  \varepsilon_j \left[ \frac{b'}{b} b\partial_b \phi_{j, b, \beta}  + \beta' \partial_\beta \phi_{j, b, \beta} \right]\\
	&+& \varepsilon_1 \left[  \frac{\frac{b'}{b} b\partial_b \phi_{1, b, \beta}  + \beta' \partial_\beta \phi_{1, b, \beta} }{c_{1,0}} - \left( \frac{b'}{b} b\partial_b \phi_{0, b, \beta}  + \beta' \partial_\beta \phi_{0, b, \beta}\right) \right]. 
\end{eqnarray*}
Note that
\begin{equation}\label{orthogoanl-varep---phi-j}
	\langle  \varepsilon_-, \phi_{j, b, \beta}\rangle_{L^2_{\rho_\beta}}  = 0, \text{for}  j =0 \text{ and } j=1,  
\end{equation}
combining this with  \eqref{estimate-beta-derive-tau}, \eqref{ODE-b-tau-proposition},  the necessary bounds in  $V_1[A, \eta, \tilde \eta](\tau),$ and Cauchy-Schwarz inequality,  we infer
\begin{equation*}
	\left| \langle \partial_\tau \varepsilon_+, \varepsilon_- \rangle_{L^2_{\rho_\beta}} \right|  \le  b^{\alpha + 3 \eta}(\tau).
\end{equation*}
Finally, we give the following estimate
\begin{eqnarray*}
	\left| \langle \Phi - \partial_\tau \varepsilon_+,\varepsilon_- \rangle_{L^2_{\rho_\beta}}  \right|  \le b^{\alpha + 3\eta}(\tau), \forall \tau \in [\tau_0, \tau^*].
\end{eqnarray*}
- For  $\langle B\left(\varepsilon \right), \varepsilon_- \rangle_{L^2_{\rho_\beta}}$ with $\varepsilon = \varepsilon_+ +\varepsilon_-$. We explicitly write $B(\varepsilon)$  in \eqref{defi-B-quadratic-appendix} as follows
\begin{eqnarray*}
	B(\varepsilon) = -3(n-2) (1 +|y|^2 Q_b) (\varepsilon_+^2 +2\varepsilon_+ \varepsilon_- +\varepsilon_-^2 ) - (d-2) |y|^2 ( \varepsilon_+^3 + 3 \varepsilon_+^2\varepsilon_- +3 \varepsilon_+ \varepsilon_-^2 + \varepsilon_-^3).     
\end{eqnarray*}
From  $\gamma$'s definition, we observe that once $d \ge 11$, one has 
$$ \gamma  \le 3.7  < 4.$$ 
\iffalse
which implies 
$$ \left\| \frac{\langle  y \rangle}{y^\gamma} \right\|_{L^3_\rho}   + \left\|y^\frac{1}{2}\frac{\langle  y \rangle}{y^\gamma}  \right\|_{L^4_\rho}  < +\infty.  $$
\fi
In addition, from the fact that $(\varepsilon, b, \beta)(\tau) \in V_1[A,\eta, \tilde \eta](\tau), \forall \tau \in [\tau_0, \tau^*]$ and \eqref{estimate-varepsilon--}, we  have 
\begin{eqnarray*}
	\left|\varepsilon_+(y) \right| \le \frac{A b^{\frac{\alpha}{2}  +\eta}(\tau) \langle y \rangle^{4}}{y^\gamma} + \frac{b^\frac{\alpha}{2}y^{2}\langle y \rangle^{2}}{y^\gamma}, \text{ and } \left|\varepsilon_-(y) \right| \le \frac{A^4 b^{\frac{\alpha}{2} + \tilde \eta }(\tau) \langle y \rangle^{4}}{y^\gamma},
\end{eqnarray*}
which yields
\begin{eqnarray*}
	\left|   \left\langle B(\varepsilon), \varepsilon_- \right\rangle_{L^2_{\rho_\beta}}   \right| \le  Cb^{\alpha + 3 \eta}.
\end{eqnarray*}
Thus,  we finish the proof of \eqref{esti-partial-varp--varep--}.  In particular, using  \eqref{estimate-beta-derive-tau} and \eqref{estimate-varepsilon--},  we get
\begin{eqnarray}
	\left|\beta' \left\langle  \varepsilon_- , \varepsilon_- \left( \frac{d+2}{2\beta} -\frac{y^2}{2} \right)  \right\rangle_{L^2_{\rho_\beta}} \right| \le b^{\alpha +3 \eta}, \forall \tau \in [\tau_0, \tau^*],
\end{eqnarray}
which  implies \eqref{esti-partialbeta--varp--varep--beta'}. Finally, by combining  \eqref{esti-partial-varp--varep--} and  \eqref{esti-partialbeta--varp--varep--beta'} we deduce \eqref{estima-derive-estima-epsilon---new} and then the proof of the Lemma follows.
\end{proof}

\medskip

\subsection{$L^\infty$ bounds}

In order to handle the nonlinear term in the $L^2_{\rho}$-energy estimate, we used the control of a weighted $L^{\infty}$-norm of $\varepsilon_-$. The rest of the section is devoted to it.
In the next step, we aim to give \textit{a priori} estimates to the infinite  part, $\varepsilon_-$. More precisely, we have the following proposition:  
\begin{lemma}[Control of the infinite dimensional part]\label{lemma-priori-estima-varep--}  Then, there exists $A_4 \ge 1$ such that for $A \ge A_3, \delta \ll 1$,  there exists $\eta_4(A,\delta) \ll 1$ such that for all $\eta \le \eta_4$, there exists $ \tilde{\eta}_4(A,\eta) \ll \eta   $ such that for all $\tilde  \eta  \le \tilde \eta_4$,  there exists $ \tau_4(A, \eta, \tilde \eta) \ge 1$, such that for all $ \tau_0 \ge \tau_5$, the following holds: assume that initial data is defined as in  \eqref{defi-initial-vaepsilon-l=1}    and the solution $(\varepsilon, b, \beta)(\tau) \in  V_1[A, \eta, \tilde \eta](\tau), \forall \tau \in [\tau_0,\tau^*]$, for some $\tau^* \ge \tau_0$ then we have the following
\begin{equation}
	\left\|   \frac{y^\gamma}{\langle y \rangle^{4}}  \varepsilon_-(., \tau)  \right\|_{L^\infty\left[  0, b^{- \tilde \eta} (\tau)  \right]}  \le \frac{A^3}{2}  I^{\frac{\alpha}{2} +\tilde \eta}(\tau), \forall \tau \in [\tau_0,\tau_1].
\end{equation} 
\end{lemma}
\begin{proof}
The proof relies on the maximum principal for the control near the origin i.e. $\left[0,b^\frac{\eta}{4} \right]$, and   pointwise estimates on $ \left[ b^\frac{\eta}{4}, b^{-\tilde \eta} \right]$.

$a)$ Let us consider $y \in \left[0,b^\frac{\eta}{4} \right]$. We apply  Proposition \ref{sub-super-solution} to obtain 
\begin{eqnarray*}
	\left|   \varepsilon(y,\tau)       \right|  \le  b^{-1}(\tau)  H \left(  \frac{y}{\sqrt{b(\tau)}}\right)  \le \frac{C b^{\frac{\alpha}{2} + \frac{\eta}{4}}(\tau) \langle y \rangle^{4}}{y^\gamma}, \, \text{for all} y \in \left[ 0,b^\frac{\eta}{4} \right].
\end{eqnarray*}
In addition, $\varepsilon_+$  can be estimated by
\begin{equation}
	\left| \varepsilon_+(y, \tau) \right| \le  \left|\varepsilon_1(\tau) \right| \left|  \frac{\phi_{1, b, \beta}}{c_{1,0}} - \phi_{0, b, \beta}  \right|.
\end{equation}
On the one hand, we use the  pointwise estimates given in Proposition \ref{propo-mathscr-L-b} 
\begin{eqnarray*}
	\left|  \frac{\phi_{1, b, \beta}}{c_{1,0}} - \phi_{0, b, \beta}  \right| \le \left|  \frac{c_{1,1}}{c_{1,0}} (\sqrt{b})^{2-\gamma}  T_1\left( \frac{y}{\sqrt{b}} \right)\right|  + \left|  \tilde \phi_{1,b, \beta}\right| + \left| \tilde \phi_{0,b,\beta}\right| \le \frac{Cb^{ 1- \frac{\epsilon}{2}}(\tau) \langle y \rangle^{2}}{y^\gamma}.
\end{eqnarray*}
On the other hand, from the compatibility 
$$ \varepsilon_1 (\tau) = - \frac{2}{\alpha} m_0 b^{\frac{\alpha}{2}} (\tau), $$
we deduce that
$$ \left| \varepsilon_+(y,\tau)   \right| \le \frac{C b^{\frac{\alpha}{2} +\frac{\eta}{2}}(\tau) \langle y \rangle^{4} }{y^\gamma}.$$
Thus, we obtain 
\begin{eqnarray}
	\sup_{y \in [0, b^{\frac{\eta}{4}}(\tau)]}  \frac{y^\gamma}{\langle
		y \rangle^{2\ell+2}} \left| \varepsilon_-(y,\tau) \right| \le CA  b^{\frac{\alpha}{2} +\frac{\eta}{4}}(\tau) \le \frac{A^3}{2} b^{\frac{\alpha}{2} + \tilde \eta}(\tau),
\end{eqnarray} 
provided that $A \ge A_4$.

$b)$  Let us consider the control on $\left[ b^{\frac{\eta}{4}}, b^{-\tilde \eta}(\tau) \right]$. On this domain, we are far the origin so we can not  use the spectrum properties of $\mathscr{L}_\infty$. The idea is inspired from \cite{BSIMRN19}. We are going to use pointwise estimates based on the semi-group. As for $\beta=\frac{1}{2}$, $\mathscr{L}_\infty$ has explicit structure. We introduce the   basis of $L^{\infty}$ 
\begin{eqnarray*}
	\phi_{0,\infty}  = \phi_{0,\infty, \frac{1}{2}} \text{ and }   \phi_{1, \infty} = \phi_{1, \infty, \frac{1}{2}},
\end{eqnarray*}
and for all $j \ge 2$ we renormalize  as in  \cite[Lemma 3.4]{BSIMRN19} that
$$ \phi_{j,\infty}(y) = \mathcal{N}_j y^{-\gamma} L_j^{\left(\frac{d}{2} -\gamma\right)} \left( \frac{y^2}{4} \right),  $$
where $L_j^{\nu}$  denoted by  the generalized Laguerre polynomial, and  the renormalisation constant $\mathcal{N}_j$ ensures that  $\|\phi_{j,\infty}\|_{L^2_{\rho}} =1$ and 
\begin{eqnarray*}
	\phi_{j,\infty} (y)= \left\{ \begin{array}{rcl}
		\alpha_j y^{-\gamma} \left(1+o(1) \right)  & \text{ as } & y \to 0  \\[0.2cm]
		\beta_j y^{2j-\gamma} (1+ o(1))   & \text{ as } & y \to +\infty,
	\end{array} \right. 
\end{eqnarray*}
with $\alpha_j$ and  $ \beta_j$ satisfies 
$$  \alpha_j \sim j^\frac{\omega}{4} \text{ and } \beta_n = \frac{j^{-\frac{\omega}{4}}}{4^jj!} \text{as } j \to +\infty.   $$

The pointwise estimates given in  Proposition \ref{propo-mathscr-L-b}  ensures  that  $\phi_{j, b, \beta}$ is very  close  to   $\phi_{j,\infty,\beta}$ on this interval  by the following 
\begin{eqnarray}
	\left| \phi_{j,b(\tau), \beta} (y)- \phi_{j,\infty, \beta}(y)     \right|   \lesssim \frac{b^{\frac{\eta}{2}}(\tau) \langle y \rangle^{4} }{y^\gamma} ,\, \forall y \ge b^{\frac{\eta}{4}}, j \le 1. \label{estimate-phi-j-b-phi-j-infty}
\end{eqnarray}
In addition, the condition 
$$\left| \beta(\tau) -\frac{1}{2}  \right| \le A I^{\eta}(\tau_0), $$
defined in the Shrinking set $V_1[A, \eta, \tilde \eta](\tau)$  shows that 
$\phi_{j, \infty, \beta}$ is close to $\phi_{j, \infty, \frac{1}{2}}:= \phi_{j, \infty}$ since
for all $j $
\begin{eqnarray}
	\left| \phi_{j, \infty, \beta } (y) - \phi_{j, \infty, \frac{1}{2}} (y)    \right| & \lesssim &  \left|\beta(\tau) -\frac{1}{2} \right|\frac{\langle y \rangle^{4}}{y^\gamma},\label{phi-beta-near-phi-1-2}\\
	\left| e^{-\frac{(2\beta) y^2}{4}} - e^{-\frac{y^2}{4}} \right|  & \lesssim &  \left| \beta(\tau) -\frac{1}{2} \right| \frac{y^2}{4} e^{-\frac{y^2}{8}},\label{rho-beta-near-rho-1-2}
\end{eqnarray}
since for all $\alpha$, it holds  $\left| e^{\alpha}  - 1 \right| \le C \alpha e^{\alpha} $, we have 
\begin{eqnarray}
	\left| \hat{\varepsilon}_1(\tau) \right| \lesssim b^\frac{\alpha}{2}(\tau), \label{estimate-hat-varep-1}
\end{eqnarray}
where $ \hat{\varepsilon}_j =  \| \phi_{j, \infty} \|^{-2}_{L^2_{\rho}} \langle \varepsilon, \phi_{j, \infty} \rangle_{L^2_{\rho}} $ is the projection of $\varepsilon$ on the basis  $\{ \phi_{j, \infty}, j \ge 0 \}$.
Hence, we use the semi-group pointwise estimates and we decompose $\varepsilon$ on the basis $\phi_{j,0}$
\begin{equation}\label{decomposition-singular}
	\varepsilon = \hat{\varepsilon}_+ + \hat{\varepsilon}_-. 
\end{equation}
%see more in \eqref{defi-hat-varpesilon+--}. 
Thus,   we will prove that
\begin{equation}\label{sup-y-delta--delta-I-hat-epsilon}
	\sup_{y \in  \left[   b^\frac{\eta}{4}(\tau), b^{-\tilde \eta} (\tau)     \right] }\left|  \frac{y^\gamma}{\langle y \rangle^{4}}  \hat{\varepsilon}_-(y,\tau)\right| \le \frac{A^3}{4} b^{\frac{\alpha}{2} + \tilde \eta}(\tau). 
\end{equation}
\iffalse
Firstly, we use \eqref{equa-varepsilon-appen} to derive 
\begin{equation}\label{equa-varepsilon--}
	\partial_\tau \varepsilon_- = \mathscr{L}_b \varepsilon_- + B(\varepsilon_+ + \varepsilon_-) + \Phi(\tau) - \partial_\tau \varepsilon_+ + \mathscr{L}_b \varepsilon_+,
\end{equation}
where $B$, $ \Phi$, and $\varepsilon_+$ defined as in \eqref{defi-B-quadratic-appendix}, \eqref{Phi-simple} and \eqref{defi-varepsilon-j}, respectively. 

\fi
Since $\varepsilon$ satisfies  \eqref{equa-varepsilon-appen}, $\hat{\varepsilon}_-$ solves
\begin{equation}\label{equa-hat-varepsilon--}
	\partial_\tau \hat{\varepsilon}_{-} = \mathscr{L}_\infty \hat{\varepsilon}_{ -}  +  \hat{B}( \hat{\varepsilon}_{+}  +  \hat{\varepsilon}_{-} ) + \left( \frac{1}{2}  -  \beta(\tau)  \right)  \Lambda_y \left(  \hat{\varepsilon}_{-} \right), 
\end{equation}
where $\mathscr{L}_\infty$ was defined in \eqref{defi-operator-L-infty} by taking $\beta=\frac{1}{2}$ and  $\hat{B}$ is defined by  
\begin{eqnarray}
	\hat{B}  &=&  -3(d-2) \left[2Q_b +y^2 Q_b^2 + \frac{1}{y^2}   \right]( \hat \varepsilon_+ + \hat{\varepsilon}_-)  + B(\hat{\varepsilon}_+ +\hat{\varepsilon}_-)   \label{defi-hat-B}\\
	& + &      \Phi(\tau)  + \mathscr{L}_\infty^{\beta} (\hat{\varepsilon}_{+} ) -\partial_\tau
	\hat{\varepsilon}_{ +} , \nonumber
\end{eqnarray}
with  $B$ and $\Phi$   defined in \eqref{defi-B-quadratic-appendix} and \eqref{Phi-simple}, respectively. By using Duhamel's formula, we get 
\begin{equation}\label{Duhamel-hat-varepsilon--}
	\hat{\varepsilon}_{ -}(\tau) = e^{(\tau-\tau_0)\mathscr{L}_\infty} \hat{\varepsilon}_{ -}(\tau_0) + \int_{\tau_0}^\tau  e^{(\tau -\tau') \mathscr{L}_\infty} \left[ \hat{B}(\hat \varepsilon_{ +} + \hat{\varepsilon}_{ -})  +\left(\frac{1}{2} -\beta(\tau') \right)\Lambda_y \hat\varepsilon_{-} \right](\tau') d\tau'.
\end{equation}
In addition, we denote $f_-$ as the part of $f$ which is orthogonal to $\phi_{0,\infty}$ and $\phi_{1,\infty}$. Then
$$ f_- (y) = f -   \sum_{j=0}^1 \langle f , \phi_{j,\infty} \rangle_{L^2_\rho} \phi_{j,\infty}. $$
In particular, if the series
$$ \sum_{j=0}^\infty  \langle f , \phi_{j,\infty} \rangle_{L^2_\rho} \phi_{j,\infty} $$
is convergent and well defined, then we can define $f_-$ pointwisely as
$$ f_- (y)   = \sum_{j=2}^\infty  \langle f , \phi_{j,\infty} \rangle_{L^2_\rho} \phi_{j,\infty}.$$
Since $\varepsilon_-$ is orthogonal to $\phi_{0,\infty}$ and $\phi_{1,\infty}$, we can write
\begin{eqnarray}
	\varepsilon_- &=& \left(e^{(\tau-\tau_0)\mathcal{L}_\infty}(\varepsilon_-(\tau_0)) \right)_- \label{express-varep---}  \\
	& + & \int_{\tau_0}^{\tau} \left(e^{(\tau-\tau')\mathscr{L}_\infty}(\hat{B}(\tau'))   \right)_-d\tau' +\int_{\tau_0}^{\tau}  \left( \left[\frac{1}{2} -\beta(\tau') \right] \Lambda \hat{\varepsilon}_-(\tau') \right)_-d\tau'.\nonumber 
\end{eqnarray}
We remark that \eqref{sup-y-delta--delta-I-hat-epsilon}  immediately  follows from 
\begin{eqnarray}
	\sup_{y \in  \left[   b^\frac{\eta}{4}(\tau), b^{-\tilde \eta}(\tau)\right]}\left|  \frac{y^\gamma}{\langle y \rangle^{4}}  \left(e^{(\tau-\tau_0)\mathscr{L}_\infty} \hat{\varepsilon}_{-}(\tau_0) \right)_-(y,\tau)\right| \le \frac{A^3}{16} b^{\frac{\alpha}{2} + \tilde \eta}(\tau), \label{estima-Duhame-hat-epsilon-tau-0}
	\\
	\sup_{y \in  \left[    b^\frac{\eta}{4}(\tau), b^{-\tilde \eta} (\tau) \right] }\left|  \frac{y^\gamma}{\langle y \rangle^{4}} \int_{\tau_0}^\tau  \left( e^{(\tau -\tau') \mathscr{L}_\infty} \left[ \hat{B}(\tau')   \right]  \right)_-d\tau' \right| \le \frac{A^3}{16} b^{\frac{\alpha}{2} + \tilde \eta}(\tau),\label{estima-Duhame-integral-hat-B}\\
	\sup_{y \in  \left[    b^\frac{\eta}{4}(\tau), b^{-\tilde \eta} (\tau) \right] }\left|  \frac{y^\gamma}{\langle y \rangle^{4}} \int_{\tau_0}^\tau  \left(e^{(\tau -\tau') \mathscr{L}_\infty} \left[ \frac{1}{2} -\beta(\tau')   \right] \Lambda_y ( \hat{\varepsilon}_{ - } ) \right)_-(\tau') d\tau' \right| \le \frac{A^3}{16} b^{\frac{\alpha}{2} + \tilde \eta}(\tau)\label{estima-Duhame-integral-1-2-Lambda_varep-beta-}.
\end{eqnarray}

To  prove  estimates (\eqref{estima-Duhame-hat-epsilon-tau-0} -\eqref{estima-Duhame-integral-1-2-Lambda_varep-beta-}),       we need to consider different  cases as \\

- The first case, we consider $\tau - \tau_0 \le \frac{\ln A}{K_0}$ \\

- The second case, we consider $\tau - \tau_0 > \frac{\ln A}{K_0}$. In addition, the second will  be  divided again by two sub-cases that  $ \frac{1}{L_0} e^{\frac{\tau-\tau_0}{2} \left(1 -\eta \left(\frac{2\ell}{\alpha} - 1 \right) \right)} \le b^{ -\tilde \eta}(\tau)  \text{  and  }  \frac{1}{L_0} e^{\frac{\tau-\tau_0}{2} \left(1 -\eta \left(\frac{2\ell}{\alpha} - 1 \right) \right)} > b^{ -\tilde \eta}(\tau)$ and in these sub-cases  also includes some smaller case that there are some large constant $L_0, K_0, R $ appear which are  fixed at the end of the proof. Let us go to the details of the proof.

\begin{center}
	\textbf{  First case $\tau - \tau_0 \le \frac{\ln A}{K_0}$   }
\end{center}

\medskip
-\textbf{ Proof of \eqref{estima-Duhame-hat-epsilon-tau-0} }:
Note that $K_0 \gg 1$ will be fixed at the end of the proof. Now, we deduce from  
\eqref{defi-initial-vaepsilon-l=1} in accordance with the decomposition \eqref{decomposition-singular}, we arrive at 
\begin{equation}\label{estimate-hat-varpe--y-gamma}
	|\hat{\varepsilon}_-(\tau_0) y^\gamma |      \le CA b^{\frac{\alpha}{2} +\eta}(\tau_0) \langle y\rangle^{4}, 
\end{equation}
where $b(\tau_0) = I(\tau_0)$ defined in \eqref{defi-I-tau}. 
Since  $\langle y\rangle^{4}$ is increasing, then, we apply Lemma  \ref{maximal-lemma} and we  obtain
\begin{eqnarray}
	\left|  e^{(\tau -\tau_0) \mathscr{L}_\infty} \hat{\varepsilon}_{ -}(\tau_0)  (y,\tau) \right| &\le & y^{-\gamma} e^{\frac{\alpha(\tau-\tau_0)}{2}} M(\hat{\varepsilon}_{ -}(\tau_0))(y) \label{estimate-math-infty-varep-beta-tau-0} \\
	& \le & CA e^{\frac{\alpha(\tau-\tau_0)}{2}} {b}^{\frac{\alpha}{2} +\eta}(\tau_0) y^{-\gamma}\frac{\int_y^\infty \langle y' \rangle^{4}  (y')^{1+\omega} e^{-\frac{(y')^2}{4}}dy'}{\int_y^\infty  (y')^{1+\omega} e^{ -\frac{(y')^2}{2}} dy'   }\nonumber \\
	&  \le  & C A e^{\frac{\alpha}{2}(\tau -\tau_0) }  b^{\frac{\alpha}{2}  +\tilde \eta}(\tau) b^{- \frac{\alpha}{2} -\tilde \eta} (\tau) b^{\frac{\alpha}{2} +\eta}(\tau_0)  y^{-\gamma}\langle y \rangle^{4}.\nonumber
\end{eqnarray}
Using  \eqref{esti-b-equivalent-I-1-tilde-eta}, we get
\begin{eqnarray}
	e^{\frac{\alpha}{2}(\tau -\tau_0)} b^{\frac{\alpha}{2} +\eta} (\tau_0) b^{-\frac{\alpha}{2} -\tilde \eta}(\tau) &\le & Ce^{\frac{\alpha}{2}(\tau -\tau_0)} e^{\left(1 -\frac{2}{\alpha} \right)\left( \left(\frac{\alpha}{2 } +\delta \right) (1 -\frac{\tilde \eta}{10})\tau_0 - \left(\frac{\alpha}{2 } +\tilde \eta \right) (1 + \frac{\tilde \eta}{10})\tau  \right)}\label{prcoess-filter-b-to-I}\\
	& \le & Ce^{-c(\eta)\tau_0}e^{\frac{\alpha}{2}(\tau -\tau_0)} e^{\left(1 -\frac{2}{\alpha} \right)\left( \left(\frac{\alpha}{2 } +\tilde \eta \right) (1 + \frac{\tilde \eta}{10})\tau_0 - \left(\frac{\alpha}{2 } +\tilde \eta \right) (1 + \frac{\tilde \eta}{10})\tau  \right)} \nonumber \\
	& \le & C e^{-c(\eta)\tau_0} e^{ \left(\frac{\alpha}{2} + \left( \frac{2 }{\alpha} - 1  \right) \left( \frac{\alpha}{2} +\tilde \eta \right)(1 +\frac{\tilde \eta}{10})  \right)  \left( \tau - \tau_0 \right)  } \nonumber \\
	&\le &  Ce^{-c(\eta)\tau_0} A^{\left(\frac{\alpha}{2} + \left( \frac{2  }{\alpha} - 1  \right) \left( \frac{\alpha}{2} +\tilde \eta \right)(1 +\tilde \eta)  \right) \frac{1}{K_0}},  \nonumber \text{ for some  } c(\eta) >0,
\end{eqnarray}
which yields
\begin{eqnarray}
	\left| \frac{y^\gamma}{ \langle y \rangle^{4}} e^{(\tau -\tau_0) \mathscr{L}_\infty} \hat{\varepsilon}_-(\tau_0)  (y,\tau) \right| &  \le  & \frac{A^3}{16} b^{\frac{\alpha}{2} +\tilde \eta}(\tau),\label{asti-tau-tau-0vareo-0}
\end{eqnarray}
provided that $K_0 \geq K_4,  A \ge A_4$. Finally, we conclude  
\eqref{estima-Duhame-hat-epsilon-tau-0}.

\medskip
\noindent
-Proof of \eqref{estima-Duhame-integral-hat-B}: for $\tau' \in [\tau_0,\tau]$, we apply  Lemma \ref{maximal-lemma} to get 
\begin{eqnarray}
	\left|e^{(\tau-\tau')\mathscr{L}_\infty}  [ \hat{B}](\tau') \right|   \le   Cy^{-\gamma} e^{\frac{\alpha(\tau-\tau')}{2}} \left\{ M(\mathbbm{1}_{(0,b^{\delta}(\tau')]}\hat B) 
	+ M(\mathbbm{1}_{y \ge b^\delta(\tau')}\hat B ) \right\}. \label{esti-semi-hat-B}
\end{eqnarray} 
To evaluate $ M(\mathbbm{1}_{(0,b^{\delta}(\tau')]}\hat B)(\tau') $,  we  apply  the result in Lemma \ref{lemma-estimate-on-hat-B} to obtain
\begin{eqnarray*}
	M(\mathbbm{1}_{[0,b^\delta(\tau')]} \hat B(\tau')) & \le &  C\left[ b^{\frac{\alpha}{2}}(\tau') M( \mathbbm{1}_{(0,b^\delta(\tau')]} y^{-\gamma})  + A^3 b^{\frac{\alpha}{2} +\tilde{\eta}}(\tau') M( \mathbbm{1}_{(0,b^\delta(\tau')]} y^{-\gamma-2}) \right.\\
	& + & \left. A^8 b^{ 2\left(\frac{\alpha}{2} +\tilde{\eta} \right)}(\tau') M(\mathbbm{1}_{(0,b^\delta(\tau')]}y^{-2\gamma})  +A^{12}b^{ 3\left(\frac{\alpha}{2} +\tilde{\eta} \right)}(\tau') M(\mathbbm{1}_{[0,b^\delta(\tau')]}y^{-3\gamma +2})   \right].
\end{eqnarray*}
For the first term on the right hand side of the above  inequality, we rewrite from \eqref{maxinum-M-f-y}   
\begin{eqnarray*}
	M( \mathbbm{1}_{[0,b^\delta(\tau')]} y^{-\gamma}) = \sup_{y \in \mathcal{I}} \frac{ \int_{\mathcal{I}} |\mathbbm{1}_{[0,b^\delta(\tau')]}| (y')^{1+\omega}  e^{-\frac{(y')^2}{4} } dy' }{\int_{\mathcal{I}} (y')^{1 +\omega} e^{-\frac{(y')^2}{4} } dy' } ,
\end{eqnarray*}
since $\mathbbm{1}_{[0,b^\delta(\tau')]} $ is   non increasing,  then, we apply the result in Lemma \ref{maximal-lemma} to get
\begin{eqnarray*}
	M( \mathbbm{1}_{(0,b^\delta(\tau')]} y^{-\gamma}) &\le &   \frac{\int_0^{b^\delta(\tau')} |\mathbbm{1}_{[0,b^\delta(\tau')]}| (y')^{1+\omega}  e^{-\frac{(y')^2}{4} } dy' }{\int_0^y(y')^{1 +\omega} e^{-\frac{(y')^2}{4} } dy' } 
	\lesssim   \frac{ (b^\delta(\tau'))^{2+\omega}   }{\int_0^y(y')^{1 +\omega} e^{-\frac{(y')^2}{4} } dy' }.  
\end{eqnarray*}
Besides that, once  $ y \ge 1$, it follows that
$$ \int_0^y(y')^{1 +\omega} e^{-\frac{(y')^2}{4} } dy' \ge C,    $$
which yields
\begin{eqnarray*}
	\left( \int_0^y(y')^{1 +\omega} e^{-\frac{(y')^2}{4} } dy'\right)^{-1}   \lesssim \frac{\langle y\rangle^{2+\omega}}{y^{2+\omega}}.
\end{eqnarray*}
Otherwise, once  $y \le 1$, we have 
\begin{eqnarray*}
	\left( \int_0^y(y')^{1 +\omega} e^{-\frac{(y')^2}{4} } dy'\right)^{-1} \lesssim \left( \int_0^y(y')^{1 +\omega}  dy'\right)^{-1} \lesssim \frac{\langle y\rangle^{2+\omega}}{y^{2+\omega}}.
\end{eqnarray*}
Then, we derive 
\begin{eqnarray*}
	M( \mathbbm{1}_{(0,b^\delta(\tau')]} y^{-\gamma}) &\le &   \frac{\int_0^{b^\delta(\tau')} |\mathbbm{1}_{[0,b^\delta(\tau')]}| (y')^{1+\omega}  e^{-\frac{(y')^2}{4} } dy' }{\int_0^y(y')^{1 +\omega} e^{-\frac{(y')^2}{4} } dy' } 
	\lesssim  \frac{(b^\delta(\tau'))^{2+\omega}\langle y\rangle^{2+\omega}}{y^{2+\omega}}.
\end{eqnarray*}
Similarly, we  have
\begin{eqnarray*}
	M( \mathbbm{1}_{[0,b^\delta(\tau')]} y^{-\gamma-2}) \lesssim   \frac{ (b^\delta(\tau'))^{\omega} \langle y \rangle^{2+\omega}}{ y^{2 +\omega}},
	\text{ and }
	M( \mathbbm{1}_{(0,b^\delta(\tau')]} y^{-2\gamma})   \lesssim   \frac{ (I^\delta(\tau'))^{\omega +2-\gamma} \langle y \rangle^{2+\omega}  }{ y^{2 +\omega}}
\end{eqnarray*}
and
\begin{eqnarray*}
	M( \mathbbm{1}_{(0,b^\delta(\tau')]} y^{-3\gamma +2})   \lesssim   \frac{ (b^\delta(\tau'))^{\omega + 4 - 2 \gamma} \langle y \rangle^{2+\omega}  }{ y^{2 +\omega}}.
\end{eqnarray*}
Combining all the related  terms with the condition that $ y \in \left[ b^{\frac{\eta}{4}}(\tau), b^{-\tilde \eta} \right]$,    we deduce
\begin{eqnarray}
	\left| e^{(\tau-\tau')\mathscr{L}_\infty} \left[ \mathbbm{1}_{(0,b^\delta(\tau')]} \hat B \right] (\tau')\right| &\lesssim & y^{-\gamma} e^{\frac{\alpha}{2}(\tau-\tau')}  b^{\delta \left( \omega +4-2\gamma\right)}(\tau')
	\frac{\langle y \rangle^{\omega+2}}{y^{\omega+2}} 
	\lesssim   y^{-\gamma} e^{\frac{\alpha}{2}(\tau-\tau')}  b^{\delta \left( \omega +4-2\gamma\right) -\frac{\eta}{4}( \omega + 2) }(\tau') \nonumber\\
	& \lesssim &  \frac{b^{\frac{\alpha}{2} + \tilde \eta}(\tau) \langle y \rangle^{4}}{y^\gamma} \left[ e^{\frac{\alpha}{2}(\tau-\tau')}  b^{\delta \left( \omega +4-2\gamma\right) -\frac{\eta}{4}( \omega + 2) }(\tau') b^{-\frac{\alpha}{2} -\tilde \eta}(\tau) \right]. \label{semi-group-1-le-b-delta-hat-B}
\end{eqnarray}
In addition,  by the same argument used in  \eqref{prcoess-filter-b-to-I} and \eqref{asti-tau-tau-0vareo-0} and that fact that $ \tau-\tau_0 \le \frac{\ln A}{K_0}$,   we have 
\begin{eqnarray*}
	\int_{\tau_0}^\tau \left[ e^{\frac{\alpha}{2}(\tau-\tau')}  b^{\delta \left( \omega +4-2\gamma\right) -\frac{\eta}{4}( \omega + 2) }(\tau') b^{-\frac{\alpha}{2} -\tilde \eta}(\tau)\right]d\tau' \lesssim A^{\frac{C(\delta, \eta, \tilde \eta)}{K_0}}.
\end{eqnarray*}
Finally, we conclude 
\begin{eqnarray*}
	\sup_{y \in  \left[    b^\frac{\eta}{4}(\tau), b^{-\tilde \eta} (\tau) \right] }\left|  \frac{y^\gamma}{\langle y \rangle^{2\ell+2}} \int_{\tau_0}^\tau  e^{(\tau -\tau') \mathscr{L}_\infty} \left| \mathbbm{1}_{(0,b^\delta(\tau')]}\hat{B}(\tau') \right|  d\tau' \right| \le \frac{A^3}{16} b^{\frac{\alpha}{2} + \tilde \eta}(\tau),
\end{eqnarray*}
provided that $K_0 \ge K_4$,  $A \ge A_5$, and $ \tau_0 \ge \tau_5(A, K_0, \delta, \eta, \tilde \eta)$. 

\medskip
\noindent 
It remains to  evaluate $M(\mathbbm{1}_{y\ge b^{\delta}(\tau')}\hat B (\tau'))$.     Using  Lemma \ref{lemma-estimate-on-hat-B} with $\ell =1$,   we get
\begin{eqnarray*}
	M(\mathbbm{1}_{y\ge b^{\delta}(\tau')}\hat B (\tau')) &\lesssim &  b^{\frac{\alpha}{2}(\tau') +4\eta}  M \left( \mathbbm{1}_{y\ge b^{\delta}(\tau')} \langle y \rangle^{4} y^{-\gamma} \right)  + b^{\alpha +\delta(1-\gamma)}   M \left( \mathbbm{1}_{y\ge b^{\delta}(\tau')} \langle y \rangle^{12} y^{-\gamma} \right) \\
	& + & b^{\frac{3\alpha}{2}-2\delta\gamma}(\tau') M \left( \mathbbm{1}_{y\ge b^{\delta}(\tau')} \langle y \rangle^{20} y^{-\gamma} \right) .  
\end{eqnarray*}
First, we observe that the function $ \mathbbm{1}_{y \ge b^{\delta}(\tau')} \langle y \rangle^{4} $ is non decreasing, we apply Lemma \ref{maximal-lemma} and we have
\begin{eqnarray*}
	M\left(\mathbbm{1}_{y \ge b^\delta(\tau')} \langle y \rangle^{4} y^{ - \gamma} \right) \lesssim 
	\frac{\int_y^\infty \langle y'  \rangle^{4} (y)'^{1+\omega} e^{-\frac{(y')^2}{4} dy'}  }{\int_y^\infty (y')^{1+\omega} e^{-\frac{(y')^2}{4} dy'} }.
\end{eqnarray*}
From a standard result on $\Gamma$ function,   we have  
\begin{eqnarray*}
	\frac{\int_y^\infty \langle y'  \rangle^{4} (y')^{1+\omega} e^{-\frac{(y')^2}{4} dy'}  }{\int_y^\infty (y')^{1+\omega} e^{-\frac{(y')^2}{4} dy'} }  &\lesssim &  \langle y \rangle^{4},
\end{eqnarray*}
which implies 
\begin{eqnarray*}
	M\left(\mathbbm{1}_{y \ge b^\delta(\tau')} \langle y \rangle^{2\ell+2} y^{ - \gamma} \right) \lesssim  \langle y \rangle^{2\ell+2}.
\end{eqnarray*}
Similarly, from  Lemma \ref{lemma-estimate-on-hat-B} and the fact  $y \le b^{-\tilde \eta}(\tau)$, we  write 
\begin{eqnarray*}
	M\left(\mathbbm{1}_{y \ge b^\delta(\tau')} \langle y \rangle^{4+8} y^{ - \gamma} \right) &\lesssim &  \langle y \rangle^{4+8} \lesssim b^{-\tilde \eta(2+6)}(\tau) \langle y \rangle^{4},\\
	M\left(\mathbbm{1}_{y \ge b^\delta(\tau')} \langle y \rangle^{6+14} y^{ - \gamma} \right) & \lesssim &  \langle y \rangle^{6+14} \lesssim b^{-\tilde \eta(4+12)}(\tau) \langle y \rangle^{4}.
\end{eqnarray*}
Thus, we derive  for all $ y \in \left[ b^{\eta}(\tau), b^{-\tilde \eta}(\tau) \right]$ 
\begin{eqnarray}
	& & \left| e^{(\tau-\tau')\mathscr{L}_\infty} \left[ \mathbbm{1}_{  y \ge b^\delta(\tau')} \hat B \right] (\tau')\right| \label{semi-group-1-ge-b-delta-hat-B} \\
	&\lesssim & y^{-\gamma} \langle y\rangle^{2+2} e^{\frac{\alpha}{2}(\tau -\tau')} \left[ b^{\frac{\alpha}{2} +4\eta}(\tau') + b^{\alpha + \delta(1 -\gamma)}(\tau') b^{-\tilde \eta(2+6)}(\tau)  + b^{\frac{3\alpha}{2} -2\delta \gamma }(\tau') b^{-\tilde \eta (4 +12)}(\tau) \right],\nonumber 
\end{eqnarray}
which implies
\begin{eqnarray}
	& &\frac{y^\gamma }{\langle y\rangle^{2+2}} \int_{\tau_0}^\tau    e^{(\tau -\tau') \mathscr{L}_\infty} M(\mathbbm{1}_{y \ge b^\delta(\tau')} \hat B(\tau')) d\tau' \label{estimate-integral-M-y-ge-hat-B} \\
	& \lesssim & b^{\frac{\alpha}{2} + \tilde \eta}(\tau)   \int_{\tau_0}^{\tau}  e^{\frac{\alpha}{2}(\tau-\tau')} b^{-\frac{\alpha}{2} -\tilde \eta}(\tau) \left[ b^{\frac{\alpha}{2} +4\eta}(\tau') + b^{\alpha + \delta(1 -\gamma)}(\tau') b^{-\tilde \eta(2+6)}(\tau)  + b^{\frac{3\alpha}{2} -2\delta \gamma }(\tau') b^{-\tilde \eta (4 +12)}(\tau) \right]d\tau'. \nonumber
\end{eqnarray}
From the assumption $ \tau - \tau_0 \le \frac{\ln A}{K_0}$ with  $K_0 $ large enough, $A \ge A_4$ and $  \alpha  \gg \delta \gg \eta \gg \tilde \eta $,  and $\tau_0 \ge \tau_4(A,K_0, \delta, \eta, \tilde \eta)$, we proceed similarly as
in \eqref{prcoess-filter-b-to-I} and \eqref{asti-tau-tau-0vareo-0} and we obtain 
\begin{eqnarray*}
	\int_{\tau_0}^{\tau}  e^{\frac{\alpha}{2}(\tau-\tau')} b^{-\frac{\alpha}{2} -\tilde \eta}(\tau) \left[ b^{\frac{\alpha}{2} +4\eta}(\tau') + b^{\alpha + \delta(1 -\gamma)}(\tau') b^{-\tilde \eta(2\ell+6)}(\tau)  + b^{\frac{3\alpha}{2} -2\delta \gamma }(\tau') b^{-\tilde \eta (4\ell +12)}(\tau) \right]d\tau'  \le \frac{A^3}{32}.
\end{eqnarray*}
Finally, we get
\begin{eqnarray*}
	\frac{y^\gamma }{\langle y\rangle^{4}} \int_{\tau_0}^\tau    e^{(\tau -\tau') \mathscr{L}_\infty} M(\mathbbm{1}_{y \ge b^\delta(\tau')} \hat B(\tau')) d\tau'  \le \frac{A^3}{32} b^{\frac{\alpha}{2} +\tilde \eta}(\tau),
\end{eqnarray*}
and   \eqref{estima-Duhame-integral-hat-B} immediately follows.

- The proof of \eqref{estima-Duhame-integral-1-2-Lambda_varep-beta-}:  
We first recall the following  identity 
$$  \varepsilon_+ + \varepsilon_- = \varepsilon = \hat{\varepsilon}_+ +\hat{\varepsilon}_-,  $$
then, we get
\begin{equation*}
	\hat{\varepsilon}_-(\tau')  =  \varepsilon_+(\tau') + \varepsilon_-(\tau') + \hat{\varepsilon}_+(\tau').
\end{equation*}
Since $(\varepsilon,b,\beta)(\tau) \in V_1[A,\eta,\tilde \eta](\tau_1), \forall \tau \in [\tau_0,
\tau_1] $,  the pointwise estimates given in Lemma \eqref{lemma-rough-estimate-bounds in shrinking-set}  and 
also \eqref{estimate-hat-varep-1}  hold, so we get  a rough estimate for all $\tau' \in [\tau_0,\tau], \tau \le \tau_1$
\begin{eqnarray}
	\left| \left( \frac{1}{2} - \beta(\tau') \right) \hat{\varepsilon}_-(\tau')  \right| \le CA^6 b^{\frac{\alpha}{2}}(\tau') I^{\eta}(\tau_0) \frac{\langle y \rangle^{4}}{y^\gamma}.
\end{eqnarray}
% Considering $ f = \left( \frac{1}{2} - \beta(\tau') \right) \hat{\varepsilon}_-(\tau') \in L^2_{\rho_{\frac{1}{2}}} $, we have the fact $\Lambda f  \in L^2_{\rho_{\frac{1}{2}}}  $ since the regulaity parabolic of the semigroup $e^{t\Delta_{d+2}}$, see more details in Remark \ref{regularity-parabolic-of-solu}. 
Now, we apply Lemma \ref{Lemma-esti-semigroup-Lambda-var} and we obtain
\begin{eqnarray*}
	\left| e^{(\tau -\tau') \mathscr{L}_\infty} \left( \left( \frac{1}{2} - \beta(\tau') \right) \hat{\varepsilon}_-(\tau') \right)   \right|  \le C A^6 b^{\frac{\alpha}{2}}(\tau') I^{\eta}(\tau_0) \frac{\langle y \rangle^{6}}{y^\gamma}, \forall \tau' \in [\tau_0,\tau),
\end{eqnarray*}
which yields
\begin{eqnarray*}
	\left|  \int_{\tau_0}^\tau e^{(\tau -\tau') \mathscr{L}_\infty} \left( \left( \frac{1}{2} - \beta(\tau') \right) \hat{\varepsilon}_-(\tau') \right) d\tau'  \right|  \le CA^6 I^{\eta}(\tau_0)\frac{\langle y \rangle^6}{y^\gamma}\int_{\tau_0}^\tau b^{\frac{\alpha}{2}}(\tau') d\tau'.
\end{eqnarray*}
Using Lemma \ref{lemma-rough-estimate-bounds in shrinking-set} and a similar estimate as in  \eqref{prcoess-filter-b-to-I}, we derive 
\begin{eqnarray*}
	\left| A^6 I^{\eta}(\tau_0)\frac{\langle y \rangle^6}{y^\gamma}\int_{\tau_0}^\tau b^{\frac{\alpha}{2}}(\tau') d\tau' \right|  \le \frac{A^3}{16} b^{\frac{\alpha}{2} +\tilde \eta} (\tau), 
\end{eqnarray*}
provided that $K_0 \ge K_4, A \ge A_4, \tau_0 \ge \tau_4(K_0,A,\eta, \tilde \eta)$ and $ \tau -\tau_0 \le \frac{\ln A}{K_0}$. 
Finally, \eqref{estima-Duhame-integral-1-2-Lambda_varep-beta-} follows.

\begin{center}
	\textbf{The second case  $\tau -\tau_0 \ge \frac{\ln A}{K_0} $}
\end{center}
As we mentioned,  this case  will be divided into two sub-cases
$$   \frac{1}{L_0} e^{\frac{\tau-\tau_0}{2} \left(1 -\eta \left(\frac{2\ell}{\alpha} - 1 \right) \right)} \le b^{ -\tilde \eta}(\tau)  \text{  and  }  \frac{1}{L_0} e^{\frac{\tau-\tau_0}{2} \left(1 -\eta \left(\frac{2\ell}{\alpha} - 1 \right) \right)} > b^{ -\tilde \eta}(\tau),$$
where $L_0$ is  large enough.

\begin{center}
	\textit{ \textbf{ The first subcase $  \frac{1}{L_0} e^{\frac{\tau-\tau_0}{2} \left(1 -\eta \left(\frac{2}{\alpha} - 1 \right) \right)} \le b^{ -\tilde \eta}(\tau) $}} 
\end{center}  
From    \eqref{esti-b-equivalent-I-1-tilde-eta} and $  \frac{1}{L_0} e^{\frac{\tau-\tau_0}{2} \left(1 -\eta \left(\frac{2}{\alpha} - 1 \right) \right)} \le b^{ -\tilde \eta}(\tau) $, we get 
\begin{eqnarray}  
	\tau < \frac{\tau_0 \left(1 - \eta \left(1-\frac{2}{\alpha}  \right)  \right)  + 2\ln L_0}{1-\left( \eta - 2\tilde \eta \left( 1+\frac{\tilde \eta}{10}\right) \right)  \left(1-\frac{2}{\alpha}  \right)  } , 
\end{eqnarray}
which yields
\begin{eqnarray}
	\frac{\ln A}{K_0}\leq \tau  -\tau_0 \le \frac{   2\tilde \eta  \tau_0  \left( \frac{2}{\alpha} -1\right)\left(1 +\frac{\tilde \eta}{10} \right) + 2\ln L_0}{1-\left(\eta -2\tilde \eta \left(1+\frac{\tilde \eta}{10} \right)\right) \left(\frac{2}{\alpha} -1 \right)}.\label{tau-tau-0-le-eta-tilde}
\end{eqnarray}
for $L_0$ large enough. According to \eqref{tau-tau-0-le-eta-tilde}, we see that this the present sub-case can be handled similarly as the fist one, since  $\tau$  is not too far from $ \tau_0$.

- The proof of \eqref{estima-Duhame-hat-epsilon-tau-0}: From \eqref{estimate-math-infty-varep-beta-tau-0}, 
we have 
\begin{eqnarray*}
	\left|  e^{(\tau -\tau_0) \mathscr{L}_\infty} \hat{\varepsilon}_{ -}(\tau_0)  (y,\tau) \right| \le C A b^{\frac{\alpha}{2} + \tilde \eta} (\tau) \frac{\langle y \rangle^{4}}{y^\gamma} \left[  e^{\frac{\alpha}{2}(\tau-\tau_0)} b^{-\frac{\alpha}{2} -\tilde \eta}(\tau_0) b^{\frac{\alpha}{2} + \eta}(\tau_0)   \right]. 
\end{eqnarray*}
The same process used for  \eqref{prcoess-filter-b-to-I} yields
\begin{eqnarray*}
	e^{\frac{\alpha}{2}(\tau -\tau_0)} b^{\frac{\alpha}{2} +\eta} (\tau_0) b^{-\frac{\alpha}{2} -\tilde \eta}(\tau) \le e^{-c(\eta)\tau_0} e^{X (\tau-\tau_0) }, \text{ with } X =\left(\frac{\alpha}{2} + \left( \frac{2 }{\alpha} - 1  \right) \left( \frac{\alpha}{2} +\tilde \eta \right)(1 +\frac{\tilde \eta}{10})  \right).
\end{eqnarray*}
From \eqref{tau-tau-0-le-eta-tilde},  we can prove that there exists $c(\eta) >0$ such that 
\begin{eqnarray*}
	& &e^{-c(\eta)\tau_0} e^{X (\tau-\tau_0) }, \text{ with } X =\left(\frac{\alpha}{2} + \left( \frac{2 }{\alpha} - 1  \right) \left( \frac{\alpha}{2} +\tilde \eta \right)(1 +\frac{\tilde \eta}{10})  \right) \\
	& \lesssim  &   e^{ -c(\eta) \tau_0 + X \left(\frac{   2\tilde \eta  \tau_0  \left( \frac{2}{\alpha} -1\right)\left(1 +\frac{\tilde \eta}{10} \right) + 2\ln L_0}{1-\left(\eta -2\tilde \eta \left(1+\frac{\tilde \eta}{10} \right)\right) \left(\frac{2}{\alpha} -1 \right)} \right)    } \le 1 ,
\end{eqnarray*}
provided that $ \tilde \eta \le  \tilde \eta_4(\eta, L_0)$ and this gives \eqref{estima-Duhame-hat-epsilon-tau-0}.

- The proof of \eqref{estima-Duhame-integral-hat-B}:  we  use \eqref{semi-group-1-le-b-delta-hat-B} and \eqref{semi-group-1-ge-b-delta-hat-B} to get
\begin{eqnarray*}
	& &  \left|  \frac{y^\gamma}{\langle y \rangle^{2+2}} \int_{\tau_0}^\tau  e^{(\tau -\tau') \mathscr{L}_\infty} \left[ \hat{B}(\hat \varepsilon_{\beta, +}, \hat{\varepsilon}_{\beta, -})   \right](\tau') d\tau' \right| \\
	&\le&   C b^{\frac{\alpha}{2} +\tilde \eta}(\tau) \left\{ \int_{\tau_0}^\tau \left[ e^{\frac{\alpha}{2}(\tau-\tau')}  b^{\delta \left( \omega +4-2\gamma\right) -\frac{\eta}{4}( \omega + 2) }(\tau') b^{-\frac{\alpha}{2} -\tilde \eta}(\tau)\right]d\tau'\right. \\
	& + & \left. \int_{\tau_0}^{\tau}  e^{\frac{\alpha}{2}(\tau-\tau')} b^{-\frac{\alpha}{2} -\tilde \eta}(\tau) \left[ b^{\frac{\alpha}{2} +4\eta}(\tau') + b^{\alpha + \delta(1 -\gamma)}(\tau') b^{-\tilde \eta(2+6)}(\tau)  + b^{\frac{3\alpha}{2} -2\delta \gamma }(\tau') b^{-\tilde \eta (16)}(\tau) \right]d\tau' \right\}.
\end{eqnarray*}
From \eqref{prcoess-filter-b-to-I} and \eqref{tau-tau-0-le-eta-tilde}, we have  
\begin{eqnarray*}
	& & \int_{\tau_0}^\tau \left[  e^{\frac{\alpha}{2}(\tau-\tau')}  b^{\delta \left( \omega +4-2\gamma\right) -\frac{\eta}{4}( \omega + 2) }(\tau') b^{-\frac{\alpha}{2} -\tilde \eta}(\tau)\right]d\tau' \\ 
	&  +   &  \int_{\tau_0}^{\tau}  e^{\frac{\alpha}{2}(\tau-\tau')} b^{-\frac{\alpha}{2} -\tilde \eta}(\tau) \left[ b^{\frac{\alpha}{2} +4\eta}(\tau') + b^{\alpha + \delta(1 -\gamma)}(\tau') b^{-\tilde \eta(8)}(\tau)  + b^{\frac{3\alpha}{2} -2\delta \gamma }(\tau') b^{-\tilde \eta (16)}(\tau) \right]d\tau'  \\
	& \le & C,
\end{eqnarray*}
provided that $ \tilde \eta \le \tilde \eta_4(\eta, \delta, L_0) $, $ A  \ge A_4$ and $ \tau_0 \ge \tau_4(A, \eta, \tilde \eta, L_0)$. Then, we infer
\begin{eqnarray*}
	\left|  \frac{y^\gamma}{\langle y \rangle^{4}} \int_{\tau_0}^\tau  e^{(\tau -\tau') \mathscr{L}_\infty} \left[ \hat{B}(\hat \varepsilon_{\beta, +}, \hat{\varepsilon}_{\beta, -})   \right](\tau') d\tau' \right| \le \frac{A^3}{16} b^{\frac{\alpha}{2} + \tilde \eta}(\tau), \forall y \in \left[ b^{\eta}(\tau), b^{-\tilde \eta}(\tau) \right] ,
\end{eqnarray*}
which implies 
\eqref{estima-Duhame-integral-hat-B}.

- Proof of \eqref{estima-Duhame-integral-1-2-Lambda_varep-beta-}: it is similar to the that of  \eqref{estima-Duhame-integral-hat-B} in the first case.
%in applying Lemma \ref{Lemma-esti-semigroup-Lambda-var} and the bound \eqref{shrinking-set-beta-tau}.   Finally, by adding all the terms we conclude \eqref{sup-y-delta--delta-I-hat-epsilon}.  

\begin{center}
	\textit{ \textbf{ The second sub-case $  \frac{1}{L_0} e^{\frac{\tau-\tau_0}{2}\left( 1 -\eta \left(\frac{2\ell}{\alpha} -1 \right)\right)} >   b^{-\tilde \eta}(\tau)  $    }}
\end{center}
We introduce $R$ large, to be fixed later, and we decompose $$\left[ b^{\eta}(\tau), b^{-\tilde \eta}(\tau) \right]  = \left[ b^{\eta}(\tau), R \right] \cup  \left[  R, b^{- \tilde \eta}(\tau) \right] .$$
Recall that  for each $ f \in L^2_\rho(\R^+) $, for each $\nu>0$, there exists $Y(v) >0$ satisfying $Y(\nu) \to +\infty$ as $\nu \to +\infty$ and  such that  $ \forall y \in [Y^{-1}(\nu),Y(\nu)]$  
\begin{equation}\label{pointwise-infinite-sum}
	e^{\nu \mathscr{L}_\infty} f (y) = \sum_{j = 0}^\infty e^{\left( \frac{\alpha}{2} - j \right) \nu }  \langle f,\phi_{j,\infty} \rangle_{L^2_\rho} \phi_{j,\infty} (y) \text{ pointwisely }.
\end{equation}

Following the above remarks,  the expression  \eqref{Duhamel-hat-varepsilon--} and the fact that  $\hat{\varepsilon}_-$ is orthogonal to $\phi_{0,\infty}$ and $ \phi_{1,\infty} $ we are led to
\begin{eqnarray}
	\hat{\varepsilon}_-(y,\tau) &=& \left(e^{(\tau-\tau_0)\mathcal{L}_\infty}(\varepsilon_-(\tau_0)) \right)_- + \int_{\tau_0}^{\tau} \left(e^{(\tau-\tau')\mathscr{L}_\infty}(\hat{B}(\tau'))   \right)_-d\tau' +\int_{\tau_0}^{\tau}  \left( \left[\frac{1}{2} -\beta(\tau') \right] \Lambda \hat{\varepsilon}_-(\tau') \right)_-d\tau'
	\nonumber\\ 
	&=& \sum_{j=2}^\infty e^{\left(\frac{\alpha}{2} - j \right)(\tau-\tau_0)} \left\langle  \hat{\varepsilon}_-(\tau_0),\phi_{j,\infty} \right\rangle_{L^2_{\rho}} \phi_{j,\infty}(y)  \nonumber\\
	& + &  \int_{\tau_0}^{\tau-L_0} \sum_{j=2}^\infty e^{\left(\frac{\alpha}{2} -j \right)(\tau-\tau')} \left\langle   \hat{B}(\tau'),\phi_{j,\infty} \right\rangle_{L^2_{\rho}} \phi_{j,\infty}(y) d\tau'  + \int_{\tau-L_0}^\tau \left( e^{(\tau-\tau')\mathscr{L}_\infty}(\hat{B}(\tau'))  \right)_- d\tau'  \nonumber \\
	& + & \int_{\tau_0}^{\tau-L_0}  \sum_{j=2}^\infty e^{\left(\frac{\alpha}{2} -j \right)(\tau-\tau')} \left\langle   \left( \left[\frac{1}{2} -\beta(\tau') \right] \Lambda \hat{\varepsilon}_-(\tau') \right) d\tau',\phi_{j,\infty} \right\rangle_{L^2_{\rho}} \phi_{j,\infty}(y) d\tau' \nonumber \\
	& + & \int_{\tau-L_0}^{\tau}  \left( \left[\frac{1}{2} -\beta(\tau') \right] \Lambda \hat{\varepsilon}_-(\tau') \right)_-d\tau'\label{writw-varep--in-sum-anf-minus-parts}
\end{eqnarray}
on $[Y^{-1}(\tau),Y(\tau)]$ ($Y(\tau) \to +\infty$ as $\tau \to +\infty$). Let us consider   $y \in \left[ b^{\eta}(\tau), R \right]$. We consider the initial data $\varepsilon(\tau_0)$ defined in \eqref{defi-initial-vaepsilon-l=1}, we have 
\begin{eqnarray*}
	\left| \left\langle \hat{\varepsilon}_-(y,\tau_0), \phi_{j,\infty}   \right\rangle_{L^2_{\rho}}  \right| \le  C b^{\frac{\alpha}{2} +\delta},  \forall j \ge 2.
\end{eqnarray*}
\iffalse

we apply  Cauchy–Schwarz inequality that follows
\begin{eqnarray*}
	\left| \langle \hat{\varepsilon}_-(\tau_0),\phi_{j,\infty} \rangle_{L^2_\rho}   \right| &\le & C \|\hat{\varepsilon}_-(\tau_0)\|_{L^2_\rho} \| \phi_{j,\infty}\|_{L^2_\rho} \le  C 4^{j} j! b^{\frac{\alpha}{2} + \delta}(\tau_0), 
\end{eqnarray*}

\begin{equation}\label{norm-L2-rho-hat-varpesilon--tau-0}
	\|\hat{\varepsilon}_{\beta, -}(\tau_0)\|_{L^2_\rho} \lesssim  C b^{\frac{\alpha}{2} + \delta }(\tau_0), 
\end{equation} 
\fi 
Note that  $\delta \gg \eta \gg \tilde \eta$ and we are in the case $\tau-\tau_0 > \frac{\ln A}{K_0}$, we   have the estimate
\begin{eqnarray*}
	& & \left| \sum_{j=2}^\infty e^{\left(\frac{\alpha}{2} -j \right)(\tau-\tau_0)} \langle  \hat{\varepsilon}_-(\tau_0),\phi_{j,\infty} \rangle_{L^2_{\rho}} \phi_{j,\infty}(y) \right|  \\
	& \le &   Cb^{\frac{\alpha}{2} +\tilde \eta} (\tau) \frac{\langle  y\rangle^4}{y^\gamma} \sum_{j=2}^{\infty} e^{(\frac{\alpha}{2} -j)(\tau-\tau_0)} b^{\frac{\alpha}{2} +\delta}(\tau_0) b^{-\frac{\alpha}{2} -\tilde \eta}(\tau) \alpha_j  (1+ y)^{2(j-1)} \label{estimate-semi-group-varep-0}\\
	& \le & C Cb^{\frac{\alpha}{2} +\tilde \eta} (\tau) \frac{\langle  y\rangle^4}{y^\gamma} \sum_{j=2}^{\infty} j^{\frac{\omega}{4}} e^{(\frac{\alpha}{2} -j)(\tau-\tau_0)} b^{\frac{\alpha}{2} +\delta}(\tau_0) b^{-\frac{\alpha}{2} -\tilde \eta}(\tau)  (1+ y)^{2(j-1)}.
\end{eqnarray*}
Similarly to the technique given  in  \eqref{prcoess-filter-b-to-I}, we have 
\begin{eqnarray*}
	e^{(\frac{\alpha}{2} -j)(\tau-\tau_0)} b^{\frac{\alpha}{2} +\delta}(\tau_0) b^{-\frac{\alpha}{2} -\tilde \eta}(\tau) &\lesssim & C  e^{(\frac{\alpha}{2} -j)(\tau-\tau_0)} e^{\left( \frac{2 }{\alpha} - 1  \right) \left( \frac{\alpha}{2} +\tilde \eta \right)(1 +\tilde \eta) (\tau-\tau_0)} \\
	& \lesssim & C e^{(\tau-\tau_0) \left(1 - j +c(\tilde \eta) \right)}  \text{ with } c(\tilde \eta) \lesssim \tilde \eta, j \ge 2,  
\end{eqnarray*}
which yields to
\begin{eqnarray}
	& & \left|  \sum_{j=2}^\infty e^{\left(\frac{\alpha}{2} -j \right)(\tau-\tau_0)} \langle  \hat{\varepsilon}_-(\tau_0),\phi_{j,\infty} \rangle_{L^2_{\rho}} \phi_{j,\infty}(y)\right| \nonumber   \\
	&\le  & C b^{\frac{\alpha}{2} +\tilde \eta}  \frac{\langle y \rangle^{4}}{y^\gamma}\sum_{j=2}^\infty  j^\frac{\omega}{4} R^{2(j-1)} \left( A^\frac{1}{K_0} \right)^{1 - j +c(\tilde \eta) } \le \frac{A^3}{32} b^{\frac{\alpha}{2} +\tilde \eta}(\tau) \frac{\langle y \rangle^4}{y^\gamma},\label{estimate-scalar-e-tau-L-epsi-tau-0}
\end{eqnarray}
provided that $K_0 \ge K_4, A \ge A_4(R, K_0)$, and $\tilde \eta \le \tilde \eta_4(R, K_0, A, \eta, \tilde \eta, L_0)$. 

\medskip
Next, observe that    $ \tau  -  (\tau - L_0)  =  L_0 \le \frac{\ln A}{K_0}$ if $A \ge A_4(L_0)$, so  we go back to the first case. Indeed, we argue similarly as in  \eqref{semi-group-1-le-b-delta-hat-B} and  \eqref{estimate-integral-M-y-ge-hat-B} to get  

\begin{eqnarray*}
	\left|   \int^{\tau}_{\tau -L_0} e^{(\tau -\tau') \mathscr{L}_\infty}  (\hat B)(\tau')  d\tau'  \right| &\le &  \frac{\langle  y\rangle^4}{y^\gamma} \int_{\tau-L_0}^\tau   e^{\frac{\alpha}{2}(\tau-\tau')} b^{\frac{\alpha}{2} +4\eta}   (\tau') d\tau\\
	& \le & C  b^{\frac{\alpha}{2} +3 \eta}(\tau) \frac{\langle y \rangle^4}{y^\gamma},
\end{eqnarray*}
provided that $ s_0 \ge s_{4}(\eta, L_0)$. This yields 
\begin{eqnarray*}
	\left|  \int^{\tau}_{\tau -L_0} \left( e^{(\tau -\tau') \mathscr{L}_\infty}  (\hat B)(\tau')    \right)_- d\tau' \right|   \le  C b^{\frac{\alpha}{2}+3\eta}(\tau)\frac{\langle y \rangle^4}{y^\gamma}.
\end{eqnarray*}
For the integral $\int_{\tau_0}^{\tau-L_0} \sum_{j=2}^\infty e^{\left(\frac{\alpha}{2} -j \right)(\tau-\tau')} \left\langle   \hat{B}(\tau'),\phi_{j,\infty} \right\rangle_{L^2_{\rho}} \phi_{j,\infty}(y) d\tau' $, we deduce from Lemma \ref{lemma-estimate-on-hat-B} that 
$$ \left| \langle \hat{B}(\tau'),\phi_{j,\infty}\rangle_{L^2_\rho}  \right| \lesssim   b^{\frac{\alpha}{2} +4 \eta}(\tau'). $$
Then, for all $y \in  [b^\eta(\tau), R  ]$, we have
\begin{eqnarray*}
	& & \left| \int_{\tau_0}^{\tau-L_0} \left( \sum_{j=2}^{\infty} e^{(\frac{\alpha}{2} -j)(\tau-\tau')}  \langle \hat{B}(\tau'), \phi_{j,\infty} \rangle_{L^2_\rho} \phi_{j,\infty}  \right) d\tau' \right| \\
	& \lesssim & \frac{\langle y\rangle^{4} b^{\frac{\alpha}{2} +\tilde \eta}(\tau)}{y^\gamma}  \sum_{j=2}^\infty |\alpha_j| R^{2(j-1)} \int_{\tau_0}^{\tau - L_0} e^{(\frac{\alpha}{2} -j)(\tau-\tau')} b^{-\frac{\alpha}{2} -\tilde \eta}(\tau)   b^{\frac{\alpha}{2} +4\eta} (\tau')    d\tau'.
\end{eqnarray*}
We repeat the techniques given in  \eqref{prcoess-filter-b-to-I} and \eqref{asti-tau-tau-0vareo-0} to obtain
\begin{eqnarray*}
	b^{-\frac{\alpha}{2} - \tilde \eta}(\tau) e^{\left( \frac{\alpha}{2} -j \right) (\tau -\tau')}  b^{\frac{\alpha}{2} +\eta } (\tau') & \lesssim & b^{-\frac{\alpha}{2} - \tilde \eta}(\tau) b^{\frac{\alpha}{2} +\tilde \eta}(\tau') e^{\left( \frac{\alpha}{2} -j \right) (\tau -\tau')} \\
	& \lesssim & e^{ \left(1  -j + c(\tilde \eta) \right)(\tau -\tau')}, \text{ with } c(\tilde \eta) \lesssim \tilde \eta.
\end{eqnarray*}
It follows that for all $j \ge 2$, we have
\begin{eqnarray*}
	\int_{\tau_0}^{\tau-L_0} b^{-\frac{\alpha}{2} - \tilde \eta} (\tau) e^{(\frac{\alpha}{2} -j) (\tau-\tau')}   b^{\frac{\alpha}{2} +\eta}(\tau') d \tau' \lesssim  \exp\left(\left[1 -j + c(\tilde \eta) \right](L_0) \right), \text{ and } 1 -j + c(\tilde \eta) < 0.
\end{eqnarray*}
By taking  $L_0 \gg R$, we get 
\begin{eqnarray*}
	&& \sum_{j=2}^\infty  |\alpha_j| R^{2(j-1)} \int_{\tau_0}^{\tau - L_0} e^{(\frac{\alpha}{2} -j)(\tau-\tau')} b^{-\frac{\alpha}{2} -\tilde \eta}(\tau)   b^{\frac{\alpha}{2} +4\eta} (\tau')    d\tau' \\
	&\lesssim &  \sum_{j=2}^\infty  j^\frac{\omega}{4} R^{2(j-1)} \exp\left(\left[1  -j +c(\tilde \eta)\right](L_0) \right)  \lesssim 1   ,
\end{eqnarray*}
provided that $ L_0  \gg R$and $\tau_0 \ge \tau_4(A, L_0, R, \delta,  \eta, \tilde{\eta})$.  Then
\begin{eqnarray*}
	\left|  \int_{\tau_0}^{\tau-L_0} \sum_{j=2}^\infty e^{\left(\frac{\alpha}{2} -j \right)(\tau-\tau')} \left\langle   \hat{B}(\tau'),\phi_{j,\infty} \right\rangle_{L^2_{\rho}} \phi_{j,\infty}(y) d\tau' \right| \le Cb^{\frac{\alpha}{2} + \tilde \eta}(\tau) \frac{\langle y \rangle^4}{y^\gamma}, y \in [b^{\frac{\eta}{4}}, R].
\end{eqnarray*}
Recall that 
$$ \left| \beta(\tau) - \frac{1}{2} \right| \le CAI^{\eta}(\tau_0),$$
so by repeating the above arguments to handle the remaining terms in \eqref{writw-varep--in-sum-anf-minus-parts},    the following result holds
\begin{eqnarray*}
	&  & \left| \int_{\tau_0}^{\tau-L_0}  \sum_{j=2}^\infty e^{\left(\frac{\alpha}{2} - j \right)(\tau-\tau')} \left\langle   \left( \left[\frac{1}{2} -\beta(\tau') \right] \Lambda \hat{\varepsilon}_-(\tau') \right) d\tau',\phi_{j,\infty} \right\rangle_{L^2_{\rho}} \phi_{j,\infty}(y) d\tau'   \right| \\
	& + & \left| \int_{\tau-L_0}^{\tau}  \left( \left[\frac{1}{2} -\beta(\tau') \right] \Lambda \hat{\varepsilon}_-(\tau') \right)_-d\tau' \right| \le C b^{\frac{\alpha}{2} + \tilde \eta  }(\tau)\frac{\langle y\rangle^4}{y^\gamma}.
\end{eqnarray*}
We conclude that
\begin{eqnarray}
	\left| \int_{\tau_0}^\tau  e^{(\tau - \tau') \mathscr{L}_\infty} \left[ \frac{1}{2} -\beta(\tau')   \right] \Lambda_y ( \hat{\varepsilon}_{\beta, - } )(\tau') d\tau' \right| \le \frac{A^3}{16} b^{\frac{\alpha}{2} +\tilde \eta}(\tau) \langle y \rangle^{4}y^{-\gamma}, \forall y \in [b^{\frac{\eta}{4}}, R].
\end{eqnarray}

\iffalse
Finally, we get the conclusion of \eqref{estima-Duhame-hat-epsilon-tau-0}-\eqref{estima-Duhame-integral-1-2-Lambda_varep-beta-} 
\fi 
\medskip
-  For the  case $y \in \left[R, b^{-\tilde \eta}(\tau)\right]$. First, we observe that once   $ y \le b^{-\tilde \eta}(\tau) \le \frac{1}{L_0} e^{\frac{\tau-\tau_0}{2}\left( 1 -\eta \left(\frac{2}{\alpha} -1 \right)\right)} $, then, there exists $ \bar \tau \in [ \tau_0, \tau - 1] $ such that 
$  y \in \left[ \frac{1}{2L_0} e^{\frac{\tau-\bar \tau}{2}\left( 1 -\eta \left(\frac{2}{\alpha} -1 \right)\right)}  ,  \frac{1}{L_0} e^{\frac{\tau - \bar \tau}{2}\left( 1 -\eta \left(\frac{2}{\alpha} -1 \right)\right)}  \right]   $. Since $ y  \ge R$ we have
$ \tau -\bar \tau \ge C(R) \to +\infty \text{ as } R \to +\infty. $
Now, write the integral equation at the initial data $\bar \tau$
\begin{equation}\label{Duhamel-hat-varepsilon--intial-bar-tau}
	\hat{\varepsilon}_{ -}(\tau) = e^{(\tau-\bar \tau)\mathscr{L}_\infty} \hat{\varepsilon}_{ -}(\bar \tau) + \int_{\bar \tau}^\tau  e^{(\tau -\tau') \mathscr{L}_\infty} \left[ \hat{B}( \tau')  +\left(\frac{1}{2} -\beta(\tau') \right)\Lambda_y \hat\varepsilon_{-} \right](\tau') d\tau'.
\end{equation}
Note that for all $j \ge 2$, we have 
\begin{eqnarray*}
	\langle  \hat{\varepsilon}_-, \phi_{j,\infty} \rangle_{L^2_\rho} = \langle \varepsilon_+, \phi_{j,\infty} \rangle_{L^2_\rho} + \langle \varepsilon_-, \phi_{j,\infty} \rangle.
\end{eqnarray*}
Using the fact that  $\varepsilon_+$ is orthogonal to $\phi_{j,b,\beta}$,  estimate \eqref{norm-phi-b-phi-nfty-i}, $\phi_{j,\infty, \beta}$'s definition given in Proposition \ref{proposition-spectral-L-infty} and $\phi_{i,\infty} = \phi_{j,\infty, \frac{1}{2}},$ we derive that
\begin{eqnarray*}
	\left| \langle \varepsilon_+,\phi_{j,\infty} \rangle_{L^2_\rho}\right|  \le C b^{\frac{\alpha}{2} +\eta}(\bar \tau), \forall j \ge 2.
\end{eqnarray*}
Moreover, using the definition of $\rho_\beta$ given in \eqref{defi-rho-y} and $\rho= \rho_\frac{1}{2}$, we apply Cauchy Schwarz inequality and estimate  \eqref{estima-epsilon--new} to derive 
\begin{eqnarray*}
	\left| \langle \varepsilon_-, \phi_{j,\infty} \rangle_{L^2_\rho} \right|  &=& \left| \int_{\mathbb{R}_+} \varepsilon_- \phi_{j,\infty} \rho dy \right| \le  C \int_{\mathbb{R}_+} |\varepsilon_-| \sqrt{\rho_\beta} |\phi_{j,\infty} | \frac{\rho}{\sqrt{\rho_\beta}}  dy\\
	&\le & C \|\varepsilon_-\|_{L^2_\rho} \sqrt{\int_{\mathbb{R}_+} \left|\phi_{j,\infty} \right|^2 \frac{\rho^2}{\rho_\beta}  dy}  \le C  b^{\frac{\alpha}{2} +\eta}.
\end{eqnarray*}
Hence, we have
$$  \left| \langle \hat{\varepsilon}_-(\bar \tau),\phi_{j,\infty} \rangle_{L^2_\rho} \right|  \le C  b^{\frac{\alpha}{2} +\eta}(\bar \tau).$$
Now, we use formula  \eqref{pointwise-infinite-sum} to get 
\begin{eqnarray*}
	\left| e^{(\tau -\bar \tau) \mathscr{L}_\infty} \hat{\varepsilon}_{-}(\bar \tau ) (y,\tau) \right| & \lesssim   & \sum_{j=2}^{\infty} e^{(\frac{\alpha}{2}-j)(\tau-\bar \tau)}  \left| \langle \hat{\varepsilon}_{ -}(\bar \tau ) \phi_{j,\infty}  \rangle_{L^2_\rho} \right|  \left| \phi_{j,\infty}(y) \right| \\
	& \lesssim  &  \sum_{j=2}^{\infty} 4^j j! e^{(\frac{\alpha}{2}-j)(\tau-\bar \tau)}  b^{\frac{\alpha}{2} + \eta}(\bar \tau)   \left| \phi_{j,\infty}(y) \right| \\
	& \lesssim &   \langle y \rangle^{4}y^{-\gamma}  b^{\frac{\alpha}{2} +\tilde \eta} (\tau) \sum_{j=2}^{\infty} \left| \beta_j  \right| e^{(\frac{\alpha}{2}-j)(\tau-\bar \tau)} b^{-\frac{\alpha}{2} -\tilde \eta}(\tau)  b^{\frac{\alpha}{2} + \eta}(\bar \tau)   y^{2(j -1)} .
\end{eqnarray*} 
Using \eqref{esti-b-equivalent-I-1-tilde-eta}, we have for all $j \ge 2$
\begin{eqnarray*}
	e^{(\frac{\alpha}{2} -j)(\tau - \bar \tau)} b^{\frac{\alpha}{2} + \eta}(\bar \tau) b^{-\frac{\alpha}{2} -\tilde \eta}(\tau)
	& \lesssim &   e^{(\frac{\alpha}{2} -j)(\tau - \bar \tau)} e^{\left(1-\frac{2}{\alpha} \right) \left[ \left( \frac{\alpha}{2} +  \eta \right)\left( 1-\frac{\tilde \eta}{10} \right) \bar \tau -\left( \frac{\alpha}{2} + \tilde \eta \right)\left( 1+\frac{\tilde \eta}{10} \right)\tau  \right]}.
\end{eqnarray*}
In addition, since  $ \tilde \eta \ll \eta$,  and $ \left(1-\frac{2}{\alpha} \right) <0$, it follows the following
\begin{eqnarray*}
	& & \left(1-\frac{2}{\alpha} \right) \left[ \left( \frac{\alpha}{2} + \eta \right)\left( 1-\frac{\tilde \eta}{10} \right) \bar \tau -\left( \frac{\alpha}{2} + \tilde \eta \right)\left( 1+\frac{\tilde \eta}{10} \right)\tau  \right]\\
	& \le & \left(1-\frac{2}{\alpha} \right) \left[ \left( \frac{\alpha}{2} +\frac{\eta}{2}  \right) \left( 1+ \frac{\tilde \eta}{10} \right) \bar \tau -\left( \frac{\alpha}{2} +\tilde \eta  \right) \left( 1+ \frac{\tilde \eta}{10} \right) \tau  \right]\\
	& = & \left(1-\frac{2}{\alpha} \right) \left(1 +\frac{\tilde \eta }{10} \right) \left[ \frac{\alpha}{2} (\bar \tau - \tau)  + \frac{\eta}{2} \bar \tau - \tilde \eta \tau   \right] \\
	& \le &  \left(1-\frac{2}{\alpha} \right) \left(1 +\frac{\tilde \eta }{10} \right)  \left[ \frac{\alpha}{2} (\bar \tau - \tau) + \frac{\eta}{2}(\bar \tau - \tau) \right] ,
\end{eqnarray*}
and 
\begin{eqnarray*}
	& & \left( \frac{\alpha}{2} -j\right) (\tau-\bar \tau)+\left(1-\frac{2}{\alpha} \right) \left(1 +\frac{\tilde \eta }{10} \right)  \left[ \frac{\alpha}{2} (\bar \tau - \tau) + \frac{\eta}{2}(\bar \tau - \tau) \right] \\
	& \le & \left[ 1 - j  +\frac{\tilde \eta}{10} \left(\frac{2}{\alpha} - 1 \right) \frac{\alpha}{2}  + \frac{\eta}{2} \left(\frac{2}{\alpha} - 1 \right) \left(1+\frac{\tilde \eta}{10} \right)  \right] (\tau-\bar \tau) \\
	& \le & \left[ 1 - j   + \eta \left(\frac{2}{\alpha} -1 \right)  \right] (\tau-\bar \tau), 
\end{eqnarray*}
and $ \left| \beta_j \right| \le C \frac{ j^{-\frac{\omega}{4}}}{4^jj!}$. 
Hence, we obtain
\begin{eqnarray*}
	\left| e^{(\tau -\bar \tau) \mathscr{L}_\infty} \hat{\varepsilon}_{-}(\bar \tau ) (y,\tau) \right| & \lesssim   & A b^{\frac{\alpha}{2} +\tilde \eta}(\tau)  \sum_{j = 2}^\infty \frac{ j^{-\frac{\omega}{4}}}{4^jj!} \left( y e^{-\frac{\tau-\bar \tau}{2} \left(1 -\eta\left(\frac{2\ell}{\alpha} -1 \right)\right)} \right)^{2(j-1)}\\
	& \lesssim  & A b^{\frac{\alpha}{2} + \tilde \eta}(\tau) \sum_{j=2}^\infty \frac{ j^{-\frac{\omega}{4}}}{4^jj!}  ( L_0^{-1} )^{2(j-\ell)} \le \frac{A^3}{16} b^{\frac{\alpha}{2} + \tilde \eta}(\tau),
\end{eqnarray*}
when $ L_0 $ is large and $ A \ge A_4$.

\medskip
From the restriction  $    \frac{1}{2L_0} e^{\frac{\tau-\bar \tau}{2}\left( 1 -\eta \left(\frac{2\ell}{\alpha} -1 \right)\right)} \le y \le b^{-\tilde \eta}(\tau),  $
it follows
\begin{eqnarray*}
	\tau < \frac{\bar \tau \left(1 - \eta \left(1-\frac{2\ell}{\alpha}  \right)  \right)  + 2\ln L_0}{1-\left( \eta - 2\tilde \eta \left( 1+\frac{\tilde \eta}{10}\right) \right)  \left(1-\frac{2\ell}{\alpha}  \right)   }, 
\end{eqnarray*}
which yields
\begin{eqnarray}
	\tau  -\bar \tau \le \frac{   2\tilde \eta  \bar \tau  \left( \frac{2\ell}{\alpha} -1\right)\left(1 +\frac{\tilde \eta}{10} \right) + 2\ln L_0}{1-\left(\eta -2\tilde \eta \left(1+\frac{\tilde \eta}{10} \right)\right) \left(\frac{2\ell}{\alpha} -1 \right)}.
\end{eqnarray}
Hence, it turns out that the second case is obtained by replacing $\tau_0$ by $\bar \tau$  
\begin{eqnarray}
	\left| \int_{\bar \tau}^\tau  e^{(\tau -\tau') \mathscr{L}_\infty} \left[ \hat{B}(\hat \varepsilon_{\beta, +} + \hat{\varepsilon}_{\beta, -})   \right](\tau') d\tau'  \right|  &\le  &\frac{A^3}{16} b^{\frac{\alpha}{2} + \tilde \eta}(\tau) \langle y \rangle^{2\ell+2} y^{-\gamma},\\
	\left| \int_{\bar \tau}^\tau  e^{(\tau -\tau') \mathscr{L}_\infty}  \left(\frac{1}{2} -\beta(\tau') \right)\Lambda_y \hat\varepsilon_{\beta,-} (\tau') d\tau' \right| & \le &  \frac{A^3}{16} b^{\frac{\alpha}{2} + \tilde \eta}(\tau) \langle y \rangle^{2\ell+2} y^{-\gamma}.
\end{eqnarray}
At final, by adding all the terms, we get 
\begin{eqnarray*}
	\left| \varepsilon_{\beta,-}(\tau)\right| \le \frac{A^3}{2} b^{\frac{\alpha}{2} +\tilde \eta}\frac{\langle y  \rangle^{2\ell+2}}{y^{\gamma}}, \forall y \in \left[ R, b^{-\tilde \eta}(\tau) \right],
\end{eqnarray*}
which concludes the third case.
\end{proof}
The final step is to focus on the \textit{a priori  estimates} on the  exterior part  $\varepsilon_e$.  More precisely, we have the following Lemma
\begin{lemma}[A priori estimates on the exterior part]\label{lemma-priori-estimate-outer-part} Let us consider $\varepsilon$ to satisfy equation \eqref{equa-varepsilon-appen} with initial data given in \eqref{initial-data}, and $\varepsilon(\tau) \in V[A,\eta, \tilde \eta](\tau)$ for all $\tau \in [\tau_0,\tau_1]$, for some $\tau_1 > \tau_0$. Then, we  have  the following  estimate 
\begin{eqnarray*}
	\||y| \varepsilon_e(\cdot,\tau)\|_{L^\infty} \le  \frac{A^4}{2} b^{\frac{\alpha}{2} +(\gamma -4) \tilde \eta }(\tau).
\end{eqnarray*}
\end{lemma}
\begin{proof}
To get the conclusion of the Lemma, we consider the  natural $(d+2)$-dimensional extension as follows
\begin{eqnarray}
	\varepsilon_{ext} (z,\tau) =  \varepsilon(y,\tau),  |z| =y \text{ and }  z \in \mathbb{R}^{d+2},  \label{radial-extension}
\end{eqnarray} 
and  we introduce 
$$   \varepsilon_{ext, e} (z,\tau) = |z| \left( 1 -  \chi_0\left(\frac{8}{3}|z| b^{\tilde \eta }(\tau)  \right)\right) \varepsilon_{ext}(z,\tau),  $$
where $\chi_0$ was defined in \eqref{defi-chi-0} and $\gamma$  defined was in   \eqref{defi-gamma}. Note that $u$'s  extension  defined in \eqref{radial-extension}  is $C^2(\mathbb{R}^{d+2})$, thanks to the parabolic  regularity of the semi-group $e^{t\Delta_{d+2}}$ and so is $\varepsilon_{ext}$ and we derive from 
\eqref{equa-varepsilon-appen} that  $\varepsilon_{ext}$ satisfies 
\begin{eqnarray*}
	\partial_\tau \varepsilon_{ext} &=& \Delta \varepsilon_{ext} - \beta(\tau) y \cdot \nabla \varepsilon_{ext} - 2\beta(\tau)\varepsilon_{ext} \\
	&-&    3(d-2) \left[ 2Q_b(|z|) + |z|^2 Q_b^2(|z|) \right]  \varepsilon_{ext} +  B(\varepsilon_{ext},|z|) + \Phi(|z|,\tau),
\end{eqnarray*}
where $\Lambda_z$ is similarly defined as in \eqref{defi-Lambda-f}, $B$  and $ \Phi$ were defined in \eqref{defi-B-quadratic-appendix}  and \eqref{Phi-simple}, respectively. Now, we introduce 
\begin{eqnarray}
	\varepsilon_{ext,e}(z, \tau) = |z|  \left(1 - \chi_{\tilde \eta} \right)  \varepsilon_{ext}, \text{ and }  \chi_{\tilde \eta} (z,\tau) = \chi_0 \left( \frac{8}{3} |z| b^{\tilde \eta}(\tau) \right),\label{defi-varep-ext-e}
\end{eqnarray}
where  $\chi_{0}$ was defined in \eqref{defi-chi-0} and we have the following facts
\begin{equation}
	\label{support-eta-tilde-chi}
	\text{supp}(1 - \chi_{\tilde \eta}) \subset \left\{ z \in \R^n \text{ such that } |z| \ge \frac{3}{8} b^{-\tilde \eta}(\tau) \right\},
\end{equation}
and $\varepsilon_{ext,e} = \varepsilon_{ext}$ for all $|z| \ge \frac{3}{4} b^{-\tilde \eta}(\tau)$.
We here mention that the conclusion of the Lemma immediately follows from
$$  \|  \varepsilon_{ext, e}(\cdot, \tau) \|_{L^\infty[b^{-\tilde \eta}(\tau),+\infty)} \le\frac{A^4}{2} b^{\frac{\alpha}{2} + (\gamma -4) \tilde \eta}  (\tau), \forall \tau \in [\tau_0, \tau_1]. $$
By using  $\varepsilon_{ext}$'s equation above,  we that $\varepsilon_{ext,e}$ exactly solves
\begin{equation}\label{equa-varepsi-ext-e}
	\partial_\tau \varepsilon_{ext,e} = \mathscr{L}_\beta (\varepsilon_{ext,e}) + \mathcal{C}(\varepsilon_{ext}) + \mathcal{N}(\varepsilon_{ext}),
\end{equation}
where $\mathcal{L}_\beta$ is defined by 
\begin{equation}\label{defi-mathscr-L-0}
	\mathcal{L}_\beta = \Delta - \beta(\tau) z \cdot \nabla -  \beta(\tau ) Id
\end{equation}
and the terms $\mathcal{C}(\varepsilon_{ext})$ and $\mathcal{N}(\varepsilon_{ext})$ are respectively defined by 
\begin{eqnarray*}
	\mathcal{C}(\varepsilon_{ext}) &=&  -  2 \text{div} (\varepsilon_{ext} \nabla( |z| (1 -\chi_{\tilde \eta})))  -  3(d-2) \left[ 2Q_b(|z|) +|z|^2Q_b^2(|z|) \right] \varepsilon_{ext,e} \nonumber
	\\
	&+&\varepsilon_{ext} \left[ \partial_\tau (1 -\chi_{\tilde \eta}) |z| + \Delta (|z|(1 -\chi_{\tilde \eta})) + \beta(\tau) z \cdot \nabla  (1 -\chi_{\tilde \eta}) |z|\right], \label{defi-communicated-term-C}
\end{eqnarray*}
and
\begin{eqnarray*}\label{defi-Phi}
	\mathcal{N}(\varepsilon_{ext}) = |z| (1 -\chi_{\tilde \eta}) \left(B(\varepsilon_{ext}) + \Phi(\cdot, \tau) \right).
\end{eqnarray*}
- \textit{The semi-group of $\mathcal{L}_\beta$:} Let us define $\mathcal{K}_\beta(\tau',\tau), \tau > \tau' \ge \tau_0$; the semi-group associated to $\mathcal{L}_\beta$  with 
$$ \mathcal{K}_\beta( \tau, \tau')f = \int_{\mathbb{R}^{d+2}} \mathcal{K}_\beta(\tau, \tau', z,z')f(z') dz' $$
and the Kernel $\mathcal{K}_\beta(\tau, \tau', z,z')$
\begin{eqnarray*}
	\mathcal{K}_\beta(\tau, \tau', z,z') = \frac{\zeta(\tau,\tau')}{\left[ 4\pi \int_{\tau'}^\tau \zeta^2(\tilde \tau, \tau') d \tilde\tau \right]^\frac{d+2}{2}} \exp \left(- \frac{\left(z'-z \zeta(\tau, \tau') \right)^2}{4 \int_{\tau'}^\tau \zeta^2(\tilde \tau,\tau') d \tilde\tau } \right),  
\end{eqnarray*}
\begin{eqnarray*}
	\zeta(\tau, \tau') =  e^{-\int_{\tau'}^\tau \beta(\tilde \tau) d \tilde \tau}, \tau > \tau'.
\end{eqnarray*}
In particular, when $ \beta \equiv \frac{1}{2}$, our situation  is the same   as the  operator considered in \cite[Lemma A.1]{MZjfa08} since 
$$ \zeta(\tau, \tau') = e^{-\frac{\tau -\tau'}{2}}, \text{ and } \int_{\tau'}^\tau \zeta^2(\tilde \tau, \tau') d \tilde \tau = 1 - e^{-(\tau-\tau')}.$$
Since our operator has the same structure as the one in \cite[Lemma A.1]{MZjfa08}, we can apply the arguments there to  get
\begin{eqnarray}
	\| \mathcal{K}_\beta(\tau', \tau)(\varphi)\|_{L^\infty} \le  \zeta(\tau,\tau') \|\varphi\|_{L^\infty} \le  e^{-\frac{(\tau -\tau')}{4}}\| \varphi\|_{L^\infty},\label{esti-semi-group-K-tau--tau}\\
	\| \mathcal{K}_\beta(\tau, \tau') \text{div}(\varphi)\|_{L^\infty}  \le C \frac{\zeta(\tau, \tau')}{\sqrt{\int_{\tau'}^\tau\zeta^2(\tilde \tau,\tau')d\tilde \tau}} \| \varphi\|_{L^\infty} \le C \frac{e^{-\frac{\tau-\tau'}{4}}}{\sqrt{\tau -\tau'}}, \label{esti-semi-group-K-tau--tau-div}
\end{eqnarray}
since \eqref{shrinking-set-beta-tau} holds for all $\tau \in [\tau_0,\tau_1]$. 
By Duhamel's formula, we now write equation \eqref{equa-varepsi-ext-e} as follows
\begin{equation}\label{duhamel-varep-ext-e}
	\varepsilon_{ext, e}(\tau) = \mathcal{K}_\beta(\tau, \tau')\varepsilon_{ext,e}(\tau')  + \int_{\tau'}^\tau \mathcal{K}_\beta(\tau,\tilde \tau) \left[ \mathcal{C}(\varepsilon_{ext}) +\mathcal{N}(\varepsilon_{ext}) \right](\tilde \tau)  d \tilde \tau.
\end{equation}
We now aim to estimate the terms involving  $\mathcal{N} (\varepsilon_{ext}(\tau))$ and $ \mathcal{C}(\varepsilon_{ext})$.

- For $\mathcal{N}$, since \eqref{support-eta-tilde-chi}  holds, we only consider $|z| \ge \frac{3}{8} b^{-\tilde \eta}(\tau)$ that will be divided into two small cases  $|z| \in \left[\frac{3}{8} b^{-\tilde \eta}(\tau), b^{-\tilde \eta}(\tau) \right]$ and $|z| \ge  b^{-\tilde \eta}(\tau)$. Since $\varepsilon(\tau) \in V_1[A,\eta, \tilde\eta](\tau), \forall \tau \in [\tau_0, \tau_1]$, we get the following 
\begin{eqnarray*}
	||z| \varepsilon_{ext}(z,\tau)| \le A^3 b^{\frac{\alpha}{2} + \gamma \tilde \eta  } \langle  |z|\rangle^4 \le C A^3 b^{\frac{\alpha}{2} + (\gamma -4) \tilde \eta }(\tau ), \forall |z| \in \left[\frac{3}{8} b^{-\tilde \eta}(\tau), b^{-\tilde \eta}(\tau) \right],
\end{eqnarray*}
and 
\begin{eqnarray*}
	\left|  |z| \varepsilon_{ext}(z,\tau)\right| \le A^4 b^{\frac{\alpha}{2} + (\gamma -4) \tilde \eta}(\tau), \forall |z| \ge b^{-\tilde \eta}(\tau) 
\end{eqnarray*}
which yields
\begin{eqnarray*}
	\| \mathcal{N}(\varepsilon_{ext}) \|_{L^\infty} \le CA^3 b^{\frac{\alpha}{2} + 10 \tilde \eta }(\tau), \forall \tau \in [\tau_0, \tau_1],
\end{eqnarray*}
provided that $\tilde \eta \le \tilde \eta_5(\alpha), \tau_0 \ge \tau_5(A, \tilde \eta)$. Using \eqref{esti-semi-group-K-tau--tau}, we deduce
$$ \left\| \mathcal{K}_\beta(\tau, \tilde \tau )
\mathcal{N}(\tilde \tau)) \right\|_{L^\infty} \le CA^3 e^{-\frac{\tilde \tau- \tilde \tau}{4}} b^{\frac{\alpha}{2} + 10 \tilde \eta }(\tilde \tau). $$

- For $ \mathcal{C}(\varepsilon_{ext})$: From \eqref{esti-semi-group-K-tau--tau-div}, we have
\begin{eqnarray*}
	\left\| \mathcal{K}_\beta(\tau, \tilde \tau) \text{div}(\varepsilon_{ext} \nabla[ (1-\chi_{\tilde \eta})|z|])(\tilde \tau) \right\|_{L^\infty} &\le & C\frac{ e^{-\frac{\tau -\tilde \tau}{4}}}{\sqrt{\tau-\tilde \tau}} \|(\varepsilon_{ext} \nabla[ (1-\chi_{\tilde \eta})|z|])(\tilde \tau) \|_{L^\infty}\\
	& \le & CA^3\frac{ e^{-\frac{\tau -\tilde \tau}{4}}}{\sqrt{\tau-\tilde \tau}} b^{\frac{\alpha}{2} + (\gamma -4) \tilde \eta  }(\tilde \tau).
\end{eqnarray*}
Similarly, we have 
\begin{eqnarray*}
	\left\| \mathcal{K}_\beta(\tau, \tilde \tau) [ 3(d-2)(2Q_b +|z|^2Q_b^2)\varepsilon_{ext,e}](\tilde \tau) \right\|_{L^\infty} \le CA^3 e^{-\frac{\tau -\tilde \tau}{4}} b^{\frac{\alpha}{2} + (\gamma -4)\tilde \eta}(\tilde \tau). 
\end{eqnarray*}
Besides that, we derive from \eqref{defi-varep-ext-e} that
$$ \|\varepsilon_{ext}(\tilde \tau) \partial_\tau (1 - \chi_{\tilde \eta})|z|\|_{L^\infty} \le C \tilde \eta A^4 b^{\frac{\alpha}{2} + (\gamma -4) \tilde \eta}(\tilde \eta) \le CA^3 b^{\frac{\alpha}{2}+(\gamma -4)\tilde \eta}(\tilde \tau). $$
Hence, we get
\begin{eqnarray}
	\left\| \mathcal{K}_\beta(\tau, \tilde \tau) \varepsilon_{ext}(\tilde \tau) \partial_\tau (1 - \chi_{\tilde \eta})|z| \right\|_{L^\infty} \le CA^3 e^{-\frac{\tau -\tilde \tau}{4}} b^{\frac{\alpha}{2} +(\gamma - 4) \tilde \eta }(\tilde \tau).
\end{eqnarray}
By the same technique, we can establish the following 
\begin{eqnarray*}
	\left\| \mathcal{K}_\beta(\tau, \tilde \tau) \varepsilon_{ext}(\tilde \tau) \left[ \Delta (|z|(1-\chi_{\tilde \eta})) + \beta(\tau) z \cdot \nabla(1-\chi_{\tilde \eta}) |z| \right]  \right\|_{L^\infty} \le CA^3 e^{-\frac{\tau -\tilde \tau}{4}} b^{\frac{\alpha}{2} + (\gamma - 4) \tilde \eta }(\tilde \tau).
\end{eqnarray*}
Taking $L^\infty$-estimate on both sides of \eqref{duhamel-varep-ext-e}, we get
\begin{eqnarray*}
	\|\varepsilon_{ext,e}(\tau)\|_{L^\infty} & \le &  e^{-\frac{\tau-\tau'}{4}} \|\varepsilon_{ext,e}(\tau')\|_{L^\infty} +CA^3 \int^{\tau}_{\tau'} e^{-\frac{\tau- \tilde\tau}{4}} b^{\frac{\alpha}{2} + (\gamma -4)\tilde \eta}(\tilde \tau)\left[ \frac{1}{\sqrt{\tau-\tilde \tau}} +1 \right] d\tilde \tau \\
	& \le & e^{-\frac{\tau-\tau'}{4}} \|\varepsilon_{ext,e}(\tau')\|_{L^\infty} + CA^3 b^{ \frac{\alpha}{2} + (\gamma -4) \tilde \eta}(\tau') \left( \sqrt{\tau -\tau'} + (\tau -\tau') \right).
\end{eqnarray*}
We now apply the technique used in \cite[Proposition 4.5]{MZjfa08}. Our the choice of the initial data in \eqref{defi-initial-vaepsilon-l=1} allows us to improve the bound on $\varepsilon_{ext, e}$ by   
$$ \|\varepsilon_{ext,e}(\tau)\|_{L^\infty} \le \frac{A^4}{2} b^{\frac{\alpha}{2} +(\gamma -4) \tilde \eta}(\tau), \forall \tau \in [\tau_0, \tau_1],  $$
provided that $A \ge A_5$ and $\tilde \eta \le \tilde \eta_5(A)$ and $\tau_0 \ge \tau_5(A,\tilde \eta)$. Finally, the conclusion of the lemma follows. 
\end{proof}

\section{ The existence of unstable blowup solutions}
In this Section, we aim to give a sketch of the proof to the existence of unstable blowup solutions to equation  \eqref{equa-Yang-Mills-} with blowup speeds $$ \lambda_\ell (t) = C(u_0) (T-t)^{\frac{2\ell}{\alpha} } \text{ as } t \to T. $$
The general strategy of the proof is the same as in the stable setting, except to some technical  modifications.  For that reason, we  only give main changes which lead to the unstable existence. Let us consider the similarity variable \eqref{similarity-variable} with $\mu (\tau) = T-t$ and $\tau= -\ln(T-t)$ then $w$ satisfies \eqref{equa-w} with $\beta \equiv \frac{1}{2}$. We  linearize around $ Q_{b(\tau)}$ given in \eqref{Q-b} by $\varepsilon = w - Q_{b(\tau)}$ and it holds that $\varepsilon$ satisfies \eqref{equa-varepsilon-appen}.  Note that all terms in the equation remain the same  with $\beta \equiv \frac{1}{2}$. In particular, the spectral analysis  of $\mathscr{L}_b$ is valid without the appearance of $\beta$. Now, we consider the composition  \eqref{defi-varepsilon-j} with $\ell \ge 2$. 

\subsection{Shrinking set for  $\ell \ge 2$ }
In this part, we modify a little bit the set in Definition \ref{shrinking-set}  to be compatible with the new setting
\begin{definition}[Shrinking set]\label{definition-V-ell} Let  $A,\eta, \tilde \eta$  be positive constants  satisfying $A \gg 1$ and $1  \gg \eta \gg \tilde \eta$,   we define $V_\ell[A, \eta, \tilde \eta](\tau)$ 
as the set of all $(\varepsilon, b) \in L^\infty \times \mathbb{R}$    satisfying:
\begin{itemize}
	\item[$(i)$]  The  dominating mode $\varepsilon_\ell $ and $b$ satisfy 
	\begin{eqnarray}
		\left|\varepsilon_\ell +  \frac{2}{\alpha} m_0 b^\frac{\alpha}{2} \right| \le A b^{\frac{\alpha}{2} + \eta}
		,\label{estimate-varepsilon-1-V-A-s-ell-ge-2}
	\end{eqnarray}
	and 
	\begin{equation}\label{shrin-king-set-b-tau}
		\frac{1}{2} \le   b I_\ell^{-1}(\tau)  \le 2,  
	\end{equation}  
	where 
	\begin{equation}\label{defi-I-tau-ell}
		I_\ell(\tau) = e^{\left(  \frac{2\ell}{\alpha}-1 \right)\tau}.
	\end{equation}
	In addition the others modes $\varepsilon_j   \in j \in \{1,...,\ell-1 \} $ satisfy 
	\begin{eqnarray}
		\left| \varepsilon_j      \right| \le A b^{\frac{\alpha}{2} + \eta}. 
	\end{eqnarray}
	
	\item[$(iii)$]   $L^2_\rho$-decay: The part $\varepsilon_- $ satisfies the following:
	$$ \|\varepsilon_-( .)\|_{L^2_{\rho_\beta}}  \le A^2 b^{\frac{\alpha}{2} +\eta}(\tau).$$
	\item[$(iv)$]  The remainders $\varepsilon_-$ given in \eqref{defi-varepsilon-j},  and $ \varepsilon_e$ satisfy 
	\begin{eqnarray*}
		\left\| y^\gamma \frac{\varepsilon_-(.,\tau)}{ \langle y \rangle^{2\ell +2}} \right\|_{L^\infty[0, b^{-\tilde \eta}(\tau)] }  & \le & A^3  b^{\frac{\alpha}{2} + \tilde \eta}(\tau) ,
		\\
		\| |y| \varepsilon_e(.,\tau)\|_{L^\infty} & \le &  A^4  b^{ \frac{\alpha}{2} +(\gamma-(2\ell+2))\tilde \eta }(\tau), \end{eqnarray*} 
	where $\varepsilon_e$ defined as in \eqref{defi-varepsilon-e}.
\end{itemize} 
\end{definition}

\subsection{Preparing initial data } 
In this part, we aim to construct a class of initial data corresponding to the Shrinking set $V_\ell[A, \eta, \tilde \eta]$. Let us consider $A \ge 1,$  and $0< \tilde \eta \ll \eta \ll \delta  \le 1$, $\alpha, m_0$ and $c_{\ell,0}$ defined as in \eqref{defi-alpha-intro},  \eqref{decompose-Phi} and \eqref{definition-c_i-j}, respectively. 
\begin{eqnarray}
& & \psi(\ell,\tau_0)  =   \chi\left( y  b^{\frac{\delta}{2}}(\tau_0) \right) \left( 1 -  \chi\left( \frac{y}{ b^\frac{\delta}{2}(\tau_0)}   \right)  \right)\label{defi-initial-vaepsilon-ell}\\
& \times &\left\{  \sum_{j=1}^{\ell-1}   A d_j   b^{\frac{\alpha}{2} +\eta}(\tau_0)   \phi_{j,b(\tau_0),\beta(\tau_0)}
-\frac{2}{\alpha}m_0 b^{\frac{\alpha}{2}}(\tau_0)\left[ 1 + d_\ell A b^{\frac{\alpha}{2} +\eta}(\tau_0)\right] \left(\frac{\phi_{\ell,b(\tau_0),\beta(\tau_0))}}{c_{\ell,0}}    -  \phi_{0,b(\tau_0),\beta(\tau_0))}   \right)  \right\}
\nonumber  
\end{eqnarray}
In particular, the class of initial data implies the following result

\begin{lemma}[Preparing the initial data]\label{intial-lemma-ell}
There exists  $ A_6 \ge 1 $, such that for all $A \ge A_6$, there exist $\eta_6(A) >0$ such that for all $\delta \le \delta_6$ there exists $\eta_6(A,\delta) >0$ such that for all $\eta \le \eta_6$ there exists $\tilde \eta_6(A, \delta, \eta)$  such that for all $\tilde \eta \le \tilde \eta_6$ there exists  $\tau_6(A,\delta, \eta, \tilde \eta) > 1$  such that for all $\tau_0 \ge \tau_6$, there exists $\mathcal{D}_A  \subset \left[-2,2 \right]^{\ell} $ such that the following properties are valid
\begin{itemize}
	\item[$(i)$] the mapping
	\begin{eqnarray*}
		\Gamma :  \mathbb{R}^{\ell}   & \to &   \mathbb{R}^{\ell }\\
		(d_1,..., d_{\ell-1})   & \mapsto   &  \left( \psi_1,..., \psi_{\ell} \right)(\tau_0),
	\end{eqnarray*}
	is  affine and one to  one   from   $\mathcal{D}_{A}$   to 	$\hat{V}[A,\eta](\tau_0) ,$ where 
	\begin{equation}\label{defi-hat-V-A-ell}
		\hat{V}[A,\eta](\tau) = \left[-A b^{\frac{\alpha}{2} +\eta}(\tau), A b^{\frac{\alpha}{2} +\eta}(\tau) \right]^{\ell}.
	\end{equation}
	In particular, we have the following property
	$$ \Gamma \left. \right|_{\partial \mathcal{D}_A} \in \partial \hat{V}_A(\tau_0),  $$
	and $\text{deg}(\Gamma \left. \right|_{\partial \mathcal{D}_A}) \ne 0$.
	\item[$(ii)$] for all $(d_1,...,d_{\ell}) \in \mathcal{D}_A$,  it follows that $\psi_\ell(\tau_0) \in V_\ell[A, \tau, \tilde \eta](\tau_0)$   with strictly  improved bounds as follows
	\begin{eqnarray*}
		\left| \psi_\ell(\tau_0) + \frac{2}{\alpha} m_0b^\frac{\alpha}{2}(\tau_0) \right| & \le & b^{ \frac{\alpha}{2} +\eta}(\tau_0),\\
		\|\psi_-(\tau_0)\|_{L^2_\rho} & \le & b^{ \frac{\alpha}{2} +\eta}(\tau_0),\\
		\left\|\frac{y^\gamma}{\langle y \rangle^{2\ell+2}} \psi_-(.,\tau_0) \right\|_{L^\infty[0,b^{-\frac{\tilde \eta}{2}}(\tau_0)]} & \le & b^{ \frac{\alpha}{2} +\eta}(\tau_0),\\
		\| \psi_e(\cdot,\tau)\|_{L^\infty} &\le &  b^{ \frac{\alpha}{2} +\eta}(\tau_0),
	\end{eqnarray*}
	and $ b(\tau_0) I^{-1}_\ell(\tau_0) \in \left[\frac{1}{4}, \frac{3}{2} \right] $.
\end{itemize}
\end{lemma}
\begin{proof}
The bounds in item (ii) immediately follow from the explicit form of $\psi(\ell, \tau_0)$ in \eqref{defi-initial-vaepsilon-ell}. In addition, the existence of $\mathcal{D}_{A}$ and mapping $\Gamma$ follows from the concentration of modes $\psi_j, j \in \{0,1,....,\ell\}$ of $\psi(\ell, \tau_0)$ and the argument is quite the same as in \cite[Proposition 4.5]{TZpre15}. We kindly refer the reader to check the details of this result.
\end{proof}
\subsection{Finite dimensional reduction}\label{propo-finite-reduction}
Since the shrinking set $V_\ell[A, \eta, \tilde \eta](\tau)$ has special properties,   the conclusion of Theorem \ref{theorem-existence-Type-II-blowup-instable} immediately  the following
$$ (\varepsilon, b)(\tau) \in V_\ell[A, \eta, \tilde \eta](\tau) \forall \tau \in [\tau_0,+\infty), \text{ for some } \tau_0 \text{ large enough}. $$
In particular, we prove in this part that this control is reduced to a finite dimensional problem on $(\varepsilon_j)_{j \in \{1,2,...,\ell \} }$
\begin{proposition}[Finite dimensional reduction]\label{propo-finite-reduction-ell-ge-2} 
There exists  $A_7 \ge 1 $,  such that for all $A \ge A_7$, there exists $\delta_7(A)$ such that for all $\delta \le \delta_7$ there exists    $\eta_7(A,\delta)$ such that for all $ \eta \le \eta_7$ there exists  $ \tilde \eta_7(A, \delta, \eta) $ such that for all $ \tilde \eta \le \tilde \eta_7 $ there exists $\tau_7(A, \delta, \eta, \tilde \eta)$ such that for all $\tau_0 \ge  \tau_7$, the following property is valid: If  $ (\varepsilon,b)$ is the solution to the coupled system (\ref{equa-varepsilon-appen}-\ref{orthogonal-condition}) with initial data  \eqref{defi-initial-vaepsilon-ell},      $(\varepsilon, b)(\tau) \in V_\ell[A,\eta, \tilde \eta](\tau)$ for all $\tau \in [\tau_0, \tau^*]$ for some $\tau^* > \tau_0$ and $ (\varepsilon, b)(\tau^*) \in \partial V_\ell[A,\eta,\tilde \eta](\tau^*)$, then we have  the following:
\begin{itemize}
	\item[$(i)$] It holds that   $\left(\varepsilon_1,...,\varepsilon_{\ell} \right)(\tau^*) \in \partial \hat{V}[A,\eta](\tau^*)$ defined as in \eqref{defi-hat-V-A-ell}.\\
	\item[$(ii)$] \textit{Transversality}: There exists $\nu_0 >0$ such that 
	$$ \left(\varepsilon_1,...,\varepsilon_{\ell} \right)(\tau^* + \nu) \notin  \partial \hat{V}[A,\eta](\tau^* +\nu), \forall \nu \in (0,\nu_0) $$
\end{itemize}
which implies
$$ (\varepsilon, b)(\tau^*+ \nu) \notin V[A,\eta,\tilde \eta](\tau^* + \nu), \forall \nu \in (0,\nu_0). $$
\end{proposition}
\begin{proof} The proof mainly relies on the \textit{priori estimate} which is the same as in Lemmas  [\ref{lemma-ODE-finite-mode}-\ref{lemma-priori-estimate-outer-part}].

- \textit{ Proof of item }(i): Let us consider    $A \ge A_7$, $\delta \le \delta_7(A) $, $ \eta \le \eta_7(A, \delta), \tilde \eta \le \tilde \eta_7(A, \delta, \eta) $, $\tau_0 \ge  \tau_7(A, \delta, \eta, \tilde \eta)$ and $(\varepsilon, b)(\tau) \in V_\ell[A, \eta, \tilde \eta](\tau) \forall \tau \in [\tau_0,\tau^*]$  such that  Lemmas \ref{lemma-L-2-rho-var--}, \ref{lemma-priori-estima-varep--}, and \ref{lemma-priori-estimate-outer-part} remain true (the technique of the proof is exactly the same and we kindly refer the reader to check the details). So, it immediately follows that   the bounds of $\varepsilon_- $ and $\varepsilon_e$  given in Definition \ref{definition-V-ell} for $V_\ell[A, \eta, \tilde \eta]$  are always   improved by   $\frac{1}{2}$-factor. In addition, we completely reproduce the argument of Lemma \ref{lemma-ODE-finite-mode} to establish the following results: For all $\tau \in [\tau_0, \tau^*]$, we have 
\begin{eqnarray}
	\left| \varepsilon_j' (\tau) - \left(  \frac{\alpha}{2} -j \right)\varepsilon_j(\tau)     \right| \le C  b^{\frac{\alpha}{2} +4\eta}(\tau), \forall \tau \in \left[\tau_0,\tau_1 \right],\label{estimate-varepsilon-'-ODE--ell}
\end{eqnarray} 
and 
\begin{equation}\label{system-vare-0-ell}
	\left\{ \begin{array}{rcl}
		& & \partial_\tau \varepsilon_\ell  -\left(  \frac{\alpha}{2} - \ell  \right)  \varepsilon_\ell =  O\left( b^{\frac{\alpha}{2} + 4 \eta}(\tau) \right),\\[0.2cm]
		& & \partial_\tau \varepsilon_\ell - \left( \frac{\alpha}{2}    \right)  \varepsilon_\ell  + m_0 \left( \frac{b_\tau}{b} -1 \right)b^\frac{\alpha}{2} = O(b^{\frac{\alpha}{2} + 4\eta}),
	\end{array}
	\right.
\end{equation}
and 
\begin{equation}\label{ODE-b-tau-proposition-ell}
	\left|  \frac{b'(\tau)}{b(\tau)}  -   \left(  1  - \frac{2 \ell}{\alpha}\right)  \right|   \le  C A b^{4\eta}(\tau).
\end{equation}
From \eqref{ODE-b-tau-proposition-ell}, we derive that
$$ \frac{3}{4}< b(\tau)I^{-1}_\ell(\tau) \le \frac{3}{2} \forall \tau \in [\tau_0, \tau^*],$$
provided that $\tau_0 \ge \tau_7(A, \eta, \tilde \eta)$.
Thus, we derive from the fact that $(\varepsilon, b)(\tau) \in \partial V_{\ell}[A, \eta, \tilde \eta](\tau^*)$ 
the following 
$$(\varepsilon_1, \varepsilon_2,...,\varepsilon_\ell)(\tau^*)  \in \partial \hat{V}[A,\eta](\tau^*) $$
which concludes item (i).

- \textit{ Proof of item } (ii): it is sufficient to prove that there exists $\nu_0$ small enough such that 
\begin{itemize}
	\item Either there exists $j \in \{1,...,\ell-1\}$, such that 
	\begin{equation}\label{transverlity-varep-j}
		\left| \varepsilon_j(\tau^* + \nu)\right|  > A b^{\frac{\alpha}{2}+\eta} (\tau_1 +\nu), \forall \nu \in (0,\nu_0),
	\end{equation}
	\item or the following holds
	\begin{equation}\label{transverlity-varep-ell}
		\left| \varepsilon_\ell (\tau^* +\nu) + \frac{2}{\alpha} m_0 b^{\frac{\alpha}{2}}(\tau_1 +\nu) \right|  > A b^{\frac{\alpha}{2} +\eta}(\tau_1 + \nu ), \forall \nu \in (0,\nu_0),
	\end{equation}
\end{itemize}
As we proved in item (i), one of the following two cases holds   
\begin{itemize}
	\item \textit{Case 1:} There exists $ j \in \{ 1,...,\ell-1\}$ such that
	$$ \varepsilon_j(\tau^*) = \sigma_j Ab^{\frac{\alpha}{2} +\eta}(\tau^*),$$
	\item \textit{Case 2:}
	$$ \varepsilon_\ell (\tau_1)  + \frac{2}{\alpha} m_0 b^\frac{\alpha}{2}(\tau_1) = \sigma_\ell Ab^{\frac{\alpha}{2} +\eta}(\tau_1),$$
\end{itemize}
where $ \sigma_j = \pm 1$.  The goal is to prove that the first case implies \eqref{transverlity-varep-j} and the second one concludes  \eqref{transverlity-varep-ell}. Indeed, we have

+ Assume that case 1 occurs. Without loss of generality, we can assume  $\sigma_j = 1$ and introduce
$$ B_j(\tau) = \varepsilon_j(\tau) - Ab^{\frac{\alpha}{2} +\eta}(\tau).$$
It is obvious that $B(\tau^*) =0$ and we also get from \eqref{estimate-varepsilon-'-ODE--ell} and \eqref{ODE-b-tau-proposition-ell},
\begin{eqnarray*}
	B_j'(\tau^*) &=&  \left(\frac{\alpha}{2} -j  \right)\varepsilon_j(\tau^*) - A \left(\frac{\alpha}{2} + \eta \right) \frac{b'(\tau^*)}{b(\tau^*)} b^{\frac{\alpha}{2} +\eta}(\tau^*) + O(b^{\frac{\alpha}{2} +4\eta}(\tau^*)) \\
	& = & Ab^{\frac{\alpha}{2} +\eta}(\tau^*) \left(  (\ell-j) -  \eta \left(1-\frac{2\ell}{\alpha} \right)\right)+O\left(b^{\frac{\alpha}{2} +4\eta}(\tau^*) \right) >0,
\end{eqnarray*}
provided that $\ell -j \ge 1$ and $\eta \le \eta_7(A,\eta, \tilde \eta)$ and $\tau^* \ge \tau_0 \ge  \tau_7(A)$. Then, $B_j( \tau^* +\nu) >0$ for all $\nu \in (0,\nu_0)$ for some $\nu_0$ small enough. Thus, \eqref{transverlity-varep-j} follows.

\medskip
+ If the case 2 occurs. We also assume $\sigma =1$ (the opposite sign is the same),  we then define
$$ B_\ell(\tau) = \varepsilon_\ell(\tau) + \frac{2}{\alpha} m_0 b^{\frac{\alpha}{2}}(\tau) -Ab^{\frac{\alpha}{2} +\eta} (\tau),$$
and it holds that $B_\ell (\tau^*) =0.$
and we derive from \eqref{system-vare-0-ell} and \eqref{ODE-b-tau-proposition-ell} that
\begin{eqnarray*}
	B'_\ell(\tau^*)  &=&   \left( \frac{\alpha}{2} -\ell \right) \varepsilon_\ell(\tau^*) +  \frac{2}{\alpha} m_0 \frac{\alpha}{2} \frac{b'}{b} b^\frac{\alpha}{2}(\tau^*) - A \left( \frac{\alpha}{2} +\eta \right) \frac{b'}{b} b^{\frac{\alpha}{2} +\eta}(\tau^*) + O(b^{\frac{\alpha}{2} +4\eta}(\tau_1))\\[0.2cm]
	&  =  & A b^{\frac{\alpha}{2 } +\eta} (\tau^*) \left[  \eta \left( \frac{2\ell}{\alpha}-1 \right)   \right]  + O(b^{\frac{\alpha}{2} +4\eta}(\tau^*))>0,
\end{eqnarray*}
since $ \eta \left( \frac{2\ell}{\alpha}-1 \right)   > 0$ and  $\tau^* \ge \tau_0 \ge \tau_7(A, \eta, \tilde \eta)$. Thus, we conclude that $B_\ell(\tau^* +\nu) >0$ for all $\nu \in (0,\nu_0)$, and \eqref{transverlity-varep-ell} follows. This concludes the proof of the Proposition.
\end{proof}
\subsection{Topological argument}
In this part, we aim to prove the existence of an initial datum $(\varepsilon,b)(\tau_0)$   that leads to the  global existence of $(\varepsilon,b)(\tau) \in V_\ell[A,\eta, \tilde \eta](\tau), \forall \tau \in [\tau_0, +\infty)$. The following is our result: 

\begin{proposition}\label{proposition-trapped-solu-l-ge1} There exist $A, \eta, \tilde \eta $ and $\delta \ll  1$  satisfying  $A \gg 1, 1 \gg  \delta \gg \eta \gg \tilde \eta >0, \tilde \eta$ and  $\tau_0(A,\eta, \tilde \eta, \delta) \gg 1$  such that there exists $ \tilde d = (d_1,...,d_\ell) \in \mathcal{D}_{A}$  defined in Lemma \ref{intial-lemma-ell} such that with  initial data $\varepsilon(\ell, \tau_0)$  defined as in \eqref{defi-initial-vaepsilon-ell},  the solution $(\varepsilon, b)$ to the coupled problem  (\ref{equa-varepsilon-appen}-\ref{orthogonal-condition}),   globally exists  and the following holds
$$ (\varepsilon, b)(\tau) \in V_\ell [A,\eta, \tilde \eta](\tau), \forall \tau \ge \tau_0.$$
\end{proposition}
\begin{proof}
The proof follows  from the topological argument which was  used  in \cite{BKnon94} and \cite{MZdm97}. Let us assume $A, \eta,  \tilde  \eta$ and $\delta $ are suitably chosen such that  Lemma \ref{intial-lemma-ell} and   Proposition \ref{propo-finite-reduction-ell-ge-2} hold true.  We now proceed to the proof by 
contradiction and we assume  that  for all $\tilde d = (d_1,...,d_\ell) \in \mathcal{D}_{A} $, there exists $\tau(\tilde d) \in [\tau_0, +\infty )$ such that  $ (\varepsilon, b) (\tau(\tilde d) ) \notin V_\ell [A, \eta, \tilde \eta](\tau(\tilde d)) $. Then, we can define for each $\tilde d \in \mathcal{D}_A$ the maximum time 
\begin{eqnarray}
	\tau^*(\tilde d) = \sup \left\{ \tau_1  \text{ such that }  (\varepsilon, b)(\tau) \in V_\ell [A,  \eta,  \tilde \eta](\tau) , \forall \tau \in [\tau_0, \tau_1] ,    \right\}
\end{eqnarray} 
Now, let the mapping  $\Pi$ be defined by 
\begin{eqnarray}
	\Pi: \mathcal{D}_{A, \tau_0}  \to \partial \left[ -1,1 \right]^{\ell }\\
	\tilde d = (d_1,...,d_\ell) \mapsto \Pi(\tilde d),  \quad 
\end{eqnarray}
where 
\begin{eqnarray*}
	\Pi(\tilde d) = (A b^{\frac{\alpha}{2} +\eta}(\tau^*(\tilde d)))^{-1} \left( \varepsilon_1,..., \varepsilon_\ell + \frac{2}{\alpha}m_0 b^\frac{\alpha}{2} \right)(\tau^*(\tilde d)).
\end{eqnarray*}
In particular,  the following properties hold:
\begin{itemize}
	\item[$(i)$] $ \Pi $ is continuous from $\mathcal{D}_{A} $ to  $\partial \left[ -1,1 \right]^{\ell }$. Indeed,  since $\tau^*(\tilde d)$'s definition implies  
	$$ (\varepsilon, b)(\tau^*(d))  \in \partial V_\ell [A, \eta, \tilde \eta] (\tau^*(d)),$$ 
	and item $(ii)$ of Proposition \ref{propo-finite-reduction}  immediately yields the result.
	\item[$(ii)$]  $\text{Deg}(\left. \Pi \right|_{\partial \mathcal{D}_{A}}  )  \ne 0 $. The result  follows from item (i) of  Lemma \ref{intial-lemma-ell}.
\end{itemize}
Thus, such a mapping $\Pi$ can not exist, since it contradicts the index theory and the conclusion of the Proposition follows.
\end{proof}

\section{Existence of  ground state}\label{ section-ground-state} 
We show in this part the asymptotic behavior of the ground state solution to \eqref{equa-u-section-2}. Let us introduce  $ Q$ to be the function satisfying  
\begin{equation}\label{equation-the ground-state}
Q''(\xi) + \frac{d+1}{\xi} Q'_\xi- 3(d-2) Q^2 - (d-2) \xi^2 Q^3 =0, \quad  Q(0)  = -1 \text{ and } Q'(0) = 0.
\end{equation}
We have the following result:
\begin{lemma}\label{lemma-ground-state}
Let $d \geq 10$, then there exists a unique solution $Q$ to equation \eqref{equation-the ground-state}  satisfying the following:
\begin{itemize}
	\item[$(i)$]  Asymptotic behavior of $Q$:
	\begin{eqnarray}
		Q(\xi) &=& - 1 + \sum_{i=1}^k  a_i \xi^{2i } + O(\xi^{2k+2}) \text{ as } \xi \to 0,\label{asym-Q-in-origin}\\
		Q(\xi) &=& - \frac{1}{\xi^2}  + q_0 \xi^{-\gamma} + O(\xi^{-3\gamma -4})   \text{ as }  \xi \to + \infty,\label{asym-Q-infinity}
	\end{eqnarray}
	\iffalse
	where $\gamma$  defined by 
	$$   \gamma = \gamma_1 =  \frac{1}{2} ( d - \sqrt{d^2 -12 d+ 24} ),$$
	
	and $g$ defined by 
	\begin{equation}\label{defi-g-}
		g =  2\alpha, 
	\end{equation}
	and $ \lambda_{1} $  defined by
	$$ \lambda_1 = -\alpha$$
	\fi
	and $ q_0 >0$.  
	% In particular,  \eqref{asym-Q-infinity} is stable under $\partial_y^k$ for all $k$:
	% \begin{equation}\label{asym-partial-x-k-Q}
		% \partial_y^k Q = \partial_y^k\left( -\frac{1}{y^2} + q_0 y^{-\gamma} \right) + O(y^{-\gamma -g-k}), \text{ as } y \to +\infty. 
		% \end{equation}
	
	\item[$(ii)$] In particular, when $d=10$, the ground state is explicitly  given by
	$$ Q_{10}(\xi) = -\frac{1}{\xi^2 +1}$$
	\item[$(iii)$] Asymptotic of $ \Lambda Q =2 Q + \xi \cdot \partial_\xi Q:$
	\begin{equation}\label{asymptotic-Lamda-Q}
		\Lambda Q(\xi) < 0, \text{ and } \Lambda Q =\left\{ \begin{array}{rcl}
			& & -2 + 4 \left(  \frac{3d-6}{3d+6} \right) \xi^2   + \sum_{i=2}^k a_i' \xi^{ 2i} + O(\xi^{2k+2}) \text{ as } \xi \to 0,   \\[0.3cm]
			& &  a_0 \xi^{-\gamma} + O(\xi^{-\gamma -g}) \text{ as }  \xi \to +\infty,   
		\end{array}  \right.     
	\end{equation}
	for some  $a_0 <0$.
	
	\noindent
	We note that $\Lambda Q$'s asymptotic at infinity is stable under $\partial^k_\xi$ for all $k \in \N$, more precisely
	\begin{equation}\label{asym-partial-y-Lambda-Q-y}
		\partial_\xi^k \Lambda Q  =\partial_\xi^k \left( a_0 \xi^{-\gamma} \right) + O(\xi^{-\gamma - g - k}) \text{ as } \xi \to +\infty. 
	\end{equation}
\end{itemize}
\end{lemma}

\begin{proof}
- The proof of item (i):  Following \cite{GMZ02}, we reformulate the ground state equation as an autonomous ODE. Indeed, let 
$$ Z(\xi) =  - \xi^2 Q(\xi), $$
then
\begin{equation}\label{equa-Z}
	Z''  + \frac{d-3}{\xi} Z' - \frac{(d-2)}{\xi^2} Z(Z-1) (Z-2) =0.
\end{equation}
Again,  apply the change of function
$$ Z(\xi) = v(x)  \text{ where  }  \xi =e^x,  $$
to get
\begin{equation}\label{equa-v}
	v''(x) + (d-4) v'(x) - (d-2) v (v-1)(v-2)=0, x \in (-\infty, +\infty).
\end{equation}
To prove global existence and asymptotic behavior of the solution, we employ the phase portrait analysis that used in \cite{BSIMRN19} for the ground-state of  the heat flow maps. First observe that \eqref{equa-v} has three critical  points $v =0,1,2$.  We choose $v =1$ to start our analysis (for $v=0,2$, the linear operator will have complex or positive eigenvalues). According to our initial condition $Q(0)=-1, Q'(0)=0$, $Q$ locally exists which in turn implies  $v$'s existence on $(-\infty, x_0)$ for some $x_0 <0$ and $|x_0|$ large enough. In particular, we  have the  boundary condition at $ -\infty$: $$ v(x) =  e^{2x} + \sum_{i=2}^k c_i e^{2ix} + O(e^{(2k+2)x} ) \text{ as } x \to -\infty. $$

This allows us to consider $v(x)$ as the flow starting at $ (0,0) $ and ending at $(1,0)$. Following \cite{BNON15}, we linearize around $1$, i.e.
$$ \epsilon = v -1,$$
then $\epsilon$ solves
\begin{equation}
	\epsilon'' + (d-4) \epsilon' - (d-2) \epsilon (\epsilon+1)(\epsilon-1) =0 .     
\end{equation}
The linearisation is given by
\begin{equation*}
	\epsilon'' + (d-4) \epsilon' + (d-2) \epsilon=(d-2)\epsilon^3,
\end{equation*}
We write the above equation in matrix form
$$  \binom{\epsilon'}{\epsilon}' = \binom{-(d-4) \quad -(d-2)}{1 \quad \quad \quad \quad \quad 0} \binom{\epsilon'}{\epsilon}+\binom{(d-2)\epsilon^3}{0}. $$
%Find eigenvalues, we solve
%$$  (\lambda) (\lambda + d-4)  + d-2 =0, $$
The eigenvalues are 
\begin{eqnarray}
	\lambda_1 &=& \frac{1}{2} ( \sqrt{d^2 -12d +24} - d +4) \label{defi-lambda-1-2} \\
	\lambda_2 &=&\frac{1}{2} (- \sqrt{d^2 -12d +24} - d +4), \nonumber
\end{eqnarray}
provided that 
$$d^2 -12d +24 \geq 0,$$
otherwise, the solution is spiral at $+ \infty$. We see that $\lambda_{1,2}(d) <0$ for all $d \geq 10$. %This is also expected  have Type II blowup solution.

- \textit{Construction  of no-escape region:}  Let us define 
$$ F(\epsilon, \epsilon') = (\epsilon', -(d-4)\epsilon' +(d-2)  (\epsilon^3 -\epsilon)).$$

We introduce a trapping region 
$$ \mathcal{S} =\{  (\epsilon,\epsilon') | \quad     \epsilon^3  - \epsilon 
\leq  \epsilon' \leq 2 (\epsilon^3 -\epsilon), \epsilon \in (-1,0) \}.$$

\begin{itemize}
	\item The lower boundary curve $ \epsilon' =  (\epsilon^3 -\epsilon)$. In the phase portrait space $(\epsilon,\epsilon'),$ we define the normal vector $\nu_{in}$ which points inward $\mathcal{S}$
	$$ \nu_{in} = ( -(3\epsilon^2 -1),1).$$
	We easily check that
	\begin{eqnarray*}
		F(\epsilon, \epsilon') \cdot \nu_{in} = (\epsilon^3 -\epsilon) 3(1 -\epsilon^2 ) >0, \forall \epsilon \in (-1,0).
	\end{eqnarray*}
	\item The upper boundary curve $ \epsilon' = 2 (\epsilon^3 -\epsilon)$. In the phase portrait space $(\epsilon,\epsilon'),$ we define the normal vector $\nu_{in}$ which points inward $\mathcal{S}$
	$$ \nu_{in} = ( 2 (3\epsilon^2 -1), -1).$$
	Then, 
	\begin{eqnarray*}
		F(\epsilon, \epsilon') \cdot \nu_{in} = (\epsilon^3 -\epsilon) ( 12 \epsilon^2 + d-10) >0, \forall \epsilon \in (-1,0) \text{ and }  m \geq 10.
	\end{eqnarray*}
	
	The vector field $F$ points always inward on the whole boundary of $\mathcal{S}$ (excluding the stationary points $(-1,0) $ and $(0,0)$). This implies that the integral curve of $F$ starting in $ \mathcal{S}$ must stay in $\mathcal{S}$.
\end{itemize}

- \textit{ The boundary conditions  trapped in $\mathcal{S}$:} Note that with initial data $Q(0)=-1, Q'(0)= 0$, $Q$ locally exists, this leads that  $\epsilon
$    exists locally, i.e., it exists on  $(-\infty, x_0)$, for some $x_0 $ near $-\infty$. In particular, we have  $Q \in C^\infty$. Using Taylor expansion together with equation \eqref{equation-the ground-state}, we see that  $Q(y)$  behaves as follows
$$ Q(\xi) = -1   +   \left(\frac{3d-6}{2d+6} \right)  \xi^2 +  \left( \frac{1}{4}. \frac{21d^2 -74d +64}{(3d+4)(d+4)}\right) \xi^4 + O(\xi^6) \text{ as } \xi \to 0.$$

- Asymptotic of the trapped solution in $\mathcal{S}$:
Let us discuss the boundary condition at $ -\infty$: 
we have
$$ \epsilon(x) = -1 + e^{2x} - \frac{3d-6}{3d+6} e^{4x} -\frac{1}{4} \frac{21d^2 -74d+64}{(3d+4)(d+4)} e^{6x} +O(e^{8x}), \text{ as } x \to -\infty.
$$
Using this asymptotic, the solution can't end up at $(-1, 0)$. In addition, at $(0,0)$, it gives the following general asymptotic of $\epsilon$:
$$ \epsilon(x) = h_+ e^{\lambda_1 x} (1 + O(e^{-2x})) +h_- e^{\lambda_2 x} (1 + O(e^{-2x})), $$
where
$$ \lambda_1 = \frac{1}{2} ( \sqrt{d^2 -12d +24} - d +4) \text{ and } \lambda_2 =\frac{1}{2} (- \sqrt{d^2 -12d +24} -d +4).$$

Assuming that $ h_+ =0,$ we derive from the shrinking set $\mathcal{S}$ that
$$ - \epsilon <   \epsilon' < - 2 \epsilon  \forall $$
Then
$$  h_- (1 + \lambda_2) > 0 \text{ and } h_- (2 +\lambda_2)  <0,$$

this contradicts the formula of $\lambda_2$.

So, $ h_+ \neq 0$. 

In addition to that, we require the same condition as $h_-$
$$ h_+ (1 + \lambda_1) > 0 \text{ and } h_+ (2 +\lambda_1)  <0 .$$
since
$$\lambda_1 + 1 < 0 \text{ and } \lambda_1 +2 >0 $$
we see that $h_+ < 0$ and we get the conclusion.

- The proof of item (ii): can be done in straightforward way, we omit the details.

- The proof of item (iii): We reformulate $Q(\xi)$ by 
\begin{equation}\label{relation-Q-epsilon-x}
	Q(\xi) = -\frac{(\epsilon(x) +1)}{e^{2x}}, \xi = e^x.
\end{equation}
Computation yields
$$ \xi Q'_\xi = - \frac{\epsilon'_x(x)}{e^{2x}} -2 Q.$$
Thus,
\begin{equation}\label{Lambda-Q-negative}
	\Lambda Q = 2Q +  y \partial_y Q = - \frac{\epsilon'_x(x)}{e^{2x}} <0 \quad \forall x \in (-\infty, \infty).
\end{equation}
Now, we aim to find the higher derivative of $\epsilon$, i.e., $ \partial^k_x \epsilon $ for all $ k \geq 1$.   In fact, $\epsilon$ satisfies the following integral equation 
\begin{equation}\label{integral-equa-epsilon}
	\epsilon (x) =  h_+ e^{\lambda_1 x} + h_- e^{\lambda_2 x} -\frac{1}{\lambda_1 -\lambda_2 }\int_{x}^\infty \left( e^{\lambda_1(x-x')} - e^{\lambda_2(x -x')} \right) g( \epsilon (x'))dx',
\end{equation}
where $g(z) = (d-2) z^3$. This gives us 
$$ \epsilon(x) =  h_+e^{\lambda_1 x} +O\left( e^{3\lambda_1 x}  \right) ,  $$
as $x \to +\infty$. In particular, applying $\partial^k_x$ to the right hand side of equation \eqref{integral-equa-epsilon} ,  we derive the following 
\begin{equation}\label{partial_x-k-epsilon}
	\partial_x^k \epsilon (x)= \partial_x^k (h_+ e^{\lambda_1 x})  + O(e^{3\lambda_1 x}), \text{ as } x \to +\infty.
\end{equation}
Let us remark that, it remains to prove \eqref{asym-partial-y-Lambda-Q-y}. Indeed, we have the following 
\begin{eqnarray*}
	\Lambda Q  &=& -  \epsilon'(x)  e^{-2x},\\
	\partial_y \Lambda  Q   & =  & -\epsilon''(x) e^{-3x} + 2 \epsilon'(x)e^{ -3x}   ,\\
	\partial_y^2 \Lambda Q  & = & -\epsilon'''(x) e^{-4x} +3 \epsilon'' e^{-4 x} + 2 \epsilon''(x) e^{-4x} -6 \epsilon'(x) e^{-4x},\\
	&  =& ( -\epsilon''' + 5\epsilon'' -6 \epsilon')e^{ -4 x},\\
	\partial_y^3 \Lambda Q & = & ( -\epsilon^{(4)} + 5\epsilon''' -6 \epsilon'')e^{ -5 x} - 4( -\epsilon''' + 5\epsilon'' -6 \epsilon')e^{ -5 x}.
\end{eqnarray*}
By induction, we can prove that
$$ \partial_y^k \Lambda Q = -  \Pi_{j=0}^{k-1} (\partial_x -2-j) \partial_x \epsilon(x)e^{(-2-k)x}, \forall k \ge 1.$$
Using  \eqref{partial_x-k-epsilon} and the fact that $\xi=e^x$, we get
\begin{eqnarray*}
	\Lambda_\xi Q(\xi)  & = &  a_0 \xi^{-\gamma} + O(\xi^{-3\gamma -4}),\\
	\partial_\xi \Lambda Q (\xi)& = & a_0 (\lambda_1 -2) \xi^{-\gamma-1} + O(\xi^{-3\gamma -5}),\\
	\partial_\xi^2 \Lambda Q(\xi) & = &  a_0(\lambda^2 - 5 \lambda_1 +6 ) \xi^{-\gamma -2} + O(\xi^{-3 \gamma -6}).
\end{eqnarray*}
In particular, we have the general case as follows: for all $ k \ge 1$
\begin{eqnarray*}
	\partial_\xi^k \Lambda Q &=& a_0 \Pi_{j=0}^{k-1} (\lambda_1 - 2-j) \xi^{\lambda_1 -2  -k } 
	+ O(\xi^{-\lambda_1-2 -k})\\
	&  =  & a_0 (-\gamma) ... (-\gamma -(k-1)) \xi^{ -\gamma-k} + O(\xi^{-3\gamma -4-k}) ,
\end{eqnarray*}
where $\gamma = 2 -\lambda_1, a_0 = - \lambda_1 h_+$.
Thus, \eqref{asym-partial-y-Lambda-Q-y} directly follows. This finishes the proof of the Lemma.
\iffalse

With $k=1$, we have

$$ \partial_y Q = \frac{2}{y^3} + q_0 (\lambda_1 -2) y^{-\gamma-1} + O(y^{-3\gamma -5}). $$

\fi
\end{proof}

\iffalse

\begin{remark}[Special case $d=10$]
When $d=10$, the ground state is explicit. Indeed, we search  a solution of the form
$$ Q(r) = \frac{1}{a_1 r^2 + a_2}. $$
Then
\begin{eqnarray*}
	& &  Q''(r)    +    \frac{d+1}{r}    Q(r)     - 3 (d-2) Q^2 - (d-2) r^2 Q^3(r) \\
	& = & \frac{-1}{(a_1 x^2 + a_2)^3} [ ( (2d-4)a_1^2 + (3d-6)a_1 +d-2 ) x^2  +a_2( 4a_1 +2a_1 d +3d -6) ]
\end{eqnarray*}
To cancellation $x^2$, we  choose $a_1 = -1 $ and we have  
$$a_2( d-10) =0,$$
since $d=10$, we see that $a_2$ is arbitrary. However, we will choose $a_2 < 0$ in order to guarantee that the solution is global. At  final
$$Q_c (r) = - \frac{1}{r^2 +c }, c>0.$$
Note that, from the above calculation and for $d \geq 11$, we cannot find $Q$  explicitly.

As a matter of fact, up to the scaling \eqref{scaling}, we normalize $Q$ so that
$$Q(0) = -1.$$
\end{remark}
\fi 
\section{Diagonalisation of $\mathscr{L}_b$}\label{proof-of-diagoligioncal-L-b}

The goal of this section is   to give a detailed proof of Proposition \ref{propo-mathscr-L-b} which is the same as the route map established in the Section 2 of \cite{CMRJAMS20}.

\subsection{Interior problem}
In the sequel, we construct eigenfunctions for $\mathscr{L}_b $ in the region $ 0 \leq  y \leq y_0 \ll 1 $. First, we introduce 
\begin{eqnarray}
w(y,\tau) = v(\xi, \tau),\quad \text{ with } \xi = \frac{y}{\sqrt{b}} \label{variable-xi-y}.
\end{eqnarray}
The interior zone can be written in terms of the blow-up variable $\xi$ as
$$ 0 \leq \xi \leq \xi_0 := \frac{y_0}{\sqrt{b}}.$$
Recall the definition of   $\mathscr{L}_b$ %definition, given in \eqref{relation-L-H}, we can write
$$ \mathscr{L}_b w(y )= \frac{1}{b} \left( H_\xi - \beta b \Lambda_\xi   \right) v.$$
where  the  Shr\"odinger type operator $H_\xi$ defined by 
\begin{equation}\label{operator-H-xi}
H = \partial_\xi^2  + \frac{d+1}{\xi} \partial_\xi -3(d-2) \left( 2Q(\xi) + \xi^2 Q^2(\xi) \right). 
\end{equation} 
\begin{lemma}[Generators of the Kernel of $H$]\label{lemma-Generation-H}  %We have the fact that $ \text{Ker}(H) = \text{Span}\{ \Lambda_\xi Q , \tilde Q(\xi) \} $ where $ \tilde Q $ defined by  \eqref{defi-Gamma}.
%Let us consider  $H$ defined as in \eqref{operator-H-xi}, then,   
There exists a family $\{ T_i\}_{i \ge 0}$ with initial element $T_0 = a_0^{-1} \Lambda_\xi Q$ belonging to the kernel of $H_\xi$ and  %with $c = a_0^{-1}$ and  $a_0 $ defined as in \eqref{asymptotic-Lamda-Q} 
such that  for all $i\in \mathbb{N}$
\begin{equation}\label{equa-T-i+1-i}
	%H (T_0) =0 \text{ and }
	H \left( T_{i+1}\right) =  T_i.
\end{equation}
Moreover, $T_{i}$ admits the expansion  
\begin{equation}\label{behavior-T-i}
	T_i  (\xi) = \left\{  \begin{array}{rcl}
		\sum_{l=0}^q t_{i,l} \xi^{2i+2l} + O(\xi^{2i + 2q+2}),  \forall q \in \N, \text{ as } \xi \to 0, \\[0.2cm]
		C_{i} \xi^{-\gamma +2i} \left(1  + O\left(\frac{\ln \xi}{\xi^2} \right) \right), \text{ as } \xi \to +\infty,
	\end{array} \right.    
\end{equation}
and the derivatives, up to order $k=3$, are such that
\begin{equation}\label{asymptotic-partial_y-T-i}
	\partial_\xi^k T_i(\xi) = \partial_\xi^k \left(  C_i \xi^{-\gamma +2i} \right)     +  O\left( \xi^{-\gamma+2i-2-k} \ln \xi \right), \text{ as } \xi \to +\infty.
\end{equation}
Here $\gamma$ and $C_j$ were defined in \eqref{defi-gamma-intro}, and  \eqref{definition-C-j-new}, respectively.\\
Let
$$ \Theta_i = \Lambda T_i   - (2i -\alpha) T_i.  $$
then, for all $k \in \{0, 1, 2\}$
\begin{equation}\label{asymptotic-Theta}
	\partial_\xi^k \Theta_i(\xi) =  O \left(  \xi^{ - \gamma +2i -k -2} \ln \xi \right) \text{ as } \xi \to +\infty.
\end{equation}
\end{lemma}
\begin{proof}
A detailed proof is to be presented in  Appendix \ref{proof-lemma-generator-H}.
%The proof is quite long and technical. For that reason, we aim to give  the detail in 
\end{proof}

The generators of the kernel of $H_\xi$ are at hand, we are in position to perform the construction of the eigenvalues and the eigenfunctions in the interior zone. More precisely, our result reads
%, i.e.   $ y \ll \sqrt{b}$   

\iffalse
$$ T_i(\xi ) \le C \frac{\xi^{2i}}{1+\xi^\gamma}$$
\fi
\begin{proposition}[Inner eigenfunctions]\label{proposition-inner-eigen-functions}
Let  $\ell \in \mathbb{N}, \ell \ge 1$, $i \in \{ 0,...,\ell\}$ and $\beta \in \left[\frac{1}{4}, \frac{3}{4} \right]$. Then,  there exists  $\epsilon_0(\beta) >0$ small enough such that for all $\epsilon \in (0, \epsilon_0)$ such that there exists $y^*(\epsilon) \ll 1$ such that for all  $0 < y_0 \le y^*$ there exist $b^*(y_0)$ and $\tilde \lambda^*(y_0)$ such that  for all  $ 0 < b < b^*(y_0)$   and $|\tilde \lambda| \leq \tilde \lambda^*$ there exists   $ \phi_{i,\text{int}}  \in  C^\infty\left(\left[0, \frac{y_0}{\sqrt{b}}\right],\R \right) $ such that the following hold:
\begin{equation}\label{equa-eigen-phi-int}
	\left( H - b \beta \Lambda \right) \phi_{i,\text{int}, \beta}=  2\beta b \left( \frac{\alpha}{2}-i +\tilde \lambda \right)\phi_{i,\text{int}, \beta},
\end{equation}
where $\phi_{i,\text{int}, \beta}$  has the following decomposition
\begin{equation}\label{inner-eigen-functins-phi-i-int}
	\phi_{i,\text{int}, \beta}(\xi) = \sum_{j=0}^i c_{i,j} (2\beta)^j b^j T_j(\xi) + \tilde \lambda  \sum_{j=0}^i  b^{j+1}  \left(c_{i,j}(2\beta)^{j+1} T_{j+1}(\xi) + S_j (\xi)\right) + bR_i(\xi),
\end{equation}
where the correction $R_i$ and $ S_j$ satisfy the following estimates
\begin{eqnarray*}
	\| S_j\|_{X^{2j+2-\gamma}_{\xi_0}} \leq C y_0^2,   \| \partial_b S_j\|_{X^{2j+4-\gamma}_{\xi_0}} \leq C,   \|\partial_{\tilde \lambda } S_j\|_{X^{2j+2-\gamma}_{\xi_0}} \leq C y^2_0, \text{ and } \|\partial_{ \beta } S_j\|_{X^{2j+2-\gamma}_{\xi_0}} \leq Cy_0^2 , \\
	\|R_i\|_{X^{-\gamma+\epsilon }_{\xi_0}} \leq C(\epsilon) , 
	\| \partial_b R_i \|_{X^{2-\gamma+\epsilon}_{\xi_0}}  \leq C(\epsilon),  
	\|\partial_{\tilde \lambda}  R_i\|_{X^{2-\gamma +\epsilon}_{\xi_0}} \leq C( \epsilon)  b, \text{ and }     \|\partial_{\beta}  R_i\|_{X^{2-\gamma +\epsilon}_{\xi_0}} \le C(\epsilon) .
\end{eqnarray*}

\end{proposition}
\begin{proof}
Due to the lengthy proof, we aim to put the details in Appendix \ref{proof-appendix-inner-eigenfunctions}.
\end{proof}

\subsection{Exterior  problem }
In this part, we aim to  construct  the eigenfunctions of $\mathscr{L}_b$  on $ [y_0 , +\infty)$, for some $y_0 \ll 1$. The following is our result 
\begin{proposition}[Outer eigenfunctions]\label{propo-outer-eigenfunctions} Let  $\ell \in \mathbb{N}, \ell \ge 1$, $i \in \{ 0,...,\ell\}$ and $\beta \in \left[\frac{1}{4}, \frac{3}{4} \right]$. Then, there exists $y^*(\beta) \ll 1$ such that for all $y_0 \le y^*$, there exist   $ b^*(y_0, \beta) $ and $ \tilde \lambda^*(y_0, \beta, b^*)$ such that for all $b \in (0, b^*)$ and $\tilde \lambda \in (-\tilde \lambda^*, \tilde \lambda^*) $, there exits  a $C^\infty\left[ y_0 ,+\infty \right)$  function $\phi_{i,out,\beta} $ satisfying
$$ \mathscr{L}_b \phi_{i,out, \beta} = \left( 2\beta \left( \frac{\alpha}{2}-i\right) +\tilde{\lambda} \right) \phi_{i,out, \beta} ,  $$
and having the following decomposition 
$$ \phi_{i,out, \beta}  = \phi_{i,\infty, \beta} + \tilde \lambda ( \tilde \phi_{i,\beta}   + \tilde R_{i,1}) + \tilde R_{i,2},$$
where  $\tilde \phi_{i,\beta}$ satisfies
\begin{equation*}
	\left(  \mathscr{L}_{\infty}^{\beta}   -   2\beta\left(    \frac{\alpha}{2} - i\right) \right)\tilde \phi_{i,\beta}= \phi_{i,\infty, \beta} \text{ with } \phi_{i,\infty, \beta} \text{ defined as in }  \eqref{phi-i-infty},
\end{equation*}
and $ \tilde R_{i,1}$ and $ \tilde R_{i,2}$ fulfil the following estimates  
\begin{eqnarray*}
	\|  \tilde R_{i,1}\|_{X^{ \gamma-d, 2i -\gamma +2}_{y_0} } \leq C |\tilde \lambda|    , \quad  \partial_{b} \tilde R_{i,1} = 0, \| \partial_{\tilde \lambda} \tilde R_{i,1}\|_{X^{\gamma-d, 2i -\gamma +2}_{y_0} } \leq C,  \| \partial_\beta  \tilde R_{i,1}  \|_{X^{ \gamma-d, 2i -\gamma +2}_{y_0} } \le C ,
\end{eqnarray*}
and 
\begin{eqnarray*}
	\|  \tilde R_{i,2}\|_{X^{ -d,a'}_{y_0} } \leq C b^{\frac{\alpha}{2}},  \quad   \| \partial_{b} \tilde R_{i,2} \|_{X^{-d,a'}_{y_0} }  \leq C b^{\frac{\alpha}{2}-1},  \| \partial_{\tilde \lambda} \tilde R_{i,2}\|_{X^{-d,a'}_{y_0} } \leq C b^\frac{\alpha}{2}, \| \partial_{\beta} \tilde R_{i,2}\|_{X^{-d,a'}_{y_0} } \leq C b^\frac{\alpha}{2},
\end{eqnarray*}
for $ a' = 2i +2 -\gamma$ and   $X^{a,a'}_{y_0}$ is a Banach space  generated by the  norm
\begin{eqnarray}
	\|f\|_{X^{a,a'}_{y_0}} &=& \sup_{y  \in [y_0,1]} y^{-a} \left\{ \sum_{i=0}^2   y^i \left| \partial^{i}_y f(y) \right|   \right\}   +  \sup_{y  \in [1, +\infty)} y^{-a'} \left\{ \sum_{i=0}^2   y^i \left| \partial^{i}_y f(y) \right|   \right\} .\label{defi-norm-X-y-0-a-a'}
\end{eqnarray}
\end{proposition}
\begin{proof}
See Appendix \ref{proof-propo-outer-eigenfunctions}.
\end{proof}

\subsection{  Matching asymptotic   }\label{mathching-asymptotic-proof}  
This part  is devoted to  conclude the proof of  the diagonalisation on $\mathscr{L}_b$.   

\begin{proof}[Proof of  Proposition \ref{propo-mathscr-L-b}:]  
Let $ i \in \{0,1,...,\ell\}$ where $\ell \in \mathbb{N}, \ell \ge 2, \beta \in \left[\frac{1}{4}, \frac{3}{4}  \right]$, $ y_0 \le y^*_1, b \le b^*_1 $  such that Propositions  (\ref{proposition-inner-eigen-functions}- \ref{propo-outer-eigenfunctions}) hold and $ \phi_{i,int}$ and $\phi_{i, out, \beta}$ are defined  in there.

\medskip
\textit{ A) The proof of item (I):} We define  
\begin{equation}\label{definition-phi-i-b}
	\phi_{i,b, \beta } (y) = \left\{      \begin{array}{rcl}
		& &      b^{-\frac{\gamma}{2}  } \phi_{i,int, \beta} \left( \frac{y}{\sqrt{b}} \right)         \text{ if  } y \in [0,y_0],     \\[0.3cm]
		& & \frac{ b^{-\frac{\gamma}{2}  } \phi_{i,int, \beta} \left( \frac{y_0}{\sqrt{b}} \right)}{\phi_{i,out, \beta }(y_0)}  \phi_{i,out, \beta}(y)    \text{ if } y \in [y_0, +\infty).
	\end{array}   \right.    
\end{equation}
The main goal is to   prove that there exists $y_0  \in (0,1)$ small enough and $b^*(y_0)>0$ such that, for all    $b \in (0, b^*)$,  there  exists a unique $ \tilde \lambda_i(b, \beta) = \tilde \lambda $ satisfying 
\begin{equation}\label{check-L-b-phi-i=lamda-i-phi-i}
	\mathscr{L}_{b} \phi_{i,b, \beta} = \left(  2 \beta \left(\frac{\alpha}{2} - i \right)+ \tilde \lambda  \right) \phi_{i,b, \beta},
\end{equation}
and  $\tilde \lambda $  satisfies  \eqref{estimat-tilde-lambda-b-beta-pro}.

\medskip
First, we observe  that  $\phi_{i, int, \beta} \in C^\infty\left(\left[ 0, \frac{y_0}{\sqrt{b}} \right]\right)$ and $\phi_{i, out, \beta} \in C^\infty[y_0, +\infty)$ and they solve the regular second order differential equations, so $\phi_{i,b, \beta} \in C^\infty[0,+\infty)$ if and only if 
\begin{equation}
	b^{-\frac{\gamma}{2} -\frac{1}{2}}  \partial_\xi \phi_{i,int, \beta} \left( \frac{y_0}{\sqrt{b}} \right)  =  b^{-\frac{\gamma}{2}}  \frac{  \phi_{i,int, \beta} \left( \frac{y_0}{\sqrt{b}} \right)}{\phi_{i, out, \beta}(y_0)} \partial_y \phi_{i,out, \beta}(y_0),
\end{equation}
this condition  ensures  $\phi_{i, b, \beta}$'s differential  is continuous at  $y_0$. In particular, it is   equivalent to
\begin{equation}
	b^{-\frac{1}{2}} \partial_\xi \phi_{i,int, \beta} \left( \frac{y_0}{\sqrt{b}}\right) \phi_{i, out, \beta}(y_0) - \phi_{i,int, \beta} \left( \frac{y_0}{\sqrt{b}}\right) \partial_y \phi_{i,out, \beta}(y_0)=0.
\end{equation}
\iffalse
This implies 
$$ F[y_0] ( \tilde \lambda, b, \beta)   =0,$$
where
$ F$ defined by

\begin{eqnarray*}
	F[y_0] ( \tilde \lambda, b)   &=& \frac{b^{-\frac{1}{2}} \partial_\xi \phi_{i,int} \left( \frac{y_0}{\sqrt{b}}\right)}{\phi_{i,int}\left( \frac{y_0}{\sqrt{b}}\right)} - \frac{\partial_y \phi_{i, out, \beta}(y_0)}{ \phi_{i,}(y_0)}\\
	&=& .
\end{eqnarray*}
\fi
Here we use the implicit function theorem by applying it to the function $\tilde F[y_0](\tilde \lambda, b, \beta)$ defined by
$$ \tilde F[y_0](\tilde \lambda, b, \beta)  =b^{-\frac{1}{2}} \partial_\xi \phi_{i,int, \beta} \left( \frac{y_0}{\sqrt{b}}\right) \phi_{i,ext, \beta}(y_0) - \phi_{i,int, \beta} \left( \frac{y_0}{\sqrt{b}}\right) \partial_y \phi_{i,out, \beta}(y_0).      $$
We firstly prove the following expansions:
\begin{eqnarray}
	\tilde F[y_0] (\tilde \lambda, b, \beta) &=& \tilde \lambda K_0 a_{i,0}(\tilde \gamma - \gamma)  y^{ -\gamma -\tilde \gamma -1}_0 ( 1 + O(y_0^2) + O(|\tilde \lambda|))\nonumber \\
	&+& O(b^{1-\frac{\epsilon}{2}} ), \label{F-y-0-lambda-b}\\
	\partial_b \tilde F[y_0](  \tilde \lambda, b, \beta) & = & O(y_0^{-2\gamma-1+\epsilon} b^{-\frac{\epsilon}{2}}) + O(|\tilde \lambda | b^{-1} y_0^{-2\gamma +3}),\label{partial-b-F-y-0-lambda-b}\\
	\partial_{\tilde \lambda } \tilde F[y_0](\tilde \lambda , b, \beta) & = & (\tilde \gamma - \gamma) K_0 a_{i,0} y_0^{ -\gamma -\tilde \gamma -1}   (1  + O(y_0^2) + O(|\tilde \lambda |)),  \label{partial-lambda-F-y-0-lambda-b} 
\end{eqnarray}
and for $\beta$-derivative  
\begin{eqnarray}
	\\
	\partial_\beta \tilde F[y_0](\tilde \lambda, b, \beta) & = &   \left\{      \begin{array}{rcl}
		\tilde \lambda 2 a_{i,1} K_0  (2 + \tilde \gamma - \gamma) y_0^{ - \tilde \gamma - \gamma +1} (1 + O(|y_0|^2) + O(|\tilde \lambda|))   &  \text{ if } & i \ge 1,  \\[0.2cm]
		+ O(b^{ 1 - \frac{\epsilon}{2} } ) &&  \\  O(|\tilde \lambda|y_0^{- \tilde \gamma - \gamma +3})+ O(|\tilde \lambda^2|  ) + O(b^{1-\frac{\epsilon}{2}})   & \text{ if } & i =0.
	\end{array}   \right.   \label{partial-beta-F-y-0-lambda-b}
\end{eqnarray}
Since the proof  of  asymptotic expansions   (\ref{F-y-0-lambda-b}-\ref{partial-beta-F-y-0-lambda-b})  are   technical and long, we complete them when we finish the proof of Proposition \ref{propo-mathscr-L-b}. Assume that the asymptotic expansions hold  for all $b \in (0, b^*(y_0)),$ $ y_0 \le y_0^*$, and $\beta \in \left[ \frac{1}{4}, \frac{3}{4} \right]$. We mention  that the expansions are  uniform in $\beta, \tilde \lambda$ and $\beta$. So, the argument from  the implicit function theorem  yields that     $\forall b \in (0,b^*(y_0))$ and $\beta \in \left[\frac{1}{4}, \frac{3}{4} \right]$,  there  exists  a unique $\tilde \lambda = \tilde \lambda (b, \beta) $ such that
$$\tilde F[y_0] (\tilde \lambda, b, \beta)=0. $$
In particular,  \eqref{F-y-0-lambda-b}  ensures that
$  \tilde \lambda(b, \beta )  = O\left(b^{1-\frac{\epsilon}{2}} \right)$  and expansions (\ref{partial-b-F-y-0-lambda-b}-\ref{partial-beta-F-y-0-lambda-b}) imply 
$$ \left| b \partial_b  \tilde \lambda (b, \beta)  \right| \lesssim_{y_0} b^{1 -\frac{\epsilon}{2}} \text{ and }  \left| \partial_\beta \tilde \lambda (b, \beta) \right| \lesssim_{y_0} 1 $$
yielding \eqref{estimat-tilde-lambda-b-beta-pro}. Next, we decompose $\phi_{i, b, \beta}$   as follows  
\begin{equation}
	\phi_{i, b, \beta}(y) = \sum_{j=0}^i c_{i,j}(2\beta)^j (\sqrt{b})^{2j-\gamma} T_{j} \left( \frac{y}{\sqrt{b}} \right)  + \tilde \phi_{i,b} (y),
\end{equation}
and we aim to prove 
\begin{equation}\label{norm-H-1rho-tilde-phi}
	\| \tilde \phi_{i,b}\|_{H^1_\rho} \le Cb^{1 -\frac{\epsilon}{2}}.   
\end{equation}
In particular,  we can specify it by 
\begin{equation*}
	\tilde \phi_{i,b, \beta} (y) = \left\{      \begin{array}{rcl}
		& &      b^{-\frac{\gamma}{2}  } \phi_{i,int, \beta} \left( \frac{y}{\sqrt{b}} \right)  - \sum_{j=0}^i c_{i,j} \left(\sqrt{b}\right)^{2j-\gamma} T_j\left( \frac{y}{\sqrt{b}} \right)        \text{ if  } y \in [0,y_0],     \\[0.3cm]
		& & \frac{ b^{-\frac{\gamma}{2}  } \phi_{i,int, \beta} \left( \frac{y_0}{\sqrt{b}} \right)}{\phi_{i, out, \beta}(y_0)}  \phi_{i, out, \beta}(y) - \sum_{j=0}^i c_{i,j} \left(\sqrt{b}\right)^{2j-\gamma} T_j\left( \frac{y}{\sqrt{b}} \right)  \text{ if } y \in [y_0, +\infty).
	\end{array}   \right.    
\end{equation*}
Now, we aim to prove that
\begin{equation}\label{point-wise-tilde-phi-i-b}
	|\partial_y^k \tilde \phi_{i,b, \beta} (y) | \le C\left(  y^{  -\gamma +2 -k} I_{y \in [0,y_0]} + y^{-\gamma +2i+2-k} I_{y \in [y_0,+\infty)} \right) b^{1 -\frac{\epsilon}{2}}, \quad  y \in \R \text{ and  } k = 0,1.
\end{equation}
Since the proofs for  $k=0$ and $k=1$ are the same, we  only give the proof of \eqref{point-wise-tilde-phi-i-b}     for the case  $k=0$ and we kindly refer the reader to check the details. Let us start the proof by considering two cases, namely,  $ y \in [0,y_0]$ and $ y \in [y_0, +\infty)$. 
\begin{enumerate}
	\item \underline{$ y \in [0,y_0]$:} write $\tilde \phi_{i,b}$ as 
	$$ \tilde \phi_{i,b} (y) =   \tilde \lambda \sum_{j=0}^i  b^{j+1-\frac{\gamma}{2}}\left[ c_{i,j} T_{j+1} \left(\frac{y}{\sqrt{b}} \right) + S_j\left(\frac{y}{\sqrt{b}} \right) \right] + b^{1-\frac{\gamma}{2}} R_i\left(\frac{y}{\sqrt{b}} \right).$$
	According to Lemma \ref{lemma-Generation-H}, we have 
	$$ \left| T_{j+1} (\xi) \right| \leq C  \xi^{-\gamma +2j+2},\quad \forall \xi \in \R^+,    $$
	so that
	$$ \left| \tilde \lambda b^{1+j-\frac{\gamma}{2}}  c_{i,j} T_{j+1}\left( \frac{y}{\sqrt{b}} \right)  \right| \le C |\tilde \lambda |  y^{2j+2 -\gamma}. $$
	Proposition \ref{proposition-inner-eigen-functions} yields
	\begin{eqnarray*}
		\left| \tilde \lambda b^{1+j-\frac{\gamma}{2}} S_j \left( \frac{y}{\sqrt{b}} \right) \right| \le C \left|\tilde \lambda   \right|  y^{ 2j+2-\gamma},\quad \textit{and} \quad b^{1 -\frac{\gamma}{2}} \left| R_i \left(\frac{y}{\sqrt{b}} \right) \right| \le C b^{1-\frac{\epsilon}{2}}y^{-\gamma +\epsilon}.
	\end{eqnarray*}
	The above three estimates allows us to infer that, for $y\in [0,y_0]$
	$$ |\tilde \phi_{i,b}(y)| \le C b^{1-\frac{\epsilon}{2}} y^{2-\gamma}. $$
	\item \underline{$ y\in [y_0, +\infty]$:} write $\tilde \phi_{i,b}$ as follows 
	\begin{eqnarray*}
		\tilde \phi_{i,b} (y) & = &  \frac{ b^{-\frac{\gamma}{2}  } \phi_{i,int} \left( \frac{y_0}{\sqrt{b}} \right)}{\phi_{i,ext}(y_0)}  \phi_{i,ext}(y) - \sum_{j=0}^i c_{i,j} \left(\sqrt{b}\right)^{2j-\gamma} T_j\left( \frac{y}{\sqrt{b}} \right) \\
		& = &\underbrace{ \phi_{i,ext} (y) - \sum_{j=0}^i c_{i,j} (\sqrt{b})^{2j-\gamma} T_{j} \left( \frac{y}{\sqrt{b}} \right)}_{=I} +\underbrace{ \left[ \frac{ b^{-\frac{\gamma}{2}  } \phi_{i,int} \left( \frac{y_0}{\sqrt{b}} \right)}{\phi_{i,ext}(y_0)} -1 \right] \phi_{i,ext}(y)}_{=II}.
	\end{eqnarray*}
	Let 
	\begin{equation}\label{defi-tilde-T-j}
		\tilde T_{j} (\xi) =T_j(\xi) - C_j \xi^{2j-\gamma}, 
	\end{equation}
	and we decompose  $\phi_{i,ext, \beta}$ as follows 
	\begin{eqnarray*}
		I &=&   \phi_{i,ext} (y) - \sum_{j=0}^i c_{i,j} (\sqrt{b})^{2j-\gamma} C_j \left( \frac{y}{\sqrt{b}} \right)^{2j-\gamma}
		-   \sum_{j=0}^i c_{i,j} (\sqrt{b})^{2j-\gamma} \tilde T_{j} \left( \frac{y}{\sqrt{b}} \right) \\
		& = & -  \sum_{j=0}^i c_{i,j} (\sqrt{b})^{2j-\gamma} \tilde T_{j} \left( \frac{y}{\sqrt{b}} \right) + \tilde \lambda ( \tilde \phi_i + \tilde R_{i,1}) + \tilde R_{i,2} 
	\end{eqnarray*}
	Lemma \ref{lemma-Generation-H} gives, for $y\geq y_0$
	\begin{eqnarray*}
		\left|   (\sqrt{b})^{2j-\gamma} \tilde T_j \left(\frac{y}{\sqrt{b}} \right)   \right| & \leq & C y^{-\gamma +2j -2 } |\ln y| b |\ln b| \\
		& \le & C y^{-\gamma +2j -2} |\ln y| b^{1 - \frac{\epsilon}{2}}.
	\end{eqnarray*} 
	%In addition to that, we use error estimates in Lemma \ref{lemma-phi-1-psi-2-outer} and  
	From Proposition \ref{propo-outer-eigenfunctions}, we deduce that
	\begin{eqnarray*}
		\left|  \tilde \phi_{i}(y) \right| &\le & C y^{2i-\gamma  } |\ln y|,\\
		\left| \tilde R_{i,1}(y)  \right| &\le &C \left( y^{-\tilde \gamma} + y^{2i-\gamma +2}  \right) |\tilde \lambda|, \\
		\left|   \tilde R_{i,2} (y)       \right| & \leq &      C ( y^{-\tilde \gamma -2-\alpha } + y^{ 2i+2 -\gamma} )b^\alpha, 
	\end{eqnarray*}
	for all $ y \geq y_0$. Since $\alpha\geq 1-\frac{\epsilon}{2}$, we obtain for all  $y \ge y_0$
	\begin{equation}
		\label{part 1}
		\left| \phi_{i, out, \beta} (y) - \sum_{j=0}^i c_{i,j} (\sqrt{b})^{2j-\gamma} T_{j} \left( \frac{y}{\sqrt{b}} \right) \right| \le C  y^{2i +2 -\gamma }|\ln y| b^{1-\frac{\epsilon}{2}}. 
	\end{equation}
	For the term II, we use the estimate obtained for I at $y=y_0$ to get
	\begin{eqnarray*}
		|\text{II}|=\left| \frac{ b^{-\frac{\gamma}{2}  } \phi_{i,int} \left( \frac{y_0}{\sqrt{b}} \right)}{\phi_{i,ext}(y_0)} - 1  \right| \le C (y_0) b^{1-\frac{\epsilon}{2}}. 
	\end{eqnarray*}
	Putting together the estimates for I and II yields
	for all $ y \ge y_0$
	$$  \left|\tilde \phi_{i,b}(y) \right|\le C(y_0) y^{-\gamma+2i+2}  b^{1-\frac{\epsilon}{2}}. $$
	Similarly for $\partial_y \tilde \phi_{i,b}$, we establish
	$$  \left|\partial_y \tilde \phi_{i,b}(y) \right|\le C(y_0) y^{-\gamma+2i+1}  b^{1-\frac{\epsilon}{2}}.$$
	
\end{enumerate}
-  Now,  we have 
%to prove the following estimate
%$$ \| \phi_{i,b} - \phi_{i,\infty}\|_{H^1_\rho} \le  C b^{1-\frac{\epsilon}{2}}.$$
%We firstly have 
$$ \| \phi_{i,b  } - \phi_{i,\infty}\|_{H^1_\rho} \le \left \|\sum_{j=0}^i c_{i,j} (\sqrt{b} )^{2j-\gamma} T_j\left( \frac{y}{\sqrt{b}}\right) - \phi_{i,\infty} \right\|_{H^1_\rho}  +\| \tilde \phi_{i,b}\|_{H^1_\rho}.  $$
Taking into account,  \eqref{norm-H-1rho-tilde-phi}, it is sufficient to establish 
$$ \left \|\sum_{j=0}^i c_{i,j} (\sqrt{b} )^{2j-\gamma} T_j\left( \frac{y}{\sqrt{b}}\right) - \phi_{i,\infty} \right\|_{H^1_\rho} \le C b^{1- \frac{\epsilon}{2}}.$$
As above, we have
\begin{eqnarray*}
	\sum_{j=0}^i c_{i,j} (\sqrt{b} )^{2j-\gamma} T_j\left( \frac{y}{\sqrt{b}}\right) - \phi_{i,\infty}  (y)   &=& \sum_{j=0}^i c_{i,j} (\sqrt{b} )^{2j-\gamma} \tilde T_j\left( \frac{y}{\sqrt{b}}\right).
\end{eqnarray*}
%From Lemma \ref{lemma-Generation-H}, we derive
% $$ \left| b^{j -\frac{\gamma}{2}} \tilde T_j \left( \frac{y}{\sqrt{b}} \right)  \right| \le C y^{-\gamma+2j -2} |\ln y | b^{1 -\frac{\epsilon}{2}}, $$
which yields, after splitting the integral in two regions $\{y\leq \sqrt{b}\}$ and $\{y\geq \sqrt{b} \}$ then using Lemma \ref{lemma-Generation-H}, to
$$ \left\|  \sum_{j=0}^i c_{i,j} (\sqrt{b})^{2j -\gamma}  \tilde T_j \left( \frac{y}{\sqrt{b}} \right)  \right\|_{H^1_\rho} \le C b^{1-\frac{\epsilon}{2}}.$$

- Now, we move to the proof of item $(iii)$ in Proposition \ref{propo-mathscr-L-b}. We distinguish two regions:
\begin{itemize}
	\item \underline{$ y \in [0,y_0]$}: from definition \eqref{definition-phi-i-b}, we have
	\begin{eqnarray*}
		\phi_{i,b}(y) = b^{   - \frac{\gamma}{2} }  \phi_{i,int} \left( \frac{y}{\sqrt{b}} \right)  &=& \sum_{j=0}^i c_{i,j} (\sqrt{b})^{2j -\gamma} T_j \left(\frac{y}{\sqrt{b}} \right) \\
		&+& \tilde \lambda \sum_{j=0}^i  b^{j+1-\frac{\gamma}{2}}\left[ c_{i,j} T_{j+1} \left(\frac{y}{\sqrt{b}} \right) + S_j\left(\frac{y}{\sqrt{b}} \right) \right] + b^{1-\frac{\gamma}{2}} R_i\left(\frac{y}{\sqrt{b}} \right).   
	\end{eqnarray*}   
	Since, for all $\xi \in \R$
	$$ | T_j(\xi)| \le C \frac{\xi^{2j}}{1 + \xi^\gamma}, $$
	we have, by Lemma \ref{lemma-Generation-H}
	$$\left|\sum_{j=0}^i b^{j -\frac{\gamma}{2}} T_j\left(\frac{y}{\sqrt{b}} \right)  \right|  \le C(i+1) \frac{\langle y\rangle^{2i}}{(\sqrt{b}+y)^{\gamma}}.$$
	and
	$$\left|\sum_{j=0}^i b^{j+1 -\frac{\gamma}{2}} T_{j+1}\left(\frac{y}{\sqrt{b}} \right)  \right| \le C(i+1) \frac{\langle y\rangle^{2i+2}}{(\sqrt{b}+y)^{\gamma}}.$$  
	For $S_j$, we have
	\begin{eqnarray*}
		\left| \tilde \lambda \sum_{j=0}^i  b^{1+j-\frac{\gamma}{2}} S_j\left(\frac{y}{\sqrt{b}} \right)\right| \le C \frac{\langle y\rangle^{2i+2}}{(\sqrt{b } + y)^\gamma},
	\end{eqnarray*}
	while for $R_i$ 
	$$\left|b^{1-\frac{\gamma}{2}} R_i \left(  \frac{y}{\sqrt{b}} \right) \right| \leq C \frac{\langle y \rangle^{\epsilon}}{(\sqrt{b}   +y)^{\gamma}}.$$
	The above allows one to conclude that 
	$$ \left| \phi_{i,b}(y) \right| \le C\frac{\langle y \rangle^{2i+2}}{(\sqrt{b}   +y)^{\gamma}}, \forall y \in [0,y_0].$$
	\item \underline{$ y\in [y_0,+\infty)$}: from the definition of $\phi_b$ on this region and the fact that 
	$$ \left|b^{-\frac{\gamma
		}{2}} \frac{\phi_{i,int}(\frac{y_0}{\sqrt{b}})}{ \phi_{i,ext}(y_0)} \right| \le C (y_0),  $$
	it is sufficient to estimate $ \phi_{i,ext}$. Recall that
	\begin{eqnarray*}
		\phi_{i,ext} (y) = \phi_{i,\infty}(y) +\tilde \lambda ( \tilde \phi_{i}(y) + \tilde R_{i,1}) + \tilde R_{i,2}.
	\end{eqnarray*}
	The asymptotic behavior of $\tilde \phi_{i} $ yields, for all $ y \ge y_0$
	$$ \left| \tilde \phi_i (y) \right| \le C \frac{y^{2i} |\ln y|}{y^\gamma} \le C(y_0) \frac{\langle y \rangle^{2i+2}}{(\sqrt{b} +y)^\gamma}.$$
	Moreover,  we have the following facts: for all $y \ge y_0$
	\begin{eqnarray*}
		\left| \tilde R_{i,1}(y) \right|  & \le &  C(y_0)|\tilde \lambda | \frac{\langle y \rangle^{2i+2}}{(\sqrt{b} +y)^\gamma},\\
		\left| \tilde R_{i,2}(y) \right| & \le &  C(y_0)b^\alpha \frac{\langle y  \rangle^{2i+2}}{(\sqrt{b} +y)^\gamma}. 
	\end{eqnarray*}
	Putting together the above estimates, one gets
	$$ \left| \partial_y \phi_{i,b}(y) \right| \le C\frac{\langle y \rangle^{2i+2}}{(\sqrt{b}   +y)^{\gamma}}, \forall y \in [y_0,\infty).$$
	A similar reasoning allows us to obtain the rest of the estimates, we omit the details.
\end{itemize} 

\textit{-  Proof  of   \eqref{F-y-0-lambda-b} :}   

First, we decompose   $\phi_{i,int, \beta}$ and $\phi_{i, out, \beta}$ by 
\begin{eqnarray*}
	\phi_{i,int, \beta} \left( \frac{y_0}{\sqrt{b}}\right) & = & b^{ \frac{\gamma}{2} } \left\{ \phi_{i,\infty, \beta}(y_0)  + \tilde \lambda \sum_{j=0}^i c_{i,j} C_{j+1}(2\beta)^{j+1}  y_0^{-\gamma+2j +2} +A_{i,1}(\tilde \lambda, y_0,b, \beta) \right\},\\
	b^{-\frac{1}{2}} \partial_\xi \phi_{i,int, \beta} \left( \frac{y_0}{\sqrt{b}}\right) & = & b^{ \frac{\gamma}{2} } \left\{ \partial_y \phi_{i,\infty, \beta}(y_0)  + \tilde \lambda \sum_{j=0}^i c_{i,j} C_{j+1} (2\beta)^{j+1}(-\gamma+2j+2)  y_0^{-\gamma+2j +1}  \right.\\
	& + & \left. A_{i,2}(\tilde \lambda, y_0,b, \beta) \right\},
\end{eqnarray*}
and 
\begin{eqnarray*}
	\phi_{i, out, \beta} \left( y_0  \right) & = & \phi_{i,\infty, \beta} (y_0) +  \tilde \lambda  K_0 y_0^{ -\tilde \gamma  } +    B_{i,1} (\tilde \lambda, y_0, b, \beta),\\
	\partial_y \phi_{i, out, \beta} (y_0) & = & \partial_y\phi_{i,\infty, \beta} (y_0) - \tilde{\lambda}  K_0 \tilde \gamma y_0^{ - \tilde \gamma -1} + B_{i,2}(\tilde \lambda, y_0, b, \beta).  
\end{eqnarray*}
where 
\begin{eqnarray*}
	A_{i,1} & = & \sum_{j=0 }^i c_{i,j}(2\beta)^{j} b^{j-\frac{\gamma}{2}} \tilde T_j \left( \frac{y_0}{\sqrt{b}} \right) + \tilde \lambda \sum_{j=0}^i   c_{i,j}(2\beta)^{j+1} b^{j+1 -\frac{\gamma}{2}} \tilde T_{j+1} \left( \frac{y_0}{\sqrt{b}}\right) \\
	& +& \tilde \lambda \sum_{j=0}^i b^{j+1-\frac{\gamma}{2}} S_j \left( \frac{y_0}{\sqrt{b}} \right) + b^{1-\frac{\gamma}{2}} R_i \left( \frac{y_0}{\sqrt{b}}\right),\\
	A_{i,2} & = & \sum_{j=0}^i c_{i,j}(2\beta)^{j} b^{j-\frac{1}{2} -\frac{\gamma}{2}} \partial_\xi \tilde T_j \left( \frac{y_0}{\sqrt{b}}\right) + \tilde \lambda  \sum_{j=0}^i c_{i,j}(2\beta)^{j+1} b^{j+\frac{1}{2} - \frac{\gamma}{2}} \partial_\xi \tilde T_{j+1} \left( \frac{y_0}{\sqrt{b}} \right)\\
	&  +  &  \tilde \lambda \sum_{j=0}^i b^{j +\frac{1}{2}-\frac{\gamma}{2}} \partial_\xi S_j\left( \frac{y_0}{\sqrt{b}} \right)  + b^{\frac{1}{2} -\frac{\gamma}{2} } \partial_\xi  R_i \left( \frac{y_0}{\sqrt{b}}\right),
\end{eqnarray*} 
and 
\begin{eqnarray*}
	B_{i,1}   & = &  \tilde \lambda (\tilde \phi_{i, \beta} - K_0 y_0^{-\tilde \gamma}) + \tilde \lambda \tilde R_{i,1}  +  \tilde R_{i,2},\\
	B_{i,2} & = & \tilde \lambda \partial_y \left(\tilde \phi_{i, \beta} - K_0 y_0^{-\tilde \gamma}  \right) + \tilde \lambda \partial_y \tilde R_{i,1} + \partial_y \tilde R_{i,2}, 
\end{eqnarray*}
and $\tilde T_j $  and  $C_j$ defined  as in \eqref{defi-tilde-T-j} and  \eqref{definition-C-j-new}, respectively.

We aim to estimate $A_i$ and $B_i$ by using the results of Propositions  \ref{proposition-inner-eigen-functions} and \ref{propo-outer-eigenfunctions}.
\begin{itemize}
	\item \underline{estimate on  $A_{i,1}$}: From  Lemma  \ref{lemma-Generation-H}, we use $T_j$'s expansion at $\infty$, to obtain   the following
	$$| \tilde T_j (\xi) | \leq C \xi^{-\gamma +2j -2} \ln \xi, $$
	for all $\xi$ large enough, i.e $\xi \geq \xi_0 >1$.  Applying the above for $\xi_0 = \frac{y_0}{\sqrt{b}}$, we derive the following
	$$  \left|\tilde  T_{j} \left( \frac{y_0}{\sqrt{b}}  \right)  \right|  \leq C  \left(  \frac{y_0}{\sqrt{b}}\right)^{-\gamma+2j -2}(  |\ln y_0 | +  |\ln b|), \forall j \geq 1,   $$
	and  for $ j=0$, we have
	$$ \left| \tilde T_{0}\left(\frac{y_0}{\sqrt{b}} \right) \right|  \leq C \left(\frac{y_0}{\sqrt{b}} \right)^{-\gamma -2}. $$
	This yields 
	\begin{eqnarray*}
		\left|  \sum_{j=0}^i    c_{i,j} b^{-\frac{\gamma}{2} +j} \tilde T_{j} \left( \frac{y_0}{\sqrt{b}} \right)  \right| \le C y^{-\gamma-2}_0  b |\ln b|.
	\end{eqnarray*}
	For the second term of $A_{i,1}$, the same   process as above gives
	$$  \left|\sum_{j=0}^i c_{i,j} b^{j+1-\frac{\gamma}{2}} \tilde T_{j+1} \left( \frac{y_0}{\sqrt{b}} \right)\right|   \leq C y^{-\gamma -2}_0  b |\ln b|.  $$
	We now estimate to $S_j$.  Accordingly  to  Proposition \ref{proposition-inner-eigen-functions}, and the definition of  $ X_{\xi_0}^{2j+2-\gamma}$, we have
	$$ \left|S_j \left(  \frac{y_0}{\sqrt{b}}\right)\right| \leq C y_0^2 \left(  \frac{y_0}{\sqrt{b}}\right)^{2j+2-\gamma},$$
	so that
	$$ \left| \tilde \lambda \sum_{j=0}^i b^{j+1-\frac{\gamma}{2}} S_j\left( \frac{y_0}{\sqrt{b}}\right) \right| \leq C\tilde  \lambda y_0^{4-\gamma}.$$
	The last term in $A_1$is to be estimated again via Proposition \ref{proposition-inner-eigen-functions}, where we have
	$$ \left| R_i \left( \frac{y_0}{\sqrt{b}} \right) \right| \leq C \left( \frac{y_0}{\sqrt{b}} \right)^{-\gamma +\epsilon}, \text{ with } \epsilon \ll 1.    $$
	%$$ \left|S_j \left(  \frac{y_0}{\sqrt{b}}\right)\right| \leq C y_0^2 \left(  \frac{y_0}{\sqrt{b}}\right)^{2j+2-\gamma}. $$
	%$$ \left| R_i \left( \frac{y_0}{\sqrt{b}} \right) \right| \leq C \left( \frac{y_0}{\sqrt{b}} \right)^{-\gamma +\epsilon}, \text{ with } \epsilon \ll 1.    $$
	Hence
	$$ \left| b^{1-\frac{\gamma}{2}} R_i \left(\frac{y_0}{\sqrt{b}} \right) \right| \leq C y_0^{-\gamma +\epsilon} b^{1 -\frac{\epsilon}{2}}.  $$ 
	Finally,  we get
	$$ A_1  = O(\tilde \lambda  y_0^{4- \gamma})   + O \left( y_0^{-\gamma -2} b^{1 -\frac{\epsilon}{2}}  \right).$$
	\item \underline{For $A_{i,2}$}: First,  by   Lemma \ref{lemma-Generation-H} we get
	
	$$ \left|\partial_\xi \tilde T_j\left(\xi_0 \right) \right| \leq C \xi_0^{ -\gamma +2j -3} |\ln \xi_0|, \forall j \geq 1, $$
	and 
	$$ \left| \partial_\xi \tilde T_0(\xi_0)  \right| \leq C |\xi_0|^{-\gamma -g-1}. $$
	Then 
	\begin{eqnarray*}
		\left| \sum_{j=0}^{i} c_{i,j} b^{-\frac{\gamma}{2} - \frac{1}{2} +j} \partial_\xi \tilde T_j(\xi_0) \right| & \leq &  C  \sum_{j=1}^i b^{j-\frac{\gamma}{2} -\frac{1}{2}}|\xi_0 |^{-\gamma+2j -3 }|\ln \xi_0| +  Cb^{-\frac{\gamma}{2} -\frac{1}{2}} |\xi_0|^{-\gamma -g-1}\\
		& \leq & C  y_0^{-\gamma -1} |\ln y_0| b |\ln b| +  C y_0^{ -\gamma -g-1} b^{\frac{g}{2}}.
	\end{eqnarray*}
	Next, we estimate the second term in  $A_{i,2}$:
	$$ \left| \tilde \lambda  \sum_{j=0}^i c_{i,j} b^{j+\frac{1}{2} - \frac{\gamma}{2}} \partial_\xi \tilde T_{j+1} \left( \frac{y_0}{\sqrt{b}} \right)  \right| \leq C |\tilde \lambda | y_0^{ -\gamma -1} |\ln y_0| b^{\frac{3}{2}} |\ln b|.  $$
	Using Proposition \ref{proposition-inner-eigen-functions} for $S_j$ 
	$$ \left| \partial_\xi S_j (\xi_0) \right| \leq C |\xi_0|^{ a -1} \le C |\xi_0|^{ 2j +1-\gamma},  $$
	we obtain 
	$$ \left| \tilde \lambda \sum_{j=0}^i b^{j +\frac{1}{2}-\frac{\gamma}{2}} \partial_\xi S_j\left( \frac{y_0}{\sqrt{b}} \right)  \right| \leq C    |\tilde \lambda | y_0^{-\gamma+1}. $$
	The last term in $A_{i,2}$ is estimated similarly and we have
	\iffalse
	on $R_i$ in Proposition \ref{propo-outer-eigenfunctions}, we have $$ \left| \partial_\xi R_i \left( \xi_0 \right) \right| \leq  C  |\xi_0|^{ -\gamma +\epsilon -1},$$
	which yields
	\fi
	$$   | b^{1-\frac{\gamma}{2}} \partial_\xi R_i(\xi_0)| \leq C y_0^{-\gamma +\epsilon -1} b^{ \frac{3}{2} -\frac{\epsilon}{2} }.  $$
	Hence, the expansion of $A_{i,2}$ is
	$$ A_{i,2} = O( |\tilde \lambda | y_0^{-\gamma+1} ) + O( y_0^{-\gamma -1}|\ln y_0| b |\ln b|) .$$
	\item \underline{For $B_{i,1}$}
	Using Lemma \ref{lemma-phi-1-psi-2-outer} we have 
	$$ \left|  \tilde \lambda ( \tilde \phi_{i}(y_0) - K_0 y_0^{-\tilde \gamma}  )\right| \leq C |\tilde \lambda | y_0^{-\tilde \gamma +2}.$$
	For the  second and third term, we use Proposition  \ref{propo-outer-eigenfunctions} 
	$$  | \tilde R_{i,1} (y_0) | \leq C |\tilde \lambda | y_0^{  -\tilde \gamma }   $$
	% Using again the  errors estimates, we can do for $ \tilde R_{i,2}$
	and
	$$ \left| \tilde R_{i,2} (y_0) \right| \le C y_0^{-\tilde \gamma -2 -\alpha} b^\alpha.  $$
	Then, $B_{i,1}$ reads as follows
	$$ B_{i,1}(y_0)  = O\left( | \tilde \lambda  | y_0^{-\tilde \gamma +2} \right) + O(|\tilde \lambda |^{2} y_0^{-\tilde \gamma}) +  O(y_0^{-\tilde \gamma -2-\alpha} b^\alpha).$$
	\item \underline{For $B_{i,2}$}: a similar reasoning gives
	$$ B_{i,2}(y_0)  = O\left( |\lambda \tilde | y_0^{-\tilde \gamma +1} \right) + O(|\tilde \lambda |^{2} y_0^{-\tilde \gamma-1}) +  O(y_0^{-\tilde \gamma -3-\alpha} b^\alpha).$$
	\iffalse 
	+ :   Using again Lemma \ref{lemma-phi-1-psi-2-outer}, we obtain
	$$ \left|   \tilde \lambda \partial_y( \tilde \phi_{i}(y_0) - K_0 y_0^{-\tilde \gamma}  )\right|  \leq C |\tilde \lambda | y_{0}^{-\tilde \gamma +1}.   $$
	
	In addition to that, we use again  error estimates in  Proposition \ref{propo-outer-eigenfunctions}, 
	$$  \left|\partial_y \tilde R_{i,1}(y_0) \right| \leq C |\tilde \lambda| y_0^{-\tilde \gamma -1},  $$
	and 
	$$ \left| \partial_y   \tilde R_{i,2}(y_0) \right| \leq   C y_0^{-\tilde \gamma -3 -\alpha}b^\alpha.  $$
	Finally, we can write $ B_{i,2}$ by the following asymptotic
	\fi 
	Putting the above expansions together,  we derive for $ \tilde F[y_0]( \tilde \lambda, b)$
	\begin{eqnarray*}
		\tilde F[y_0](\tilde \lambda, b, \beta)  &=& ( \partial_y\phi_{i,\infty, \beta}(y_0) +  \tilde \lambda \sum_{j=0}^i c_{i,j} (2\beta)^{j+1} C_{j+1}(-\gamma+2j+2)  y_0^{-\gamma+2j +1} +      A_2(\tilde \lambda, y_0, b, \beta) )\\
		& \times  & ( \phi_{i,\infty, \beta} (y_0) +  \tilde \lambda  K_0 y_0^{ -\tilde \gamma  } +    B_1 (\tilde \lambda, y_0, b, \beta)  )
		\\
		& - & ( \partial_y \phi_{i,\infty, \beta} (y_0)  - \tilde{\lambda}  K_0 \tilde \gamma y_0^{ - \tilde \gamma -1} + B_2(\tilde \lambda, y_0, b, \beta) )\\
		& \times & (\phi_{i,\infty, \beta} (y_0) + \tilde \lambda \sum_{j=0}^i c_{i,j} C_{j+1}(2\beta)^{j+1}  y_0^{-\gamma+2j +2}+ A_1(\tilde \lambda, y_0, b, \beta))\\
		& = & \tilde \lambda K_0 a_{i,0}(\tilde \gamma - \gamma)  y^{ -\gamma -\tilde \gamma -1}\left( 1 + O(y_0^2) +O(|\tilde \lambda|) \right) \\
		&+& O\left( y_0^{-2\gamma-2} b^{1 -\frac{\epsilon}{2}} \right) + O\left(y_0^{-\gamma-\tilde \gamma-3-\alpha}  b^{\alpha}\right).
	\end{eqnarray*}
\end{itemize} 
The proofs of  \eqref{partial-b-F-y-0-lambda-b} and \eqref{partial-lambda-F-y-0-lambda-b} follow the same outline.

\end{proof}

\section{ Maximum principal}

The main goal in this section is to use Maximum principal to construct    the sub solution and the  super solution  to  \eqref{equa-varepsilon-appen}   on   the interval $\left[0,b^\frac{\eta}{4}(\tau) \right]$ that leads to  suitable  estimates for $\varepsilon$.

\begin{proposition}[Sub and super solutions]\label{sub-super-solution}
Let us consider $\eta, \tilde \eta $ be  positive constants such that $1 \gg \eta \gg \tilde \eta$,   $A \ge 1$.  We assume furthermore that      $ \varepsilon $ is the solution to \eqref{equa-varepsilon-appen} on $[\tau_0, \tau_1]$     with initial data given in  \eqref{initial-data} and  the flow $(b,\beta)(\tau) \in (C^1(\tau_0,\tau_1])^2 $ satisfy   $(\varepsilon, b, \beta)(\tau) \in V[A,\eta, \tilde \eta](\tau)$ for all $\tau \in [\tau_0, \tau_1]$. Then, there exists  $H(\xi)$ satisfying 
$$ \left| H(\xi)  \right|  \le  C(\eta) \left[ b(\tau)\frac{\xi^2}{1+\xi^\gamma}  +  \frac{b^\frac{\eta}{4}(\tau) }{1 + \xi^\gamma}  \right] ,  \text{ for all } \xi \in \R_+,       $$
such that 
\begin{equation}\label{Maximun-Pro-e}
	\left| \varepsilon(y,\tau) \right| \le  b^{-1}(\tau) H\left(\frac{y}{\sqrt{b(\tau)}} \right), \forall y \in [0, b^{\frac{\eta}{4}} (\tau) ),
\end{equation}
where $Q_b$ defined as in \eqref{Q-b}.  In other words, \eqref{Maximun-Pro-w} remains true  with $\tau$ lager than $\tau_1$ as long as $w$ exists and $b$ satisfies the hypothesis of the Proposition. 
\end{proposition}
\begin{proof} First,  we claim that the following 
\begin{equation}\label{Maximun-Pro-w}
	\left| w(y,\tau) - Q_{b(\tau)}(y) \right| \le  b^{-1}(\tau) H\left(\frac{y}{\sqrt{b(\tau)}} \right), \forall y \in [0, b^{\frac{\eta}{4}} (\tau) )
\end{equation}
implies  \eqref{Maximun-Pro-e}. We also mention that the  proof is similar to the one in \cite{Biernat_2020}  where the authors   constructed sub-solution and super-solution to equation \eqref{equa-w} on the small interval $[0,b^\frac{\eta}{4} (\tau)]$.    Let us   consider  the blowup  variable $(\xi,\tau)$ and  
$$ w(y,\tau)  = \frac{1}{b(\tau)}  \omega\left( \frac{y}{\sqrt{b(\tau)}}, \tau  \right) = \frac{1}{b(\tau)} \omega (\xi,\tau).  $$
Introducing $v = \omega - Q(\xi)$ and    $v $  reads
\begin{equation}\label{equa-v-xi-tau}
	b(\tau) \partial_\tau v   =  \partial_{\xi}^2 v + \frac{d+1}{\xi} \partial_\xi v  -3(d-2) (2Q +\xi^2Q^2)  v +    B(v) +  \theta(\tau) \Lambda_\xi Q +  \theta(\tau) \Lambda_\xi v,
\end{equation}  
where $ B$ defined as in \eqref{defi-B-quadratic-appendix} and $\theta(\tau) $ defined by 
\begin{equation}\label{defi-theta-tau-}
	\theta(\tau) = \beta(\tau)\left( b'(\tau) - b(\tau) \right).
\end{equation}
We also introduce the operator $\mathcal{P}$ as follows
\begin{equation}\label{defi-mathcal-P}
	\mathcal{P}(v) :=  \partial_\xi^2 v + \frac{d+1}{\xi} \partial_\xi v   -  3(d-2) (2Q +\xi^2Q^2) v +\bar B(v) + \theta(\tau) \Lambda_\xi Q  + \theta(\tau) \Lambda_\xi v   -  b(\tau) \partial_\tau v,
\end{equation}
%\begin{eqnarray}
%\bar B (z) := -3(d-2) (1 + \xi^2 Q ) z^2 -(d-2)\xi^2 z^3 .  \label{defi-bar-B}
%\end{eqnarray}
%Note that, we have the following law for $b(\tau)$
%$$ \left| \frac{b_\tau}{b}  - \frac{2}{\alpha} \left( \frac{\alpha}{2} -l \right) \right|  \le C A  e^{ \left[   \left( 1 -\frac{\epsilon}{2}\right) \frac{2}{\alpha} \left( \frac{\alpha}{2} -l \right)  \right] \tau}, \forall \tau \in \left[  \tau_0, \tau_1  \right].   $$
In order to construct the sub-solution and the super-solution, we need to construct  two functions as follows: let  $Q$ be  the ground  state satisfying \eqref{equation-the ground-state},   $Q(0)=-1$,$Q'(0) =0$, and  we introduce
$$ Q_\sigma =   \frac{1}{\sigma} Q\left( \frac{\xi}{\sqrt{\sigma}} \right), \sigma >0,$$
exactly solves  \eqref{equation-the ground-state} thanks to the scaling \eqref{scaling-law}.  
\iffalse
We next do the following linearisation 
$$ z(\xi,\tau)  =  v(\xi,\tau)  -  Q(\xi).    $$
From \eqref{equa-v-xi-tau}, we derive that $z$ reads
\begin{equation}\label{equa-z}
	b(\tau) \partial_\tau z  =  \partial^2_\xi z + \frac{d+1}{\xi} \partial_\xi  z - 3(d-2) ( 2Q +\xi^2Q^2  ) z   + \bar{B}(z) +  \mu(\tau) \Lambda_\xi Q +  \mu(\tau) \Lambda_\xi z ,
\end{equation} 

We next introduce

Formally, we have the fact that $z$ is the small perturbation. So, \eqref{equa-v-xi-tau} leads to the following toy model 
\begin{eqnarray}
	\partial_\xi^2  z  + \frac{d+1}{\xi} \partial_\xi z -3(d-2) (2 Q + \xi^2 Q^2) z = - \mu \Lambda_\xi Q. \label{equa-toy-z} 
\end{eqnarray}
Bearing in mind that $z$  only solves \eqref{equa-toy-z}. We search  the solution $z$ under the following form
$$  z(\xi,\tau)  = \mu(\tau) H(\xi),    $$ 
that leads to the following equation
\begin{equation}\label{equa-so-H}
	H''  + \frac{d+1}{\xi} H' -3(d-2)( 2Q + \xi^2 Q^2) H= - \Lambda_\xi Q :=T(\xi).
\end{equation}
\fi
Next, define 
\begin{equation}\label{defi-H-0}
	H_{0}(\xi)= \Lambda_\xi Q_\sigma(\xi),
\end{equation}
satisfying  
\begin{equation}\label{equa-H-0}
	H_0''  + \frac{d+1}{\xi} H_0' - 3(d-2) ( 2Q_\sigma + \xi^2 Q_\sigma^2) H_0=0,
\end{equation}
and  let $H_1(\xi) $ solve the following
\begin{equation}\label{equa-H-1}
	H_1''  + \frac{d+1}{\xi} H_1' - 3(d-2) ( 2Q_\sigma + \xi^2 Q_\sigma^2) H_1 = T(\xi),
\end{equation}
where $T(\xi) = -\Lambda_\xi Q$. In particular,  $H_1$ is explicitly  given by
\begin{equation}\label{formula-H-1}
	H_1(\xi) = H_0 \int_0^\xi \frac{\mathscr{L}(T)(\xi')}{H_0 (\xi')} d\xi',
\end{equation}
where $\mathcal{L}$ was defined in  \eqref{defi-math-H-inverse}
Using  \eqref{defi-H-0}  and  \eqref{formula-H-1},   $H_0$ and $H_1$ have the following asymptotics:
\begin{eqnarray}
	H_0 \left(  \xi \right)  = \left\{ \begin{array}{rcl}
		-\frac{2}{\sigma} &  \text{ as } & \xi \to 0,  \\
		a_0 \sigma^{\alpha} \xi^{-\gamma}   &  \text{ as } & \xi \to \infty,
	\end{array} \right. \label{comport-H-0}
\end{eqnarray}
where $a_0<0$ and $\alpha=2-\gamma$
%were defined in \eqref{defi-alpha}; 
and
\begin{eqnarray}
	H_1 \left(\xi   \right)  = \left\{ \begin{array}{rcl}
		\frac{\xi^2}{d+2} &  \text{ as } & \xi \to 0,  \\
		\frac{-a_0}{2(d+2-\gamma)} \xi^{2-\gamma}  & \text{ as } & \xi \to \infty.
	\end{array} \right.\label{comport-H-1}
\end{eqnarray}
\iffalse
In particular, we   also have
\begin{eqnarray}
	\Lambda_\xi  H_0( \xi ) = a_0 \beta^{\alpha} (2-\gamma) 
\end{eqnarray}
\fi
\iffalse
For $\xi\sim 0$, we have
\begin{equation}\frac{1}{\alpha}(\Lambda_\xi Q)\left(\frac{\xi}{\sqrt{\alpha}}\right)\sim \textit{Constant} 
\end{equation}
and
\begin{equation}
	\mathscr{L}(T) (\xi) \sim \xi  
\end{equation}
so that
\begin{equation}
	H_1(\xi)\sim \xi^2   
\end{equation}
\fi
Inspired by \cite{Biernat_2020},  we define 
\begin{equation}\label{defi-z-+--}
	v^+(\xi,\tau)  = \theta^+(\tau) H_1(\xi)  - M(\eta) b^\frac{\eta}{4}(\tau)  H_0 \left( \xi \right)      \text{ and } v^{-} = \theta^-(\tau) H_1(\xi)  +   M (\eta) b^{\frac{\eta}{4}}(\tau)  H_0(\xi), 
\end{equation}
where  
\begin{eqnarray}
	\theta^+(\tau) =  b(\tau) \left[ \beta (2\beta-1)  -4\beta^2 \frac{\ell}{\alpha} - b^{\frac{\eta}{8}}(\tau)\right] \text{ and }  \theta^-(\tau) = b(\tau) \left[  \beta (2\beta-1)  -4\beta^2 \frac{\ell}{\alpha} + b^{\frac{\eta}{8}}(\tau)\right].\label{defi-theta-+}
\end{eqnarray}
Note  that $ H_0(\xi) = \Lambda_\xi  Q_{\sigma} < 0$ see  more \eqref{asymptotic-Lamda-Q}.  In particular,   our aim is to prove 
\begin{eqnarray*}
	w^+(y,\tau)  =  Q_{b(\tau)}(y) + \frac{1}{b(\tau)}v^{+}\left( \frac{y}{\sqrt{b(\tau)}}\right) \text{ and } w^- (y,\tau) =  Q_{b(\tau)}(y)  + \frac{1}{b(\tau)}v^{-}\left( \frac{y}{\sqrt{b(\tau)}}\right)
\end{eqnarray*}
are respectively the super-solution and the sub-solution to \eqref{equa-w} which immediately implies \eqref{Maximun-Pro-w}. 
\iffalse
$v^+\left( x\right) (v^-)$ is  super-solution (sub-solution, respectively)  to \eqref{equa-v-xi-tau}  on the ball $ \{ |y| \le  b^{\eta} (\tau)   \}$.  We firstly handle  the plus sign. 
at page 2756
\fi
Following  \cite{Biernat_2020}, it is sufficient to check that
\begin{itemize}
	\item[$(i)$]  $ \mathcal{P}(v^+) < 0 \quad   (\mathcal{P}(v^-) >0),  \forall \tau \in [\tau_0, \tau_1]$ and $ \xi \le b^{\frac{\eta}{4} -\frac{1}{2}}(\tau)$. 
	\item[$(ii)$] Initial estimate: 
	$ \frac{1}{b(\tau_0)} v^- \left(\frac{y}{\sqrt{b(\tau_0)}}, \tau_0 \right) <   w(y,\tau_0)  -  Q_{b(\tau_0)}(\xi) < \frac{1}{b(\tau_0)} v^+\left(\frac{y}{\sqrt{b(\tau_0)}} ,\tau_0\right), \forall  y \le b^{\frac{\eta}{4} }(\tau_0) .$
\end{itemize}
We remark that  the proof  of the estimates on $v^-$ are quite the  same as for $v^+$. Thus, we only handle the latter. 

- Proof of $(i)$: plugging $v^+$ into \eqref{defi-mathcal-P},  we  get
\begin{eqnarray*}
	\mathcal{P}(v^+) &=&   \theta^+(\tau) \partial^2_\xi H_1 - M b^\frac{\eta}{4}(\tau) \partial^2_\xi H_0  + \theta^+(\tau) \left( \frac{d+1}{\xi} \partial_\xi H_1 \right)   - M \frac{\eta}{4}(\tau) \frac{d+1}{\xi} \partial_\xi H_0\\
	& - & 3(d-2) \left[  2Q +\xi^2 Q^2  \right] \left( \mu^+(\tau) H_1 - M b^\frac{\eta}{4}(\tau) H_0 (\xi)\right)  + \theta(\tau) \Lambda_\xi Q  + \bar B(v^+) \\
	& + & \theta(\tau) \left[ \theta^+ \Lambda_\xi H_1 - M b^\frac{\eta}{4} \Lambda_\xi H_0  \right]  - b(\tau) \left[ \partial_\tau \theta^+(\tau) H_1 - M \partial_\tau b^\frac{\eta}{4}(\tau) H_0\left( \xi \right)  \right] \\
	& = & - 3(d-2) \left[ 2Q + \xi^2 Q^2 - (2Q_\sigma + \xi^2 Q_\sigma^2)    \right] \left( \theta^+(\tau) H_1(\xi) - M b^\eta(\tau) H_0(\xi)  \right) \\
	&+&  \bar{B}(v^+)  +  \left[  \theta - \theta^+ \right] \Lambda_\xi Q   +\theta(\tau) \left[ \theta^+ \Lambda_\xi H_1 - M b^\frac{\eta}{4} \Lambda_\xi H_0  \right]   - b(\tau) \left[\partial_\tau \theta^+ H_1 - M \partial_\tau b^\frac{\eta}{4} H_0 \right],
\end{eqnarray*}
where  the simplification comes from the facts that  $H_0$  and $ H_1$ solve
\eqref{equa-H-0} and \eqref{equa-H-1}, respectively.  Since $\xi \le b^{\frac{\eta}{4} -\frac{1}{2}}(\tau)$ with $\eta \ll 1$ and $b(\tau) \to 0$, the range of $\xi$ will be large, and we should divide it into two cases $\xi = O(1)$ and $ \xi \gg 1$.   

+ For the  case  $\xi \gg 1 $,  we derive on the one hand,  from  the  definitions of $\theta $, $\theta^+ $ in \eqref{defi-theta-tau-} \eqref{defi-theta-+}, and \eqref{ODE-b-tau-proposition}, that    
\begin{eqnarray*}
	\theta(\tau) -  \theta^+(\tau) = b^{1+\frac{\eta}{8}} + O(b^{1+4\eta}).
\end{eqnarray*}
Thus we derive  from  \eqref{asymptotic-Lamda-Q} that
$$   \left[ \theta(\tau) - \theta^+ \right] \Lambda_\xi Q  =    a_0  b^{1+\frac{\theta}{8}}(\tau) \xi^{-\gamma} + O(\xi^{-\gamma} b^{1+4\eta}) + O(b^{1+\frac{\eta}{8}} \xi^{-\gamma-2}),
$$
where $a_0 < 0$. On the other hand, by a similar argument,  we have
\begin{eqnarray*}
	- b(\tau) \left[\partial_\tau \theta^+ H_1 - M \partial_\tau b^\frac{\eta}{4} H_0 \right] =O \left(  b^{1+\frac{\eta}{4}}\xi^{-\gamma}\right),\\
	\theta(\tau) \left[ \theta^+ H_1 - M b^\frac{\eta}{4} H_0 \right]  = 0\left(   b^{1 + \frac{\eta}{4}} \xi^{-\gamma} \right) \text{ as } \xi \to \infty.
\end{eqnarray*}
Thus, we conclude the following inequality
\begin{eqnarray*}
	\mathcal{P}(z^+) \le -  3(d-2) \left[ 2Q + \xi^2 Q^2 - (2Q_\sigma + \xi^2 Q_\sigma^2) \right] v^+ + \bar B(v^+), 
\end{eqnarray*}
provided that $b, \eta $  small  enough.  

Next, note that   $\xi \le b^{\frac{\eta}{4} -\frac{1}{2}}(\tau)$ implies  the following
$$ | b(\tau) \xi^2| \le b^{\frac{\eta}{2}}(\tau). $$
Hence,   by using $v^+$'s definition  in \eqref{defi-z-+--}, along with \eqref{comport-H-0} and \eqref{comport-H-1}, we obtain 
%the asymptotics of $H_0$ and $H_1$

\begin{eqnarray*}
	v^+ (\xi, \tau) =  - M a_0 \sigma^{\alpha} b^{\frac{\eta}{4}}\xi^{-\gamma}   +  O(b^{\eta} \xi^{-\gamma-2}) + O(b^{2\beta}\xi^{-\gamma}), \text{ as } \xi \to +\infty.   
\end{eqnarray*}
In addition, recall that
\begin{eqnarray*}
	2Q(\xi) + \xi^2 Q^2(\xi) & \sim  & -\frac{1}{\xi^2} + q_0^2 \xi^{-2\gamma +2} + o(\xi^{-2\gamma +2}),\\
	2Q_\sigma(\xi) + \xi^2 Q_\sigma^2(\xi) & \sim & -\frac{1}{\xi^2} + q_0^2 \sigma^{\gamma-2}  \xi^{-2\gamma+2} + o(\xi^{-2\gamma +2}),
\end{eqnarray*}
as $\xi \to +\infty$. Then, fixing $\sigma$ less that 1, we derive 
\begin{eqnarray*}
	2Q(\xi) + \xi^2 Q^2(\xi) - \left[2Q_\beta(\xi) + \xi^2 Q_\beta^2(\xi) \right]  = q_0^2 \left(  1 -\sigma^{\gamma-2}\right) \xi^{-3\gamma +2} + o(\xi^{-2\gamma +2}). 
\end{eqnarray*}
Thus, 
\begin{eqnarray}
	-  3(d-2) \left[ 2 Q +  \xi^2 Q^2 - (2Q_\sigma + \xi^2 Q_\sigma^2)    \right] v^+ = m_0 b^\frac{\eta}{4}(\tau) \xi^{-3\gamma+2} + o(\xi^{-3\gamma+2}), \text{ as } \xi \to +\infty. \label{asymptoic-poten-z-+}
\end{eqnarray}
where $m_0 = 3(d-2)q_0^2(1 -\sigma^{\gamma-2})M a_0 \sigma^\gamma <0$. 
Next, we study  $\bar B(v^+)$ defined by
\begin{eqnarray*}
	\bar B(v^+) &=& -3(d-2) (1 + \xi^2 Q(\xi)) (v^+)^2 - (d-2) \xi^2 (v^+)^3\\
	& = & \left( m_1 b^{\frac{\eta}{2}} + m_2 b^{\frac{3\eta}{4}} \right) \xi^{-3\gamma+2} + o(\xi^{-3\gamma+2}), \text{  as  } \xi \to +\infty,
\end{eqnarray*}
where 
\begin{eqnarray*}
	m_1 = -3(d-2)q_0 M^2 a_0^2 \sigma^{2\alpha} \text{ and } m_1 = -(d-2) M^3 a_0^3 \sigma^{3\alpha}.
\end{eqnarray*}
Finally, we derive
$$ \mathcal{P}(v^+) < 0,$$
provided that $\xi \gg 1$, $b \le b_1 \ll 1 $ and $\eta \ll 1$. 
\iffalse
Now, we focus on the case   $\xi  \gg 1$.    
\begin{eqnarray*}
	2Q + \xi^2 Q^2 - (2Q_\alpha + \xi^2 Q_\alpha^2) \sim \frac{C_\alpha}{\alpha} \xi^{ -\gamma }, \text{ with }  C_\alpha >0,
\end{eqnarray*}
and 
\begin{eqnarray*}
	\mu^+ H_1(\xi)  - M b^\theta(\tau) H_0(\xi) \sim  C_M \xi^{-\gamma}, \text{ with } C_M >0. 
\end{eqnarray*}
Therefore, we obtain
$$ \mathcal{P}(z^+)  <0, \text{ as } \xi \gg 1.$$
\fi

+ For $ \xi = O(1)$, i.e., $\xi \in [0,K]$ for some $K>0$ large enough: thanks to the  smallness of $b$,  
the dominating term in $v^+$ is  $-M b^\frac{\eta}{4} (\tau) H_0(\xi) > 0 $. Besides that, we have
\begin{eqnarray*}
	\partial_\sigma \left( 2Q_\sigma + \xi^2 Q_\sigma^2    \right)  =  -\frac{1}{\sigma} \Lambda_\xi Q_\sigma -  \xi^2 \frac{Q_\sigma}{\sigma} \Lambda_\xi Q_\sigma = -\frac{\Lambda_\xi Q_\sigma }{\sigma} \left( 1 + \xi^2 Q_\sigma (\xi) \right).
\end{eqnarray*}
Note that the construction of $Q$ (see more in \eqref{relation-Q-epsilon-x}) ensures  $ \xi^2Q_\sigma(\xi) >-1 $ and we derive
$$\partial_\sigma  \left( 2Q_\sigma + \xi^2 Q_\sigma^2    \right)  >0, $$
from which we infer the existence of $m_3(\sigma,K) >0$  such that 
\iffalse
since    $\partial_\sigma  Q_\sigma = -\frac{1}{2\sigma} \Lambda_\xi Q_\sigma >0$, we deduce that for all  $\xi \in [0,K]$, with $K$ fixed and large, then, we have  
$$  Q(\xi) - Q_\sigma(\xi) < 0,    $$
provided that $\beta >1$. We see that   $[0,K]$ is compact so there exists  $m_3 >0$ such that
\fi
$$ 2Q + \xi^2 Q^2 - (2Q_\alpha + \xi^2 Q_\alpha^2) \ge  m_3 \text{ with } \sigma <1.  $$
Since $[0,K]$ is compact, we get
\begin{eqnarray*}
	v^+ = \theta^+ H_1(\xi)  - M b^{\frac{\eta}{4} } H_0 \ge m_4(\sigma, K)b^\frac{\eta}{4} \text{ with } m_4 >0.
\end{eqnarray*}
Thus,  we get
\begin{eqnarray*}
	- 3(d-2) \left[ 2Q + \xi^2 Q^2 - (2Q_\alpha + \xi^2 Q_\alpha^2)   \right] v^+ \le  -3(d-2)m_3 m_4 b^\frac{\eta}{4}.  
\end{eqnarray*}
This concludes $\mathcal{P}(z^+) < 0$ for the case $\xi = O(1)$. 

- Proof of $(ii)$: Notice that for $M(\eta)$ large enough and $\sigma $ small,  we have
$$ \frac{v^+\left( \frac{y}{\sqrt{b(\tau_0)}}\right)}{b(\tau_0)} >0, $$
and from  $\varepsilon(\tau_0)$'s definition in \eqref{initial-data}, we see that it vanishes when $y \le b(\tau_0)$ and it is sufficient to check it for $y \in [b^\frac{\delta}{2}(\tau_0), b^{\frac{\eta}{4}}(\tau_0) ]$ giving
$$ \frac{y}{\sqrt{b(\tau_0)}} \to +\infty.$$
Thus, using  \eqref{comport-H-0} and \eqref{comport-H-1}, we can find a $c_0$ such that 
\begin{eqnarray*}
	\frac{v^+\left(\frac{y}{\sqrt{b(\tau_0)}}\right)}{b(\tau_0)} \ge c_0 b^{\frac{\alpha}{2} + 
		\frac{\eta}{4}}(\tau_0) y^{-\gamma}.
\end{eqnarray*}
On the other hand, $\varepsilon$'s formula and the fact that $ y \le b^\frac{\eta}{4}(\tau_0)$ imply
$$ \left| \frac{\phi_{\ell, b(\tau_0), \beta(\tau_0)}}{c_{\ell, 0}} - \phi_{0, b(\tau_0), \beta(\tau_0)} \right| \lesssim b^\frac{\eta}{4}(\tau_0). $$
This yields
\begin{eqnarray*}
	\left|\varepsilon(\tau_0) \right| \lesssim b^{\frac{\alpha}{2} +\frac{\eta}{2}} (\tau_0) y^{-\gamma}.
\end{eqnarray*}

\end{proof}

\appendix

\section{Details on pointwise estimates}

In the sequel, we give details to some pointwise estimates used in our paper.
\begin{lemma}\label{lemma-estimate-on-hat-B} Let us consider $\hat{B}$ defined as in  \eqref{defi-hat-B},  $(\varepsilon, b, \beta)(\tau) \in V[A,\eta, \tilde{\eta}] (\tau),$ for all $\tau\in [\tau_0, \tau^*]$ for some $\tau^* \ge \tau_0$, and $\delta  >  \eta > \tilde \eta$. Then, there exists $\tau_9(A, \delta, \eta, \tilde \eta) \ge 1$, such that for all $\tau_0 \ge \tau_9$,  the following  holds
\begin{eqnarray*}
	\left| \mathbbm{1}_{ y \in \left(0, b^{\delta}\right]} \hat B(y,\tau) \right|  
	& \le  &  \frac{Cb^{\frac{\alpha}{2}}}{y^{\gamma }} + \frac{CA^3b^{\frac{\alpha}{2} +\tilde\eta} (\tau )}{y^{\gamma+2}} + \frac{CA^6 b^{2\left(\frac{\alpha}{2} + \tilde \eta \right)}(\tau)}{y^{2\gamma}} +  \frac{CA^9 b^{3\left(\frac{\alpha}{2} + \tilde \eta \right)}(\tau)}{y^{3\gamma -2}} ,\\
	\left| \mathbbm{1}_{ y \ge b^{\delta}} \hat B(y,\tau) \right| & \le &  C\frac{ b^{\frac{\alpha}{2} + 4 \eta }\langle y \rangle^{2\ell+2}}{y^\gamma} +  \frac{Cb^{\alpha +\delta(1 -\gamma)  } \langle y \rangle^{4\ell+8} }{y^{ \gamma}} + \frac{Cb^{ \frac{3\alpha}{2} - 2\gamma \delta} \langle y \rangle^{6\ell +14}}{y^{\gamma}}, \forall \tau  \in [\tau_0, \tau^*].
\end{eqnarray*}
\end{lemma}
\begin{proof}
Let us consider $\delta \gg  \eta \gg \tilde \eta$. First,  since $(\varepsilon, b, \beta)(\tau) \in V_\ell[A,\eta, \tilde \eta](\tau)$ and by applying   Lemma  \ref{lemma-ODE-finite-mode} we have  
\begin{eqnarray*}
	\left|  \beta'(\tau)   \right|   \lesssim  A b^{4\eta}(\tau)  \text{ and }
	\left| \frac{b'}{b} - 2\beta\left( 1- \frac{2\ell}{\alpha}  \right)   \right|  \lesssim   Ab^{4\eta} (\tau).  
\end{eqnarray*}
Write
%\eqref{hat-varepsilon-j} below
\begin{eqnarray}
	\hat \varepsilon_{\beta, j}   = \| \phi_{j, \infty, \beta}\|_{L^2_{\rho_\beta}}^{-2} \langle  \varepsilon, \phi_{j,\infty, \beta}         \rangle_{L^2_{\rho_\beta}}. 
\end{eqnarray}
We observe that even though $\phi_{j, b, \beta}$ is not orthogonal to $\phi_{k, \infty, \beta}, k \ne j$, we have
$$\left| \langle \phi_{j, b, \beta} ,\phi_{k, \infty, \beta} \rangle_{L^2_{\rho_\beta}} \right|  \le \int_{y \le b^{\delta}} \left|\phi_{j, b, \beta} \phi_{k, \infty, \beta} \right| \rho_\beta dy +\int_{y \ge  b^{\delta}} \left|\phi_{j, b, \beta} \phi_{k, \infty, \beta} \right| \rho_\beta dy \lesssim b^\delta.$$
Thus, we use   pointwise estimates in Lemma \ref{lemma-rough-estimate-bounds in shrinking-set} to obtain 
\begin{eqnarray}
	\left| \hat{\varepsilon}_{\beta,j} (\tau)\right|  \le CA b^{\frac{\alpha}{2} + \tilde \eta},  \forall j < \ell, \quad 
	\hat{\varepsilon}_{\beta,\ell}  =  \frac{\varepsilon_\ell}{c_{\ell,0}} + O(b^{\frac{\alpha}{2} + 4 \eta}), \text{ and }  
	\hat{\varepsilon}_{\beta, 0} =   - \varepsilon_\ell  + O(b^{\frac{\alpha}{2} + 4 \eta}).\label{estimate-hat-varep-beta-j}
\end{eqnarray}
In particular, repeating the technique in  Lemma \ref{lemma-ODE-finite-mode} we get 
\begin{eqnarray}
	\left\{ \begin{array}{rcl}
		\partial_\tau \hat{\varepsilon}_{\beta,\ell}  & = & 2 \beta \left( \frac{\alpha}{2}  -  \ell \right) \hat{\varepsilon}_{\beta,\ell} + O(b^{\frac{\alpha}{2} + 4\eta}),   \\
		\partial_\tau \hat{\varepsilon}_{\beta, 0} &  =  &  2\beta \frac{\alpha}{2}\hat{\varepsilon}_{\beta, 0}  + \left[\frac{b'}{b} -2\beta \right] m_0 b^\frac{\alpha}{2} + O(b^{\frac{\alpha}{2} +4\eta}).
	\end{array}
	\right. \label{partial-hat-varep-beta-ell}
\end{eqnarray}
\iffalse 
We will prove them at the end of the proof, and now come back to the main purpose. Let us consider the proof in two case  that $ y \le b^{\delta}$ and $ y \ge b^{\delta}$. \iffalse
\begin{eqnarray*}
	\frac{b_\tau}{b} -1 &=& -\frac{2l}{\alpha} + O\left(Ab^{\eta} \right),\\
	\hat{\varepsilon}_0 &=& -m_0 \varepsilon_l(\tau) + O(b^{\frac{\alpha}{2}+\eta})\\
	\hat{\varepsilon}_l &=& m_0 \frac{\varepsilon_l}{c_{l,0}} + O(b^{\frac{\alpha}{2}+\eta}),
\end{eqnarray*}
and  rough estimates on $\hat{\varepsilon}_+$, and $ \hat{\varepsilon}_-$ as follows
\begin{eqnarray*}
	\left|\hat{\varepsilon}_+(y,\tau) \right|  &\le & \frac{CAb^{\frac{\alpha}{2} +\eta} (\tau) \langle y \rangle^{2l+2}}{y^\gamma} + \frac{Cb^{\frac{\alpha}{2}} y^2\langle y \rangle^{2l}}{y^\gamma},\\
	\left| \hat{\varepsilon}_-(y,\tau)\right|   & \le & \frac{CA^3 b^{\frac{\alpha}{2} + \tilde{\eta}}(\tau) \langle y  \rangle^{2l+2}}{y^\gamma}.
\end{eqnarray*}
\fi 
\fi
\medskip
- The first case: $ y \in \left( 0, b^{\delta}(\tau)\right] $.  From \eqref{defi-hat-B}, we   have 
\begin{eqnarray*}
	& & \left| \hat{B}(\hat{\varepsilon}_{\beta, +} + \hat{\varepsilon}_{\beta, -}) \right|\\
	& \le &  \left|3(d-2)\left( 2Q_b +y^2Q_b^2 + \frac{1}{y^2} \right) (\hat{\varepsilon}_{\beta, +} + \hat{\varepsilon}_{\beta, -}) \right| + \left| B(\hat{\varepsilon}_{\beta, +} + \hat{\varepsilon}_{\beta, -})\right| +   \left|  \Phi + \mathscr{L}_\infty^\beta \hat{\varepsilon}_{\beta, +}  -  \partial_\tau \hat{\varepsilon}_{\beta, +}    \right| .
\end{eqnarray*}
From $Q$'s asymptotic given in Lemma  \ref{lemma-ground-state} and \eqref{Q-b}, we  get 
$$  \left| 2Q_b + y^2 Q_b + \frac{1}{y^2}  \right| \lesssim y^{-2}.$$
Besides that, since $\hat{\varepsilon}_{\beta,+}  + \hat{\varepsilon}_{\beta, -}  =  \varepsilon =  \varepsilon_+  + \varepsilon_- $, and   the  pointwise estimates in Lemma \ref{lemma-rough-estimate-bounds in shrinking-set}, we obtain   
\begin{eqnarray*}
	\left| \hat{\varepsilon}_{\beta,+}  + \hat{\varepsilon}_{\beta, -} \right| \le C \left( A^4 b^{\frac{\alpha}{2} +\tilde \eta}(\tau) y^{-\gamma} + b^{\frac{\alpha}{2} } y^{2-\gamma} \right),
\end{eqnarray*}
Thus,  the following is valid
\begin{eqnarray*}
	\left|3(d-2)\left(2Q_b +y^2Q_b^2 + \frac{1}{y^2} \right) (\hat{\varepsilon}_{\beta,+}  + \hat{\varepsilon}_{\beta, -}) \right|  \le   C \left( A^4 b^{\frac{\alpha}{2} + \tilde \eta} y^{-2-\gamma} + b^{\frac{\alpha}{2}}(\tau) y^{-\gamma}   \right) .
\end{eqnarray*}
Similarly, we have 
\begin{eqnarray*}
	\left|B(\hat{\varepsilon}_{\beta,+}  + \hat{\varepsilon}_{\beta, -})   \right| \le C \left( |\varepsilon|^2 + y^2 |\varepsilon|^3  \right) \le C \left(A^8 b^{\alpha +2\tilde \eta} y^{-2\gamma}  + A^{12} b^{\frac{3}{2}\alpha + 3\tilde \eta}  y^{2-3\gamma}\right).
\end{eqnarray*}
For the last term, we immediately deduce from \eqref{estimate-hat-varep-beta-j} and \eqref{partial-hat-varep-beta-ell} that
\begin{eqnarray*}
	\left| \Phi + \mathscr{L}_\infty^\beta \hat{\varepsilon}_{\beta, +}  -  \partial_\tau \hat{\varepsilon}_{\beta, +}  \right| \le  C b^{\frac{\alpha}{2}} y^{-\gamma}.
\end{eqnarray*}
By adding all  related terms, we conclude   the estimate on $ \mathbbm{1}_{y \in (0,b^{\delta}]} \hat B$.

\medskip
- The second case:  $y \in \left[ b^\delta(\tau) +\infty \right)$.  Regarding  \eqref{decompose-Phi}, we can improve it  as follows
$$  \Phi  =  \left[\frac{b'}{b}  -2\beta \right] m_0 b^\frac{\alpha}{2} \phi_{0, \infty, \beta} + \tilde{ \Phi}(y,\tau),  $$
where 
\begin{eqnarray*}
	\left| \tilde{ \Phi}(y,\tau) \right| \le C  b^{\frac{\alpha}{2} +  \delta  } y^{-\gamma} \langle y \rangle^{2\ell +2} , \forall y \ge b^\delta(\tau).
\end{eqnarray*}
In addition to that, we have
\begin{eqnarray*}
	\left| \mathscr{L}_{\infty}^{\beta} \hat{\varepsilon}_{\beta, +} - \sum_{j=0}^\ell \left( \frac{\alpha}{2} -j \right) \hat{\varepsilon}_{\beta, j} \phi_{j,\infty, \beta} \right| \lesssim  b^{ \frac{\alpha}{2} +4\eta }\langle y \rangle^{2\ell +2}y^{-\gamma},\\
	\left| \partial_\tau  \hat{\varepsilon}_{\beta, +} -  \sum_{j=0}^\ell \left( \frac{\alpha}{2} -j \right) \hat{\varepsilon}_{\beta, j} \phi_{j,\infty, \beta}   \right| \lesssim b^{\frac{\alpha}{2} + 4\eta}\langle y \rangle^{2\ell +2}y^{-\gamma}.
\end{eqnarray*}
Finally, we conclude 
\begin{eqnarray*}
	\left| \Phi +  \mathscr{L}_\infty^{\beta} \hat{\varepsilon}_{\beta, +} - \partial_\tau \hat{\varepsilon}_{\beta, +} \right| \le  \frac{CA b^{\frac{\alpha}{2} +4\eta} (\tau) \langle y \rangle^{2\ell+2}}{y^\gamma}.
\end{eqnarray*}
Next, we study the estimate involving $Q_b$. Using \eqref{asym-Q-infinity} and the facts that  $y \ge b^{\delta}, \delta \ll 1$, we get
$$ \xi = \frac{y}{\sqrt{b}} \gg 1,$$
then, it follows
\begin{eqnarray*}
	Q_b (y) = -\frac{1}{y^2} \left( 1 + \tilde Q \right),
\end{eqnarray*}
where  
\begin{eqnarray*}
	\left|   \tilde Q(y)  \right| \le C(\delta) b^\frac{\alpha}{2} y^{2 -\gamma}  \le b^{\frac{\alpha}{2} -\gamma \delta} \le C b^{\delta}(\tau), y \ge b^{\delta}.
\end{eqnarray*}
Hence,  we have 
\begin{eqnarray*}
	\left| 2Q_b(y) + y^2Q_b^2(y) + \frac{1}{y^2}  \right|  \le Cb^{\delta},
\end{eqnarray*}
which implies 
\begin{eqnarray*}
	\left| \left( 2Q_b(y) + y^2Q_b^2(y) + \frac{1}{y^2} \right)(\hat{\varepsilon}_{\beta, +} + \hat{\varepsilon}_{\beta, -}) \right|  \le C b^{\frac{\alpha}{2} + \delta}    \langle y \rangle^{2\ell+2}y^{-\gamma}.
\end{eqnarray*}
Similarly, since  
\begin{eqnarray*}
	\left| 1  + y^2 Q_b(y)   \right| \le C b^\frac{\alpha}{2} (\tau) y^{2-\gamma} \le C b^{\delta}, \forall y \ge b^\delta(\tau),
\end{eqnarray*}
we deduce
\begin{eqnarray*}
	\left|  B(\hat{\varepsilon}_{\beta, +}+\hat{\varepsilon}_{\beta, -})\right|  & \lesssim &  b^\delta \left( \frac{b^{\alpha} y^4 \langle y \rangle^{4\ell+4}}{y^{2\gamma}} + \frac{A^8 b^{\alpha +2\tilde \eta} \langle y\rangle^{4\ell+4}}{y^{2\gamma}} \right) + y^2 \left( \frac{b^{\frac{3\alpha}{2}} y^6 \langle y \rangle^{6\ell+6}}{y^{3\gamma}} + \frac{A^{12} b^{\frac{3\alpha}{2} +3\tilde \eta} \langle y\rangle^{6\ell+6}}{y^{3\gamma}} \right) \\
	& \lesssim & b^{\alpha +\delta} \langle y \rangle^{2\ell+8}y^{-2\gamma} +  b^{\frac{3\alpha}{2}} \langle y \rangle^{6\ell+14}{y^{-3\gamma}}.
\end{eqnarray*}
In particular, once $ y \ge b^\delta(\tau')$, it follows
$  \frac{1}{y^\gamma} \lesssim  b^{-\gamma \delta}.   $
By adding the related  bounds, we conclude the estimate  on $\mathbbm{1}_{y \ge b^\delta} \hat{B}$. This achieves the proof of the Lemma.
\end{proof}
\section{Detail on spectral analysis computation of $\mathscr{L}_\infty$}\label{Appendix-proof-spectrum-L-infty}
In this part, we aim to give a complete computation to formulate constant in Proposition \ref{proposition-spectral-L-infty}.  Let us consider the following quadratic equation
\begin{equation}\label{equa-gamma}
\gamma^2 - d  \gamma +3(d-2) =0.
\end{equation}
The equation has two distinct solutions 
$$ \left\{ \begin{array}{rcl}
\gamma_1   & = &  \frac{1}{2} ( d - \sqrt{d^2 -12 d+ 24} ),    \\
\gamma_2  & =& \frac{1}{2}  ( d +\sqrt{d^2 -12 d +24}).
\end{array}        \right. $$
We remark that  $\frac{1}{r^{\gamma_2}} $ does not belong to $ H^1_\rho $, but $ \frac{1}{r^{\gamma_1}}$ does.  In addition, we also define 
\begin{equation*}\label{defi-gamma}
\gamma = \gamma_1 =  \frac{1}{2} ( d - \sqrt{d^2 -12 d+ 24} ), \text{ and } \tilde  \gamma = \gamma_2 =  \frac{1}{2} ( d + \sqrt{d^2 -12 d+ 24} ).
\end{equation*}
From  $\gamma$'s formula above, we can get the first eigenfunction and eigenvalue as follows
$$ \phi_{0,\beta,\infty}(r) = \frac{1}{r^\gamma} \text{ and } \lambda_{0,\beta,\infty} = 2 \beta \left(  \frac{1}{2} ( \gamma  -2 ) \right):= 2\beta \left(\frac{ \alpha}{2}  \right).$$
Following \cite{CMRJAMS20} (also \cite{CGMNARXIV20-a}, and \cite{CGMNARXIV20-b}), we search the eigenfunctions and eigenvalues in the following forms
\begin{eqnarray}\label{form-phi-beta-infty}
\phi_{i,\beta,\infty}(r)  = \sum_{j=0}^i a_{i,j} (2\beta)^j r^{2j-\gamma}, \text{ and } a_{i,i} =1,  \text{ and } \lambda_{i,\beta, \infty} = 2\beta \left(\frac{\alpha}{2} - i \right), \forall i \geq 0.
\end{eqnarray}
Plugging the  form  \eqref{form-phi-beta-infty} into the following relation
\begin{equation}\label{relation-L-infty-alpha-i}
\mathscr{L}_\infty^\beta \phi_{i,\beta,\infty} = 2\beta \left(\frac{\alpha}{2} - i  \right) \phi_{i,\beta,\infty}, 
\end{equation}
we get  

\begin{eqnarray}
\mathscr{L}_\infty^\beta \phi_{i,\beta,\infty}&=&\sum_{j=0}^i a_{i,j}  (2\beta)^j  A_j r^{2(j-1)-\gamma} +\sum_{j=0}^i a_{i,j} (2\beta)^j (-2\beta-\beta(2j-\gamma))r^{2j-\gamma} \nonumber\\
&=& 2\beta \left(\frac{\alpha}{2} - i  \right)\sum_{j=0}^i a_{i,j}  (2\beta)^j r^{2j-\gamma} , \label{equality-L-beta-infty-lamfa-beta-infinity}
\end{eqnarray}
where $A_j=(2j-\gamma)(2j-\gamma-1)+(d+1)(2j-\gamma)+3(d-2)$. Fix $0\leq j \leq i-1$.

\medskip
+ For $j =i$: We choose $ a_{i,i}=1$, then,   we get
$$ (-2\beta -\beta (2i -\gamma )) = 2\beta \left( \frac{\alpha}{2} -i \right), $$
then, \eqref{equality-L-beta-infty-lamfa-beta-infinity}  is satisfied.

\medskip
+ For all $j \leq i-1$:  \eqref{equality-L-beta-infty-lamfa-beta-infinity} yields
\begin{eqnarray*}
a_{i,j+1} (2\beta)^{j+1} A_{j+1} + a_{i,j}(2\beta)^j  ( -2\beta -\beta(2j-\gamma)) = 2\beta \left( \frac{\alpha}{2} - i \right) a_{i,j} (2\beta)^j,
\end{eqnarray*}
which yields 
\begin{eqnarray}
a_{i,j+1} A_{j+1} = (  j - i )a_{i,j}.
\end{eqnarray}
By a simple recurrence, we obtain 
\[a_{i,j}=(-1)^{i-j}\Pi_{k=j}^{i-1}\frac{A_{k+1}}{i-k}=\frac{(-1)^{i-j}}{(i-j)!}\Pi_{k=j+1}^{i}A_{k}.\]
Thus
\begin{eqnarray}
a_{i,j} = \frac{(-1)^j}{(i-j)!} \nonumber  \times \Pi_{k=j+1}^{i-1} \left( (2k -\gamma) (2k -\gamma -1) + (d+1) ( 2k-\gamma) +3(d-2) \right).\label{defi-a-i-j}
\end{eqnarray}
In particular, in  \eqref{equality-L-beta-infty-lamfa-beta-infinity},  the  order $r^{-\gamma-2}$ remains. However, its coefficient is equal to $0$, since  $\gamma$ solves \eqref{equa-gamma}. Finally, \eqref{relation-L-infty-alpha-i} is  completely satisfied by the choice  of $a_{i,j}$ above.
\iffalse
Indeed, from \eqref{relation-L-infty-alpha-i}, we have 
\[a_{i,j+1}=-\frac{i-j}{A_{j+1}}a_{i,j},\]
where  and we then get
\[a_{i,j}=\Pi_{k=j}^{i-1}\frac{a_{i,k}}{a_{i,k+1}}\]
and, for $j\leq k \leq i-1$, we also have $\frac{a_{i,k}}{a_{i,k+1}}=-\frac{A_{k+1}}{i-k}$. Hence, we derive 
\[a_{i,j}=(-1)^{i-j}\Pi_{k=j}^{i-1}\frac{A_{k+1}}{i-k}=\frac{(-1)^{i-j}}{(i-j)!}\Pi_{k=j+1}^{i}A_{k}\]
so that
\[a_{i,j}=\frac{(-1)^{i-j}}{(i-j)!}\Pi_{k=j+1}^{i}\left( (2k -\gamma) (2k -\gamma -1) + (d+1) ( 2k-\gamma) +3(d-2) \right).\]
\fi
Next, we aim to decompose $a_{i,j}$ as follows:  First, we deduce from   \eqref{equa-gamma} that
\begin{eqnarray*}
& & (2k -\gamma) (2k -\gamma -1) + (d+1) ( 2k-\gamma) +3(d-2) \\
&=& 2k.  2k +  2k d  - 4k \gamma  =  2k ( 2k + d - 2 \gamma   )  =  4 k \left( \frac{d}{2}  -  \gamma  +k \right).
\end{eqnarray*}
Then,  $a_{i,j}$ is decomposed as
\begin{eqnarray*}
a_{i,j} = \frac{(-1)^{i-j}}{(i-j)!} \Pi_{k=j+1}^i 4 k  \left( \frac{d}{2}  -  \gamma  +k \right) = \frac{(-1)^{i-j}}{(i-j)!} 4^{i-j} \frac{i!}{j!} \frac{(\frac{d}{2} - \gamma)_i}{(\frac{d}{2} - \gamma)_j}=c_{i,j} C_j,
\end{eqnarray*}
where   $ c_{i,j}  = \frac{ (-1)^{i - j} 4^{i} i! \left(\frac{d}{2} -\gamma \right)_i! }{(i-j)!}$, $C_j=\frac{ 1 }{ 4^{j} j! \left(\frac{d}{2} -\gamma \right)_j! } $, and 
$$\left( \frac{d}{2} -\gamma\right)_i! =\left( \frac{d}{2} -\gamma +1 \right)\left( \frac{d}{2} - \gamma +2 \right)...\left( \frac{d}{2} -\gamma +i\right) \text{ and }  \left( \frac{d}{2} -\gamma\right)_0!  =1. $$

\section{Poisson kernel for Laguerre  expansions}\label{secion-poisson-kernel}
In this part, we aim to provide some pointwise estimates involving    semi-group $e^{\tau \mathscr{L}_\infty}$ with $\mathscr{L}_\infty $ defined as in  \eqref{defi-operator-L-infty}. Recall that for $f \in L^1(\R_+,x^\omega e^{-x} dx)$, we have the following presentation
\begin{eqnarray}
\left[e^{(\tau-\tau_0)\mathcal{L}_\infty} \right]f (y,\tau) =  2^{\omega +1} y^{-\gamma} e^{\frac{\alpha}{2} (\tau-\tau_0)}\int_0^\infty P_{\frac{\omega}{2}} \left( \frac{y^2}{4}, \frac{x^2}{4}, e^{-(\tau-\tau_0)}  \right) [f(x) x^{\gamma}  ] x^{\omega+1} e^{-\frac{x^2}{4}} dx .\label{semi-f}
\end{eqnarray}
where $P_\zeta  $ is defined by 
\begin{eqnarray*}
P_{\zeta} \left( \frac{y^2}{4}, \frac{x^2}{4}, r \right) & = & e^{-\frac{r}{1-r}\left(\frac{y^2}{4}+\frac{x^2}{4}\right)}\frac{(-r\frac{x^2\cdot y^2}{16})^{-\frac{\zeta}{2}}}{1-r} J_{\zeta}\left(\frac{2(-r\frac{x^2\cdot y^2}{16})^{\frac{1}{2}}}{1-r}\right),\\
&=& \frac{(\sqrt{r\frac{y^2}{4} \frac{x^2}{4}})^{-\zeta}}{1-r} e^{-\frac{r}{1-r}\left(\frac{y^2}{4}+\frac{x^2}{4}\right)} i^{\zeta} J_\zeta \left( \frac{2r^\frac{1}{2} \frac{y}{2} \frac{x}{2}}{1-r} i \right) \\
&  =  &  \frac{4^\zeta}{\sqrt{r}^\zeta (1-r) (yx)^\zeta}   e^{-\frac{r}{1-r}\left(\frac{y^2}{4}+\frac{x^2}{4}\right)} \mathbf{I}_\zeta\left(\frac{r^\frac{1}{2}y x}{2(1-r)} \right),
\end{eqnarray*}
and $\mathbf{I}_\zeta$ is the function of imaginary argument corresponding to 
\begin{equation}
\mathbf{I}_\zeta(z)=   \frac{\left(\frac{1}{2}z\right)^\zeta}{\Gamma(\zeta +\frac{1}{2}) \Gamma(\frac{1}{2}) }   \int_0^\pi \cosh\left( z \cos \theta\right)   \sin^{2\zeta} (\theta) d\theta,    
\end{equation}
provided that $\text{Re}(\zeta +\frac{1}{2}) >0$, the reader can check the formula at  page 79,  formula $(2)$ in \cite{WbookCUP1922}.  We have the following result
\begin{lemma}[Maximal estimate, \cite{BSIMRN19}]\label{maximal-lemma}
Let us consider $ f \in L^2_{\rho}$ with $\rho = \rho_\frac{1}{2}$ defined in  \eqref{defi-rho-y}, then, 
\begin{equation}\label{maxinum-M-f-y}
	\left|  e^{(\tau -\tau_0)\mathscr{L}_\infty} f(y)  \right|    \le C y^{-\gamma} e^{\frac{\alpha}{2}(\tau -\tau_0)} [Mf](y), \forall y >0, \tau > \tau_0,
\end{equation}
where $ \alpha$ and  $\gamma$  were  defined in \eqref{defi-alpha-intro} and \eqref{defi-gamma-intro}, respectively, and $Mf$  is given by 
\begin{equation}\label{defi-maximan-function}
	Mf(y) = \sup_{y \in \mathcal{I}} \frac{\int_{\mathcal{I}} |f(y') (y')^\gamma | (y')^{1+\omega}e^{-\frac{(y')^2}{4}}dy'}{\int_{\mathcal{I}} (y')^{1+\omega} e^{-\frac{(y')^2}{4}}dy'}, \omega = \sqrt{d^2-12d +24},
\end{equation}
here the supremum is taken over all sub-intervals $\mathcal{I}$ containing $y$. In particular, if $|f(y) y^\gamma |$  is a non decreasing function, then  the supremum   in \eqref{defi-maximan-function} is attained by $\mathcal{I} = [y,+\infty)$. Otherwise, if  $ |f(y) y^\gamma |$ is a non increasing, then  the supremum is attained by $\mathcal{I}= [0,y]$.
\end{lemma}
\begin{proof}
The proof is quite the same as for Lemma $VI.2$ in  \cite{BSIMRN19}. 
\end{proof}
Next, we will estimate the growth  of the action $e^{(\tau -\tau') \mathscr{L}_\infty}$ to $\Lambda$:
\begin{lemma}\label{Lemma-esti-semigroup-Lambda-var} Let us consider $f \in L^2_\rho$ and $\Lambda f \in L^2_\rho$ where $\Lambda $ was defined in \eqref{defi-Lambda-f}.  Assume further more that 
\begin{equation}\label{condition-f-le-mathcal-B}
	\left|  f(y)\right|  \le     \mathcal{B} \frac{\langle y \rangle^{2\ell +2}}{y^\gamma}, \text{ for some } \mathcal{B} \in \R^*_+.
\end{equation}
Then, it holds that   for $\tau > \tau'$
\begin{eqnarray}
	\left| e^{(\tau-\tau') \mathscr{L}_\infty} (\Lambda f)    \right| \lesssim e^{\tau -\tau'}\mathcal{B} \langle y \rangle^{2 \ell +3} y^{-\gamma}.
\end{eqnarray}
\end{lemma}
\begin{proof}
Recall from \eqref{defi-Lambda-f} that   $ \Lambda f = y  \partial_y f  + 2f   $ and apply  \eqref{semi-f} in deriving
\begin{eqnarray*}
	\left[e^{(\tau-\tau')\mathcal{L}_\infty} \right] \Lambda f (y) = 2^{\omega +1 } y^{-\gamma} e^{\frac{\alpha}{2}(\tau-\tau')} \int_0^{\infty}  P_{\zeta} \left( \frac{y^2}{4}, \frac{x^2}{4}, r \right) \left( x \partial_x f + 2f \right)x^{\omega+1+\gamma} e^{-\frac{x^2}{4}} dx,
\end{eqnarray*}
where $r = e^{-(\tau-\tau')}$ and $  \zeta  = \frac{\omega}{2}$. First,  Lemma \ref{maximal-lemma} results in
\begin{eqnarray*}
	\left|  2^{\omega +1 } y^{-\gamma} e^{\frac{\alpha}{2}(\tau-\tau')} \int_0^{\infty}  P_{\zeta} \left( \frac{y^2}{4}, \frac{x^2}{4}, r \right) f(x) x^{\omega+1+\gamma} e^{-\frac{x^2}{4}} dx  \right|  \lesssim \mathcal{B} e^{\frac{\alpha}{2}(\tau - \tau')}\langle y \rangle^{2\ell + 2} y^{-\gamma}. 
\end{eqnarray*}
Then, it is sufficient to prove that 
$$ \left| I(y)\right| \lesssim \mathcal{B} \langle y \rangle^{2\ell +3} \text{ where }  I =\int_0^{\infty}  P_{\zeta} \left( \frac{y^2}{4}, \frac{x^2}{4}, r \right) \left( x \partial_x f(x) \right)x^{\omega+1+\gamma} e^{-\frac{x^2}{4}} dx. $$
Using the integration by parts provided that the functions go to $0$ at $+\infty$ and $0$, we get
\begin{eqnarray*}
	I &=& -\int_0^\infty f(x) \partial_x \left(  P_\zeta \left( \frac{y^2}{4}, \frac{x^2}{4}, r \right)  x^{\omega+2+\gamma} e^{-\frac{x^2}{4}}   \right)   dx\\
	& = &- \int_0^\infty f(x) \partial_x\left( P_\zeta\left( \frac{y^2}{4}, \frac{x^2}{4}, r \right) \right) x^{\omega+2+\gamma} e^{-\frac{x^2}{4}}dx. \\
	&-&(\omega+2+\gamma)\int_0^\infty f(x) P_\zeta\left( \frac{y^2}{4}, \frac{x^2}{4}, r \right) x^{\omega+1+\gamma}  e^{-\frac{x^2}{4}} dx\\
	&+&\frac{1}{2}\int_0^\infty f(x) P_\zeta\left( \frac{y^2}{4}, \frac{x^2}{4}, r \right) x^{\omega+3+\gamma} e^{-\frac{x^2}{4}}dx.
\end{eqnarray*}
We now  explicitly compute  
\begin{eqnarray*}
	\partial_x\left( P_\zeta\left( \frac{y^2}{4}, \frac{x^2}{4}, r \right) \right)&=&  \partial_x\left( \frac{4^\zeta}{\sqrt{r}^{\zeta} (1-r) (yx)^{\zeta}}   e^{-\frac{r}{1-r}\left(\frac{y^2}{4}+\frac{x^2}{4}\right)} \mathbf{I}_\zeta\left(\frac{r^\frac{1}{2}y x}{2(1-r)} \right)\right)\\
	&=&\partial_x\left( \frac{4^\zeta}{\sqrt{r}^\zeta (1-r) (yx)^\zeta}   e^{-\frac{r}{1-r}\left(\frac{y^2}{4}+\frac{x^2}{4}\right)}\right) \mathbf{I}_\zeta\left(\frac{r^\frac{1}{2}y x}{2(1-r)} \right)\\
	&+&\frac{4^\zeta}{\sqrt{r}^\zeta (1-r) (yx)^\zeta}   e^{-\frac{r}{1-r}\left(\frac{y^2}{4}+\frac{x^2}{4}\right)} \partial_x\left(\mathbf{I}_\zeta\left(\frac{r^\frac{1}{2}y x}{2(1-r)} \right)\right)\\
	&=&\partial_x\left( \frac{4^\zeta}{\sqrt{r}^\zeta (1-r) (yx)^\zeta}   e^{-\frac{r}{1-r}\left(\frac{y^2}{4}+\frac{x^2}{4}\right)}\right) \mathbf{I}_\zeta\left(\frac{r^\frac{1}{2}y x}{2(1-r)} \right)\\
	&+&\frac{4^\zeta}{x\sqrt{r}^\zeta (1-r) (yx)^\zeta}   e^{-\frac{r}{1-r}\left(\frac{y^2}{4}+\frac{x^2}{4}\right)}\frac{r^\frac{1}{2}yx }{2(1-r)}\mathbf{I}'_\zeta\left(\frac{r^\frac{1}{2}y x}{2(1-r)} \right),
\end{eqnarray*}
from equality $(3)$ at page 79 in \cite{WbookCUP1922} 
\begin{equation*}
	\frac{r^\frac{1}{2}y x}{2(1-r)}    \mathbf{I}'_\zeta\left(\frac{r^\frac{1}{2}y x}{2(1-r)} \right)=\zeta \mathbf{I}_\zeta\left(\frac{r^\frac{1}{2}y x}{2(1-r)} \right)+\frac{r^\frac{1}{2}y x}{2(1-r)} \mathbf{I}_{\zeta+1}\left(\frac{r^\frac{1}{2}y x}{2(1-r)}\right),
\end{equation*}
we infer
\begin{eqnarray*}
	\partial_x\left( P_\zeta\left( \frac{y^2}{4}, \frac{x^2}{4}, r \right) \right)&=&\partial_x\left( \frac{4^\zeta}{\sqrt{r}^\zeta (1-r) (yx)^\zeta}   e^{-\frac{r}{1-r}\left(\frac{y^2}{4}+\frac{x^2}{4}\right)}\right) \mathbf{I}_\zeta\left(\frac{r^\frac{1}{2}y x}{2(1-r)} \right)\\
	&+&\frac{4^\zeta}{x\sqrt{r}^\zeta (1-r) (yx)^\zeta}   e^{-\frac{r}{1-r}\left(\frac{y^2}{4}+\frac{x^2}{4}\right)}\zeta \mathbf{I}_\zeta\left(\frac{r^\frac{1}{2}y x}{2(1-r)} \right)\\
	&+&\frac{4^\zeta}{x\sqrt{r}^\zeta (1-r) (yx)^\zeta}   e^{-\frac{r}{1-r}\left(\frac{y^2}{4}+\frac{x^2}{4}\right)}\frac{r^\frac{1}{2}y x}{2(1-r)} \mathbf{I}_{\zeta+1}\left(\frac{r^\frac{1}{2}y x}{2(1-r)}\right).
\end{eqnarray*}
Besides, we have
\begin{eqnarray*}
	\partial_x\left( \frac{4^\zeta}{\sqrt{r}^\zeta (1-r) (yx)^\zeta}   e^{-\frac{r}{1-r}\left(\frac{y^2}{4}+\frac{x^2}{4}\right)}\right) &=& \frac{4^\zeta(-\zeta)}{\sqrt{r}^\zeta (1-r) (yx)^\zeta x} e^{-\frac{r}{1-r}\left(\frac{y^2}{4}+\frac{x^2}{4} \right)} \\
	& - & \frac{r}{1-r} \frac{x}{2} \frac{4^\zeta}{\sqrt{r}^\zeta (1 -r) (yx)^\zeta} e^{-\frac{r}{1-r}\left(\frac{y^2}{4}+\frac{x^2}{4} \right)}.
\end{eqnarray*}
At final, we arrive to
\begin{eqnarray*}
	x\partial_x\left( P_\zeta\left( \frac{y^2}{4}, \frac{x^2}{4}, r \right) \right)&=&  -\frac{r}{1-r} \frac{x^2}{2} \frac{4^\zeta}{\sqrt{r}^\zeta (1 -r) (yx)^\zeta} e^{-\frac{r}{1-r}\left(\frac{y^2}{4}+\frac{x^2}{4} \right)}\mathbf{I}_\zeta\left(\frac{r^\frac{1}{2}y x}{2(1-r)} \right)\\
	&+&\frac{4^\zeta}{\sqrt{r}^\zeta (1-r) (yx)^\zeta} \frac{r^\frac{1}{2}y x}{2(1-r)}  e^{-\frac{r}{1-r}\left(\frac{y^2}{4}+\frac{x^2}{4}\right)} \mathbf{I}_{\zeta+1}\left(\frac{r^\frac{1}{2}y x}{2(1-r)}\right)\\
	&=&-\frac{r x^2}{2(1-r)}P_\zeta\left( \frac{y^2}{4}, \frac{x^2}{4}, r \right)+\frac{ry^2x^2}{8(1-r)}P_{\zeta+1}\left( \frac{y^2}{4}, \frac{x^2}{4}, r \right).
\end{eqnarray*}
Plugging  into $I$'s formula, we get  
\begin{eqnarray}
	I& = &\frac{r}{2(1-r)} \int_0^\infty f(x)  P_\zeta\left( \frac{y^2}{4}, \frac{x^2}{4}, r \right)  x^{\omega+3+\gamma} e^{-\frac{x^2}{4}}\nonumber  \\
	&-&\frac{ry^2}{8(1-r)}\int_0^\infty f(x)  P_{\zeta+1}\left( \frac{y^2}{4}, \frac{x^2}{4}, r \right)  x^{\omega+3+\gamma} e^{-\frac{x^2}{4}} \nonumber\\
	&-&(\omega+2+\gamma)\int_0^\infty f(x) P_\zeta\left( \frac{y^2}{4}, \frac{x^2}{4}, r \right) x^{\omega+1+\gamma}  e^{-\frac{x^2}{4}} \nonumber  \\
	&+&\frac{1}{2}\int_0^\infty f(x) P_\zeta\left( \frac{y^2}{4}, \frac{x^2}{4}, r \right) x^{\omega+3+\gamma} e^{-\frac{x^2}{4}}\nonumber\\
	&= &  \int_0^\infty \left[ \frac{P_\zeta\left( \frac{y^2}{4}, \frac{x^2}{4}, r \right)}{2(1 -r )} - \frac{r y^2}{8(1-r)} P_{\zeta +1}\left( \frac{y^2}{4}, \frac{x^2}{4}, r \right) \right] f(x) x^{\omega +3+\gamma} e^{-\frac{x^2}{4}} dx  \label{intergral-I}  \\
	& - & (\omega+2+\gamma)\int_0^\infty f(x) P_\zeta\left( \frac{y^2}{4}, \frac{x^2}{4}, r \right) x^{\omega+1+\gamma}  e^{-\frac{x^2}{4}} dx.\nonumber
\end{eqnarray}
It also follows from Lemma \ref{maximal-lemma} that
\begin{eqnarray*}
	\left|  \int_0^\infty f(x) P_\zeta\left( \frac{y^2}{4}, \frac{x^2}{4}, r \right) x^{\omega+1+\gamma}  e^{-\frac{x^2}{4}} dx \right|  \lesssim M(f),
\end{eqnarray*}
and from \eqref{condition-f-le-mathcal-B}, we have
\begin{eqnarray}
	M(f)(y) & = & \sup_{y \in \mathcal{I}} \frac{\int_{\mathcal{I}} |f(y') (y')^\gamma | (y')^{1+\omega}e^{-\frac{(y')^2}{4}}dy'}{\int_{\mathcal{I}} (y')^{1+\omega} e^{-\frac{(y')^2}{4}}dy'} \lesssim  \mathcal{B}\frac{\int_{y}^\infty \langle y' \rangle^{2\ell +2}(y')^{1+\omega}e^{-\frac{(y')^2}{4}}dy'}{\int_{y}^\infty (y')^{1+\omega} e^{-\frac{(y')^2}{4}}dy'} \\
	& \lesssim & \mathcal{B} \langle y \rangle^{2\ell + 2}.\label{estimate-for mathcal-L-f-poison-Kernel}
\end{eqnarray}
Now, it remains to  prove the following:
\begin{equation}\label{goal-integral-I}
	\left|  I_1(y)\right| \lesssim \frac{\mathcal{B}}{\sqrt{1 -r }}  \langle y \rangle^{2\ell+3},
\end{equation}
where
\begin{eqnarray*}
	I_1(y)= \int_0^\infty \left[ \frac{P_\zeta\left( \frac{y^2}{4}, \frac{x^2}{4}, r \right)}{2(1 -r )} - \frac{r y^2}{8(1-r)} P_{\zeta +1}\left( \frac{y^2}{4}, \frac{x^2}{4}, r \right) \right] f(x) x^{\omega +3+\gamma} e^{-\frac{x^2}{4}} dx 
\end{eqnarray*}
We recall formulae $(2)$ at page 77 and formula $(2)$ at page 203 in \cite{WbookCUP1922} regarding the function $I_\zeta$ of imaginary argument
\begin{equation}\label{bound-on-I-zeta}
	\begin{array}{rcl}
		cz^\zeta \leq \mathbf{I}_\zeta(z) \leq C z^\zeta  & \text{ if } z \in [0,1],  \\
		cz^{-\frac{3}{2}} e^z \leq \left|\mathbf{I}_\zeta(z) - \frac{e^z}{\sqrt{2\pi z}}  \right|  \leq C z^{-\frac{3}{2}} e^z  & \text{ if } z \in [1,+\infty) ,
	\end{array}  
\end{equation}
%where is sharper  than  formula $(4.4)$ at page 238 in \cite{MTAMS69}.
In particular,  \eqref{bound-on-I-zeta} implies
\begin{equation}\label{functional-H}
	C^{-1} H_\zeta(y,x,r) \le  P_\zeta \left(\frac{y^2}{4},\frac{x^2}{4},r \right) \le C H_\zeta(y,x,r),
\end{equation}
where 
\begin{equation}\label{defi-H-function}
	H_\zeta(y,x,r)= \left\{ \begin{array}{rcl}
		(1-r)^{-\zeta-1}e^{-\frac{r(\frac{y^2}{4}+\frac{x^2}{4})}{1-r}}   & \text{ if } x \in \left[0,\frac{2(1-r)}{ \sqrt{r}y}\right]  \\[0.2cm]
		\frac{(4r)^{-\frac{\zeta}{2}-\frac{1}{4}} (yx)^{-\zeta-\frac{1}{2}}}{e(\sqrt{1-r})} e^{\frac{-r\frac{y^2}{4}+ \frac{1}{2}(r)^{\frac{1}{2}}(yx)-r\frac{x^2}{4}}{1-r}}   & \text{ if } z \in \left[\frac{2(1-r)}{\sqrt{r}y},+\infty\right) 
	\end{array}  
	\right..
\end{equation}
\iffalse
Now, we aim to estimate to integrals involving  $I$ provided that $ \varepsilon(\tau') \in V[A,\eta, \tilde \eta](\tau')$, that
$$ \left| x^\gamma \hat{\varepsilon}_-(x,\tau') \right| \le A^3 b^{\frac{\alpha}{2}+\tilde \eta}(\tau') \left(1 + x^{2l+2} \right).$$
The main goal is to prove the following

Let us start to the proof of  \eqref{goal-integral-I}. We observe that the second integral in \eqref{intergral-I}  can be estimated by $M(w(\tau'))$ by 
the  inequality  $(85)$ in \cite{BSIMRN19} (originally proved in \cite{MTAMS69}):
\begin{eqnarray*}
	\left| \int_0^\infty w(\tau') P_\zeta\left( \frac{y^2}{4}, \frac{x^2}{4}, r \right) x^{\omega+1+\gamma}  e^{-\frac{x^2}{4}} dx \right| \le  M(w(\tau')) \le C A^3 b^{\frac{\alpha}{2} +\tilde \eta}(\tau')(1 + y^{2\ell +2}).
\end{eqnarray*}
Thus, it remains to estimate first one.   Accordingly to the change of $H$'s behavior   in  \eqref{defi-H-function},
we decompose the integral as follows

\iffalse
we consider  two cases: $y \in \left[0,\frac{2(1-r)}{\sqrt{r} y} \right]$   and $ y \in  \left[\frac{2(1-r)}{\sqrt{r} y}, +\infty \right) $

We deduce from \eqref{functional-H} that

$$I_1 = \frac{r}{1-r} \int_0^\infty P_\zeta \left(\frac{y^4}{4}, \frac{x^2}{4}, r  \right) x^\gamma \hat{\varepsilon}(\tau')  x^{\omega +3} e^{-\frac{x^2}{4}} dx.  $$
\fi
\fi 
Using \eqref{condition-f-le-mathcal-B}, we estimate 
\begin{eqnarray*}
	& &\left|I_1 (y,\tau,\tau')\right| \\
	&\lesssim &  \mathcal{B} \left( \int_0^{\frac{2(1-r)}{\sqrt{r}y}} \left| \frac{P_\zeta\left( \frac{y^2}{4}, \frac{x^2}{4}, r \right)}{2(1 -r )} - \frac{r y^2}{8(1-r)} P_{\zeta +1}\left( \frac{y^2}{4}, \frac{x^2}{4}, r \right)  \right| (1 + x^{2\ell +2}) x^{\omega + 3 + \gamma} e^{ -\frac{x^2}{4}} dx \right.\\
	& + &  \left.\int_0^{\frac{2(1-r)}{\sqrt{r}y}} \left| \frac{P_\zeta\left( \frac{y^2}{4}, \frac{x^2}{4}, r \right)}{2(1 -r )} - \frac{r y^2}{8(1-r)} P_{\zeta +1}\left( \frac{y^2}{4}, \frac{x^2}{4}, r \right)  \right| (1 + x^{2\ell +2}) x^{\omega + 3 + \gamma} e^{ -\frac{x^2}{4}} dx     \right) .
\end{eqnarray*}
+ For the  integral on  $\left[0, \frac{2(1-r)}{\sqrt{r}y} \right]$, we use  the first asymptotic in  \eqref{defi-H-function} to obtain
\begin{eqnarray*}
	& & \int_0^{\frac{2(1-r)}{\sqrt{r}y}} \left| \frac{P_\zeta\left( \frac{y^2}{4}, \frac{x^2}{4}, r \right)}{2(1 -r )} - \frac{r y^2}{8(1-r)} P_{\zeta +1}\left( \frac{y^2}{4}, \frac{x^2}{4}, r \right)  \right| \left(1 + x^{2\ell +2}\right) x^{\omega + 3 +\gamma} e^{ -\frac{x^2}{4}} dx \\
	&\le & C (1- r)^{-\zeta -2} \int_{0}^{\frac{2(1-r)}{\sqrt{r}y}} \left(1 + \frac{ry^2}{(1 -r)} \right) e^{-\frac{r}{1-r} \left(\frac{y^2}{4} + \frac{x^2}{4}\right)} (1 +x^{2\ell +2})x^{\omega +3} e^{-\frac{x^2}{4}} dx.
\end{eqnarray*}
On the one hand,   once 
$ r \in (0,\frac{1}{4})$,  it immediately follows that
\begin{eqnarray*}
	& &(1-r)^{-\zeta-2}\int_0^{\frac{2(1-r)}{\sqrt{r}y}}  \left(1 + \frac{ry^2}{(1 -r)} \right) e^{ - \frac{r}{1-r}\left( \frac{y^2}{4} +\frac{x^2}{4}\right)} (1 + x^{2\ell +2}) x^{\omega + 3} e^{ -\frac{x^2}{4}} dx \\
	&\le &  C \int_0^{\frac{2(1-r)}{\sqrt{r}y}} (1 + x^{2\ell +2}) x^{\omega + 3}  e^{-\frac{x^2}{4} -\frac{r x^2}{4(1-r)} } dx \le  C \int_0^\infty (1 +x^{2\ell +2}) x^{\omega +3} e^{ -\frac{x^2}{4}} dx  \le C,
\end{eqnarray*}
where $C$ is independent of $r$. On the other hand,  once  $r \ge \frac{1}{4}$ and by a change  of variable $z = \frac{\sqrt{r}}{2\sqrt{1-r}} x$ we obtain
\begin{eqnarray*}
	& & \int_0^{\frac{2(1-r)}{\sqrt{r}y}} \left(1 + \frac{ry^2}{(1-r)} \right) e^{ - \frac{r}{1-r}\left( \frac{y^2}{4} +\frac{x^2}{4}\right)} (1 + x^{2\ell +2}) x^{\omega + 3} e^{ -\frac{x^2}{4}} dx \\
	& \le & C e^{-\frac{r y^2}{4(1-r)}} \int_0^{\frac{\sqrt{1-r}}{y}} \left(1 + \frac{r y^2}{(1-r)} \right) e^{-z^2 \left( 1+\frac{1-r}{r} \right)} (1 + \frac{(1-r)^{\ell +1}}{r^{\ell +1}}z^{2\ell +2}) \frac{(1-r)^{\frac{\omega}{2}+2}}{r^{\frac{\omega}{2}+2}} z^{\omega +3} dz\\
	& \le & C (1 -r)^{\frac{\omega}{2}+2} e^{-\frac{r}{4(1-r)} y^2} \left(1 + \frac{r y^2}{(1-r)} \right) \int_0^{\frac{\sqrt{1-r}}{y}} e^{-z^2 \left( 1+\frac{1-r}{r} \right)} (1 + (1-r)^{\ell +1} z^{2\ell +2}) z^{\omega +3} dz \\
	& \le &  (1- r)^{\zeta +2} e^{-\frac{r}{4(1-r)}y^2} \int_0^\infty e^{-z^2} (1 + z^{2\ell +2+\omega+3}) dz \le C (1 -r)^{\zeta +2}, \text{ with } \zeta = \frac{\omega}{2}.
\end{eqnarray*}
Finally, we obtain
\begin{eqnarray*}
	\left|\int_0^{\frac{2(1-r)}{\sqrt{r}y}} \left| \frac{P_\zeta\left( \frac{y^2}{4}, \frac{x^2}{4}, r \right)}{2(1 -r )} - \frac{r y^2}{8(1-r)} P_{\zeta +1}\left( \frac{y^2}{4}, \frac{x^2}{4}, r \right)  \right| \left(1 + x^{2\ell +2}\right) x^{\omega + 3 +\gamma} e^{ -\frac{x^2}{4}} dx  \right| \lesssim  \frac{\langle y \rangle^{2\ell +3}}{\sqrt{1-r}}.
\end{eqnarray*}

\medskip
- For the integral on $\left[\frac{2(1-r)}{\sqrt{r}y}, +\infty \right)$, we apply  \eqref{bound-on-I-zeta} and    \eqref{defi-H-function} and noticing that $z = \frac{\sqrt{r} yx}{2(1 -r)} \ge 1 $
\begin{eqnarray*}
	\frac{x}{2(1-r)} P_\zeta \left( \frac{y^2}{4}, \frac{x^2}{4}, r\right)  = \frac{4^\zeta}{2\sqrt{r}^\zeta (1-r) (yx)^\zeta} \left( \frac{x}{2(1-r)} \right) e^{-\frac{r}{4(1-r)} ( y^2 +x^2)} \left( \frac{e^{z}}{\sqrt{2\pi z}}  + O(e^{z} z^{-\frac{3}{2}})\right),
\end{eqnarray*}
and
\begin{eqnarray*}
	\frac{r y^2 x }{8(1-r)} P_{\zeta +1} \left( \frac{y^2}{4}, \frac{x^2}{4}, r\right) = \frac{r y^2 x }{8(1-r)} \frac{4^{\zeta+1}}{\sqrt{r}^{\zeta+1}(1-r) (yx)^{\zeta +1}} e^{-\frac{r}{4(1-r)} ( y^2 +x^2)} \left( \frac{e^{z}}{\sqrt{2\pi z}}  +O (e^{z} z^{-\frac{3}{2}})\right).
\end{eqnarray*}
Then 
\begin{eqnarray*}
	& &\left| x \frac{P_\zeta\left( \frac{y^2}{4}, \frac{x^2}{4}, r \right)}{2(1 -r )} - x\frac{r y^2}{8(1-r)} P_{\zeta +1}\left( \frac{y^2}{4}, \frac{x^2}{4}, r \right)  \right| \\
	& \le & C \left| \frac{x - \sqrt{r} y }{(1-r)} \right| \frac{4^\zeta}{2\sqrt{r}^\zeta (1-r) (yx)^\zeta}  e^{-\frac{r}{4(1-r)} ( y^2 +x^2)} \frac{e^{z}}{\sqrt{2\pi z}}\\
	& +  & C \frac{x}{(1-r)} \frac{4^\zeta}{2\sqrt{r}^\zeta (1-r) (yx)^\zeta}  e^{-\frac{r}{4(1-r)} ( y^2 +x^2)} \frac{e^{z}}{z^{\frac{3}{2}}} \\
	& +& C \frac{\sqrt{r}y}{(1-r)}\frac{4^\zeta}{2\sqrt{r}^\zeta (1-r) (yx)^\zeta}  e^{-\frac{r}{4(1-r)} ( y^2 +x^2)} \frac{e^{z}}{z^{\frac{3}{2}}} 
	\\
	&\le &  C  \frac{r^{-\frac{\zeta}{2}-\frac{1}{4} } (yx)^{-\zeta -\frac{1}{2}}}{\sqrt{1-r}} \left|\frac{x -\sqrt{r} y}{1-r} \right| e^{\frac{-\frac{r y^2}{4} + \frac{1}{2} \sqrt{r} y x - \frac{r x^2}{4}}{1-r}}\\
	&+&  C r^{-\frac{\zeta}{2}-\frac{3}{4}}(yx)^{-\zeta-\frac{3}{2}} \left|\frac{x+\sqrt{r}y}{\sqrt{1-r}}\right|e^{\frac{-\frac{r y^2}{4} + \frac{1}{2} \sqrt{r} y x - \frac{r x^2}{4}}{1-r}}
\end{eqnarray*}
\iffalse
\begin{eqnarray*}
	&& \frac{P_\zeta\left( \frac{y^2}{4}, \frac{x^2}{4}, r \right)}{2(1 -r )} - \frac{r y^2}{8(1-r)} P_{\zeta +1}\left( \frac{y^2}{4}, \frac{x^2}{4}, r \right)\\
	&=&\frac{4^\zeta}{2\sqrt{r}^\zeta (1-r) (yx)^\zeta}   e^{-\frac{r}{1-r}\left(\frac{y^2}{4}+\frac{x^2}{4}\right)}\left( \mathbf{I}_\zeta\left(\frac{r^\frac{1}{2}y x}{2(1-r)} \right)-\frac{y\sqrt{r}}{x}\mathbf{I}_{\zeta+1}\left(\frac{r^\frac{1}{2}y x}{2(1-r)} \right) \right)\\
	&=&\frac{4^\zeta}{2\sqrt{r}^\zeta (1-r) (yx)^\zeta}   e^{-\frac{r}{1-r}\left(\frac{y^2}{4}+\frac{x^2}{4}\right)}\left( \left(\mathbf{I}_\zeta\left(\frac{r^\frac{1}{2}y x}{2(1-r)} \right)-\frac{e^z}{\sqrt{2\pi z}}\right)-\frac{y\sqrt{r}}{x}\left(\mathbf{I}_{\zeta+1}\left(\frac{r^\frac{1}{2}y x}{2(1-r)} \right)-\frac{x}{y\sqrt{r}}\frac{e^z}{\sqrt{2\pi z}} \right)\right)
\end{eqnarray*}
with $z=\frac{r^\frac{1}{2}y x}{2(1-r)}$

\begin{eqnarray*}
	\frac{y \sqrt{r}}{x} I_{\zeta+1} (z)  = \frac{y \sqrt{r}}{x} \left( \frac{e^z}{\sqrt{2\pi z}} + I_{\zeta+1} - \frac{e^z}{\sqrt{2\pi z}} \right) 
\end{eqnarray*}
we 
\begin{eqnarray*}
	y^{-\zeta -\frac{1}{2}} \frac{r^{-\frac{\zeta}{2} -\frac{1}{4}} }{\sqrt{1-r}} \int_{\frac{2(1-r)}{\sqrt{r}y}}^\infty e^{\frac{-r\frac{y^2}{4}+ \frac{1}{2}(r)^{\frac{1}{2}}(yx)-r\frac{x^2}{4}}{1-r}}  (1 + x^{2\ell +2}) x^{\omega +3}  e^{-\frac{x^2}{4}} dx 
\end{eqnarray*}
\fi
Notice that  we  have the following  identity 
$$ e^{\frac{-r\frac{y^2}{4}+ \frac{1}{2}(r)^{\frac{1}{2}}(yx)-r\frac{x^2}{4}}{1-r}} e^{-\frac{x^2}{4}} = e^{-\frac{1}{4} \left(\frac{x -\sqrt{r} y }{\sqrt{1-r}} \right)^2},$$
and by a change of variable $z = \frac{x -\sqrt{r} y}{\sqrt{1-r}}$, we get
\begin{eqnarray*}
	& & \int_{\frac{2(1-r)}{\sqrt{r}y}}^\infty \left| x\frac{P_\zeta\left( \frac{y^2}{4}, \frac{x^2}{4}, r \right)}{2(1 -r )} -x \frac{r y^2}{8(1-r)} P_{\zeta +1}\left( \frac{y^2}{4}, \frac{x^2}{4}, r \right)  \right| (1 + x^{2\ell +2}) x^{\omega + 2} e^{ -\frac{x^2}{4}} dx \\
	& \le & C \frac{y^{-\zeta -\frac{1}{2}} r^{-\frac{\zeta}{4}-\frac{1}{4}}  }{\sqrt{1-r}} \int_{\frac{2(1-r)}{\sqrt{r}y}}^\infty  \left|\frac{x -\sqrt{r} y}{1-r} \right| (1 + x^{2\ell +2}) x^{\zeta + \frac{3}{2}} e^{-\frac{1}{4} \left(\frac{x -\sqrt{r} y }{\sqrt{1-r}} \right)^2} dx\\
	&+& C\frac{y^{-\zeta-\frac{3}{2}} r^{-\frac{\zeta}{2}-\frac{3}{4}}}{\sqrt{1-r}}\int_{\frac{2(1-r)}{\sqrt{r}y}}^\infty  \left|\frac{x+\sqrt{r}y}{\sqrt{1-r}}\right|(1 + x^{2\ell +2}) x^{\zeta + \frac{1}{2}} e^{-\frac{1}{4} \left(\frac{x -\sqrt{r} y }{\sqrt{1-r}} \right)^2} dx\\
	& \le & C\frac{(\sqrt{r}y)^{-\zeta -\frac{1}{2}}   }{\sqrt{1-r}} \int_{ \frac{\frac{2(1-r)}{\sqrt{r} y} - \sqrt{r} y}{\sqrt{1-r}} }^{\infty}  
	|z|(1 + (z\sqrt{1-r}+\sqrt{r}y)^{2\ell +2}) \left| z\sqrt{1-r}+\sqrt{r}y \right|^{\zeta + \frac{3}{2}}e^{-\frac{z^2}{4}} dz\\
	&+& C \frac{(\sqrt{r}y)^{-\zeta-\frac{3}{2}}}{\sqrt{1-r}}\int_{ \frac{\frac{2(1-r)}{\sqrt{r} y} - \sqrt{r} y}{\sqrt{1-r}} }^{\infty}  (1 + (z\sqrt{1-r}+\sqrt{r}y)^{2\ell +2})  \left| z\sqrt{1-r}+\sqrt{r}y \right|^{\zeta + \frac{3}{2}}e^{-\frac{z^2}{4}} dz.
	%\\
	%&\leq& C \frac{(\sqrt{r}y)^{-\zeta -\frac{1}{2}}   }{\sqrt{1-r}} \int_{ \frac{\frac{2(1-r)}{\sqrt{r} y} - \sqrt{r} y}{\sqrt{1-r}} }^{\infty} |z| (1+|z|^{2l+2}+|\sqrt{r}y|^{2l+2})(|z|^{\zeta + \frac{3}{2}}+|\sqrt{r}y|^{\zeta + \frac{3}{2}})e^{-\frac{z^2}{4}} dz\\
	%&+& C \frac{(\sqrt{r}y)^{-\zeta-\frac{3}{2}}}{\sqrt{1-r}}\int_{ \frac{\frac{2(1-r)}{\sqrt{r} y} - \sqrt{r} y}{\sqrt{1-r}} }^{\infty} (|z|+|\sqrt{r}y|)(1+|z|^{2\ell +2}+|\sqrt{r}y|^{2\ell+2})(|z|^{\zeta + \frac{1}{2}}+|\sqrt{r}y|^{\zeta + \frac{1}{2}})e^{-\frac{z^2}{4}} dz \\
	%& \le & \frac{C (1 + y^{2\ell +2 +1})}{\sqrt{1-r}}.
\end{eqnarray*}
Now, we have
\begin{eqnarray*}
	|z|(1 + (z\sqrt{1-r}+\sqrt{r}y)^{2\ell +2}) \left| z\sqrt{1-r}+\sqrt{r}y \right|^{\zeta + \frac{3}{2}}  \lesssim \langle y \rangle^{2\ell + 2}\langle z \rangle^{2\ell + 3} \left[ |z \sqrt{1-r}|^{\zeta +\frac{3}{2}} + |\sqrt{r} y|^{\zeta +\frac{3}{2}} \right] 
\end{eqnarray*}
and 
\begin{eqnarray*}
	& & (\sqrt{r}y)^{-\zeta -\frac{1}{2}}\int_{ \frac{\frac{2(1-r)}{\sqrt{r} y} - \sqrt{r} y}{\sqrt{1-r}} }^{\infty} \langle z \rangle^{2\ell + 3}  |z \sqrt{1-r}|^{\zeta +\frac{3}{2}} e^{-\frac{z^2}{4}} dz\\
	&\lesssim & (\sqrt{r}y)^{-\zeta -\frac{1}{2}} \sqrt{1-r}^{\zeta+\frac{3}{2}} \int_{ \frac{\frac{2(1-r)}{\sqrt{r} y} - \sqrt{r} y}{\sqrt{1-r}} }^{\infty} \langle z \rangle^{2\ell +3 +\zeta +\frac{3}{2}}e^{-\frac{z^2}{4}} dz\\[0.2cm]
	& \lesssim &   X^{\zeta +\frac{1}{2}} \int_{2X-\frac{1}{X}}^\infty  \langle z \rangle^{2\ell +3 +\zeta +\frac{3}{2}}e^{-\frac{z^2}{4}} dz  \lesssim 1, \text{ with } X = \frac{\sqrt{1-r}}{\sqrt{r} y} >0,
\end{eqnarray*}
yielding
\begin{eqnarray*}
	& & \frac{(\sqrt{r}y)^{-\zeta -\frac{1}{2}}   }{\sqrt{1-r}} \int_{ \frac{\frac{2(1-r)}{\sqrt{r} y} - \sqrt{r} y}{\sqrt{1-r}} }^{\infty}  
	|z|(1 + (z\sqrt{1-r}+\sqrt{r}y)^{2\ell +2}) \left| z\sqrt{1-r}+\sqrt{r}y \right|^{\zeta + \frac{3}{2}}e^{-\frac{z^2}{4}} dz \\
	& \lesssim & \frac{\langle y \rangle^{2\ell +3}}{\sqrt{1-r}}.
\end{eqnarray*}
Similarly we have
\begin{eqnarray*}
	& &\frac{(\sqrt{r}y)^{-\zeta-\frac{3}{2}}}{\sqrt{1-r}}\int_{ \frac{\frac{2(1-r)}{\sqrt{r} y} - \sqrt{r} y}{\sqrt{1-r}} }^{\infty}  (1 + (z\sqrt{1-r}+\sqrt{r}y)^{2\ell +2})  \left| z\sqrt{1-r}+\sqrt{r}y \right|^{\zeta + \frac{3}{2}}e^{-\frac{z^2}{4}} dz 
	\lesssim   \frac{\langle y\rangle^{2\ell+3}}{\sqrt{1-r}}.
\end{eqnarray*}
Thus, we obtain
\begin{eqnarray*}
	\int_{\frac{2(1-r)}{\sqrt{r}y}}^\infty \left| x\frac{P_\zeta\left( \frac{y^2}{4}, \frac{x^2}{4}, r \right)}{2(1 -r )} -x \frac{r y^2}{8(1-r)} P_{\zeta +1}\left( \frac{y^2}{4}, \frac{x^2}{4}, r \right)  \right| (1 + x^{2\ell +2}) x^{\omega + 2} e^{ -\frac{x^2}{4}} dx \lesssim \frac{\langle  y \rangle^{2\ell+3}}{\sqrt{1-r}}.
\end{eqnarray*}
By adding all related terms, we conclude \eqref{goal-integral-I}  and with it the proof of the Lemma is accomplished.
\end{proof}

\section{Generators of  the Kernel of $H$}\label{proof-lemma-generator-H}
%In this part, we focus on the complete proof of  Lemma \ref{lemma-Generation-H}.  
We construct  the family  $\{T_{i+1}\}$  via the  recursive formula 
\begin{equation}\label{recurrence-T-i+1-i}
T_{i+1} = H^{-1} T_i, \quad \text{ and } T_0 = c \Lambda_\xi Q
\end{equation}
here $c$ is some constant which will be chosen later. In other words, we have for $i\geq1$
\begin{equation*}\label{T-i-H--1-T-0}
T_{i} =   H^{-i} (T_0), \quad i\geq 1. 
\end{equation*}
The operator   $H^{-1}$ is  explicitly  given by 
\begin{equation}\label{defi-H-1-f}
{H}^{-1} f (\xi) =  \Lambda Q(\xi) \int_0^\xi \frac{{L} f(\xi')}{\Lambda Q(\xi')} d\xi',
\end{equation}
with ${L}$ 
\begin{equation}\label{defi-math-H-inverse}
{L} f(\xi) = \frac{1}{\xi^{d+1} \Lambda  Q(\xi)} \int_0^\xi f(\xi') \Lambda  Q (\xi')  (\xi')^{d+1} d\xi'.    
\end{equation}
We start now our induction argument.
For $i=0$, we have by assumption 
$$T_0 = c \Lambda Q(\xi). $$
%for some $c \neq 0$, will be fixed later.  
\iffalse
Let us recall the following asymptotic 
\begin{equation*}
\left\{   \begin{array}{rcl}
	\Lambda Q(\xi)   &=  & -2 + \sum_{i=1}^k a_i' \xi^{2i} + O(\xi^{2k+2}), \text{ for some } k \geq 1, \text{ as } \xi \to 0,  \\[0.2cm]
	\Lambda Q (\xi)  & = & a_0 \xi^{ -\gamma} + O(\xi^{-\gamma -g})  \text{ as } \xi \to +\infty,
\end{array} \right.
\end{equation*}
\fi
The asymptotic
\eqref{asymptotic-Lamda-Q} yields 
\begin{equation}\label{asymptotic-T-0}
T_0(\xi) = \left\{ \begin{array}{rcl}
	& & -2 c  +\sum_{i=1}^k a_i'' \xi^{2i} + O(\xi^{2k+2})  \text{ as }   \xi \to 0,  \\[0.2cm]
	& & a_0 c \xi^{-\gamma} + O\left( \xi^{-\gamma - 2\alpha}\right)         \text{ as }   \xi \to \infty. 
\end{array}  \right.
\end{equation}
%where  $\gamma$ and $g$  defined as in \eqref{defi-gamma-intro} and \eqref{defi-g-}, respectively. Now,  we   start the  induction.  
This proves \eqref{behavior-T-i} for the case $i=0$.
Now, let us suppose that \eqref{behavior-T-i} is true for some $k$  and let us prove that it holds true  for $k+1$.\\

-  At $\infty$: we have
\begin{eqnarray*}
\mathscr{L}(T_k) &=& 
%\frac{1}{ \xi^{d+1} \Lambda Q} \int_0^\xi   T_k (\Lambda Q) (\xi')^{d+1} d\xi',\\
\frac{C_k}{-2 \gamma +d+2+2k} \xi^{ - \gamma +1+2k} \left(1+O\left( \frac{\ln \xi}{\xi^2}\right) \right).
\end{eqnarray*}
\iffalse
Then, we have
\begin{eqnarray*}
\frac{\mathscr{L}(T_k)}{\Lambda Q} = \frac{C_k}{(-2 \gamma +d+2+2k) a_0} \xi^{1+ 2k} \left(1+O\left( \frac{\ln \xi}{\xi^2}\right) \right).
\end{eqnarray*}

This yields
\begin{eqnarray*}
H^{-1} (T_k) = \frac{C_k}{ 4 (k+1) ( \frac{d}{2} - \gamma +k+1)} \xi^{-\gamma+2(k+1)} \left(1+O\left( \frac{\ln \xi}{\xi^2}\right) \right).
\end{eqnarray*}
\fi
Hence
\begin{eqnarray*}
T_{k+1} =  H^{-1} (T_k) = \frac{C_k}{4 (k+1) ( \frac{d}{2} - \gamma +k+1)} \xi^{-\gamma+2(k+1)} \left(1+O\left( \frac{\ln \xi}{\xi^2}\right) \right)
\end{eqnarray*}
The above expansion yields
$$ C_{k+1} = \frac{C_k}{4 (k+1) ( \frac{d}{2} - \gamma +k+1) },$$
so that
%Thus, \eqref{behavior-T-i} holds for all $k =i+1 0$.  According to the principle of mathematical induction, the behaviors are hold for all $i \ge 0$.  In particular, we have 
\begin{eqnarray*}
C_{k} = \frac{c a_0 }{4^k k! (\frac{d}{2}  - \gamma)_k!}.
\end{eqnarray*}
Choosing  $ c = \frac{1}{a_0}$ concludes the first part of the proof.\\
In a similar fashion and upon using that $\partial_\xi^kT_{i+1}=H^{-1}T_i$, one can establish the mentioned result regarding the asymptotic behavior of the derivatives of $T_i$. We omit the details.
\section{ Inner eigenfunctions computation}\label{proof-appendix-inner-eigenfunctions}
Now we prove Proposition \ref{proposition-inner-eigen-functions}. 
%Let us consider   $ \phi_{i,int}  $ defined as in  \eqref{inner-eigen-functins-phi-i-int}. Now,  the 
Taking into account \eqref{inner-eigen-functins-phi-i-int}, the eigenvalue problem  \eqref{equa-eigen-phi-int} reads 
\begin{eqnarray*}
0 &=& \left\{  H - b \beta \Lambda -  2\beta b\left(  \frac{\alpha}{2} - i  + \tilde \lambda  \right)  \right\} \phi_{i,\text{int}} 
= \sum_{j=0}^i c_{i,j}(2\beta)^{j} b^j \left\{ H - b \beta \Lambda - 2\beta  b \left(\frac{\alpha}{2} -i +\tilde{\lambda }   \right)\right\} T_j \\
&+& \tilde \lambda \sum_{j=0}^i c_{i,j}(2\beta)^{j+1} b^{j+1} \left\{  H -  b \beta  \Lambda - 2\beta b \left(\frac{\alpha}{2} -i +\tilde{\lambda }   \right) \right\}T_{j+1} 
\\
&+& \tilde \lambda \sum_{j=0}^i b^{j+1} \left\{  H - b \beta  \Lambda - 2\beta b\left( \frac{\alpha}{2} -i  +\tilde \lambda  \right)  \right\} S_j\\
& +& b \left\{  H - \beta  b \Lambda - 2\beta b\left( \frac{\alpha}{2} -i +\tilde \lambda  \right)  \right\} R_i.
\end{eqnarray*}
Since $T_{i+1} = H^{-1} T_i$ and $H (T_0)  =0$, we get 
%Using  \eqref{recurrence-T-i+1-i} and 
\begin{eqnarray*}
& & \sum_{j=0}^i  c_{i,j} (2\beta)^j b^{j} \left\{H - \beta b \Lambda - 2\beta b\left( \frac{\alpha}{2} - i  + \tilde\lambda  \right)  \right\} T_j \\
&=&   \sum_{j=0}^i c_{i,j} (2  \beta)^j b^j \left\{ H T_j - \beta b ((2j -\alpha) T_j + \Theta_j) - 2 \beta b \left(  \frac{\alpha}{2}-i \right)T_j - 2\beta b\tilde \lambda T_j \right\}\\[0.1cm]
\iffalse
& = & \sum_{j=0}^{i-1} c_{i,j+1} (2\beta)^{j+1} b^{j+1} T_j + \sum_{j=0}^i c_{i,j} (2\beta)^{j+1}b^{j+1} (i-j) T_j -  \frac{1}{2 } \sum_{j=0}^i c_{i,j} (2\beta)^{j+1}b^{j+1} \Theta_j\\
& - & \tilde \lambda \sum_{j=0}^i c_{i,j} (2\beta)^{j+1} b^{j+1} T_j\\
& =& \sum_{j=0}^{i-1} (2\beta)^{j+1} b^{j+1} T_j \left\{ c_{i,j+1} +c_{i,j}(i-j) \right\} -\frac{1}{2 } \sum_{j=0}^i c_{i,j} (2 \beta)^{j+1} b^{j+1} \Theta_j \\
& - &\tilde \lambda \sum_{j=0}^i c_{i,j} (2\beta)^{j+1} b^{j+1} T_j\\
\fi
& =& - \frac{1}{2 } \sum_{j=0}^i c_{i,j} (2\beta)^{2j}b^{j+1} \Theta_j - \tilde \lambda \sum_{j=0}^i c_{i,j} (2\beta)^{j+1} b^{j+1} T_j  ,
\end{eqnarray*}
where we used the fact that $c_{i,j} $ defined in   \eqref{definition-c_i-j}  satisfies  $ c_{i,j+1} + c_{i,j} (i-j) =0$. Thus, the construction of $S_j$ and $R_i$  reduces the following equations (for all $j \leq i$)
\iffalse
\begin{eqnarray*}
&  & \tilde \lambda \sum_{j=0}^i b^{j+1} \left\{ H S_j  - 2\beta b \left(  \frac{\alpha}{2} - i + \tilde \lambda \right) \left(S_j + c_{i,j} (2\beta)^{j+1} T_{j+1} \right) - \beta b\Lambda \left( S_j + c_{i,j}(2\beta)^{j+1} T_{j+1} \right) \right\} \\
& & + b \left\{  H -  \beta b \Lambda  - 2\beta b\left(  \frac{\alpha}{2} -i +\tilde \lambda  \right)  \right\} R_i - \frac{1}{2 } \sum_{j=0}^i c_{i,j}(2\beta)^{j+1} b^{j+1} \Theta_j =0.
\end{eqnarray*}
\fi
\begin{eqnarray}
H S_j &=&  2 \beta b \left\{  \left( \frac{\alpha}{2} - i  + \tilde \lambda + \frac{1}{2} \Lambda \right) ( S_j + c_{i,j}(2\beta)^{j+1} T_{j+1})  \right\},\label{equa-H-S-j-b-Lambda-S-j}\\
H R_i  & = & 2\beta  b  \left(  \frac{\alpha}{2} -i +\tilde \lambda + \frac{1}{2}  \Lambda  \right) R_i  + \frac{1}{2} \sum_{j=0}^i c_{i,j}(2\beta)^{j+1} b^j \Theta_j.\label{equa-H-R-i-b-Lambda-R-i}
\end{eqnarray}
We note that  the construction of  $S_j$ and $R_i$ follows \cite{CMRJAMS20} (see also \cite{CGMNARXIV20-a}) which relies on the Banach fixed point theorem in the functional space  $X^a_{\xi_0}$ for some $ a \in \R$ and $ \xi_0 >0$ and where the norm is given by
\begin{equation}\label{defi-norm-X-a-y-0}
\| f\|_{X^a_{\xi_0}} = \sup_{\xi \in [0,\xi_0]} \sum_{i=0}^2   \frac{ \left| (  \xi \partial_\xi )^i f (\xi)  \right|  }{  \langle \xi \rangle^{a} },
\end{equation}
with $ \langle  \xi  \rangle = \sqrt{1+ \xi^2}  $.
For sake of shortness and since the determination of both $S_j$ and $R_i$ follows the same reasoning, we only consider $S_j$ in the sequel.\\ 

\medskip
\begin{center}
\textbf{ Step 1: Construction of $S_j$}
\end{center}
Identity \eqref{equa-H-S-j-b-Lambda-S-j} can be put in the form
\begin{equation}\label{equa-S-j-rought}
S_j =2 \beta b H^{-1} \left[  \left(   \frac{\alpha}{2} - i + \tilde \lambda  + \frac{1}{2} \Lambda     \right) \left( S_j + c_{i,j}(2\beta)^{j+1}  T_{j+1} \right) \right].
\end{equation}
Now, write $S_j$ as 
\begin{equation}\label{decompose-S_j}
S_j =  L(S_j) = L(0) + D L(S_j),
\end{equation}
where
\begin{eqnarray*}
L(0) &=&  b c_{i,j}(2\beta)^{j+2}  H^{-1} \left[   \left(   \frac{\alpha}{2} -i + \tilde \lambda + \frac{1}{2} \Lambda     \right) T_{j+1}  \right],\\[0.2cm]
DL(S_j)  &=& b (2\beta)  H^{-1} \left[    \left(   \frac{\alpha}{2} -i + \tilde \lambda + \frac{1}{2} \Lambda     \right)  S_j   \right].
\end{eqnarray*}
Our goal is to prove that for all $j \le i \le \ell$
\begin{eqnarray}
\| L(0)\|_{X^{2j-\gamma+2}_{\xi_0}}  & \leq &  C  y_0^2,\label{estimate-L(0)}\\
\| DL(S_j)\|_{ X^{2j - \gamma+2}_{\xi_0}}  & \leq &  C  y_0^2\| S_j\|_{ X^{2j-\gamma+2}_{\xi_0}}.\label{estimate-DL-S-j}
\end{eqnarray}
Once these estimates are established, we apply the  Banach fixed point theorem to  $L(S_j)$   mapping the ball $ B(0, 2 C   y_0^2)$ into itself with $ y_0 \leq \frac{1}{2\sqrt{C}}$. This yields the existence and uniqueness of   $S_j$ satisfying
%, solving   \eqref{equa-H-S-j-b-Lambda-S-j}with the following estimate
\begin{equation}\label{estimate-S-j-X-2j-}
\|S_j\|_{X^{2j-\gamma+2}_{\xi_0}} \leq   2 C  y_0^2. 
\end{equation}

We are now in position to prove  \eqref{estimate-L(0)} and \eqref{estimate-DL-S-j}:

\bigskip
- \textit{ Proof of   \eqref{estimate-L(0)}:} take  $a = 2j +2 - \gamma $, we get
\begin{eqnarray*}
\| L(0)\|_{X^{a}_{\xi_0}} \leq | c_{i,j}| b (2\beta)^{j+2} \left(    \left|  \frac{\alpha}{2} -i  +\tilde{\lambda}  \right| \|H^{-1} (T_{j+1})\|_{X^{a}_{\xi_0}}   + \frac{1}{2} \| H^{-1}(\Lambda T_{j+1})\|_{X^{a}_{\xi_0}}\right).
\end{eqnarray*}
Lemma \ref{lemma-contu-H-1-in-X-a-xi-0} yields
\begin{eqnarray*}
\| L(0)\|_{X^{a}_{\xi_0}} \leq C  \|T_{j+1}\|_{X^{a-2}_{\xi_0}} \le C  b \xi_0^2  \|T_{j+1}\|_{X^{a}_{\xi_0}}. 
\end{eqnarray*}
From lemma \ref{lemma-Generation-H} and   $X^a_{\xi_0}$'s definition, one infers 
$$ \|T_{j+1}\|_{X^{a}_{\xi_0}}  \leq C .$$
Plugging  $ \xi_0 = \frac{y_0}{\sqrt{b}}$, we get 
$$\| L(0)\|_{X^{a}_{\xi_0}} \leq C   y_0^2,$$
which concludes \eqref{estimate-L(0)}.

- \textit{ Proof of \eqref{estimate-DL-S-j}:}  we argue as in the proof of \eqref{estimate-L(0)}. Indeed, we apply Lemma \ref{lemma-contu-H-1-in-X-a-xi-0} so that
\begin{eqnarray*}
\| DL(S_j) \|_{X^{a}_{\xi_0}}  & \leq &  C b (2\beta) ( \|  H^{-1}(S_j)  \|_{X^{a}_{\xi_0}} +  \|  H^{-1}(\Lambda S_j)  \|_{X^{a}_{\xi_0}} )\\
& \leq & C y_0^2 \| S_j\|_{X^a_{\xi_0}},
\end{eqnarray*}
We conclude \eqref{estimate-S-j-X-2j-} as above.\\
Our task now is to establish the desired estimates for $\partial_{\tilde{ \lambda}}  S_j$, $\partial_b S_j  $ and $\partial_\beta S_j$. Since the proofs are quite the same we only estimate  $\partial_b S_j  $. Apply $\partial_b $ to  both sides of \eqref{decompose-S_j} to get 

%\medskip
%- \textit{  For  $\partial_b S_j  $:}  Taking $\partial_b $ to  equality \eqref{decompose-S_j}, we get
\begin{eqnarray*}
\partial_b S_j &=&   c_{i,j}(2\beta)^{j+2} H^{-1} \left[\left( \frac{\alpha}{2} - i  + \tilde \lambda +\frac{1}{2} \Lambda \right) T_{j+1} \right]  + (2\beta) H^{-1} \left[\left( \frac{\alpha}{2} -i + \tilde \lambda +\frac{1}{2} \Lambda \right) S_{j} \right] \\
& + & b (2\beta) H^{-1} \left[\left( \frac{\alpha}{2} -i  + \tilde \lambda + \frac{1}{2} \Lambda \right) \partial_b S_{j} \right]. 
\end{eqnarray*}
Using Lemma \ref{lemma-contu-H-1-in-X-a-xi-0} with   $\tilde a = 2j +4 - \gamma$,    we derive 
\begin{eqnarray*}
\|\partial_b S_j \|_{X^{\tilde a}_{\xi_0}}  & \leq & C\left(  \| T_{j+1}\|_{ X^{\tilde a-2}_{\xi_0}} + \| S_j\|_{X^{\tilde a-2}_{\xi_0}} + b \| \partial_b S_j \|_{X^{\tilde a-2}_{\xi_0}}  \right)\\
& \le & C(1 + y_0^2) + C b \xi_0^2 \|\partial_b S_j\|_{X^{\tilde a}_{\xi_0}}
\leq    C ( 1 + y_0^2)  + C y_0^2 \| \partial_b S_j \|_{X^a_{\xi_0}} ),
\end{eqnarray*}
this implies that $\|\partial_b S_j \|_{X^{2j+4-\gamma}_{\xi_0}}  \leq  C, $
provided $ y_0 \leq y_0^* $ is small enough.
\iffalse
\medskip
\textit{ For $\partial_{\tilde \lambda} S_j $:} It is quite similar to the estimate  for $\partial_b S_j$. Indeed, we apply  $ \partial_{\tilde \lambda} $ to \eqref{decompose-S_j}, then, we derive
\begin{eqnarray*}
\partial_{\tilde \lambda}  S_j & = & b (2\beta)^{j+2} c_{i,j} H^{-1} \left[  T_{j+1} \right]  + b (2\beta) H^{-1} \left[ \left( \frac{\alpha}{2} -i+ \tilde{\lambda}  +\frac{1}{2} \Lambda \right) \partial_{\tilde \lambda } S_j \right]  +  b (2\beta) H^{-1} \left[ S_j \right].
\end{eqnarray*}
Taking $ X^{a}_{\xi_0} $ norm of both sides of the equality and  using Lemma \ref{lemma-contu-H-1-in-X-a-xi-0} with $a =2j+2-\gamma$,   we deduce
\begin{eqnarray*}
\|\partial_{\tilde{ \lambda}} S_j\|_{X^{a}_{\xi_0}}  & \leq & C  b    \left( \|T_{j+1} \|_{X^{a-2}_{\xi_0}} +  \|S_j\|_{X^{a-2}_{\xi_0}}  +  \|\partial_{\tilde \lambda } S_j\|_{X^{a-2}_{\xi_0}}  \right)\\
& \leq  & C  b\xi_0^2   \left( \| T_{j+1}  \|_{X^a_{\xi_0}}  +  \|S_j  \|_{X^a_{\xi_0}}   + \| \partial_{\tilde{\lambda}} S_j \|_{X^a_{\xi_0}} \right)\\
& \leq  & C  y_0^2 \left( 1 + \| \partial_{\tilde{\lambda}} S_j\|_{X^a_{\xi_0}} \right) \le C y_0^2.  
\end{eqnarray*}
provided that $ y_0 \leq y_0^* $ small enough.
\fi

\noindent 
\medskip
+ For $\partial_\beta S_j$: Applying $\partial_\beta $ to \eqref{decompose-S_j}, we get
\begin{eqnarray*}
\partial_{\beta}  S_j & = & b 2(j+2) (2\beta)^{j+1} c_{i,j} H^{-1} \left[ \left( \frac{\alpha}{2} -i +\tilde{\lambda}  +\frac{1}{2} \Lambda \right) T_{j+1} \right]  + 2b H^{-1} \left[ \left( \frac{\alpha}{2} -i+\tilde{\lambda}   +\frac{1}{2} \Lambda \right) S_j \right]\\
& + & b (2\beta) H^{-1} \left[ \left( \frac{\alpha}{2} -i  + \tilde{\lambda}  + \frac{1}{2} \Lambda \right)  \partial_{\beta } S_j \right].
\end{eqnarray*}
using the boundedness of $j$, we get (with $a =2j+2-\gamma$)
\begin{eqnarray*}
\|\partial_{\beta} S_j\|_{X^{a}_{\xi_0}}  & \leq & C  b    \left( \|T_{j+1} \|_{X^{a-2}_{\xi_0}} +  \|S_j\|_{X^{a-2}_{\xi_0}}  +  \|\partial_{\beta} S_j\|_{X^{a-2}_{\xi_0}}  \right)\\
& \leq  & C  b\xi_0^2   \left( \| T_{j+1}  \|_{X^a_{\xi_0}}  +  \|S_j  \|_{X^a_{\xi_0}}   + \| \partial_{\beta} S_j \|_{X^a_{\xi_0}} \right)\\
& \leq  & C  y_0^2 \left( 1 + \| \partial_{\beta} S_j\|_{X^a_{\xi_0}} \right) \le C y_0^2.  
\end{eqnarray*}
provided that $ y_0 \leq y_0^* $ small enough.

\begin{center}
\textbf{ Step 2: construction   of $R_i$}
\end{center} 
Taking $H^{-1}$ to   \eqref{equa-H-R-i-b-Lambda-R-i},  we get
\begin{eqnarray}
R_i &=& b (2\beta)  H^{-1} \left[  \left(  \frac{\alpha}{2}-i +\tilde{ \lambda} + \frac{1}{2} \Lambda \right) R_i   \right]  + \frac{1}{2} \sum_{j=0}^i c_{i,j} (2\beta)^{j+1} b^j  H^{-1}(  \Theta_j)\label{mapping-J-R-i}
\\& = & J (R_i).   \nonumber
\end{eqnarray}
Let us consider  $a = -\gamma + \epsilon > -\gamma$, then, we apply Lemma   \ref{lemma-contu-H-1-in-X-a-xi-0}
\begin{eqnarray*}
\| J(R_i)\|_{X^a_{\xi_0}} & \leq & C \left(  b (2\beta) \| H^{-1}  R_i \|_{X^{a}_{\xi_0}}   + \sum_{j=0}^i (2\beta)^{j+1} b^j \| H^{-1} (\Theta_j)\|_{X^{a}_{\xi_0}} \right)\\
& \leq  & C \left(  b \|   R_i \|_{X^{a-2}_{\xi_0}}   + \sum_{j=0}^i b^j \|  \Theta_j \|_{X^{a-2}_{\xi_0}} \right).
\end{eqnarray*}
Now, we derive from \eqref{asymptotic-Theta} that
\begin{eqnarray*}
\| \Theta_j \|_{X^{-2-\gamma +\epsilon}_{\xi_0}} \leq  C(\epsilon) \xi^{2j}_{0}.
\end{eqnarray*}
Thus, we derive
\begin{eqnarray}
\|J(R_i)\|_{X^{a}_{\xi_0}}  \leq C  y_0^2 \|R_i\|_{X^a_{\xi_0}} + C(\epsilon).
\end{eqnarray}
Taking  $ y_0  $ small enough,  $J $ maps   the ball $ B(0, 2 C(\epsilon))  $ into itself. In addition to that,  it is similar   to prove $J$ is a contraction. Hence, by using  Banach fixed point theorem, we imply   the existence and the uniqueness  of $R_i$ satisfying
$$ \| R_i\|_{X^{-\gamma +\epsilon}_{\xi_0}} \leq 2 C( \epsilon).$$
Similarly for $S_j$,  we  can  respectively take $\partial_b$, $ \partial_{\tilde{\lambda}}$, and $\partial_\beta$  to  \eqref{mapping-J-R-i} by using $R_i \in B(0, 2 C(\epsilon))  $,   and we  get
\begin{eqnarray*}
\|\partial_b R_i\|_{ X^{-\gamma+2+\epsilon}_{\xi_0}}  & \leq &   C( \epsilon),\\
\|\partial_{\tilde{\lambda}} R_i\|_{ X^{-\gamma+2+\epsilon}_{\xi_0}}  & \leq & C( \epsilon) b,\\
\|\partial_{\beta} R_i\|_{ X^{-\gamma+2+\epsilon}_{\xi_0}}  &\le & C(\epsilon),
\end{eqnarray*}
where the constant $C(\epsilon)$ is universal. Finally,    we  conclude the proof of the Proposition \ref{proposition-inner-eigen-functions} . $\square$
\begin{eqnarray*}
\partial_b R_i&=&(2\beta)H^{-1}\left[\left( \frac{\alpha}{2} -i +\tilde{\lambda}  +\frac{1}{2} \Lambda \right)R_i \right
]+ b(2\beta)H^{-1}\left[\left( \frac{\alpha}{2} -i +\tilde{\lambda}  +\frac{1}{2} \Lambda \right)\partial_b R_i \right
]\\
&+&\frac{1}{2}\sum_{j=0}^{i} c_{i,j}j(2\beta)^{j+1} b^{j-1}H^{-1} \Theta_j,
\end{eqnarray*}
and
\begin{eqnarray*}
\partial_\beta R_i&=&2b H^{-1}\left[\left( \frac{\alpha}{2} -i +\tilde{\lambda}  +\frac{1}{2} \Lambda \right)R_i \right
]+ b(2\beta)H^{-1}\left[\left( \frac{\alpha}{2} -i +\tilde{\lambda}  +\frac{1}{2} \Lambda \right)\partial_\beta R_i \right
]\\
&+&\sum_{j=0}^i c_{i,j}(j+1)(2\beta)^{j} b^{j} H^{-1}\Theta_j,
\end{eqnarray*}
and
\begin{eqnarray*}
\partial_{\tilde{\lambda}} R_i&=& (2\beta) b H^{-1} R_i+ b(2\beta)H^{-1}\left[\left( \frac{\alpha}{2} -i +\tilde{\lambda}  +\frac{1}{2} \Lambda \right)\partial_{\tilde{\lambda}} R_i \right
].
\end{eqnarray*}

\noindent
\medskip
The rest of this part is devoted to the results which are used to complete the proof of Proposition \ref{proposition-inner-eigen-functions}. 
\begin{lemma}[Continuity  of $H^{-1}$  in $X^a_{\xi_0}$]\label{lemma-contu-H-1-in-X-a-xi-0} For all $ a  \geq -\gamma$, we have
\begin{equation}\label{defi-H-1f}
	\| H^{-1} f\|_{X^a_{\xi_0}} \leq  C \sup_{\xi \in [0,\xi_0]} \langle \xi  \rangle^{ 2 -a} |f(\xi)|,
\end{equation}
and
\begin{equation}\label{defi-H-1-Lambda}
	\| H^{-1} \left( \Lambda f  \right)\|_{X^a_{\xi_0}} \leq C(a) \sup_{\xi \in [0,\xi_0]}  \left[ \langle \xi  \rangle^{ 2 - a} |f(\xi)|+\langle \xi \rangle^{3-a} | \partial_\xi f (\xi)|\right].
\end{equation}
In particular
\begin{equation}
	\| H^{-1} f\|_{X^a_{\xi_0}} \leq C   \|f\|_{X^{a-2}_{\xi_0}},
\end{equation}
and
\begin{equation}
	\| H^{-1} \left( \Lambda f  \right)\|_{X^a_{\xi_0}} \leq C   \|f\|_{X^{a-2}_{\xi_0}}.
\end{equation}
Using the fact that, for $ \xi_0 \geq  1$, we have
$$  \|f\|_{X^{a-2}_{\xi_0}}  \leq  2 \xi_0^2 \|f\|_{X^{a}_{\xi_0}}, $$
Formulae \eqref{defi-H-1f} and \eqref{defi-H-1-Lambda} read, for $ \xi_0 \geq  1$
\begin{equation}
	\| H^{-1} f\|_{X^a_{\xi_0}} \leq C  \xi_0^2 \|f\|_{X^{a}_{\xi_0}},
\end{equation}
and
\begin{equation}
	\| H^{-1} \left( \Lambda f  \right)\|_{X^a_{\xi_0}} \leq C   \xi_0^2 \|f\|_{X^{a}_{\xi_0}}.
\end{equation}
\end{lemma}
\begin{proof}
Regarding \eqref{defi-norm-X-a-y-0}, we have
%=   \langle \xi \rangle^{-a} | g| + \langle \xi \rangle^{1-a} | \partial_\xi g| + \left| \langle \xi \rangle^{2-a} \partial_\xi^2 g + \langle \xi \rangle^{-a}  \xi \partial_\xi g \right|\\[0.2cm]
\begin{eqnarray*}
	\sum_{j=0}^2  \frac{ \left| ( \xi \partial_\xi )^i g(\xi) \right|}{ \langle \xi \rangle^{a} }
	& \leq &  \langle \xi \rangle^{-a} | g(\xi)|  + 2 \langle \xi \rangle^{ 1 - a} | \partial_\xi g | + \langle \xi \rangle^{2 - a} | \partial_\xi^2 g|.
\end{eqnarray*}
In addition to that, we derive from \eqref{defi-H-1-f}
that   
\begin{eqnarray*}
	\partial_\xi \left( H^{-1} f \right)  &=&  \partial_\xi \Lambda Q \int_0^\xi \frac{\mathscr{L} f}{ \Lambda Q }(\xi') d\xi'  + \mathscr{L}(f)(\xi),\\[0.2cm]
	\partial_\xi^2 \left( H^{-1} f \right)  &= & \partial_\xi^2   \Lambda Q \int_0^\xi \frac{\mathscr{L} f}{ \Lambda Q }(\xi') d\xi'   + \frac{ \partial_\xi \Lambda Q }{\Lambda Q}\mathscr{L}(f) + \partial_\xi ( \mathscr{L}(f)),
\end{eqnarray*}
where  $ \mathscr{L}$ defined as in \eqref{defi-math-H-inverse}. We remark that \eqref{defi-H-1f} follows from:    for all $\xi \in [0,\xi_0]:$
\begin{eqnarray}
	\langle \xi \rangle^{-a} \left| H^{-1} f (\xi) \right| & \leq &  \frac{C}{|a|} \sup_{\xi \in [0,\xi_0]} \xi^{2-a} (1 -\xi^\gamma) |f(\xi)|,\label{esti-aH-f-C-2-2-f-1}\\[0.2cm]
	\langle  \xi \rangle^{1-a}  \left|  \partial_\xi (H^{-1} f(\xi)) \right|
	& \leq & \frac{C}{|a|} \sup_{\xi \in [0,\xi_0]} \xi^{2-a} (1 -\xi^\gamma) |f(\xi)|,\label{esti-aH-f-C-2-2-f-2}\\[0.2cm]
	\langle  \xi \rangle^{2-a}  \left|  \partial_\xi^2 (H^{-1} f(\xi)) \right|
	& \leq & \frac{C}{|a|} \sup_{\xi \in [0,\xi_0]} \xi^{2-a} (1 -\xi^\gamma) |f(\xi)|,\label{esti-aH-f-C-2-2-f-3}
\end{eqnarray}
where $C$ does not depend on $\xi_0$. Let us start with the proof of these estimates:

\medskip
- \textit{ The proof of \eqref{esti-aH-f-C-2-2-f-1}:  }   From   $\mathscr{L}$'s formula in \eqref{defi-math-H-inverse},  we  have 
\begin{eqnarray*}
	| \mathscr{L} (f) | & \leq &  \frac{1}{
		|\xi|^{d+1} |\Lambda Q|} \int_0^\xi  \langle  \xi'\rangle^{2-a}| f(\xi')| \langle \xi' \rangle^{a-2}   | \Lambda Q(\xi') |(\xi')^{d+1} d\xi'\\
	& \leq  & \sup_{\xi \in [0,\xi_0]} \left\{ \langle \xi \rangle^{2-a } |f(\xi)| \right\}\frac{1}{
		|\xi|^{d+1}| \Lambda Q|} \int_0^\xi \langle \xi' \rangle^{a-2}   | \Lambda Q(\xi') |(\xi')^{d+1} d\xi'\\
	&=&\sup_{\xi \in [0,\xi_0]} \left\{ \langle \xi \rangle^{2-a } |f(\xi)| \right\} \tilde{L}(\xi).
\end{eqnarray*}
Plugging this estimate to $H^{-1} f$, we obtain
\begin{eqnarray*}
	\left| H^{-1} f(\xi) \right|\leq \left[   \sup_{\xi \in [0,\xi_0]} \langle \xi \rangle^{2-a} |f(\xi)|  \right]   | \Lambda Q| \int_0^\xi  \frac{\tilde{L}(\xi')}{|\Lambda Q|} d\xi'.
\end{eqnarray*}  
Hence,  it is sufficient to prove
$$ | \Lambda Q| \int_0^\xi  \frac{\tilde{L}(\xi')}{|\Lambda Q|} d\xi' \leq C \langle \xi \rangle^{a},$$
where $C$ does not depend on $\xi_0$. Indeed,  we  consider two cases where $  \xi_0 \ll 1 $ and $\xi_0 \gg 1$:

+ The case $  \xi_0 \ll 1 $: We have the following for all $\xi \in [0,\xi_0]$
\begin{eqnarray*}
	\frac{1}{C} \leq |\Lambda  Q| \leq C,
\end{eqnarray*}
and 
\begin{equation*}
	\tilde{L}(\xi)=\frac{1}{ |\xi|^{d+1} |\Lambda Q|} \int_0^\xi  \langle \xi' \rangle^{a-2} |\Lambda Q| (\xi')^{d+1} d\xi' \leq C\xi,
\end{equation*}
this yields 
\begin{eqnarray*}
	| \Lambda Q| \int_0^\xi  \frac{\tilde{L}(\xi')}{|\Lambda Q|} d\xi'  \leq C\xi^2\leq 2C.
\end{eqnarray*}
which concludes the case $\xi_0 \ll 1 $.

+ The case $ \xi_0 \gg 1$:  We observe that there exists $M >0$ such that for all $ \xi \in [M,\xi_0]$
$$ \frac{1}{C}  \xi^{-\gamma} \leq  |\Lambda Q| \leq C \xi^{-\gamma}.$$
Then, we have
\begin{eqnarray*}
	\int_0^\xi  \frac{\tilde{L}(\xi')}{|\Lambda Q|} d\xi'   &=&\int_0^M \frac{\tilde{L}(\xi')}{|\Lambda Q|} d\xi'   + \int_M^\xi \frac{\tilde{L}(\xi')}{|\Lambda Q|} d\xi' \\
	& \leq & C(M) +C  \int_M^\xi  \tilde{L}(\xi') (\xi')^{\gamma} d\xi'.
\end{eqnarray*}
Besides that, we estimate  $\tilde{L}(\xi') ,$ for all $\xi' \in [ M, \xi_0]$ as follows 
\begin{eqnarray*}
	\tilde{L}(\xi) & = &  \frac{1}{\xi^{d+1} \Lambda Q} \left( \int_0^M\langle \xi' \rangle^{a-2} |\Lambda Q | (\xi')^{d+1} d\xi'  + \int_M^\xi \langle \xi' \rangle^{a-2} |\Lambda Q | (\xi')^{d+1} d\xi'  \right)\\
	& \leq  & C(M) \left( \xi^{-d-1+\gamma} + \xi^{a-1}  \right),
\end{eqnarray*}
it follows that
\begin{eqnarray*}
	\int_M^\xi \tilde{L}(\xi') (\xi')^{\gamma } d\xi'  \leq C (M) ( 1+ \xi^{-d+2\gamma} + \xi^{a+\gamma}).
\end{eqnarray*}
Thus, we derive
\begin{eqnarray*}
	| \Lambda Q| \int_0^\xi  \frac{\tilde{L}(\xi')}{|\Lambda Q|} d\xi'  \leq C(M,a) \left( \xi^{-\gamma} + \xi^{a } \right).
\end{eqnarray*}
Finally, we have
$$  |H^{-1} f(\xi)| \leq  C(a)  \langle \xi \rangle^a ,  $$
provided that
$$  a > -\gamma  .$$

- The  proofs of \eqref{esti-aH-f-C-2-2-f-2} and \eqref{esti-aH-f-C-2-2-f-3} are the same.

\medskip
Now, it remains to prove that if $f \in X^{a}_{\xi_0}$, then we have
$$  \| f \|_{X^{a-2}_{\xi_0}} \leq C (a) \xi_0^2 \| f\|_{X^a_{\xi_0}}  ,$$
provided that $\xi_0 \geq 1$. Indeed, this comes from 
$$ \langle \xi  \rangle ^{2 - a} = \langle \xi  \rangle ^{-a} \langle \xi  \rangle ^{2}  \leq  2 \xi_0^2  \langle \xi  \rangle ^{-a}  ,$$
provided that $\xi_0 \geq 1$.  Thus, we have

\begin{eqnarray*}
	\|f\|_{X^{a-2}_{\xi_0}}  = \sup_{\xi \in [0,\xi_0]} \sum_{j=0}^2  \left| \frac{(\langle \xi \rangle \partial_\xi )^i f}{\langle \xi \rangle^{a-2}} \right| \leq \sup_{\xi \in [0,\xi_0]} \sum_{j=0}^2 2\xi_0^2 \left| \frac{(\langle \xi \rangle \partial_\xi )^i f}{\langle \xi \rangle^{a}} \right| \leq 2\xi_0^2 \|f\|_{X^{a}_{\xi_0}}.
\end{eqnarray*}
This concludes the proof of the Lemma. 
\end{proof}

In the following Lemma, we aim to estimate $ \partial_\xi^3 S_j$ and $\partial_\xi^3 R_i$:
\begin{lemma}[Higher estimates for $\partial_\xi^3 S_j$ and $\partial_\xi^3  R_i$] Let us consider  $S_j$ and $R_i$ which satisfying  \eqref{equa-S-j-rought} and \eqref{mapping-J-R-i}.  Furthermore, we assume  that the following estimates hold
\begin{eqnarray*}
	\| S_j \|_{X^{2j+2-\gamma}_{\xi_0}}  \le C y_0^2  \text{ and } \|R_i\|_{X^{\epsilon - \gamma}_{\xi_0}} \le C, \text{ with } \xi_0 = \frac{y_0}{\sqrt{b}}.
\end{eqnarray*}
Then, the following holds: for all $\xi \in [0,\xi_0]$
\begin{eqnarray}
	\left| \partial_\xi^3 S_j ( \xi)   \right| \le C  \langle \xi \rangle^{2j-\gamma -1},\label{estimate-partial-xi-3-S-j}\\
	\left| \partial_\xi^3 R_i (\xi) \right| \le C  \langle \xi \rangle^{\epsilon -\gamma -3}. \label{estimate-partial-xi-3-R-i}
\end{eqnarray}
\end{lemma}
\begin{proof}   

- The proof of  \eqref{estimate-partial-xi-3-S-j}:  We  remark that  when   $ b \to 0 $, $\xi_0  \to +\infty$. Then, we will consider two  situations, namely, $  \xi   \ll  1 $ and $ \xi \gg 1$ . Recall the   inverse  formula  
\begin{eqnarray*}
	S_j = bH^{-1} ( f ),  f  =  \left( \frac{\alpha}{2}  -i +\tilde \lambda  +\frac{1}{2} \Lambda  \right) (S_j + c_{i,j} T_{j+1})
\end{eqnarray*}
and write $ \partial_\xi^3 S_j $ as follows
\begin{eqnarray*}
	b^{-1} \partial_\xi^3 S_j (\xi )  &=&   \partial_\xi^3   \Lambda Q  \int_0^\xi \frac{\mathcal{L}(f)}{\Lambda Q } d\xi' +  \frac{ \partial_\xi^2   \Lambda Q }{ \Lambda Q}   \mathcal{L}(f) +\partial_\xi f     - \frac{(d+1)(d+2)}{\xi^2} \mathcal{L}(f)  - \frac{(d+1)\partial_\xi \Lambda Q}{\xi \Lambda Q} \mathcal{ L} (f)\\
	&  - & \frac{(d+1)}{\xi} f.
\end{eqnarray*}
+  The case $\xi \in [0,1]$:   We observe that   when $\xi \le 1$, it follows that
$  |f (\xi)| \le  C,    $
then, plugging to \eqref{equa-S-j-rought}, we obtain
$$   |S_j(\xi)  | = |H^{-1} f| \le C\xi^2, $$
Hence, we refine the behavior near  $ 0$  as follows
\begin{eqnarray*}
	| f (\xi)| \le  C \xi^2,\\
	\left| \mathcal{L}(f) (\xi) \right| \le C  \xi^3, 
\end{eqnarray*}
continuing this process we enhance the behavior to
\begin{eqnarray*}
	| f (\xi)| \le  C \xi^{2j+2},\\
	\left| \mathcal{L}(f) (\xi) \right| \le C  \xi^{2j+3}, 
\end{eqnarray*}
(we can  get a precise behavior for $S_j$  at $0$ by $S_j (\xi) = O(\xi^{2j+4})$  ). 
Then, it is easy to derive 
$$ b^{-1} \partial^3_\xi S_j (\xi) =O ( \xi^{2j+1} ) \text{ as } \xi \to 0.  $$

+ The case  $\xi \ge 1$:  we have the following fact
\begin{eqnarray*}
	\left|  S_j(\xi ) \right|  \le   C \langle \xi \rangle^{2j+2-\gamma} \text{ and }
	\left| T_{j+1} (\xi)  \right|  \le  C\langle \xi \rangle^{2j+2-\gamma},    
\end{eqnarray*}
which yields
$$ |f (\xi)| \le C \langle \xi \rangle^{2j+2-\gamma}.  $$
Since
\begin{eqnarray*}
	\left| \mathcal{L}(f) (\xi)  \right| \le C\langle \xi \rangle^{2j+3 -\gamma}, 
	\text{ and }
	\left|  S_j(\xi ) \right| &  \le &  C \langle \xi \rangle^{2j+2-\gamma},
\end{eqnarray*}
we get 
\begin{eqnarray*}
	\left|\partial_{\xi}^3  S_j(\xi ) \right| &  \le &  C b \langle \xi \rangle^{2j+1-\gamma} \le  C  \langle \xi\rangle^{ 2j -\gamma-1 },
\end{eqnarray*}
due to the fact that 
$$ b \xi^2 = y_0 \le 1.$$
Thus, we conclude the proof of \eqref{estimate-partial-xi-3-S-j}. By the same technique, we derive \eqref{estimate-partial-xi-3-R-i}. This finalizes the proof of the Lemma.
\end{proof}
\section{ Outer eigenfunctions construction }\label{proof-propo-outer-eigenfunctions}

This paragraph is devoted to  give the complete proof to Proposition \ref{propo-outer-eigenfunctions}:
\begin{proof}
Now, let   $\phi_{i,out, \beta}$ be of the form
\begin{equation}\label{form-phi-i-out-beta}
	\phi_{i,out, \beta}  (y )  = \phi_{i,\infty,\beta}  (y )  + \tilde{ \lambda} (  \tilde{ \phi}_{i, \beta}  (y) +   R_{i,1}( y)  ) +  R_{i,2}(y),
\end{equation}
where $R_{i,1}$ and $R_{i,2}$ are to be constructed. Rewrite  \eqref{defi-mathscr-L-b-radial}  as  
$$ \mathscr{L}_b = \mathscr{L}_\infty^\beta   - 3(d-2) \left( \frac{1}{y^2} + 2 Q_b + Q_b^2 y^2     \right), $$
where $\mathscr{L}_\infty^\beta$    was defined in \eqref{defi-operator-L-infty}
and  $\mathscr{L}_{i, ext}^{\beta} = \mathscr{L}^\beta_\infty - 2\beta \left( \frac{\alpha}{2}- i\right)$.  Plugging \eqref{form-phi-i-out-beta} into    
$$ \left[ \mathscr{L}_{b} - 2 \beta \left( \frac{\alpha}{2} -i  \right) -\tilde \lambda  \right]( \phi_{i,out, \beta}) =0,$$
which yields 
\begin{eqnarray*}
	& & \mathscr{L}_{i, ext}^{\beta} \phi_{i,out, \beta} - \tilde{\lambda } \phi_{i,out, \beta}  - 3 (d-2)  \left( \frac{1}{y^2} + 2 Q_b + Q_b^2 y^2     \right) \phi_{i,out, \beta}
	=  \left(  \mathscr{L}_{i, ext}^{\beta} R_{i,1}   -\tilde{\lambda} (\tilde{\phi}_{i, \beta} + R_{i,1}) \right) \\
	& + & \left(  \mathscr{L}_{i, ext}^{\beta}  R_{i,2} -   3 (d-2)  \left( \frac{1}{y^2} + 2 Q_b + Q_b^2 y^2     \right) \phi_{i,out} -\tilde{\lambda } R_{i,2}  \right) =0,
\end{eqnarray*}
where we used $\mathscr{L}^\beta_{i,ext} \phi_{i,\infty} =0$ as well as $  \mathscr{L}^\beta_{i, ext}(\tilde \phi_{i,\beta}) = \phi_{i, \infty, 
	\beta}$.       
Thus, it is sufficient to   construct $R_{i,1}$ and $R_{i,2}$ satisfying 
\begin{eqnarray*}
	\mathscr{L}_{i, ext}^{\beta}( R_{i,1} )  & = &  \tilde{\lambda} R_{i,1}   + \tilde{\lambda} \tilde{\phi}_{i,\beta} \\[0.2cm]
	\mathscr{L}_{i, ext}^{\beta}(R_{i,2})  &=& \left(  \tilde{\lambda}  +3 (d-2)  \left( \frac{1}{y^2} + 2 Q_b + Q_b^2 y^2     \right) \right) R_{i,2} \\
	&+&   3 (d-2)  \left( \frac{1}{y^2} + 2 Q_b + Q_b^2 y^2     \right) (   \phi_{i,\infty,\beta} + \tilde{ \lambda} (\tilde{\phi}_{i,\beta } + R_{i.1})),
\end{eqnarray*}
or equivalently
%$\mathscr{L}_{i,ext}^{-1} $, we derive
\begin{eqnarray}
	R_{i,1} &=& \tilde{\lambda} \mathscr{L}^{-1}_{i,ext} (R_{i,1})  + \tilde{\lambda} \phi_{i,\infty,\beta}, \label{equa-R-i-1}\\
	R_{i,2}  &=& \mathscr{L}_{i,ext}^{-1} (H_1 R_{i,2}) + \mathscr{L}_{i,ext}^{-1} (H_2)\label{equa-R-i-2}.
\end{eqnarray}
Here 
\begin{eqnarray*}
	H_1 (y)  &=&  \left(  \tilde{\lambda}  +3 (d-2)  \left( \frac{1}{y^2} + 2 Q_b + Q_b^2 y^2     \right) \right) \\
	H_2 (y) & = &  3 (d-2)  \left( \frac{1}{y^2} + 2 Q_b + Q_b^2 y^2     \right)   ( \phi_{i,\infty,\beta} + \tilde{ \lambda} (\tilde{\phi}_{i,\beta} + R_{i,1}) ).
\end{eqnarray*}

\textit{ Step 1: Construction of $ R_{i,1}$:}
The construction is based on  Banach fixed point theorem on Banach space $X^{a,a'}_{y_0}$ equipped with the norm introduced in \eqref{defi-norm-X-y-0-a-a'}.   Now, denote the right hand side of \eqref{equa-R-i-1}  by $ K (R_{i,1})$, and apply Lemma \ref{continnuity-L-i-ext--1} with  $ a =  {\gamma}  -d$ and $a' = 2i+2-\gamma$ 
$$  \| K(R_{i,1}) \|_{X^{a,a'}_{y_0}} \leq   C_1(a,a')|\tilde{\lambda}| \left( \|R_{i,1}\|_{X^{a,a'}_{y_0}}   +    \|\phi_{i,\infty,\beta}\|_{X^{a,a'}_{y_0}}  \right) \text{ and }   \| \phi_{i,\infty,\beta}\|_{ X^{a,a'}_{y_0}} \leq    C_2 . $$
Then, $K$ maps  the ball $ B(0, 2 C_1  C_2 |\tilde{\lambda}|)$ into itself   provided that   $| \tilde{\lambda}| \leq  \min \left( \frac{1}{2}, \frac{1}{2C_1}\right)$. In addition to that,  for all $ X_{1}, X_2 \in B(0, 2 C_1  C_2 |\tilde{\lambda}|)$, we have
\begin{eqnarray*}
	\left\| K(X_1) - K(X_2) \right\|_{X^{a,a'}_{y_0}}  \lesssim |\tilde \lambda | \left\| X_1 - X_2 \right\|_{X^{a,a'}_{y_0}}.
\end{eqnarray*}
Hence, for $
\tilde \lambda $ small enough, $K$ is a contraction and  the existence of $  R_{i,1}    $  follows with the bound
$$ \|   R_{i,1}  \|_{X^{a,a'}_{y_0}}  \leq   2C_1 C_2 |\tilde{\lambda}|.  $$ 
Now, we establish the estimates for $\partial_b R_{i,1}, \partial_b R_{i,1} $ and $\partial_\beta R_{i,1} $. 

\noindent
\medskip
- \textit{ For $ \partial_b R_{i,1}$:} From  \eqref{equa-R-i-1},  we see that $R_{i,1}$  is independent of  $  b$ so $ \partial_b R_{i,1}  =0.$

\noindent
\medskip
-  \textit{ For $ \partial_{\tilde \lambda}  R_{i,1}$:}
Applying   $ \partial_{\tilde{ \lambda}}  $ to  \eqref{equa-R-i-1},  we obtain
$$ \partial_{\tilde{\lambda} } R_{i,1}  =  \left(\mathscr{L}_{i,ext}^\beta\right)^{-1} (R_{i,1}) + \tilde{\lambda} \left(\mathscr{L}_{i,ext}^\beta\right)^{-1}( \partial_{\tilde{\lambda}} R_{i,1} ) + \phi_{i,\infty, \beta}.  $$
Lemma   \ref{continnuity-L-i-ext--1} implies 
$$   \|  \partial_{\tilde{\lambda} } R_{i,1}  \|_{X^{a,a'}_{y_0}}    \leq   C_1 \|  R_{i,1} \|_{X^{a,a'}_{y_0}}  + C_1 |\tilde{\lambda}| \|\partial_{\tilde{\lambda}} R_{i,1}\|_{X^{a,a'}_{y_0}}  +\|\phi_{i, \infty, \beta}\|_{X^{a, a'}_{y_0}},  $$
hence
$$    \|  \partial_{\tilde{\lambda} } R_{i,1}  \|_{X^{a,a'}_{y_0}} \leq C. $$

\noindent
\medskip
-  \textit{ For $ \partial_\beta   R_{i,1}$:} Applying  $ \partial_\beta $ to  \eqref{equa-R-i-1},  we obtain
$$ \partial_{\beta } R_{i, 1} =  \tilde\lambda \partial_{\beta} \left( \left(\mathscr{L}_{i, ext}^\beta \right)^{-1}\right) (R_{i,1}) + \tilde{\lambda} \left(\mathscr{L}_{i, ext}^\beta \right)^{-1}( \partial_{\beta} R_{i,1} ) +\tilde \lambda \partial_{\beta} \phi_{i,\infty,\beta}.  $$
Applying $X^{a,a'}_{y_0}$ norm to the above equality and   
Lemma   \ref{continnuity-L-i-ext--1}, we deduce 
\begin{equation*}
	\| \partial_{\beta } R_{i,1} \|_{X^{a,a'}_{y_0}}  \leq C_1 |\tilde \lambda | \|R_{i,1} \|_{X^{a,a'}_{y_0}}+C_2 |\tilde{\lambda}|\| \partial_{\beta } R_{i,1} \|_{X^{a,a'}_{y_0}} +   |\tilde{\lambda}| \| \partial_\beta \phi_{i, \infty, \beta} \|_{X^{a,a'}_{y_0}}. 
\end{equation*}
On the one hand,   we have 
$$   \|  R_{i,1} \|_{X^{a,a'}_{y_0}}   +     \| \partial_\beta \phi_{i, \infty, \beta} \|_{X^{a,a'}_{y_0}}  \le C.     $$ 
Finally,      we get 
\begin{equation*}
	\|  \partial_\beta  R_{i,1} \|_{X^{a,a'}_{y_0}} \le C |\tilde \lambda|.   
\end{equation*}
The construction and estimates on $R_{i,2}$ are very similar to those established above and are left to the reader.

\textit{- Step 2: Construction of  $R_{i,2}$: }

The construction is also based on the  Banach fixed point theorem on the Banach space $X^{a,a'}_{y_0}$    with  $ a= -\tilde{\gamma}  -2 -\alpha,$ and $ a'= 2i +2-\gamma$.    First,  we    define  the   right hand  side  of \eqref{equa-R-i-2}   to be $J(R_{i,2})$.     Using 
Lemma  \ref{continnuity-L-i-ext--1}, we have the following
\begin{eqnarray*}
	\| J(R_{2,i})\|_{X^{a,a'}_{y_0}}     \leq     C _1  \left(  \| H_1  R_{i,2}\|_{X^{a,a'}_{y_0}}     +   \| H_2\|_{X^{a,a'}_{y_0}}   \right).
\end{eqnarray*} 
We now aim to prove the following estimates 
\begin{eqnarray}
	\| H_1  R_{i,2}\|_{X^{a,a'}_{y_0}}    & \leq &  C(y_0) b^{\alpha} \| R_{i, 2}\|_{X^{a,a'}_{y_0}}  \label{estimat-Theta-1-R-i-2},\\
	\|H_2\|_{X^{a, a'}_{y_0}}   & \leq & C(y_0) b^{\alpha}. \label{estima-Theta-2}
\end{eqnarray}

- \textit{The proof of  \eqref{estimat-Theta-1-R-i-2}}: From     Lemma \ref{lemma-ground-state},   we have 
\begin{eqnarray*}
	Q_b (y)= \frac{1}{b} Q\left(  \frac{y}{\sqrt{b}}   \right) = \frac{1}{b} \left(     -\frac{1}{ \left( \frac{y}{\sqrt{b}}\right)^2}       +q_0 \left( \frac{y}{\sqrt{b}}\right)^{-\gamma}  + O_{b\to 0} \left(\frac{y}{\sqrt{b}} \right)^{-\gamma -g} \right), 
\end{eqnarray*}
$\text{ since } y \ge y_0,   \frac{y}{\sqrt{b}} \to +\infty $. This implies that for all $y \ge y_0$ and $ b \in (0,b^*(y_0))$
\begin{eqnarray*}
	\left| 2 Q_b(y)    -  \left( -\frac{2}{y^2}  + 2q_0 y^{-\gamma} b^{\frac{\alpha}{2}}\right)   \right|  & \leq &  C(y_0)  b^\alpha ,\\
	\left|     y^2 Q_b^2(y)    - \left( \frac{1}{y^2}  - 2q_0 y^{-\gamma} b^{\frac{\alpha}{2}}  \right)     \right|    & \leq &  C(y_0)  b^{\alpha},
\end{eqnarray*}
since $g = - 2\lambda_1 = 2(\gamma -2)$.  Thus, for all $ y \geq y_0$, we have 
\begin{eqnarray}
	\left|  2 Q_b(y) + Q_b^2(y) y^2 +\frac{1}{y^2} \right| \leq C(y_0)   b^{ \alpha} \lesssim b^\frac{\alpha}{2}. \label{estima-Q-b-y-2-Q-b-2-1-y-2-y-0}
\end{eqnarray}

\noindent
- \textit{Proof of  \eqref{estima-Theta-2}:}  Recall that 
$$ \|   R_{i,1}  \|_{X^{-\tilde \gamma  ,a'}_{y_0}}  \leq   C|\tilde{\lambda}|,$$  
which implies
$$ \|   R_{i,1}  \|_{X^{a  ,a'}_{y_0}}  \leq   C(y_0)|\tilde{\lambda}| \text{ with } a = -\tilde \gamma -2-\alpha.$$
Similarly, we also have 
\begin{eqnarray*}
	\| \phi_{i,\infty,\beta} \|_{X^{a  ,a'}_{y_0}} + \|\tilde{\phi}_{i,\beta} \|_{X^{a  ,a'}_{y_0}} \le C(y_0).
\end{eqnarray*}
Thus, the above estimates and \eqref{estima-Q-b-y-2-Q-b-2-1-y-2-y-0} immediately conclude \eqref{estima-Theta-2}.

Using  estimates \eqref{estimat-Theta-1-R-i-2} and \eqref{estima-Theta-2},  we get
\begin{eqnarray*}
	\| J(R_{2,i})\|_{X^{a,a'}_{y_0}}     \leq C_1(y_0) b^{\alpha}  \left(  \|  R_{i,2}\|_{X^{a,a'}_{y_0}}     +  1   \right).
\end{eqnarray*}
Consequently,  once  $b \ll 1 $, $J$ becomes a  contraction from the ball $B(C_1 b^\alpha,0)$ to itself. Thus, it follows Banach fixed point theorem the existence of $R_{i,2}$ satisfying
\begin{equation}
	\| R_{i,2}  \|_{X^{a,a'}_{y_0}} \leq C b^{\frac{\alpha}{2}}.
\end{equation}
Next, we focus on evaluating  
$\partial_{\tilde{\lambda}} R_{i,2}$, $\partial_{b} R_{i,2}$  and $\partial_{\beta} R_{i,2}$:

+  For $\partial_{\tilde{\lambda}} R_{i,2} $:  We have
\begin{eqnarray*}
	\partial_{\tilde{\lambda}} R_{i,2}= \mathscr{L}_{i,ext}^{-1} ( R_{i,2})+\mathscr{L}_{i,ext}^{-1} (H_1 \partial_{\tilde{\lambda}} R_{i,2}) + \mathscr{L}_{i,ext}^{-1}(\partial_{\tilde{\lambda}}H_2)
\end{eqnarray*}
then
\begin{eqnarray*}
	\| \partial_{\tilde{\lambda}} R_{i,2} \|_{X^{a,a'}_{y_0}} &\leq& C_1 \| R_{i,2} \|_{X^{a,a'}_{y_0}}+C_2\|H_1 \partial_{\tilde{\lambda}} R_{i,2}\|_{X^{a,a'}_{y_0}}+C_3\|\partial_{\tilde{\lambda}}H_2\|_{X^{a,a'}_{y_0}}\\
	&\leq& C_1 b^{\frac{\alpha}{2}}+C_2 b^{\frac{\alpha}{2}} \| \partial_{\tilde{\lambda}} R_{i,2} \|_{X^{a,a'}_{y_0}}+C_3 b^{\frac{\alpha}{2}}
\end{eqnarray*}
which yields to
\begin{equation*}
	\| \partial_{\tilde{\lambda}} R_{i,2} \|_{X^{a,a'}_{y_0}}  \leq C  b^{\frac{\alpha}{2}}.
\end{equation*}

- For $\partial_b R_{i,2}:$
We have
\begin{eqnarray*}
	\partial_b R_{i,2}=\mathscr{L}_{i,ext}^{-1} (\partial_b H_1 R_{i,2})+\mathscr{L}_{i,ext}^{-1} (H_1 \partial_b  R_{i,2})+\mathscr{L}_{i,ext}^{-1}(\partial_b H_2)
\end{eqnarray*}
then
\begin{eqnarray*}
	\| \partial_b R_{i,2} \|_{X^{a,a'}_{y_0}} &\leq& C_1 \|\partial_b H_1 R_{i,2} \|_{X^{a,a'}_{y_0}}+C_2\|\Theta_1 \partial_b R_{i,2}\|_{X^{a,a'}_{y_0}}+C_3\|\partial_b H_2\|_{X^{a,a'}_{y_0}}\\
	&\leq& C_1 b^{\frac{\alpha}{2}-1}+C_2 b^{\frac{\alpha}{2}} \| \partial_{\tilde{\lambda}} R_{i,2} \|_{X^{a,a'}_{y_0}}+C_3 b^{\frac{\alpha}{2}-1}\leq C b^{\frac{\alpha}{2}-1}.
\end{eqnarray*}

- For $\partial_{\beta} R_{i,2}$:  We have
\begin{eqnarray*}
	\partial_{\beta} R_{i,2}=\partial_{\beta}(\mathscr{L}_{i,ext}^{-1} ( \Theta_1 R_{i,2}))+\mathscr{L}_{i,ext}^{-1} ( \Theta_1 \partial_{\beta} R_{i,2})+\partial_{\beta}(\mathscr{L}_{i,ext}^{-1} ( \Theta_2)),
\end{eqnarray*}
then we use Lemma \ref{continnuity-L-i-ext--1} to get
\begin{eqnarray*}
	\| \partial_{\beta} R_{i,2} \|_{X^{a,a'}_{y_0}} \leq C (y_0)\left( b^\frac{\alpha}{2} \|R_{i,2}\|_{X^{a,a'}_{y_0}} + b^\frac{\alpha}{1} \|\partial_{\beta} R_{i,2}\|_{X^{a,a'}_{y_0}} + b^\alpha  \right) \le C b^\frac{\alpha}{2}  .
\end{eqnarray*}
Finally,  we conclude the proof of the Proposition.
\end{proof}
In the  sequel, we aim to   complete the results  used in the proof of Proposition \ref{propo-outer-eigenfunctions}. To be begin with, we  need the following result on the resonance  of $\mathscr{L}^\beta_{i,ext}$. For sake of shortness we set

\begin{equation}\label{mathscr-L-ext}
\mathscr{L}_{i,ext}^\beta u=  \left(  \mathscr{L}_{\infty}^{\beta}   -   2\beta\left(    \frac{\alpha}{2} - i\right) \right) u.
\end{equation}
We have 
\begin{lemma}[Resonance of $\mathscr{L}_{i,ext}^\beta $]\label{lemma-phi-1-psi-2-outer}
We consider $i \in \N$, then, there exists   $\tilde {\psi}_{i,\beta} $ such   that  it solves $ \mathscr{L}_{i, ext}^\beta \tilde {\psi}_{i,\beta}= 0$. Moreover,  we have 
$$  \text{Ker}\hspace{0.04cm}(\mathscr{L}_{i,ext}^\beta)  = \text{Span}\{ \phi_{i,\infty,\beta},  \tilde \psi_{i, \beta} \}, $$
where $ \phi_{i,\infty, \beta}$ is the i-th  eigenfunction of 
$\mathscr{L}_\infty^\beta$, given   in Proposition \ref{proposition-spectral-L-infty}; and     $\tilde \psi_{i,\beta}$ has the following asymptotic:
\begin{equation}\label{asymptotic-tilde-phi}
	\tilde{\psi}_{i,\beta}(y) = \left\{  \begin{array}{rcl}
		& & \frac{y^{\gamma -d}}{a_{i, 0}(d-2\gamma)} ( 1 + O(y^2)) \text{ as } y \to 0, \\[0.2cm]
		& &  -\frac{2}{a_{i,i}(2\beta)^{i}} y^{-2i + \gamma -(d+2)} e^{2\beta\frac{y^2}{4}} \left[ 1+ O( y^{-2}) \right]  \text{ as }  y \to +\infty
	\end{array}
	\right.
\end{equation}
where  $  \gamma  $ and $a_{i,j}$ were  defined in    \eqref{defi-gamma-intro} and \eqref{defi-a-i-j-intro} .
\iffalse
In particular,  $ \left(\mathscr{L}_{i,ext}^{\beta}\right)^{-1  } $ is explicitly formulated 
\begin{eqnarray}
	\left(\mathscr{L}_{i,ext}^{\beta}\right)^{-1  }  f(y)  &=& - \phi_{i,\infty,\beta} \int_1^y   f(y' )\frac{\tilde{ \psi}_{i,\beta}(y')}{W (y')} dy' +\tilde  \psi_{i,\beta} \int_1^y  f(y') \frac{\phi_{i,\infty,\beta}(y')}{W (y')} dy',\label{ inverse-mathscrL-ext } \\
	&+ &   y_0 \phi_{i,\infty,\beta}  + y_1 \tilde \psi_{i,\beta},  \nonumber
\end{eqnarray}
where $W$ is the Wronskian corresponding to $\mathscr{L}_{i,ext}^{\beta}$, explicitly given by
$$  W(y) = y^{-(d+1)} e^{2\beta \frac{y^2}{4}} .$$
\fi 
In particular, there exists  a  solution $\tilde \phi_{i,\beta}$ to  $\mathscr{L}_{i,ext}^{\beta} \tilde \phi_{i}  = \phi_{i,\infty,\beta}$, satisfying the following asymptotic
\begin{equation}\label{asimptotics}
	\tilde \phi_{i,\beta}(y) = \left\{ \begin{array}{rcl}
		&& K_0  y^{-\tilde \gamma}  (1 + O(y^2)  )  \text{ as } y \to 0, \\
		& &  K_{\infty} y^{2i -\gamma} ( \ln y + O(1) )  \text{ as } y  \to +\infty,
	\end{array}
	\right.     
\end{equation}
where  $ K_0 = \frac{1}{(2\gamma-d)(d+2 -2\gamma)} , K_\infty =2 (2\beta)^i$ and $\tilde \gamma$ is defined  by
\begin{equation}\label{defi-gamma-tilde}
	\tilde  \gamma =  \frac{1}{2} ( d + \sqrt{d^2 -12 d+ 24} ).
\end{equation}
\end{lemma}

\begin{proof}
Recall  that   $ \phi_{i, \infty, \beta}$ solves  $ \mathscr{L}_{i,ext}^\beta \phi_{i,\infty, \beta}=0$,   this follows from the fact that $\phi_{i,\infty}$ is  the i-th eigenfunction of $\mathscr{L}_\infty$. According to  $\mathscr{L}_{i,ext}^\beta $,  the  
Wronskian is given by
\begin{equation}\label{defi-wronski-determinant}
	W(y) = y^{-(d+1)} e^{2\beta \frac{y^2}{4}}.
\end{equation}
\iffalse
Indeed, we write $ \mathscr{L}_{i,ext}^\beta f=0$ as follows
\begin{equation}
	f''(\xi)=a(\xi)f'(\xi)++b(\xi)f(\xi)
\end{equation}
where $a(\xi)=-\frac{d+1}{\xi}+\beta\xi$ and $b(\xi)=2\beta-\frac{3(d-2)}{\xi^2}+2\beta(\frac{\alpha}{2}-i)$. The Wronskian is given by
\begin{equation}\label{}
	W(\xi)=e^{\int^{\xi} a(\xi') d\xi'}=\xi^{-(d+1)} e^{2\beta \frac{\xi^2}{4}}
\end{equation}
\fi 
Then, we can formulate   an   independent linear  solution $ \tilde{\psi}_{i,\beta}$  to $ \mathscr{L}_{i,ext}^\beta   \tilde{\psi}_{i,\beta} =0$     
\begin{equation}\label{formulate-tilde-phi-i-beta}
	\tilde  \psi_{i,\beta}  (y) =  - \phi_{i,\infty,\beta} \int_1^y \frac{(y')^{-(d+1)} e^{2\beta\frac{(y')^2}{4}}}{ \phi^2_{i,\infty,\beta}(y')} dy'. 
\end{equation}
From \eqref{phi-i-infty}, we derive 
\begin{equation}\label{asymptotic-phi-i-infty-beta}
	\phi_{i, \infty,\beta} (y)= \left\{ \begin{array}{rcl} & &
		a_{i, 0} y^{-\gamma} \left( 1 + O(y^2) \right) \text{ as } y \to 0,\\[0.2cm]
		& & a_{i,i} (2\beta)^{i} y^{2i -\gamma}  \left( 1 + O(y^{-2}) \right)   \text{ as } y \to +\infty,
	\end{array}
	\right.
\end{equation}
where $a_{i,j}$'s   general formula given in \eqref{defi-a-i-j-intro}, and we plug the above fact into \eqref{formulate-tilde-phi-i-beta} to get  
\begin{eqnarray}
	\tilde \psi_i (y)   = \left\{   \begin{array}{rcl}
		& & \frac{y^{\gamma - d}}{a_{i, 0} (2\gamma -d) } \left(1  + O(y^2) \right)  \text{ as } \xi  \to 0 ,\\[0.2cm]
		& & -\frac{2}{a_{i,i} (2\beta)^i}  y^{-2i + \gamma - d-2} e^{2\beta\frac{y^2}{4}} \left( 1 + O(y^{-2}) \right) \text{ as } \xi \to +\infty.
	\end{array}    \right. \label{asymptotic-tilde-psi-i-beta}
\end{eqnarray}
In particular, it is easy to see that 
$$ \text{ Ker} (\mathscr{L}^\beta_{i,ext}) = \text{Span}\{ \phi_{i,\infty, \beta}, \tilde \psi_{i,\beta} \}, $$
which leads to 
$$ \left(\mathscr{L}_{i,ext}^{\beta}\right)^{-1} f(y) = - \phi_{i,\infty, \beta} \int_1^y   f(y' )\frac{\tilde{ \psi}_{i,\beta}(y')}{W (y')} d\xi' + \tilde  \psi_{i,\beta} \int_y^{+\infty} f(y') \frac{\phi_{i,\beta,\infty}(y')}{W (y')} dy'  + c_1 \phi_{i,\beta,\infty} + c_2 \tilde \psi_{i,\beta},$$
and since we need to construct a special solution with explicit asymptotics, we choose $c_1=c_2=0$, then, we can omit the generality and write 
\begin{equation}\label{inverse-mathcam-i-ext}
	\left(\mathscr{L}_{i,ext}^{\beta}\right)^{-1} f(y) = - \phi_{i,\infty, \beta} \int_1^y   f(y' )\frac{\tilde{ \psi}_{i,\beta}(y')}{W (y')} dy' + \tilde  \psi_{i,\beta} \int_y^{+\infty} f(y) \frac{\phi_{i,\beta,\infty}(y')}{W (y')} dy'.
\end{equation}
Thus, the solution $\tilde \phi_{i,\beta}$  to $ \mathscr{L}^\beta_{i,ext} \tilde \phi_{i, \beta} = \phi_{i, \infty,\beta} $ can be written 
$$ \tilde \phi_{i,\beta} = \left(\mathscr{L}_{i,ext}^{\beta}\right)^{-1}(\phi_{i,\infty, \beta}). $$

\noindent 
\medskip
- \textit{Behavior at $0$}: From \eqref{asymptotic-phi-i-infty-beta} and \eqref{asymptotic-tilde-psi-i-beta}, we have 
\begin{eqnarray*}
	-\phi_{i, \infty, \beta} \int_{1}^{y} \frac{\phi_{i,\infty, \beta} (\xi')\tilde \psi_{i,\infty,\beta}(\xi')}{W(\xi')} d\xi'=\frac{a_{i,0}}{2(2\gamma-d)}y^{-\gamma+2}(1+O(\xi^2))\text{ as } y \to 0
\end{eqnarray*}
and also
\begin{eqnarray*}
	\tilde  \psi_{i,\beta} \int_y^\infty  \frac{\phi_{i,\infty, \beta}^2(\xi')}{W (\xi')} d\xi'=\frac{a_{i,i}}{(2\gamma-d)(-2\gamma+d+2)} y^{\gamma-d}(1+O(\xi^2))\text{ as } y  \to 0,
\end{eqnarray*}
noting that $\gamma-d=-\frac{1}{2}(d+\sqrt{d^2-12d+24})=-\tilde{\gamma}$, we get
$$ \tilde \phi_{i,\beta}(y)  =  K_0 y^{-\tilde \gamma} (1 + y^2 ) \text{ as } y  \to 0,    $$
where $K_0 = \frac{1}{(2\gamma - d)(d+2-2\gamma)} $.

\noindent
\medskip
- \textit{Behavior at $+\infty$}: Using  \eqref{asymptotic-phi-i-infty-beta} and \eqref{asymptotic-tilde-psi-i-beta} again, we obtain 
\begin{eqnarray*}
	-\phi_{i, \infty, \beta} \int_{1}^{y} \frac{\phi_{i,\infty, \beta} (\xi')\tilde \psi_{i,\infty,\beta}(\xi')}{W(\xi')} d\xi'= 2a_{i,i}(2\beta)^i y^{2i-\gamma} \left( \ln y +O(1)\right) \text{ as } y \to +\infty,  
\end{eqnarray*}
and for the second one
\begin{eqnarray*}
	\tilde  \psi_{i,\beta} \int_y^\infty  \frac{\phi_{i,\infty, \beta}^2(\xi')}{W (\xi')} d\xi' =  -\frac{2}{a_{i,i}(2\beta)^i} y^{-2i +\gamma-d-2}  e^{\frac{2\beta y^2}{4}} O( y^{4i -2\gamma +d} e^{\frac{-2\beta y^2}{4}}) \text{ as } \xi \to +\infty. 
\end{eqnarray*}
Thus, we have
$$ \tilde \phi_{i,\beta}(y) = 2 (2\beta)^i y^{2i-\gamma} \left( \ln \xi +O(1) \right) \text{ as} \xi \to +\infty,$$
and this concludes the proof of the Lemma. \end{proof}
It remains only to prove the continuity of $ \mathscr{L}_{i,ext}^{-1}$ in  the Banach space $X^{a,a'}$, a result that we have used above.
%In the below, we aim to prove  which we mainly used to get the conclusion of Proposition \ref{propo-outer-eigenfunctions}
\begin{lemma}[Continuity of $ \mathscr{L}_{i,ext}^{-1}$]\label{continnuity-L-i-ext--1}
For all  $ a \leq {\gamma}-d$,   $a \neq  -d-2$ and $ a' > 2i - \gamma$ with $ y_0 \in (0,1)$ and $\beta \in \left( \frac{1}{4}, \frac{3}{4} \right) $, we  have the following estimate
$$  \left\|    \left( \mathscr{L}_{i, ext}^{\beta} \right)^{-1}    f    \right\|_{X^{a,a'}_{y_0}}  \leq     C(a,a',\beta) \| f \|_{X^{a,a'}_{y_0}} \text{ and }  \left\|    \partial_\beta \left( \mathscr{L}_{i,  ext}^{\beta} \right)^{-1}    f    \right\|_{X^{a,a'}_{y_0}}  \leq     C(a,a',\beta) \| f \|_{X^{a,a'}_{y_0}}. $$ 
\end{lemma}
\begin{proof}
Define $ g = \left( \mathscr{L}_{i, ext}^{\beta} \right)^{-1} f $ and recall from  \eqref{inverse-mathcam-i-ext} that
\begin{eqnarray*}
	g(y) = \left(\mathscr{L}_{i,ext}^{\beta}\right)^{-1} f(y) = - \phi_{i,\infty, \beta} \int_1^y  f(\xi' )\frac{\tilde{ \psi}_{i,\beta}(\xi')}{W (\xi')} d\xi' + \tilde  \psi_{i,\beta} \int_y^{+\infty} f(\xi) \frac{\phi_{i,\beta,\infty}(\xi')}{W (\xi')} d\xi'.
	%,  W \text{ defined as in } \eqref{defi-wronski-determinant} , 
\end{eqnarray*}
Then 

\begin{eqnarray*}
	\partial_y g(y)&=&-\partial_y\phi_{i,\infty, \beta} \int_1^y  f(\xi' )\frac{\tilde{ \psi}_{i,\beta}(\xi')}{W (\xi')} d\xi' + \partial_y \tilde  \psi_{i,\beta} \int_y^{+\infty} f(\xi) \frac{\phi_{i,\beta,\infty}(\xi')}{W (\xi')} d\xi' \phi_{i,\infty, \beta}\\
	& &-2\phi_{i,\infty, \beta}f(y )\frac{\tilde{ \psi}_{i,\beta}(y)}{W (y)}
\end{eqnarray*}
and
\begin{eqnarray*}
	\partial_y^2 g(y)&=&-\partial_y^2\phi_{i,\infty, \beta} \int_1^y  f(\xi' )\frac{\tilde{ \psi}_{i,\beta}(\xi')}{W (\xi')} d\xi' + \partial_y^2 \tilde  \psi_{i,\beta} \int_y^{+\infty} f(\xi) \frac{\phi_{i,\beta,\infty}(\xi')}{W (\xi')} d\xi' \phi_{i,\infty, \beta}\\
	& &-3\partial_y\phi_{i,\infty, \beta}f(y )\frac{\tilde{ \psi}_{i,\beta}(y)}{W (y)}-3\partial_y \tilde  \psi_{i,\beta}f(y) \frac{\phi_{i,\beta,\infty}(y)}{W (y)}\\
	& &-2\phi_{i,\infty, \beta}\partial_yf(y )\frac{\tilde{ \psi}_{i,\beta}(y)}{W (y)}+2 \phi_{i,\infty, \beta}f(y ) \frac{\tilde{ \psi}_{i,\beta}(y)\partial_y W(y)}{W^2(y)}.
\end{eqnarray*}

In order to establish the desired estimate, we will only need to control higher order derivatives, namely, $y^{2-a }\partial_y^2 g(y)$ for $y \in [y_0, 1]$ and $y^{2-a' }\partial_y^2 g(y)$ for $y \in [1, \infty)$.\\
\iffalse
For $y \in [y_0, 1],$ we have
\begin{eqnarray*}
	\left| y^{-a } g(y)  \right| &\le & y^{-a} |\phi_{i,\infty, \beta}| \int^{1}_y (\xi')^{-a}|f(\xi')|\frac{ | (\xi')^a| | \tilde \psi_{i,\beta}(\xi')|}{W(\xi')} d\xi'\\ 
	&+& y^{-a}|\tilde{ \psi}_{i,\beta}| \int_y^{1} (\xi')^{-a}|f(\xi')| \frac{|(\xi')^{a}||\phi_{i,\beta,\infty}(y)|}{W(\xi')} d\xi'\\
	& +& y^{-a}|\tilde{ \psi}_{i,\beta}| \int_{1}^\infty (\xi')^{- a'}|f(\xi')| \frac{|(\xi')^{ a'}||\phi_{i,\beta,\infty}(y)|}{W(\xi')} d\xi'\\
	& \lesssim &  \sup_{y \in [y_0,1]} |y^{-a} f(y)|+ \sup_{y \in [1,+\infty)}  |y^{-a'} f(y)|,
\end{eqnarray*}
provided that $ a \le -\tilde \gamma$ and $ a' \ge  2i -\gamma $
\fi
\begin{itemize}
	\item \underline{$y \in [y_0, 1]$}, we have
	\begin{eqnarray*}
		\left| y^{2-a } \partial_y^2 g(y)  \right| &\lesssim& y^{2-a } |\partial_y^2\phi_{i,\infty, \beta}| \int^{1}_y (\xi')^{-a}|f(\xi')|\frac{ | (\xi')^a| | \tilde \psi_{i,\beta}(\xi')|}{W(\xi')} d\xi'\\
		&+& y^{2-a}|\partial_y^2\tilde{ \psi}_{i,\beta}| \left(\int_y^{1} (\xi')^{-a}|f(\xi')| \frac{|(\xi')^{a}||\phi_{i,\beta,\infty}(y)|}{W(\xi')} d\xi'+\int_1^{\infty}(\xi')^{-a'}|f(\xi')| \frac{|(\xi')^{a'}||\phi_{i,\beta,\infty}(y)|}{W(\xi')} d\xi'\right)\\
		&+& y^{2-a } |f(y)|\left(|\partial_y\phi_{i,\infty, \beta}|\frac{|\tilde{ \psi}_{i,\beta}(y)|}{W (y)}+ |\partial_y\tilde  \psi_{i,\beta}| \frac{|\phi_{i,\beta,\infty}(y)|}{W (y)}+|\phi_{i,\infty, \beta}| \frac{|\tilde{ \psi}_{i,\beta}(y)\partial_y W(y)|}{W^2(y)}\right)\\
		&+&y^{1-a } |\partial_y f(y)|\left(y|\phi_{i,\infty, \beta}|\frac{|\tilde{ \psi}_{i,\beta}(y)|}{W (y)}\right)\\
		& \lesssim &  \sup_{y \in [y_0,1]} |y^{-a} f(y)|+\sup_{y \in [y_0,1]} |y^{1-a} \partial_y f(y)|+ \sup_{y \in [1,+\infty)}  |y^{-a'} f(y)|\lesssim \| f \|_{X^{a,a'}_{y_0}},
	\end{eqnarray*}
	provided that $ a \le  \gamma-d$ and $ a' \ge  2i -\gamma $.\\
	\item \underline{$y\in [1,+\infty)$}, we have
	\begin{eqnarray*}
		\left| y^{2-a '} \partial_y^2 g(y)  \right|
		&\lesssim& y^{2-a '} |\partial_y^2\phi_{i,\infty, \beta}| \int_1^{\infty} (\xi')^{-a'}|f(\xi')|\frac{ | (\xi')^{a'}| | \tilde \psi_{i,\beta}(\xi')|}{W(\xi')} d\xi'\\
		&+& y^{2-a'}|\partial_y^2\tilde{ \psi}_{i,\beta}| \int_{1}^{\infty} (\xi')^{-a'}|f(\xi')| \frac{|(\xi')^{a'}||\phi_{i,\beta,\infty}(y)|}{W(\xi')} d\xi'\\
		&+& y^{-a' } |f(y)|\left(y^2|\partial_y\phi_{i,\infty, \beta}|\frac{|\tilde{ \psi}_{i,\beta}(y)|}{W (y)}+y^2 |\partial_y\tilde  \psi_{i,\beta}| \frac{|\phi_{i,\beta,\infty}(y)|}{W (y)}+y^2|\phi_{i,\infty, \beta}| \frac{|\tilde{ \psi}_{i,\beta}(y)\partial_y W(y)|}{W^2(y)}\right)\\
		&+&y^{1-a' } |\partial_y f(y)|\left(y|\phi_{i,\infty, \beta}|\frac{|\tilde{ \psi}_{i,\beta}(y)|}{W (y)}\right)\\
		&\lesssim& \sup_{y \in [1,+\infty)}  |y^{-a'} f(y)|+\sup_{y \in [1,+\infty)}  |y^{1-a'} \partial_y f(y)|\lesssim \| f \|_{X^{a,a'}_{y_0}}
	\end{eqnarray*}
\end{itemize}
which yields
$$\|    g    \|_{X^{a,a'}_{y_0}} = \|    \left( \mathscr{L}_{i, ext}^{\beta} \right)^{-1}    f    \|_{X^{a,a'}_{y_0}}  \leq     C(a,a') \| f \|_{X^{a,a'}_{y_0}} $$
as claimed. Finally, we finish the proof of the Lemma. 
\end{proof}

\newpage

\bibliographystyle{alpha}
\bibliography{mybib}

\end{document}